%% file: Mapping_class_group_extension.tex
\documentclass[11pt]{amsart}

\usepackage{amscd,amssymb,amsmath,amsfonts,amsthm}
\usepackage{thmtools}
\usepackage{graphicx,import}
\usepackage{subcaption}
\usepackage[top=2.5cm, bottom=2.5cm, left=2cm, right=2cm]{geometry}
\usepackage[colorlinks=true]{hyperref}
\usepackage[nameinlink]{cleveref}

\makeatletter
\AddToHook{cmd/appendix/before}{\def\cref@section@alias{appendix}}
\makeatother

\makeatletter
\AddToHook{cmd/appendix/before}{\def\cref@subsection@alias{appendix}}
\makeatother

\usepackage{import}
\usepackage{csquotes} 
\usepackage[backend=biber,style=alphabetic,giveninits=true,maxbibnames=99]{biblatex}
\renewbibmacro{in:}{}
\DeclareFieldFormat{pages}{#1}
\AtEveryBibitem{%
  \clearfield{pagetotal}
  \clearlist{location}
}
\usepackage{tikz}
\usetikzlibrary{arrows.meta}
\usepackage{array}
\usepackage[all]{xy}
\usepackage{tikz-cd}
\usepackage{float}
\usepackage{enumitem}

\usepackage{xcolor}

\definecolor{myred}{RGB}{255,100,100}

\DeclareMathOperator{\GL}{GL}
\DeclareMathOperator{\PGL}{PGL}
\DeclareMathOperator{\id}{id}

\DeclareMathOperator{\Aut}{Aut}

\DeclareMathOperator{\End}{End}

\DeclareMathOperator{\ad}{ad}
\DeclareMathOperator{\coad}{coad}
\DeclareMathOperator{\inv}{inv}

\newcounter{intro}
\newtheorem{theoremeintro}[intro]{Theorem}

\newtheorem{theoreme}{Theorem}[section]
\newtheorem{proposition}[theoreme]{Proposition}
\newtheorem{lemme}[theoreme]{Lemma}
\newtheorem{corollaire}[theoreme]{Corollary}

\newtheorem*{theoreme*}{Theorem}

\theoremstyle{definition}

\newtheorem{remarqueintro}[intro]{Remark}

\newtheorem{definition}[theoreme]{Definition}

\newtheorem{remarque}[theoreme]{Remark}

\newtheorem*{notation*}{Notation}
\Crefname{theoremeintro}{Thm.}{Thm.}
\Crefname{corollaireintro}{Cor.}{Cor.}
\Crefname{propositionintro}{Prop.}{Prop.}
\Crefname{remarqueintro}{Rem.}{Rem.}

\Crefname{theoreme}{Thm.}{Thm.}
\Crefname{proposition}{Prop.}{Prop.}
\Crefname{lemme}{Lem.}{Lem.}
\Crefname{corollaire}{Cor.}{Cor.}

\Crefname{definition}{Def.}{Def.}
\Crefname{notation}{Not.}{Not.}

\Crefname{remarque}{Rem.}{Rem.}

\Crefname{subsection}{Subsection}{Subsection}
\Crefname{subsubsection}{§}{§}

\creflabelformat{enumi}{#2(#1)#3}
\counterwithin{figure}{section}

\allowdisplaybreaks

\title{Central extensions of mapping class groups of surfaces from stated skein algebras}

\author{Joris Moulai}
\address{IMAG, Univ. Montpellier, CNRS, Montpellier, France}
\email{joris.moulai@umontpellier.fr}

\begin{document}

\begin{abstract}
	Let $\Sigma$ be a surface of genus $g$ with zero or one boundary component and $n$ marked points, and $H$ a finite-dimensional factorizable ribbon Hopf algebra. 
	We compute the central extension of the mapping class group of $\Sigma$, associated with a linearization of the projective representation constructed from the stated skein algebra of $\Sigma$ and $H$. 
	Our proof is purely two-dimensional, and makes no use of TQFT arguments.
\end{abstract}


\maketitle

\tableofcontents

\section*{Introduction}\label{sec-introduction}

Let $H$ be a finite-dimensional factorizable ribbon Hopf algebra, $\Sigma$ a compact oriented surface possibly with a finite set of marked points, and $\Gamma$ its mapping class group. 
The goal of this article is to identify the cohomology class associated with a minimal central extension $\widetilde{\Gamma}$ of $\Gamma$ on which we can define a linearization $\tilde{\rho}$ of the projective representation $\rho$ of $\Gamma$ constructed from the stated skein algebra of $\Sigma$ and $H$. 

\smallskip

This is a classical problem in quantum topology. 
It is a general feature of topological quantum field theories (TQFTs) to give rise to functors with anomalies (see \cite{turaev_quantum_2010} and \cite{kerler_non-semisimple_2001}), that is, the image of the gluing of two cobordisms by the functor depends on a gluing data, e.g. lagrangian subspaces.
In particular, TQFTs provide so-called \enquote{quantum} projective representations of the mapping class groups, hence associated central extensions which are characterized by cohomology classes in $H^2(\Gamma, A)$, where $A$ is an abelian group.
For the Reshetikhin–Turaev TQFT, the cohomology classes of the minimal central extensions have been computed for closed surfaces in \cite{masbaum_central_1995} and are equal to 12 times a generator of $H^2(\Gamma, \mathbb{Z})$.
The topological description of this central extension in terms of decorated mapping cylinders was further studied in \cite{gilmer_maslov_2011}.
For the non semisimple TQFTs defined in \cite{de_renzi_renormalized_2018} and \cite{de_renzi_3-dimensional_2022}, the cohomology classes of the corresponding central extensions must be the same since the functor anomaly is solved in a formally similar way by means of the signature of bounding 4-manifolds. 

\smallskip

The computation of the anomalies of the aforementioned quantum representations of surface mapping class groups has relied on the three-dimensional constructions of the corresponding TQFTs.
However, these quantum representations can be defined directly in terms of skein algebras of surfaces (see e.g. \cite{blanchet_topological_1995}, \cite{marche_introduction_2021}, and \cite{de_renzi_3-dimensional_2022}), including surfaces with a finite set of marked points, so that it is a natural problem to solve their anomalies for all finite type surfaces by a purely two-dimensional, skein theoretic, approach.

\smallskip

In another direction, the above skein algebras are associated with modular categories, which by definition satisfy some finiteness assumptions and thus cover only a portion of what skein theory can produce. 
Indeed, skein algebras can be associated with non-modular categories like those of finite dimensional representations of unrestricted quantum groups, which have been the subject of an intense activity these last years, motivated by the construction of generalizations of the above TQFTs to cobordisms endowed with holonomy representations of the fundamental group into a Lie group (see e.g. \cite{jordan_quantum_2025}, \cite{baseilhac_noetherian_2025}, and \cite{baseilhac_structure_2026}). 
The description of the anomaly of the associated quantum actions of mapping class groups is still an open problem, even in the most basic (and most fundamental) situation, namely, for the (quantum Teichmüller algebras or the) Kauffman bracket skein algebras of punctured surfaces at roots of unity. 
As a consequence, the quantum invariants of surface diffeomorphisms that these algebras define are for now only defined up to some ambiguity (\cite{baseilhac_quantum_2018}, \cite{bonahon_representations_2007}, \cite{bonahon_asymptotics_2021}, and \cite{ishibashi_cyclic_2026}). 
Only the related framework of infinite-dimensional quantum Teichmüller theory has been considered, where the central extensions of the mapping class groups of punctured surfaces have been described by Funar-Kashaev in \cite{funar_centrally_2014}.

\smallskip

Let us put our results (stated above) into this context. 
Let $\Sigma = \Sigma_{g,n}^{s, \bullet}$ be the compact oriented surface of genus $g \geq 1$ with $n \geq 0$ marked points and $s \in \{0, 1 \}$ boundary component, and let $H$ be a finite-dimensional factorizable ribbon Hopf algebra over the algebraically closed field $\mathbb{K}$. 
We consider the stated skein algebra $\mathcal{S}_H^{\mathrm{st}}(\Sigma_{g,n}^{1, \bullet})$ introduced in \cite{baseilhac_noetherian_2025}. 
It is isomorphic to a quantum graph algebra $\mathcal{L}_{g, n}(H)$ (\cite[Thm. 6.5]{baseilhac_noetherian_2025}), which provides very convenient tools for the representation theory of $\mathcal{S}_H^{\mathrm{st}}(\Sigma_{g,n}^{1, \bullet})$: the factorizability hypothesis of $H$ guarantees that we have an isomorphism $\mathcal{L}_{g, n}(H) \simeq \End_\mathbb{K}((H^*)^{\otimes g}) \otimes H^{\otimes n}$, and $\mathcal{L}_{g, n}(H)$ has a natural induced representation on $(H^*)^{\otimes (g+n)}$.
That representation allows \enquote{simple} diagrammatic computations, and we will see that it gives rise to a projective representation of the mapping class group $\Gamma_{g,n}^s$ of $\Sigma_{g,n}^{s, \bullet}$, for $s \in \{0,1\}$. 
It is the anomaly (and the related central extension of $\Gamma_{g,n}^s$) of this quantum representation that we compute. 
The isomorphism $\mathcal{L}_{g, n}(H) \simeq \mathcal{S}_H^{\mathrm{st}}(\Sigma_{g,n}^{1, \bullet})$ immediately provides a skein theoretic formulation of our results. 
In the case $n=0$, the quantum representation coincides with the one defined by the Kerler-Lyubashenko TQFTs for $H$-$\mathrm{mod}$ (\cite[Sec. 6]{faitg_holonomy_2024}), so this solves the first \enquote{natural problem} discussed above (see \Cref{remintro-lien_avec_Lyubashenko} for more details).

\smallskip

In view of generalizations of our method and results to infinite-dimensional skein algebras (e.g. the Kauffman bracket algebra at roots of unity), let us mention that small quantum groups $u_\epsilon(\mathfrak{g})$ of complex semisimple Lie algebras $\mathfrak{g}$ at roots of unity $\epsilon$ of odd order are examples of finite-dimensional factorizable ribbon Hopf algebras (\cite{lyubashenko_invariants_1995}), and there is a quantum graph algebra $\mathcal{L}_{g, n}^\epsilon(\mathfrak{g})$, which is a central extension of $\mathcal{L}_{g,n}(u_\epsilon(\mathfrak{g}))$ and satisfies an exact sequence of algebras (see \cite{baseilhac_structure_2026}):

\begin{equation}\label{suite_exacte_generalisation}
	0 \longrightarrow \mathcal{O}^+(G)^{\otimes (2g+n)} \xrightarrow{~ (\mathbb{F}\mathrm{r}_\epsilon^{*})^{\otimes (2g+n)} ~} \mathcal{L}_{g,n}^{\epsilon} \xrightarrow{~ \pi^{\otimes (2g+n)} ~} \mathcal{L}_{g,n}(u_\epsilon(\mathfrak{g})) \longrightarrow 1.
\end{equation}
where $G$ is the simply-connected Lie group with Lie algebra $\mathfrak{g}$, and $\mathcal{O}^+(G)$ is the augmentation ideal of $\mathcal{O}(G)$, which consists of the functions $G \to \mathbb{C}$ vanishing at the unit element of $G$.
A natural strategy is to lift (with a proper meaning) all our computations to linear mapping class group actions defined from $\mathcal{L}_{g, n}^\epsilon(\mathfrak{g})$.

\smallskip 

Our approach to compute the anomalies can be explained as follows. 
Given a projective representation $\rho : G \to \PGL(V)$, the obstruction to find a linearization of $\rho$, i.e. a linear representation $\tilde{\rho} : G \to \GL(V)$ such that $\rho = \pi \circ \tilde{\rho}$, where $\pi : \GL(V) \to \PGL(V)$ is the projection, is entirely characterized by the cohomology class of the central extension $G^\rho := \{ (g,M) \in G \times \GL(V) ~ \vert ~ \rho(g) = \pi(M) \}$.
In the case where the extension $G^\rho$ is non-trivial, one may wonder to what extent the projective representation $\rho$ is not linearizable. 
The task is therefore to find an explicit minimal central extension of $G$ on which one can define a linear representation that linearizes $\rho$. 
We find such a central extension by following a method already considered by Funar-Kashaev (\cite{funar_centrally_2014}) in the context of infinite-dimensional quantum Teichmüller theory, which we now recall.

\smallskip

Let $\langle F \vert R \rangle$ be a presentation of the group $G$. 
We can always construct a linear representation $\tilde{\rho} : F \to \GL(V)$ such that the following diagram commutes:     
\[  
\begin{tikzcd}
    F \arrow[r, "\tilde{\rho}"] \arrow[d, "\pi"] & \GL(V) \arrow[d, "\pi"] \\
    \langle F \vert R \rangle \arrow[r, "\rho"] &  \PGL(V)
\end{tikzcd}
\] 
Denote by $f_i, i \in I$, the generators of $F$.
Since $\rho$ is a projective representation, the image by $\tilde{\rho}$ of a relator $r = \prod_{1}^{n}f_{i_k}^{\varepsilon_{i_k}}$ in $R$ is a scalar multiple of the identity map of $V$, say $\alpha(r) \id_V$. 
If $\alpha(r)$ is not equal to $1$, if possible we renormalize the image by $\tilde{\rho}$ of the generators $f_{i_k}$ in order to have $\alpha(r) = 1$, and this preserves the commutative diagram. 
We do this for as many relators $r$ in $R$ as possible.
Denote by $E$ the quotient of $F$ by the normal subgroup generated by these relators. 
Using the remaining relators, we construct the group $A$ by which $G$ will be extended. 
This gives us a minimal extension $0 \to A \to E \to G \to 1$ on which we can linearize the projective representation $\rho$. 

\subsection*{Main results and organization of the paper}
Let $\Sigma_{g,n}^s$ be a compact oriented surface of genus $g$ with $s \in \{ 0, 1 \}$ boundary component and $n$ marked points, and $\Gamma_{g,n}^s$ its mapping class group.
Let $H$ be a finite-dimensional factorizable ribbon Hopf algebra.

\smallskip

In \Cref{sec-groupe_modulaire}, we recall the key results on $\Gamma_{g,n}^s$ that we will need: the Gervais presentation of $\Gamma_{g,n}^s$ from \cite{harer_second_1983}, \cite{gervais_presentation_1996}, and \cite{funar_centrally_2014}, the Harer presentation of the universal central extension $\widetilde{\Gamma}_{g}^{s,\mathrm{univ}}$  of $\Gamma_g^s$, and basic facts on $H^2(\Gamma_{g,n}^s, \mathbb{Z})$.

\smallskip

In \Cref{sec-Lgn_et_Alekseev}, we recall the definition of the graph algebras $\mathcal{L}_{g, n}(H)$, the Alekseev morphism $\Phi_{g,n}$, and some facts about representations of $\mathcal{L}_{g, n}(H)$, including one denoted by $\mathcal{R}_{g,n}^1(V)$.

\smallskip

In \Cref{sec-construction_representations_projectives}, using $\mathcal{R}_{g,n}^1(V)$ and the Gervais presentation of $\Gamma_{g,n}^1$, we construct a projective representation $\rho_{g,n}^1 : \Gamma_{g,n}^1 \to \PGL(V_{g,n}^1)$, where $V_{g,n}^1 := (H^*)^{\otimes g} \otimes V_1 \otimes \cdots \otimes V_n$ with $\{ V_i \}_{i = 1, \cdots, n}$ a collection of simple $H$-modules. 
For the surface $\Sigma_{g,n}$ with no boundary component ($s = 0$), the capping morphism (\cite[Prop. 3.19]{farb_primer_2012}) and the Birman exact sequence (\cite[Thm. 4.6]{farb_primer_2012}) allow us to restrict $\rho_{g,n}^1$ to a projective representation $\rho_{g,n} := \rho_{g,n}^0 : \Gamma_{g,n} \to \PGL(V_{g,n})$, where $V_{g,n} := V_{g,n}^0 \subset V_{g,n}^1$ is a submodule of invariant elements for an action of $H$, on which a subalgebra $\mathcal{L}_{g, n}^{\mathrm{inv}}(H) \subset \mathcal{L}_{g, n}(H)$ of invariant elements under the right coadjoint action of $H$ acts, via a representation denoted by $\mathcal{R}_{g,n}(V) := \mathcal{R}_{g,n}^0(V) : \mathcal{L}_{g, n}^{\mathrm{inv}}(H) \to \End_{\mathbb{K}} (V_{g,n})$.
The construction of $\rho_{g,n}^1$ and $\rho_{g,n}$ was already done in the case $n = 0$ in \cite{faitg_derived_2026}, we extend it to the case $n > 0$ (see in particular \Cref{prop-action_twists_de_Dehn_non_separants_comme_conjugaison_et_proprietes,prop-twist_le_long_du_bord_et_push_dans_noyau}).
\smallskip

Our main results are obtained in \Cref{sec-calcul_extension_associee}, where we prove \Cref{thmintro-extension_qui_linearise_et_linearisation} below by direct computation, by using the Gervais presentation of $\Gamma_{g,n}^s$ and a diagrammatic description of $\mathcal{L}_{g, n}(H)$ introduced in \cite[Sec. 3]{faitg_holonomy_2024}, that we recall in \Cref{ann-description_diagrammatique_Lgn}.
This diagrammatic description is based on knot maps $\mathfrak{i}_\gamma : \mathcal{L}_{0, 1}(H) \to \mathcal{L}_{g, n}(H)$ associated with oriented simple closed curves $\gamma \subset \Sigma_{g,n}^1$.
Diagrammatic computations are developped in \Cref{ann-calcul_diagrammatique}.

\smallskip
Denote by $v$ the ribbon element of $H$.
Consider in particular the knot map $\mathfrak{i}_{m_{g+i}}$ where $m_{g+i}$ is a small loop around the $(g+i)$-th marked point. 
Given a left integral $\lambda : H \to \mathbb{K}$ and $V := \{ V_i \}_{i = 1, \cdots, n}$ a collection of simple $H$-modules, denote by $l$ the tuple $(l_0, \cdots, l_n) \in \mathbb{N}^{n+1}$ where $l_0$ is the order of the element $\bigl( \lambda(v^{-1}) \lambda(v)^{-1} \bigr)^2$ and $l_i$ is the order of $\mathfrak{i}_{m_{g+i}} \bigl( (v \triangleright \lambda) \lambda(v)^{-1} \bigr)$ acting on $V_{g,n}^s$ by $\mathcal{R}_{g,n}^s(V)$.
By convention, $l_i = 0$ if the associated element is of infinite order.
\begin{theoremeintro}[\Cref{subsubsec-gammatildegns2}]\label{thmintro-extension_qui_linearise_et_linearisation}
	Let $V := \{ V_i \}_{i = 1, \cdots, n}$ be a collection of simple $H$-modules. 
	For all $g \geq 2$, $n \geq 0$, and $s \in \{ 0,1 \}$, a minimal central extension $\widetilde{\Gamma}_{g,n}^s(2,l)$ of $\Gamma_{g,n}^s$ on which we can define a linear representation $\tilde{\rho}_{g,n}^s(2)$ that linearizes $\rho_{g,n}^s$ is given by the following presentation:
	\begin{itemize}
		\item Generators: $\{ \tau_{\gamma} ~ \vert ~ \text{non-separating simple closed curve } \gamma \subset \Sigma_{g,n}^s \} \cup \{T, E_1, \cdots, E_n\}$.
		\item Relations: 
		\begin{enumerate}
			\item All 0 and 1-braid relations; \label{itemintro-0_et_1-tresse}
			\item One lantern relation, if $g \geq 3$; \label{itemintro-lanterne}
			\item $r_c = T^2$, where $r_c$ is the 3-chain relator; \label{itemintro-rc_egale_Tcarre}
			\item $T^{l_0} = 1$ (the unit element); \label{itemintro-T_puissance_l0}
			\item The element $T$ is central; \label{itemintro-T_est_central}
			\item $r_{p_i} = E_i$, where $r_{p_i}$ is the $i$-th puncture relator for each puncture $p_i$; \label{itemintro-rpi_egale_Ei}
			\item $E_i^{l_i} = 1$, for all $i \in [\![ 1, n]\!]$; \label{itemintro-Ei_puissance_li}
			\item $E_i$ is central, for all $i \in [\![ 1, n]\!]$. \label{itemintro-Ei_est_central}
		\end{enumerate}
	\end{itemize}
	By definition $T^{0}$ and $E_i^0$ are the empty word. 
	Moreover, the linearization of $\rho_{g,n}^s$ is defined by:
	\[ 
	\begin{array}{l|ccl}
		\tilde{\rho}_{g,n}^s(2): & \widetilde{\Gamma}_{g,n}^s(2,l) & \longrightarrow & \GL(V_{g,n}^s) \\
		& \tau_\gamma & \longmapsto & \mathcal{R}_{g,n}^s(V)(\hat{\tau}_\gamma), \\
		& T & \longmapsto & \frac{\lambda(v^{-1})}{\lambda(v)} \id_{V_{g,n}^s}, \\
		& E_i & \longmapsto & \xi_i \id_{V_{g,n}^s}.
	\end{array}
	\]
	where $\hat{\tau}_\gamma = \mathfrak{i}_\gamma \bigl( (v \triangleright \lambda) \lambda(v)^{-1} \bigr)$ for any non-separating simple closed curve $\gamma \in \Sigma_{g,n}^s$, and $\xi_i \in \mathbb{K}$ is such that $\mathcal{R}_{g,n}^s(V) \bigl( \mathfrak{i}_{m_{g+i}} \bigl( (v \triangleright \lambda) \lambda(v)^{-1} \bigr) \bigr) = \xi_i \id_{V_{g,n}^s}$, for all $i \in [\![ 1,n ]\!]$.
\end{theoremeintro}
\noindent
Note that when $n = 0$, the collection $V$ of simple $H$-modules is empty, and the generators $E_i$ and the relations \labelcref{itemintro-rpi_egale_Ei,itemintro-Ei_puissance_li,itemintro-Ei_est_central} do not appear in the presentation of $\widetilde{\Gamma}_{g}^s(2,l_0)$, so the generators $E_i$ do not appear in $\tilde{\rho}_{g}^s(2):= \tilde{\rho}_{g,0}^s(2)$ either. 
In that case, we will remove $n$, which is equal to $0$, from the notation, e.g. $\widetilde{\Gamma}_{g}^s(2,l_0) := \widetilde{\Gamma}_{g,0}^s(2,l_0)$.

\begin{remarqueintro}
	For any non-separating simple closed curve $\gamma \subset \Sigma_{g,n}^s$, the order of the Dehn twist $\tau_\gamma$ through the linear representation $\tilde{\rho}_{g,n}^s(2)$ is given by the order of the ribbon element $v \in H$ (see \Cref{rem-ordre_tau_chapeau_gamma}). 
\end{remarqueintro}
\noindent
For all $g \geq 2$ and $n = 0$, let $\widetilde{\Gamma}_g^s(l_0)$ be the central extension of $\Gamma_g^s$ defined by the same presentation as $\widetilde{\Gamma}_{g}^s(2,l_0)$ except for the relation \labelcref{itemintro-rc_egale_Tcarre} which is replaced by $r_c = T$. 
Denote by $\widetilde{\Gamma}_{g,n}^s(l_0)$ its pullback induced by the morphism $\Gamma_{g,n}^s \xrightarrow{\text{Forget}} \Gamma_g^s$ forgetting the marked points.
Let $c_{\widetilde{\Gamma}_{g,n}^s(2,l)} \in \bigoplus_{i=0}^n H^2(\Gamma_{g,n}^s, \mathbb{Z}_{l_i})$ be the cohomology class associated with the central extension $\widetilde{\Gamma}_{g,n}^s(2,l)$ of \Cref{thmintro-extension_qui_linearise_et_linearisation} (where $\mathbb{Z}_{l_i} := \mathbb{Z} / l_i \mathbb{Z}$ and $\mathbb{Z}_0 := \mathbb{Z}$), $c_{\widetilde{\Gamma}_{g,n}^s(l_0)}$ the cohomology class associated with $\widetilde{\Gamma}_{g,n}^s(l_0)$, and $[e_i]$ the $i$-th Euler class. 

\begin{theoremeintro}[\Cref{thm-cohomologie_Gammagns2l}]\label{propintro-cohomologie_Gammatildegs2l}
		For all $g \geq 2$, $n \geq 0$, and $s \in \{ 0, 1 \}$, we have:
		\[ c_{\widetilde{\Gamma}_{g,n}^s(2,l)} = 2 c_{\widetilde{\Gamma}_{g,n}^s(l_0)} \oplus [e_1](l_1) \oplus \cdots \oplus [e_n](l_n) \in \bigoplus_{i=0}^n H^2(\Gamma_{g,n}^s, \mathbb{Z}_{l_i}), \]
		where $[e_i](l_i)$ is the pushout of $[e_i]$ induced by the map $\mathbb{Z} \to \mathbb{Z}_{l_i}$. 
\end{theoremeintro}
\noindent
Since $\widetilde{\Gamma}_{g}^s(l_0)$ is the pushout of $\widetilde{\Gamma}_{g}^s(0)$ induced by the map $\mathbb{Z} \to \mathbb{Z}_{l_0}$, a good description of $\widetilde{\Gamma}_{g}^s(l_0)$ is given by the fact that the cohomology class $c_{\widetilde{\Gamma}_g^s(0)}$ associated with $\widetilde{\Gamma}_{g}^s(0)$ is a generator of $H^2(\Gamma_g^s, \mathbb{Z})$ (see \Cref{prop-cohomologie_tildeGammags_genere_H^2}). 
This last statement follows from results of Harer in \cite{harer_second_1983} when $g \geq 4$, and the cases $g = 2$ and $3$ are discussed in \cite{masbaum_central_1995}.

\smallskip

We treat the cases $g = 1$, $n = 0$, and $s \in \{ 0,1 \}$ separately because the presentations of $\Gamma_{1}^{s}$ are quite different, but we obtain the same results as in genus $g \geq 2$ (see \Cref{subsubsec-gammatilde12}).
We do not cover the cases $g = 1$, $n > 1$, and $s \in \{ 0,1 \}$, because in these cases the known presentations of $\Gamma_{1,n}^s$ use separating curves (\cite{gervais_finite_2001}).

\begin{remarqueintro}\label{remintro-lien_avec_Lyubashenko}
	When $n = 0$, the projective representation $\rho_{g} := \rho_{g,0}$ of $\Gamma_g$ is isomorphic to the one obtained with the Kerler-Lyubashenko TQFT for $H$-$\mathrm{mod}$ (\cite[Sec. 6]{faitg_projective_2020}, \cite[App. A]{faitg_derived_2026}).
	A byproduct of our work is a direct computation of the cohomology class of the projective representation of this TQFT, which is 12 times (the pushout of) a generator of $H^2(\Gamma_g, \mathbb{Z})$. 
	In fact, the construction of the TQFT needs a distinguished element $\mathcal{D}$ (see e.g. \cite[(5.2.13) and Lem. 6.3.3]{kerler_non-semisimple_2001}, or \cite[Thm. 3.3]{de_renzi_3-dimensional_2022}). 
	In our context, $\mathcal{D}$ is a square root of $\lambda(v) \lambda(v^{-1})$ (see \cite[Thm. 2.9]{de_renzi_renormalized_2018}) and so $\lambda(v^{-1}) \lambda(v)^{-1} = (\mathcal{D} \lambda(v)^{-1})^2$.
	Then, by expressing the linearization $\tilde{\rho}_g(2)$ in terms of $\mathcal{D} \lambda(v)^{-1}$ the proof of \Cref{propintro-cohomologie_Gammatildegs2l} gives a cohomology class equal to $4$ times a generator, and thus the Meyer class (see \Cref{thm-H^2_grand_genre}).
	Finally, since the anomaly of the TQFT is corrected by adding an integer weight to each 3-manifold, we obtain an extension whose cohomology class is 12 times a generator of $H^2(\Gamma_g, \mathbb{Z})$ (see \cite[Sec. 4 and 5]{masbaum_central_1995} and \cite[Sec. 3]{gervais_presentation_1996}).
\end{remarqueintro}

\subsection*{Acknowledgments}

I firstly want to thank my PhD advisor, Stephane Baseilhac, for his support and his constant help since the beginning.
I would also like to thank Marco De Renzi for our various discussions and his explanations about TQFTs. 
Finally, I would like to thank Matthieu Faitg for his explanations about the construction of projective representations and help with diagrammatic computations. 

\section{Presentation and (co-)homology of the mapping class group of a surface}\label{sec-groupe_modulaire}

In this section, we recall some facts about mapping class groups of surfaces.  
For details, we refer to \cite[Chap. 1 to 5]{farb_primer_2012}.

\medskip

We denote by $\Sigma_{g,n}^s$ the compact oriented surface of genus $g$ with $n$ punctures (or marked points) and $s$ boundary components.
We will not write the index $n$ or the exponent $s$ when they are equal to 0.

\smallskip

The \emph{mapping class group} of $\Sigma_{g,n}^s$ is the group of isotopy classes of diffeomorphisms of $\Sigma_{g,n}^s$ that fix the boundary pointwise, where the isotopies must also fix the boundary pointwise.
Note that a mapping class can permute the punctures.

\smallskip

We denote by $\Gamma_{g,n}^s$ the subgroup of the mapping class group of $\Sigma_{g,n}^s$ that fixes the punctures.
As with the surface $\Sigma_{g,n}^s$, we will not write the index $n$ or the exponent $s$ of $\Gamma_{g,n}^s$ when they are equal to 0.
Note that it is equal to the mapping class group of $\Sigma_{g,n}^s$ when $n = 0$ or $1$. 

\smallskip

First, we recall some relations between left Dehn twists along simple closed curves (i.e. without self-intersection), where the left Dehn twist along $\gamma$ is denoted by $\tau_\gamma$.

\smallskip

The following figures represent surfaces embedded in $\Sigma_{g,n}^s$ where the relations take place.

\begin{proposition}[0 and 1-braid relation]\label{prop-relation_0_et_1-tresse}
    Let $a$ and $b$ be two simple closed curves. 
	We have the following relations in $\Gamma_{g,n}^s$: 
    \begin{itemize}
        \item \emph{0-braid relation}: If $i(a,b) = 0$ then $\tau_a \tau_b = \tau_b \tau_a$,
        \item \emph{1-braid relation}: If $i(a,b) = 1$ then $\tau_a \tau_b \tau_a = \tau_b \tau_a \tau_b$,
    \end{itemize}
    where $i(a,b)$ denotes the geometric intersection number between the curves $a$ and $b$.
\end{proposition}

\begin{proposition}[Lantern relation]\label{prop-relation_lanterne}
	Let $b_1$, $b_2$, $b_3$, $b_4$, $x$, $y$, and $z$ be the simple closed curves arranged as in \Cref{fig-courbes_relation_lanterne_sur_surface_minimale}.
	We have the following relation in $\Gamma_{g,n}^s$, called the \emph{lantern relation}:
    \[ \tau_{b_1} \tau_{b_2} \tau_{b_3} \tau_{b_4} = \tau_x \tau_y \tau_z. \] 
	We denote by $r_l := \tau_{b_1} \tau_{b_2} \tau_{b_3} \tau_{b_4} \tau_z^{-1} \tau_y^{-1} \tau_x^{-1}$ the associated relator.
\end{proposition}

\begin{figure}[H]
	\centering
	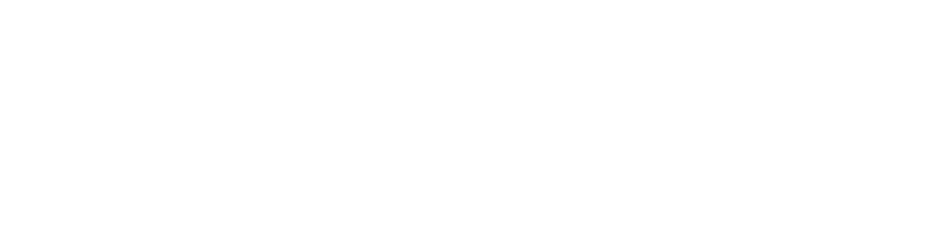
	\caption{Lantern relation}
	\label{fig-courbes_relation_lanterne_sur_surface_minimale}
\end{figure}

The following relation is an immediate consequence of the lantern relation and the capping morphism (\cite[Prop. 3.19]{farb_primer_2012}):

\begin{corollaire}[Puncture relation]\label{prop-relation_piqure}
    Let $b_1$, $b_2$, $b_3$, $x$, $y$, and $z$ be the simple closed curves arranged as in \Cref{fig-courbes_relation_piqure_sur_surface_minimale}.
    We have the following relation in $\Gamma_{g,n}^s$, called the \emph{puncture relation}:
    \[ \tau_{b_1} \tau_{b_2} \tau_{b_3} = \tau_x \tau_y \tau_z. \] 
	We denote by $r_p := \tau_{b_1} \tau_{b_2} \tau_{b_3} \tau_z^{-1} \tau_y^{-1} \tau_x^{-1}$ the associated relator.
\end{corollaire}

\begin{figure}[H]
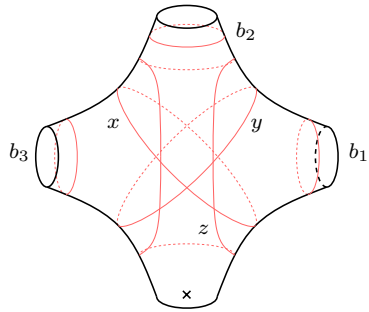

	\centering
	
	\caption{Puncture relation}
	\label{fig-courbes_relation_piqure_sur_surface_minimale}
\end{figure}

\begin{proposition}[3-chain relation]\label{prop-relation_3-chaine}
	Let $a$, $b$, $c$, $d$, and $e$ be the simple closed curves arranged as in \Cref{fig-courbes_relation_3-chaine_sur_surface_minimale}.
	We have the following relation in $\Gamma_{g,n}^s$, called the \emph{3-chain relation}:
	\[ (\tau_{a} \tau_{b} \tau_{c})^4 = \tau_d \tau_e. \] 
	We denote by $r_c := (\tau_{a} \tau_{b} \tau_{c})^4 \tau_e^{-1} \tau_d^{-1}$ the associated relator.
\end{proposition}

\begin{figure}[H]
	\centering
	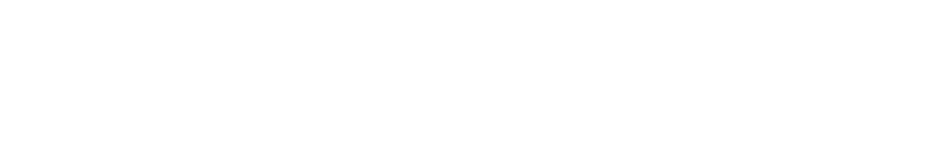
	\caption{3-chain relation}
	\label{fig-courbes_relation_3-chaine_sur_surface_minimale}
\end{figure}
\noindent
We will use the following presentation of $\Gamma_{g,n}^s$.
 
\begin{theoreme}[{{ \cite{harer_second_1983}, \cite[Thm. B]{gervais_presentation_1996}, and \cite[Lem. 2.7]{funar_centrally_2014}}}]\label{thm-presentation_de_Gervais}
	For all $g \geq 2$, $n \geq 0$, and $s \geq 0$, the group $\Gamma_{g,n}^s$ has the following presentation: 
    \begin{itemize}
        \item Generators: All Dehn twists $\tau_{\gamma}$ along non-separating simple closed curves $\gamma \subset \Sigma_{g,n}^s$.
        \item Relations: 
        \begin{enumerate}
            \item All 0 and 1-braid relations,
            \item One lantern relation, if $g \geq 3$,
            \item One 3-chain relation,
            \item One puncture relation for each puncture $p_i$.
        \end{enumerate}
    \end{itemize}
\end{theoreme}
\noindent
In genus 1, we have the following presentations:
\begin{equation}\label{eq-presentation_Gamma11_et_Gamma1}
    \Gamma_1^1 = \langle \tau_a, \tau_b ~ \vert ~ \tau_a \tau_b \tau_a = \tau_b \tau_a \tau_b \rangle ~ \text{ and } ~ \Gamma_1 = \langle \tau_a, \tau_b ~ \vert ~ \tau_a \tau_b \tau_a = \tau_b \tau_a \tau_b \text{ and } (\tau_a \tau_b)^6 = 1 \rangle.
\end{equation}

\smallskip

Recall that for all $g \geq 3$, $n \geq 0$, and $s \geq 0$, the group $\Gamma_{g,n}^s$ is \emph{perfect}, which means $\Gamma_{g,n}^s = [\Gamma_{g,n}^s, \Gamma_{g,n}^s]$, where $[\Gamma_{g,n}^s, \Gamma_{g,n}^s]$ is the commutator subgroup of $\Gamma_{g,n}^s$.
This implies that $H_1( \Gamma_{g,n}^s, \mathbb{Z}) = 0$, for all $g \geq 3$. 
In genus $1$ and $2$, we have the following isomorphisms (see \cite[Rem. 5.2 and Thm. 5.1]{korkmaz_second_2003}):
\begin{equation}\label{eq-H_1_petit_genre}
    H_1(\Gamma_{1,n}^s, \mathbb{Z}) \simeq
    \begin{cases}
        \mathbb{Z}_{12} & \text{ if } s = 0, \\
        \mathbb{Z}^s & \text{ else,}
    \end{cases}
    \qquad
    H_1(\Gamma_{2,n}^s, \mathbb{Z}) \simeq \mathbb{Z}_{10}.
\end{equation}

In particular, when $g \geq 3$, $n \geq 0$, and $s \geq 0$, because $\Gamma_{g,n}^s$ is perfect, it has a \emph{universal central extension} $\widetilde{\Gamma}_{g,n}^{s,\mathrm{univ}}$.
When $g \geq 4$, $n = 0$, and $s \geq 0$, we have the following presentation of $\widetilde{\Gamma}_{g}^{s,\mathrm{univ}}$. 

\begin{theoreme}[{{\cite[Sec. 5 - (2)]{harer_second_1983}}}]\label{thm-presentation_ECU}
 	For all $g \geq 4$ and $s \geq 0$, the universal central extension $\widetilde{\Gamma}_{g}^{s,\mathrm{univ}}$ of $\Gamma_{g}^s$ has the following presentation: 
    \begin{itemize}
        \item Generators: All Dehn twists $\tau_{\gamma}$ along non-separating simple closed curves $\gamma \subset \Sigma_{g}^s$.
        \item Relations: 
        \begin{enumerate}
            \item All 0 and 1-braid relations,
            \item One lantern relation.
        \end{enumerate}
    \end{itemize}
\end{theoreme}

Next, consider the second homology group of $\Gamma_{g,n}^s$ with $\mathbb{Z}$-coefficient.

\begin{theoreme}[{{\cite{harer_second_1983} and \cite{korkmaz_second_2003}}}]\label{thm-H_2_grand_genre}
	For all $g \geq 4$, $n \geq 0$, and $s \geq 0$, we have an isomorphism: 
	\[ 	H_2(\Gamma_{g,n}^s, \mathbb{Z}) \simeq \mathbb{Z}^{n+1}. \]
\end{theoreme}
\noindent
In genus $2$ and $3$, we have the following isomorphisms (see \cite[Thm. 1.3]{korkmaz_second_2003} and \cite[Thm. 4.9 and Cor. 4.10]{sakasai_lagrangian_2012}):
\begin{equation}\label{eq-H_2_petit_genre}
    \begin{aligned}
        & H_2(\Gamma_{2}^s, \mathbb{Z}) \simeq \mathbb{Z}_2 & \text{ where } s \in \{0,1 \}, \\
        & H_2(\Gamma_{3}^s, \mathbb{Z}) \simeq \mathbb{Z} \oplus \mathbb{Z}_2 & \text{ where } s \in \{0,1 \}. 
    \end{aligned}
\end{equation}

Recall that for all $g \geq 1$, $n \geq 0$, and $s \geq 0$, the $i$-th Euler class $[e_i] \in H^2(\Gamma_{g,n}^s, \mathbb{Z})$ is the cohomology class of the central extension:
\begin{equation}\label{eq-definition_classe_Euler}
	 0 \longrightarrow \langle \tau_{\partial_i} \rangle \longrightarrow \Gamma_{g,n-1}^{s+1} \xrightarrow{~\mathrm{Cap}~} \Gamma_{g,n}^s \longrightarrow 1,
\end{equation}
where $\partial_i$ denotes the boundary of a disk centered at the $i$-th puncture and $\mathrm{Cap} : \Gamma_{g,n-1}^{s+1} \to \Gamma_{g,n}^s$ is the capping morphism (\cite[Prop. 3.19]{farb_primer_2012}).

\smallskip
\noindent
Also, recall that the Meyer class $[\tau] \in H^2(\Gamma_g, \mathbb{Z})$ is constructed by using the signature of certain 4-manifolds which fiber over a disk with two holes (see \cite{meyer_signatur_1964}).
We also denote by $[\tau]$ the pullback of the Meyer class by the map $\Gamma_{g,n}^s \to \Gamma_g$, which is the composition of the capping morphism and the morphism $\mathrm{Forget} : \Gamma_{g,n}^s \to \Gamma_{g}^s$ forgetting the marked points.

\smallskip
\noindent
The fact that $\Gamma_{g,n}^s$ is perfect in genus $g \geq 3$, \Cref{thm-H_2_grand_genre}, and the universal coefficient theorem imply the isomorphism in the following statement. 
For the generators, see \cite{korkmaz_second_2003}.

\begin{theoreme}\label{thm-H^2_grand_genre}
	For all $g \geq 4$, $n \geq 0$, and $s \geq 0$, we have an isomorphism: 
	\[ 	H^2(\Gamma_{g,n}^s, \mathbb{Z}) \simeq \mathbb{Z}^{n+1}. \]
	The generators of $H^2(\Gamma_{g,n}^s, \mathbb{Z})$ are one fourth of the Meyer class $[\tau]$ and the $n$ Euler classes $[e_i]$, for all $i \in [\![1, n]\!]$. 
\end{theoreme}
\noindent
In particular when $n = 0$, it follows from \Cref{thm-H_2_grand_genre} that for all $g \geq 4$ the group $H^2(\Gamma_g^s, \mathbb{Z})$ is generated by the cohomology class associated with the universal central extension $\widetilde{\Gamma}_{g}^{s,\mathrm{univ}}$.
By \Cref{thm-H^2_grand_genre} it is equal to $[\tau]/4$. 

\smallskip
\noindent
In genus $g \leq 3$, the group $H^2(\Gamma_{g,n}^s, \mathbb{Z})$ is known only for certain values of $n$ and $s$. 
In genus 3 the group $H^2(\Gamma_{3,n}^s, \mathbb{Z})$ contains a factor isomorphic to $\mathbb{Z}^{n+1}$, and in genus 2 the group $H^2(\Gamma_{2,n}^s, \mathbb{Z})$ contains a factor isomorphic to $\mathbb{Z}_{10} \oplus \mathbb{Z}^{n}$.
More precisely, we have (see \cite[Cor. 4.4 and 5.3]{korkmaz_second_2003}, \cite[Lem. 2.11]{funar_centrally_2014}, and \cite[Thm. 5.2]{ramirez_amalgam_2018}):

\begin{proposition}\label{prop-H^2_petit_genre}
    \begin{enumerate}[leftmargin=*]
        \item For $g = 3$, one fourth of the Meyer class $[\tau]$ and the $n$ Euler classes $[e_i]$, for all $i \in [\![1, n]\!]$, generate the factor $\mathbb{Z}^{n+1} \subset H^2(\Gamma_{3,n}^s, \mathbb{Z})$. 
        When $n = 0$ and $s \in \{ 0, 1 \}$, the factor $\mathbb{Z}$ generated by one fourth of the Meyer class $[\tau]$ is isomorphic to $H^2(\Gamma_{3}^s, \mathbb{Z})$.
        \item For $g = 2$, one half of the Meyer class $[\tau]$ and the $n$ Euler classes $[e_i]$, for all $i \in [\![1, n]\!]$, generate the factor $\mathbb{Z}_{10} \oplus \mathbb{Z}^{n} \subset H^2(\Gamma_{2,n}^s, \mathbb{Z})$. 
		When $n = 0$ and $s \in \{ 0, 1 \}$, the factor $\mathbb{Z}_{10}$ generated by one half of the Meyer class $[\tau]$ is isomorphic to $H^2(\Gamma_{2}^s, \mathbb{Z})$.
        \item For $g = 1$, $H^2(\Gamma_{1}, \mathbb{Z}) \simeq \mathbb{Z}_{12}$ is generated by one fourth of the Meyer class $[\tau]$, and $H^2(\Gamma_{1}^1, \mathbb{Z}) = 0$.
    \end{enumerate}
\end{proposition}

\section{The graph algebras and its representations}\label{sec-Lgn_et_Alekseev}

Let $(H, \mu, 1_H, \Delta, \varepsilon, S)$ be a finite-dimensional ribbon Hopf algebra over the field $\mathbb{K}$ (see \cite[Sec. VIII.2 and VIII.4]{kassel_quantum_1995} and \cite[§4.1.C]{chari_guide_2000}). 
Starting from \Cref{subsec-representations_Lgn}, we will further assume that $H$ is factorizable (see \Cref{def-factorisable}).

\begin{notation*}
	The product $\mu$ of $H$ will often be implicit.
	We denote the iterated co-product by $\Delta^{(k)} := (\Delta \otimes \id_H) \circ \Delta^{(k-1)}$. 
	By convention $\Delta^{(1)} := \id_H$ and $\Delta^{(2)} := \Delta$. 
	We use Sweedler's notation: $\Delta^{(k)} (h) = h_{(1)} \otimes \cdots \otimes h_{(k)}$. 
	We denote the evaluation pairing by $\langle - , - \rangle : H^* \otimes H \to \mathbb{K}$.
\end{notation*}
\noindent
We recall the following basic relations for future references.
The \emph{universal R-matrix} $R := \sum_a r_a \otimes r^a \in H^{\otimes 2}$ satisfies the following properties: 
\begin{align}
	& R \Delta = \Delta^{op} R, \\
	& (\Delta \otimes \id_H)(R) = R_{13} R_{23}, \qquad (\id_H \otimes \, \Delta)(R) = R_{13}R_{12}, \label{eq-R-matrice_et_co-produit} \\
	& (S \otimes \id_H)(R) = (\id_H \otimes \, S^{-1})(R) = R^{-1}, \qquad (S \otimes S)(R) = R, \label{eq-R-matrice_et_antipode} \\
	& (\varepsilon \otimes \id_H)(R) = (\id_H \otimes \, \varepsilon)(R) = 1_H, \label{eq-R-matrice_et_co-unite} \\
	& R_{12}R_{13}R_{23} = R_{23}R_{13}R_{12}. \label{eq-Yang-Baxter}
\end{align}
The \emph{Drinfeld element} $u := \sum_a S(r^a)r_a \in H$ satisfies the following properties:
\begin{equation}\label{eq-propriete_u}
	u^{-1} = S^{-2}(r^a)r_a = S^{-1}(r^a)S(r_a) = r^a S^2(r_a), \qquad \forall h \in H, \, S^2(h) = u h u^{-1}.
\end{equation} 
The \emph{ribbon element} $v \in H$ is central, invertible, and satisfies the following properties: 
\begin{equation}\label{eq-propriete_v}
v^2 = uS(u), \qquad \Delta(v) = (R_{21}R)^{-1} v \otimes v, \qquad \varepsilon(v) = 1, \qquad S(v) = v.
\end{equation}
The \emph{pivotal element} $g := uv^{-1}$ satisfies the following properties: 
\begin{equation}\label{eq-propriete_g}
 \Delta(g) = g \otimes g, \qquad S(g) = g^{-1}, \qquad \forall h \in H, \, S^2(h) = g h g^{-1}.
\end{equation}

\subsection{The graph algebras \texorpdfstring{$\mathcal{L}_{g,n}(H)$}{Lgn(H)}}\label{subsec-Lgn}

A reference on the algebras $\mathcal{L}_{g, n}(H)$ is \cite[§3.1, 3.2, and 4.1]{baseilhac_noetherian_2025}. 

\begin{definition}
	The \emph{loop algebra} $\mathcal{L}_{0,1}(H)$ is the dual vector space $H^*$ endowed with the following product: 
    \[\varphi \psi = \sum_{i,j} \bigl( r^j S(r^i)\triangleright \varphi \bigr) \star (r_j \triangleright \psi \triangleleft r_i),\]
    where $\triangleright$ (resp. $\triangleleft$) is is the \emph{left} (resp. \emph{right}) \emph{coregular action} of $H$ on $H^*$ defined by $\langle h \triangleright \varphi,x \rangle = \langle \varphi, xh \rangle$ (resp. $\langle h \triangleleft \varphi,x \rangle = \langle \varphi, hx \rangle$). 
    It is an associative algebra, with neutral element the co-unit $\varepsilon$.
\end{definition}
\noindent
The algebra $\mathcal{L}_{0,1}(H)$ is an $H$-module-algebra for the \emph{right coadjoint action}:
\[ \forall h \in H, \, \forall \varphi \in \mathcal{L}_{0, 1}(H), ~ \coad^r(h) (\varphi) := S(h_{(2)}) \triangleright \varphi \triangleleft h_{(1)}, \]
which means: 
\[ \forall \varphi, \, \psi \in \mathcal{L}_{0, 1}(H), \, \forall h \in H, ~ \coad^r(h)(\varphi\psi) = \coad^r(h_{(1)})(\varphi) \coad^r(h_{(2)})(\psi), \]
where we used the product of $\mathcal{L}_{0, 1}(H)$.

\smallskip
\noindent
An element $\varphi \in \mathcal{L}_{0, 1}(H)$ is invariant under the right coadjoint action if:
\[ \forall h \in H, ~ \coad^r(h)(\varphi) = \varepsilon(h) \varphi, \]
which we denote by $\varphi \in \inv \coad^r$.
A well-known equivalent formulation of the invariance under the right coadjoint action is: 
\begin{equation}\label{eq-reformulation_invariance_coad^r}
	\varphi \in \inv \coad^r \iff \forall x,y \in H, \varphi(xy) = \varphi \bigl( S^{-2}(y)x \bigr).
\end{equation}
We have the following result which will be usefull in \Cref{subsec-calcul_relations}.
\begin{lemme}[{{\cite[Lem. 4.3]{faitg_derived_2026}}}]\label{lem-produit_L01_si_un_est_invariant}
	For all $\varphi \in \mathcal{L}_{0, 1}(H)$ and $\psi \in \inv \coad^r$, we have: 
	\[ \varphi \psi = \varphi \star \psi = \psi \varphi = \psi \star \varphi. \]
\end{lemme}

\begin{definition}\label{def-L10}
	The \emph{handle algebra} $\mathcal{L}_{1,0}(H)$ is the vector space $H^* \otimes H^*$ endowed with the product determined by the following conditions:
    \begin{enumerate}[itemsep=0.2em]
         \item $\mathcal{L}_{0,1}(H) \otimes \varepsilon$ and $\varepsilon \otimes \mathcal{L}_{0,1}(H)$ are subalgebras of $\mathcal{L}_{1,0}(H)$,
         \item $(\varepsilon \otimes \alpha)(\beta \otimes \varepsilon) = \sum_{i,j,k,l} (r^l r_k\triangleright \beta \triangleleft r_i r_j) \otimes (r^k S(r^i) \triangleright \alpha \triangleleft r^j r_l)$,
         \item $(\beta \otimes \varepsilon)(\varepsilon \otimes \alpha) = (\beta \otimes \alpha).$ \label{itemL10:betafoisalpha_egale_betaotimesalpha}
    \end{enumerate}
	It is an associative algebra, with neutral element $\varepsilon \otimes \varepsilon$.
\end{definition}
\noindent
The algebra $\mathcal{L}_{1, 0}(H)$ is an $H$-module-algebra for the right coadjoint action induced by the co-product:
\begin{equation}\label{eq-coad^r_sur_L01}
	\forall h \in H, \, \forall \varphi \otimes \psi \in \mathcal{L}_{1,0}(H), ~ \coad^r(h)(\varphi \otimes \psi) = \coad^r(h_{(1)})(\varphi) \otimes \coad^r(h_{(2)})(\psi).
\end{equation}
Moreover, the embeddings of $\mathbb{K}$-vector spaces: 
\begin{equation}\label{eq-injections_L10}
	\begin{array}{l|ccl}
		\mathfrak{i}_{a}: & \mathcal{L}_{0, 1}(H) & \longrightarrow & \mathcal{L}_{1, 0}(H) \\
		& \alpha & \longmapsto & \varepsilon \otimes \alpha
	\end{array}
	\qquad
	\begin{array}{l|ccl}
		\mathfrak{i}_{b}: & \mathcal{L}_{0, 1}(H) & \longrightarrow & \mathcal{L}_{1, 0}(H) \\
		& \beta & \longmapsto & \beta \otimes \varepsilon
	\end{array}
\end{equation}
are morphisms of $H$-module-algebras and by \labelcref{itemL10:betafoisalpha_egale_betaotimesalpha} of \Cref{def-L10} we have:
\[ \mathfrak{i}_{b}(\beta) \, \mathfrak{i}_{a}(\alpha) = \beta \otimes \alpha. \]

We recall the definition of a braided tensor product (see \cite[Lem. 9.2.12]{majid_foundations_2000}).
The category $\mathrm{mod}$-$H$ of right $H$-modules is braided, with the braiding:
\[ 
\begin{array}{l|ccl}
	c_{M,N}: & M \otimes N & \longrightarrow & N \otimes M \\
	& m \otimes n & \longmapsto & \sum_{a} n.r_a \otimes m.r^a.
\end{array} 
\]
Let $(M, \mu_M, 1_M)$ and $(N, \mu_N, 1_N)$ be two right $H$-module-algebras (i.e. algebras in $\mathrm{mod}$-$H$). 
The maps:
\begin{align*}
	& \mu_{M \tilde{\otimes} N} : (M \otimes N) \otimes (M \otimes N) \xrightarrow{\id_M \otimes \, c_{N,M} \otimes \, \id_N} (M \otimes M) \otimes (N \otimes N) \xrightarrow{\mu_M \otimes \, \mu_N} M \otimes N, \\ 
	& 1_{M \tilde{\otimes} N} := 1_M \otimes 1_N
\end{align*} 
endow the vector space $M \otimes N$ with a structure of $H$-module-algebra, denoted by $M \, \tilde{\otimes} \, N$ and called the \emph{braided tensor product} of $M$ and $N$. 
This operation is associative.
Explicitly, the product of $M \, \tilde{\otimes} \, N$ is: 
\begin{equation*}
	\forall (m \, \tilde{\otimes} \, n) \in M \, \tilde{\otimes} \, N, \, \forall (m' \, \tilde{\otimes} \, n') \in M \, \tilde{\otimes} \, N, ~  (m \, \tilde{\otimes} \, n) (m' \, \tilde{\otimes} \, n') = \sum_a m(m'.r_a) \, \tilde{\otimes} \, (n.r^a)n'.
\end{equation*}
The embeddings of $\mathbb{K}$-vector spaces: 
\begin{equation*}
	\begin{array}{l|ccl}
		i_M: & M & \longrightarrow & M \, \tilde{\otimes} \, N \\
		& m & \longmapsto & m \, \tilde{\otimes} \, 1_N
	\end{array}
	\qquad
	\begin{array}{l|ccl}
		i_N: & M & \longrightarrow & M \, \tilde{\otimes} \, N \\
		& n & \longmapsto & 1_M \, \tilde{\otimes} \, n
	\end{array}
\end{equation*}
are morphisms of $H$-module-algebras and we have: 
\begin{equation*}
	i_M(m) \, i_N(n) = m \, \tilde{\otimes} \, n, \qquad i_N(n) \, i_M(m) = \sum_a i_M(m.r_a) \, i_N(n.r^a).
\end{equation*}

\begin{definition}
	The \emph{graph algebra} of $\Sigma_{g,n}^1$ is $\mathcal{L}_{g,n}(H) :=  \mathcal{L}_{1,0}(H)^{\tilde{\otimes} \, g} \, \tilde{\otimes} \, \mathcal{L}_{0,1}(H)^{\tilde{\otimes} \, n}$.
\end{definition}
\noindent
The algebra $\mathcal{L}_{g,n}(H)$ is an $H$-module-algebra for the right coadjoint action induced by the co-product (see e.g. \labelcref{eq-coad^r_sur_L01} for $\mathcal{L}_{1, 0}(H)$).

\smallskip
\noindent
The embeddings $\mathfrak{i}_{a}$ and $\mathfrak{i}_{b}$ of \labelcref{eq-injections_L10} extend to the following embeddings of $H$-module-algebras:  
\begin{equation}\label{eq-definition_injections_canoniques}
	\mathfrak{i}_{a_j} : \mathcal{L}_{0, 1}(H) \longrightarrow \mathcal{L}_{g, n}(H), 
	\qquad 
	\mathfrak{i}_{b_j} : \mathcal{L}_{0, 1}(H) \longrightarrow \mathcal{L}_{g, n}(H),
	\qquad 
	\mathfrak{i}_{m_{g+k}}: \mathcal{L}_{0, 1}(H) \longrightarrow \mathcal{L}_{g, n}(H),
\end{equation}
defined by: $\forall \alpha, \beta, \psi \in \mathcal{L}_{0, 1}(H)$,
\begin{align*}
		& \mathfrak{i}_{a_j}(\alpha) = \varepsilon^{\otimes 2(j-1)} \otimes \varepsilon \otimes \alpha \otimes \varepsilon^{\otimes 2(g-j)+n}, \qquad \mathfrak{i}_{m_{g+k}}(\psi) = \varepsilon^{\otimes 2g+k-1} \otimes \psi \otimes \varepsilon^{\otimes n-k}, \\
		& \mathfrak{i}_{b_j}(\beta) = \varepsilon^{\otimes 2(j-1)} \otimes \beta \otimes \varepsilon \otimes \varepsilon^{\otimes 2(g-j)+n}.
\end{align*}
These embeddings correspond to the based (at a fixed point in $\partial\Sigma_{g,n}^1$) curves of $\Sigma_{g,n}^1$ represented in \Cref{fig-courbes_standards_Lgn,fig-courbes_standards_Lgn_ruban} below.
\begin{figure}[H]
	\centering
	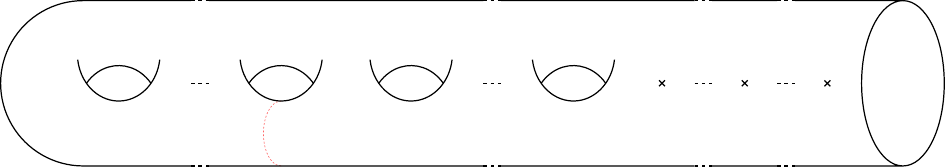
	\caption{Standard curves of $\Sigma_{g,n}^1$}
	\label{fig-courbes_standards_Lgn}
\end{figure}

\begin{figure}[H]
	\centering
	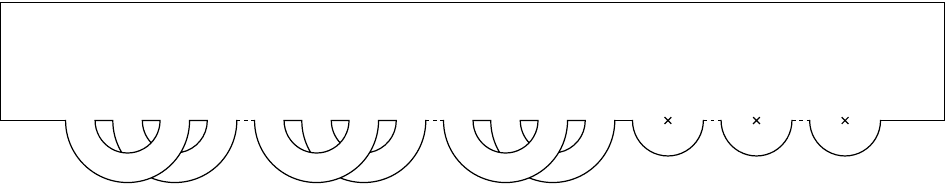
	\caption{Standard curves in the ribbon surface presentation of $\Sigma_{g,n}^1$}
	\label{fig-courbes_standards_Lgn_ruban}
\end{figure}
\noindent
These embeddings have the following properties: 
\begin{align}
 	& \mathfrak{i}_{a_j}, \mathfrak{i}_{b_j}, \mathfrak{i}_{m_{g+k}} : \mathcal{L}_{0,1}(H) \to \mathcal{L}_{g,n}(H) \text{ are morphisms of algebras}, \label{eq-injections_canoniques_morphismes_algebres}\\
 	& \mathfrak{i}_{b_1}(\beta_1) \mathfrak{i}_{a_1}(\alpha_1) \cdots \mathfrak{i}_{b_g}(\beta_g) \mathfrak{i}_{a_g}(\alpha_g) \mathfrak{i}_{m_{g+1}}(\psi_1) \cdots \mathfrak{i}_{m_{g+n}}(\psi_n) = \beta_1 \otimes \alpha_1 \otimes \cdots \otimes \beta_g \otimes \alpha_g \otimes \psi_1 \otimes \cdots \otimes \psi_n, \label{eq-generateurs_Lgn}\\
 	& \mathfrak{i}_{a_i}(\varphi) \mathfrak{i}_{b_i}(\psi) = \sum_{i,j,k,l} \mathfrak{i}_{b_i}(r^l r_k \triangleright \psi \triangleleft r_i r_j) \, \mathfrak{i}_{a_i} \bigl( r^k S(r^i) \triangleright \varphi \triangleleft r^j r_l \bigr), \label{eq-commutation_ai_bi}\\
 	& \mathfrak{i}_{y_j}(\varphi) \mathfrak{i}_{x_i}(\psi) = \sum_{i,j,k,l} \mathfrak{i}_{x_i} \bigl( S(r_k r_l) \triangleright \psi \triangleleft r_i r_j \bigr) \, \mathfrak{i}_{y_j} \bigl( S(r^i r^k) \triangleright \varphi \triangleleft r^j r^l \bigr), \label{eq-commutation_yj_xi}
\end{align}
for all $\alpha_i, \beta_i, \psi_j, \varphi, \psi \in \mathcal{L}_{0,1}(H)$, where $i < j$ and $x_i$ (resp. $y_j$) is either $a_i$, $b_i$, or $m_i$ (resp. $a_j$, $b_j$, or $m_j$). 

\medskip

There is a topological interpretation of the algebra $\mathcal{L}_{g, n}(H)$ in terms of the stated skein algebra $S^{\mathrm{st}}_H(\Sigma_{g,n}^{1,\bullet})$, defined by using the Reshetikhin-Turaev functor on stated $H$-colored oriented ribbon graphs in $\Sigma_{g,n}^{1,\bullet} \times [0,1]$, where $\Sigma_{g,n}^{1,\bullet}$ is the surface $\Sigma_{g,n}^{1}$ with a fixed base point on its boundary.

\begin{theoreme}[{{\cite[Thm. 6.5]{baseilhac_noetherian_2025}}}]\label{thm-isomorphisme_avec_stated_skein}
	There is an (explicit) isomorphism between the graph algebra $\mathcal{L}_{g, n}(H)$ and the stated skein algebra $S^{\mathrm{st}}_H(\Sigma_{g,n}^{1,\bullet})$. 
\end{theoreme}


\subsection{The Alekseev morphism}\label{subsec-Alekseev}
In \Cref{sec-construction_representations_projectives}, we will construct a projective representation of the mapping class groups $\Gamma_{g,n}^s$, for $s \in \{0,1 \}$, thanks to a map $\Phi_{g,n}$ on $\mathcal{L}_{g, n}(H)$ which can be defined inductively from the following simpler maps $\Phi_{0,1}$ and $\Phi_{1,0}$. 

\medskip

Consider the map:
\begin{equation}\label{eq-phi01}
	\begin{array}{l|ccl}
		\Phi_{0,1}: & \mathcal{L}_{0, 1}(H) & \longrightarrow & H \\
		& \varphi & \longmapsto & (\varphi \otimes \id_H)(RR_{21}).
	\end{array}
\end{equation}

\begin{definition}\label{def-factorisable}
	We say that $H$ is \emph{factorizable} if $\Phi_{0,1}$ is an isomorphism.
\end{definition}
\noindent
There are plenty examples of factorizable ribbon Hopf algebras, e.g. given by the small quantum groups $u_\epsilon(\mathfrak{g})$, where $\mathfrak{g}$ is a complex semisimple Lie algebra and $\epsilon$ a primitive odd root of unity (see \cite[Cor. A.3.3]{lyubashenko_invariants_1995}).

\smallskip
\noindent
The map $\Phi_{0,1}$ is a morphism of algebras.
Moreover, if we endow $H$ with the \emph{right adjoint action} $\ad^r(h)(x) := S(h_{(1)}) x h_{(2)}$, then $\Phi_{0,1}$ is a morphism of $H$-module-algebras.  

\medskip
\noindent
We will use the \emph{Heisenberg double} $\mathcal{H}(H^*)$ of $H^*$, which is the $\mathbb{K}$-vector space $H^* \otimes H$ endowed with the product $(\varphi \otimes x)(\psi \otimes y) := \varphi \star (x_{(1)} \triangleright \psi ) \otimes x_{(2)}y$,
where $\star$ is the product of $H^*$ (i.e. $\varphi \star \psi = (\varphi \otimes \psi) \circ \Delta$, with $\Delta$ the co-product of $H$).
It is an associative algebra, with neutral element $\varepsilon \otimes 1_H$.
The map:
\begin{equation}\label{eq-representation_Montgomery}
	\begin{array}{l|ccl}
		\Psi: & \mathcal{H}(H^*) & \longrightarrow & \End_\mathbb{K}(H^*) \\
		& \varphi \otimes x & \longmapsto & \bigl( \psi \mapsto \varphi \star (x \triangleright \psi) \bigr)
	\end{array}
\end{equation}
is an embedding of algebras (see \cite[Lem. 9.4.2]{montgomery_hopf_1993}).
Since $H$ is finite-dimensional, by equality of dimensions it follows that $\Psi$ is an isomorphism.

\medskip
\noindent
Define the map: 
\[ 
\begin{array}{l|ccl}
	\Phi_{1,0}: & \mathcal{L}_{1, 0}(H) & \longrightarrow & \mathcal{H}(H^*) \\
	& (\varphi \otimes \psi) & \longmapsto & \sum_{i,j,k}(r^ir^j\triangleright \varphi \triangleleft r_k r_i) \otimes \bigl( r^k r_j \Phi_{0,1}(\psi) \bigr).
\end{array}
\]
It is a morphism of algebras (even $H$-module-algebras) (see \cite[Prop. 3.8]{baseilhac_noetherian_2025}).

\begin{theoreme}[{{\cite[§3.3]{faitg_projective_2020}, \cite[Sec. 5]{baseilhac_noetherian_2025}}}]\label{thm-isomorphisme_phign}
	\begin{enumerate}[leftmargin=*]
		\item The morphisms of algebras $\Phi_{0,1}$ and $\Phi_{1,0}$ extend to a morphism of algebras $\Phi_{g,n} : \mathcal{L}_{g,n}(H) \to \mathcal{H}(H^*)^{\otimes g} \otimes H^{\otimes n}$, called the \emph{Alekseev morphism}.
		\item Let us also assume that $H$ is factorizable. 
		Then $\Phi_{g,n}$ is an isomorphism. 
		Therefore, it induces an isomorphism between $\mathcal{L}_{g, n}(H)$ and $\End_\mathbb{K}(H^*)^{\otimes g} \otimes H^{\otimes n}$. \\
		In particular, the algebra $\mathcal{L}_{g, 0}(H)$ is isomorphic to $\End_\mathbb{K}(H^*)^{\otimes g} \simeq M_{\dim(H)}(\mathbb{K})^{\otimes g}$.
	\end{enumerate}
\end{theoreme}

\subsection{Representations of \texorpdfstring{$\mathcal{L}_{g, n}(H)$}{Lgn}}\label{subsec-representations_Lgn}

Recall that $H$ is a finite-dimensional ribbon Hopf algebra. 
From now on, $H$ is also assumed to be factorizable.

\begin{proposition}\label{prop-representations_irreductibles_Lgn}
	The irreducible representations of $\mathcal{L}_{g, n}(H)$ are isomorphic to $(H^*)^{\otimes g} \otimes V_1 \otimes \cdots \otimes V_n$, where $\{ V_i \}_{i = 1, \cdots, n}$ is a collection of simple $H$-modules. \\
	In particular, there is a unique irreducible representation of $\mathcal{L}_{g, 0}(H)$ up to isomorphism, which is $(H^*)^{\otimes g}$.
\end{proposition}

\begin{proof}
    If $A$ and $B$ are algebras, the finite-dimensional irreducible representations of $A \otimes B$ are the tensor products of the finite-dimensonal irreducible representations of $A$ and $B$ (see \cite[Thm. 3.10.2]{etingof_introduction_2011}). 
	Also, recall that a matrix algebra $M_n(\mathbb{K})$ has a unique irreducible representation, which is $\mathbb{K}^n$. 
	Since $\Phi_{g,n}$ is an isomorphism of algebras (see \Cref{thm-isomorphisme_phign}), $\Psi$ is an isomorphism (see under \labelcref{eq-representation_Montgomery}), and $H$ is finite-dimensional, the result follows. 
\end{proof}
\noindent
Let $X := \{X_i\}_{i = 1, \cdots, n}$ be a collection of $H$-modules and $\mathcal{R}_i : H \to \End_\mathbb{K}(X_i)$ the associated representations.
A representation of $\mathcal{L}_{g, n}(H)$ on $(H^*)^{\otimes g} \otimes X_1 \otimes \cdots \otimes X_n$ is defined by:
\begin{equation}\label{eq-definition_representation_phign}
	\begin{array}{l|ccl}
        \mathcal{R}_{\Phi_{g,n}}^1(X) : & \mathcal{L}_{g, n}(H) & \longrightarrow & \End_\mathbb{K} \big( (H^*)^{\otimes g} \otimes X_1 \otimes \cdots \otimes X_n \big) \\
		& x & \longmapsto & (\Psi^{\otimes g} \otimes \mathcal{R}_1 \otimes \cdots \otimes \mathcal{R}_n) \circ \Phi_{g,n} (x).
    \end{array}
\end{equation}

\medskip

We now introduce another representation of $\mathcal{L}_{g, n}(H)$ which was initially obtained by braided tensor product of representations of $\mathcal{L}_{1, 0}(H)$ and $\mathcal{L}_{0, 1}(H)$. 

\begin{proposition}[{{\cite[Prop. 4.10]{faitg_derived_2026}}}]\label{prop-definition_representation_combinee}
	For any collection $X := \{ X_i \}_{i = 1, \cdots, n}$ of $H$-modules, there is a representation $\mathcal{R}_{g,n}^1(X) : \mathcal{L}_{g, n}(H) \to \End_\mathbb{K}((H^*)^{\otimes g} \otimes X_1 \otimes \cdots \otimes X_n)$ defined by the following formulas:
	\begin{align*}
        & \mathfrak{i}_{a_j}(\varphi) \cdot (\psi_1 \otimes \cdots \otimes \psi_{g} \otimes x_1 \otimes \cdots \otimes x_n) = \coad^r ({r_a}_{(1)})(\psi_1) \otimes \cdots \otimes \coad^r ({r_a}_{(j-1)})(\psi_{j-1}) \\
		& \qquad \qquad \qquad \qquad \qquad \qquad \qquad \otimes \Phi_{0,1} \bigl( \coad^r(r^a r^b)(\varphi) \bigr) \triangleright \psi_j \triangleleft r_b \otimes \psi_{j+1} \otimes \cdots \otimes \psi_{g} \otimes x_1 \otimes \cdots \otimes x_n, \\[.8em]
        & \mathfrak{i}_{b_j}(\varphi) \cdot (\psi_1 \otimes \cdots \otimes \psi_{g} \otimes x_1 \otimes \cdots \otimes x_n) = \coad^r( {r_a}_{(1)})(\psi_1) \otimes \cdots \otimes \coad^r({r_a}_{(j-1)})(\psi_{j-1}) \\
		& \qquad \qquad \qquad \qquad \qquad \qquad \qquad \qquad \qquad \qquad \otimes \coad^r(r^a)(\varphi) \psi_j \otimes \psi_{j+1} \otimes \cdots \otimes \psi_{g} \otimes x_1 \otimes \cdots \otimes x_n , \\[.8em]
		& \mathfrak{i}_{m_{g+k}}(\varphi) \cdot (\psi_1 \otimes \cdots \otimes \psi_{g} \otimes x_1 \otimes \cdots \otimes x_n) = \coad^r({r_a}_{(1)})(\psi_1) \otimes \cdots \otimes \coad^r({r_a}_{(g)})(\psi_{g}) \\
		& \qquad \qquad \qquad \otimes S({r_a}_{(g+1)}) . x_1 \otimes \cdots \otimes S({r_a}_{(g+k-1)}) . x_{k-1} \otimes \Phi_{0,1} \bigl(\coad^r(r^a)(\varphi) \bigr) . x_k \otimes x_{k+1} \otimes \cdots \otimes x_n,
    \end{align*}
	where the curves $a_j$, $b_j$, and $m_{g+k}$ are as in \Cref{fig-courbes_standards_Lgn}, $j \in [\![ 1, g ]\!]$, and $k \in [\![ 1, n ]\!]$.
	If $n = 0$, the collection $X$ is empty so we will write more briefly $\mathcal{R}_{g}^1 := \mathcal{R}_{g,0}^1(\emptyset)$.
\end{proposition}

We will now construct a representation of $\mathcal{L}_{g, n}(H)$ that lifts any $\mathcal{R}_{g,n}^1(V)$, where $V := \{V_i\}_{i = 1, \cdots, n}$ is a collection of simple $H$-modules.
Let  $\{ X_i \}_{i = 1, \cdots, n}$ be the collection of $H$-modules defined by $X_i := {}_HH$, the left regular representation of $H$, for all $i \in [\![ 1, n ] \!]$. 
Consider the representation defined in \Cref{prop-definition_representation_combinee} with this collection of $H$-modules:
\begin{equation}\label{eq-representation_relevee_a_H}
	\mathcal{R}_{g,n}^1(H) : \mathcal{L}_{g, n}(H) \longrightarrow \End_\mathbb{K} \bigl( (H^*)^{\otimes g} \otimes {}_HH^{\otimes n} \bigr).
\end{equation}
We can compare $\mathcal{R}_{g,n}^1(H)$ with $\mathcal{R}_{g,n}^1(V)$:
\begin{proposition}\label{prop-representation_relevee_a_H}
	For all collections $V := \{ V_i \}_{i = 1, \cdots, n}$ of simple $H$-modules, the representation $\mathcal{R}_{g,n}^1(H)$ lifts the representation $\mathcal{R}_{g,n}^1(V)$.
\end{proposition}

\begin{proof}
	Since $H$ is finite-dimensional, the Krull-Schmidt theorem (see \cite[Thm. 3.8.1]{etingof_introduction_2011}) implies that the left regular representation ${}_HH$ admits a direct sum decomposition into indecomposable $H$-modules (the PIMs) ${}_HH \simeq \bigoplus_{i=1}^n n_i P_i$, which is unique up to the order of the $P_i$ and isomorphism. 
	Furthermore, these indecomposable modules are the projective covers of the simple $H$-modules (see \cite[Thm. 9.2.1]{etingof_introduction_2011}).  

	\smallskip
	\noindent
	Let $V := \{ V_i \}_{i = 1, \cdots, n}$ be a collection of simple $H$-modules and $P := \{ P_i \}_{i = 1, \cdots, n}$ the collection of projective coverings $P_i$ associated with the simple $H$-modules $V_i$. 
	Let $i_j : P_j \hookrightarrow {}_HH$ and $\pi_j : {}_HH \to P_j$ be the associated injections and projections.
	We can restrict and project $\mathcal{R}_{g,n}^1(H)$ to obtain the map:
	\[
	\begin{array}{l|ccl}
		\mathcal{R}_{g,n}^1(P): & \mathcal{L}_{g, n}(H) & \longrightarrow & \mathrm{End}_\mathbb{K} \bigl( (H^*)^{\otimes g} \otimes P_1 \otimes \cdots \otimes P_n \bigr) \\
		& x & \longmapsto & \bigl( \id_{(H^*)^{\otimes g}} \otimes \, \pi_1 \otimes \cdots \otimes \pi_n \bigr) \circ \mathcal{R}_{g,n}^1(H)(x) \circ \bigl( \id_{(H^*)^{\otimes g}} \otimes \, i_1 \otimes \cdots \otimes i_n \bigr).
	\end{array}
	\]
	Denote by $\mathrm{Rad}(P_i) := \pi_i(\mathrm{Rad}(H))$, where $\mathrm{Rad}(H)$ is the radical of $H$, and let $\mathbf{P}_i := (H^*)^{\otimes g} \otimes P_1 \otimes \cdots \otimes P_{i-1} \otimes \mathrm{Rad}(P_i) \otimes P_{i+1} \otimes \cdots \otimes P_n$, for all $i \in [\![1, n]\!]$. 
	Since each PIM $P_j$ is a submodule of ${}_HH$ it is clear that $\mathcal{R}_{g,n}^1(P)$ is a subrepresentation of $\mathcal{R}_{g,n}^1(H)$.
	Moreover, since $\mathrm{Rad}(P_i)$ is a submodule of $P_i$, for all $i \in [\![ 1,n]\!]$, $\varphi \in H^*$, and $v \in \mathbf{P}_i$, we have $\mathcal{R}_{g,n}(P)(\mathfrak{i}_{x_r}(\varphi))(v) \in \mathbf{P}_i$, where $x_r$ is either $a_j$, $b_j$, or $m_{g+k}$.

	\smallskip
	\noindent
	Therefore, the representation $\mathcal{R}_{g,n}^1(P)$ factors by $\bigoplus_{i=1}^n \mathbf{P}_i$ and yields a representation of $\mathcal{L}_{g, n}(H)$ on $\End_\mathbb{K}((H^*)^{\otimes g} \otimes V_1 \otimes \cdots \otimes V_n).$
	The formulas of the representation constructed from $\mathcal{R}_{g,n}^1(H)$ are the same as the formulas for the representation $\mathcal{R}_{g,n}^1(V)$, which concludes the proof.
\end{proof}

\begin{remarque}\label{rem-isomorphisme_combinee_phign_et_irreductible}
It is proved in \cite[App. A]{faitg_derived_2026} that the representation $\mathcal{R}_{g,n}^1(X)$ is isomorphic to $\mathcal{R}_{\Phi_{g,n}}^1(X)$ (with an explicit isomorphism).
In particular, by \Cref{prop-representations_irreductibles_Lgn}, for any collection $V := \{ V_i \}_{i = 1, \cdots, n}$ of simple $H$-modules the representation $\mathcal{R}_{g,n}^1(V)$ is irreducible. 
\end{remarque}

\noindent
We will use the irreducible representations $\mathcal{R}_{g,n}^1(V)$ in \Cref{subsec-representation_projective_rhogn1,subsec-representation_projective_rhogn} to define projective representations of $\Gamma_{g,n}^1$.
The reason why we use $\mathcal{R}_{g,n}^1(V)$ instead of $\mathcal{R}_{\Phi_{g,n}}^1(V)$ is that it allows diagrammatic computations for the actions of the elements of $\mathcal{L}_{g, 0}(H) \subset \mathcal{L}_{g, n}(H)$.
Let us explain this.

\medskip
\noindent
The embeddings defined in \labelcref{eq-definition_injections_canoniques} and the relations \labelcref{eq-injections_canoniques_morphismes_algebres,eq-generateurs_Lgn,eq-commutation_ai_bi,eq-commutation_yj_xi} imply that $\mathcal{L}_{g, n}(H)$ is generated as a $\mathbb{K}$-vector space by the elements: 
\begin{equation}\label{eq-generateurs_Lgn_pour_representation_induite}
	\mathfrak{i}_{b_1}(\beta_1) \cdots \mathfrak{i}_{b_g}(\beta_g) \mathfrak{i}_{a_1}(\alpha_1)\cdots \mathfrak{i}_{a_g}(\alpha_g) \mathfrak{i}_{m_{g+1}}(\psi_1)\cdots \mathfrak{i}_{m_{g+n}}(\psi_n), \quad \forall \alpha_j, \beta_j, \psi_k \in H^*.
\end{equation}
Consider the $\mathbb{K}$-vector subspace $\mathcal{A} \subset \mathcal{L}_{g, n}(H)$ defined by: 
\begin{equation*}
    \mathcal{A} := \mathrm{Vect}_\mathbb{K} \bigl\{ \mathfrak{i}_{a_1}(\alpha_1)\cdots \mathfrak{i}_{a_g}(\alpha_g) ~ \vert ~ \alpha_j \in H^* \bigr\}.
\end{equation*}
The relations \labelcref{eq-injections_canoniques_morphismes_algebres,eq-generateurs_Lgn,eq-commutation_ai_bi,eq-commutation_yj_xi} imply that $\mathcal{A}$ is a subalgebra of $\mathcal{L}_{g, n}(H)$ isomorphic to $\mathcal{L}_{0, g}(H)$.
The algebra $\mathcal{A}$ has a representation on $\mathbb{K}$ defined by: 
\begin{equation*}\label{eq-representation_de_A}
	\forall \alpha_j \in H^*, ~ \mathfrak{i}_{a_j}(\alpha_j) \cdot 1_\mathbb{K} := \alpha_j (1_H) \, 1_\mathbb{K}.
\end{equation*}
The algebra $\mathcal{L}_{g, n}(H)$ is a right $\mathcal{A}$-module via right multiplication in $\mathcal{L}_{g, n}(H)$ so we define:
\begin{definition}\label{def-definition_representation_induite}
	We call the morphism:
	\[ 
	\begin{array}{l|ccl}
		\mathcal{R}_{\mathrm{ind}_{g,n}}^1: & \mathcal{L}_{g, n}(H) & \longrightarrow & \End_\mathbb{K} \bigl( \mathcal{L}_{g, n}(H) \otimes_\mathcal{A} \mathbb{K} \bigr) \\
		& x & \longmapsto & (y \otimes z \mapsto xy \otimes z)
	\end{array}
	\] 
	the \emph{induced representation} of $\mathcal{L}_{g, n}(H)$.
\end{definition}
\noindent
Since $\mathcal{L}_{g, n}(H)$ is the $\mathbb{K}$-vector space generated by the elements \labelcref{eq-generateurs_Lgn_pour_representation_induite}, using the relation \labelcref{eq-commutation_yj_xi} to swap $\mathfrak{i}_{a_j}(\alpha_j)$ and $\mathfrak{i}_{m_{g+k}}(\psi_k)$, we obtain:
\begin{equation*}
	\mathcal{L}_{g, n}(H) \otimes_\mathcal{A} \mathbb{K} = \mathrm{Vect}_\mathbb{K} \bigl\{ \mathfrak{i}_{b_1}(\beta_1) \cdots \mathfrak{i}_{b_g}(\beta_g) \mathfrak{i}_{m_{g+1}}(\psi_1) \cdots \mathfrak{i}_{m_{g+n}}(\psi_n) \otimes 1_\mathbb{K} ~ \big\vert ~ \beta_j, \psi_k \in H^* \bigr\} \simeq \mathcal{L}_{0, g+n}(H),
\end{equation*}
where the isomorphism is defined by:
\begin{equation}\label{eq-isomorphisme_representation_induite_L0g+n}
	\mathfrak{i}_{b_1}(\beta_1) \cdots \mathfrak{i}_{b_g}(\beta_g) \mathfrak{i}_{m_{g+1}}(\psi_1) \cdots \mathfrak{i}_{m_{g+n}}(\psi_n) \otimes 1_\mathbb{K} \longmapsto \mathfrak{i}_{m_1}(\beta_1) \cdots \mathfrak{i}_{m_g}(\beta_g) \mathfrak{i}_{m_{g+1}}(\psi_1) \cdots \mathfrak{i}_{m_{g+n}}(\psi_n).
\end{equation}
Thus $\mathcal{L}_{g, n}(H) \otimes_\mathcal{A} \mathbb{K} \simeq (H^*)^{\otimes (g+n)}$ as a $\mathbb{K}$-vector space. 
By a proof similar to the one of \Cref{prop-definition_representation_combinee}, one can check that the action of $\mathcal{L}_{g, n}(H)$ on $(H^*)^{\otimes (g+n)}$ via the representation $\mathcal{R}_{\mathrm{ind}_{g,n}}^1$ has the following formulas:
\begin{equation}\label{eq-formules_representation_induite}
	\begin{aligned}
		& \mathfrak{i}_{a_j}(\varphi) \cdot (\psi_1 \otimes \cdots \otimes \psi_{g+n}) = \coad^r({r_a}_{(1)})(\psi_1) \otimes \cdots \otimes \coad^r({r_a}_{(j-1)})(\psi_{j-1}) \\
		& \qquad \qquad \qquad \qquad \qquad \qquad \qquad \qquad \qquad \quad \otimes \Phi_{0,1} \bigl( \coad^r (r^a r^b)(\varphi) \bigr) \triangleright \psi_j \triangleleft r_b \otimes \psi_{j+1} \otimes \cdots \otimes \psi_{g+n}, \\[.8em]
		& \mathfrak{i}_{b_j}(\varphi) \cdot (\psi_1 \otimes \cdots \otimes \psi_{g+n}) = \coad^r({r_a}_{(1)})(\psi_1) \otimes \cdots \otimes \coad^r({r_a}_{(j-1)})(\psi_{j-1}) \\
		& \, \qquad \qquad \qquad \qquad \qquad \qquad \qquad \qquad \qquad \qquad \qquad \qquad \quad \otimes \coad^r(r^a)(\varphi) \psi_j \otimes \psi_{j+1} \otimes \cdots \otimes \psi_{g+n}, \\[.8em]
		& \mathfrak{i}_{m_{g+k}}(\varphi) \cdot (\psi_1 \otimes \cdots \otimes \psi_{g+n}) = \coad^r({r_a}_{(1)})(\psi_1) \otimes \cdots \otimes \coad^r({r_a}_{(g+k-1)})(\psi_{g+k-1}) \\
		& \qquad \qquad \qquad \qquad \qquad \qquad \qquad \qquad \qquad \qquad \qquad \quad\otimes \coad^r(r^a)(\varphi) \psi_{g+k} \otimes \psi_{g+k+1} \otimes \cdots \otimes \psi_{g+n},
	\end{aligned}
\end{equation}
where the curves $a_j$, $b_j$, and $m_{g+k}$ are as in \Cref{fig-courbes_standards_Lgn}, $j \in [\![1, g]\!]$, and $k \in [\![1, n]\!]$. 

\smallskip
\noindent
By the very definition of $\mathcal{R}_{\mathrm{ind}_{g,n}}^1$, we will see that we can use diagrammatic computations for this representation (see \Cref{subsubsec-representation_induite}). 
On another hand, since $H$ is factorizable, $\Phi_{0,1}$ is an isomorphism.
We then consider the representation:
\begin{equation}\label{eq-representation_induite_tiree_en_arriere}
	\begin{array}{l|ccl}
		& \mathcal{L}_{g, n}(H) & \longrightarrow & \End_{\mathbb{K}} \bigl( (H^*)^{\otimes g} \otimes H^{\otimes n} \bigr) \\
		& x & \longmapsto & \bigl( \id_{(H^*)^{\otimes g}} \otimes \Phi_{0,1}^{\otimes n} \bigr) \circ \mathcal{R}_{\mathrm{ind}_{g,n}}^1 \circ \bigl( \id_{(H^*)^{\otimes g}} \otimes (\Phi_{0,1}^{-1})^{\otimes n} \bigr)(x).
	\end{array}
\end{equation}
The formulas of the actions of $\mathfrak{i}_{a_j}(\varphi)$ and $\mathfrak{i}_{b_j}(\varphi)$ in \Cref{prop-definition_representation_combinee} and for the representation \labelcref{eq-representation_induite_tiree_en_arriere} are the same.
Therefore, the diagrammatic computations can also be used for the actions by $\mathcal{R}_{g,n}^1(V)$ and $\mathcal{R}_{g,n}^1(H)$ of elements in $\mathcal{L}_{g, 0}(H) \subset \mathcal{L}_{g, n}(H)$.

\smallskip
One may ask if the representation \labelcref{eq-representation_induite_tiree_en_arriere} and the representation $\mathcal{R}_{g,n}^1(H)$ (see \labelcref{eq-representation_relevee_a_H}) are isomorphic. 
At that point we do not know if that is the case.

\section{Construction of the projective representations \texorpdfstring{$\rho_{g,n}^1$}{rhogn1} and \texorpdfstring{$\rho_{g,n}$}{rhogn}}\label{sec-construction_representations_projectives}

Recall that $H$ is a finite-dimensional factorizable ribbon Hopf algebra. 
A reference on the construction of the projective representations in the case $n = 0$ is \cite[§5.2]{faitg_derived_2026}.

\subsection{The maps \texorpdfstring{$\mathfrak{i}_\gamma$}{igamma} and the action of \texorpdfstring{$\Gamma_{g,n}^1$}{Gammagn1}}\label{subsec-applications_igamma_et_action_de_Gammagn1}

In \Cref{subsec-Lgn}, we defined the embeddings $\mathfrak{i}_{x_r}$ where $x_r$ is a curve $a_j \subset \Sigma_{g,n}^1$ (resp. $b_j$, or $m_{g+k}$) which is simple, closed, based (at a point in $\partial\Sigma_{g,n}^1$), and oriented (see \labelcref{eq-definition_injections_canoniques} and \Cref{fig-courbes_standards_Lgn}).

\smallskip
\noindent
To define the embeddings $\mathfrak{i}_{x_r^{-1}}$ along the same curve $a_j$ (resp.  $b_j$ or $m_{g+k}$) equipped with the opposite orientation, denoted $a_j^{-1}$ (resp.  $b_j^{-1}$ or $m_{g+k}^{-1}$), we need an antipode of $\mathcal{L}_{0, 1}(H)$ (i.e. an inverse of $\id_{\mathcal{L}_{0, 1}(H)}$ for the convolution product).

\smallskip
\noindent
The algebra $\mathcal{L}_{0, 1}(H)$ is equipped with the antipode defined by the following map: 
\[
\begin{array}{l|ccl}
	S_{\mathcal{L}_{0,1}(H)}: & \mathcal{L}_{0,1}(H) & \longrightarrow & \mathcal{L}_{0,1}(H) \\
	& \varphi & \longmapsto & \sum_i S_{H^*} \bigl( S(r_i) \triangleright \varphi \triangleleft r^i u^{-1} \bigr),
\end{array}
\]
where $u = \sum_{i} S(r^i)r_i$ is the Drinfeld element and $S_{H^*}$ is the antipode of $H^*$. 

\smallskip
\noindent
We define the embeddings $\mathfrak{i}_{a_j^{-1}}$ (resp. $\mathfrak{i}_{b_j^{-1}}$ and $\mathfrak{i}_{m_{g+k}^{-1}}$) for the curves $a_j$ (resp. $b_j$ and $m_{g+k}$) equipped with the opposite orientation by:
\[ \mathfrak{i}_{a_j^{-1}} = \mathfrak{i}_{a_j} \circ S_{\mathcal{L}_{0,1}(H)} \quad (\text{resp. } \mathfrak{i}_{b_j^{-1}} = \mathfrak{i}_{b_j} \circ S_{\mathcal{L}_{0,1}(H)} ~ \text{ and } ~ \mathfrak{i}_{m_{g+k}^{-1}} = \mathfrak{i}_{m_{g+k}} \circ S_{\mathcal{L}_{0,1}(H)}). \]

\medskip

From now on, we fix a base point on $\partial \Sigma_{g,n}^1$ for $\pi_1(\Sigma_{g,n}^1)$.
A simple closed oriented and based (at this base point of $\partial{\Sigma_{g,n}^1}$) curve of $\Sigma_{g,n}^1$ is said to be \emph{positively oriented} if its orientation near the basepoint looks like 
\tikz[baseline=-2ex]{
    \fill (0,0) circle (1.5pt);
    \draw[-{Stealth}, myred] (0,0) -- (-0.3,-0.3);
	\draw[myred] (-0.3,-0.3) -- (-0.5,-0.5);
    \draw[-{Stealth}, myred] (0.5,-0.5) -- (0.15,-0.15);
	\draw[myred] (0.2,-0.2) -- (0,0);
} when we draw it in the ribbon surface (see \Cref{fig-courbes_standards_Lgn_ruban}); otherwise the curve is said to be \emph{negatively oriented}. 
\begin{figure}[H]
    \centering
	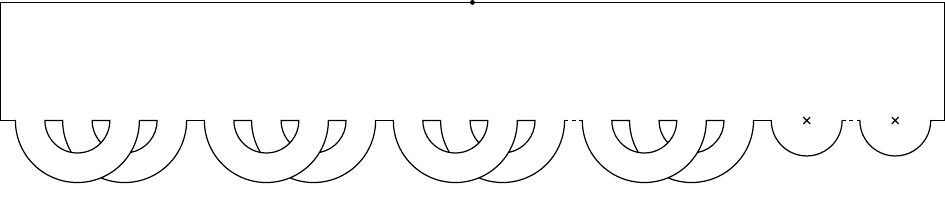
    \caption{Example of a simple closed based and positively oriented curve in the ribbon surface, with normalization $N(z) = 3 - 3 = 0$}
	\label{fig-exemple_courbe_positivement_orientee}
\end{figure}
An element $\gamma \in \pi_1(\Sigma_{g,n}^1)$ is said to be \emph{simple} if it is represented by a simple closed and based curve which we still denote by $\gamma$. 
Let $\gamma \in \pi_1(\Sigma_{g,n}^1)$ be simple, see \Cref{fig-exemple_courbe_positivement_orientee} for an example. 
The \emph{normalization} of $\gamma$ is $N(\gamma) \in \mathbb{Z}$ defined by: 
\begin{equation*}
	N(\gamma) :=
    \begin{cases}
        N_\cup(\gamma) - N_\cap(\gamma) & \text{ if $\gamma$ is positively oriented,} \\
        N_\cap(\gamma^{-1}) - N_\cup(\gamma^{-1}) & \text{ if $\gamma$ is negatively oriented,}
	\end{cases}
\end{equation*}
where: 
\begin{itemize}
    \item $N_\cup(\gamma)$ is the number of times the curve $\gamma$ makes a cup from right to left when drawn in the ribbon surface,
 	\item $N_\cap(\gamma)$ is the number of times the curve $\gamma$ makes a cap when drawn in the ribbon surface without counting the potential cap formed at the base point.
\end{itemize}
The integer $N(\gamma)$ does not depend on the choice of a simple representative of $\gamma$ in $\Sigma_{g,n}^1$ (see \cite[App. B.2]{faitg_derived_2026}).
Consequently, the integer $N(\gamma)$ is well-defined for all simple elements $\gamma \in \pi_1(\Sigma_{g,n}^1)$.

\begin{definition}\label{def-injection_igamma}
	Let $\gamma \in \pi_1(\Sigma_{g,n}^1)$ be simple. 
    Write $\gamma = \gamma_1 \cdots \gamma_k$ where each $\gamma_i$ is a usual generator ($a_j$, $b_j$, or $m_{g+k}$) of $\pi_1(\Sigma_{g,n}^1)$ or its inverse. 
	We define: 
    \[
    \begin{array}{l|ccl}
        \mathfrak{i}_{\gamma}: & \mathcal{L}_{0,1}(H) & \longrightarrow & \mathcal{L}_{g,n}(H) \\
		& \varphi & \longmapsto & \varphi_{(1)}(v^{N(\gamma)}) \mathfrak{i}_{\gamma_1}(\varphi_{(2)}) \cdots  \mathfrak{i}_{\gamma_k}(\varphi_{(k+1)}).
    \end{array}
    \]
\end{definition}
\noindent
The map $\mathfrak{i}_{\gamma}$ does not depend on the factorization $\gamma = \gamma_1 \cdots \gamma_k$ (see \cite[Prop. 5.14]{faitg_derived_2026}).
It satisfies the following properties:

\begin{proposition}[{{\cite[Lem. 5.4, Prop. 5.6, and Cor. 5.8]{faitg_derived_2026}}}]\label{prop-proprietes_igamma}
	Let $\gamma \in \pi_1(\Sigma_{g,n}^1)$ be simple.
    \begin{enumerate}[leftmargin=*]
        \item 
        \begin{enumerate}
            \item If $\gamma$ is positively oriented, then $\mathfrak{i}_{\gamma}$ is a morphism of $H$-module-algebras.
			\item If $\gamma$ is negatively oriented, then $\mathfrak{i}_\gamma$ is a morphism of $H$-modules and satisfies: 
            \[ \forall \varphi, \psi \in \mathcal{L}_{0, 1}(H), ~ \mathfrak{i}_\gamma(\varphi \psi) = \sum_i \mathfrak{i}_\gamma \bigl( \coad^r(r_i)(\psi) \bigr) \, \mathfrak{i}_\gamma \bigl( \coad^r(r^i)(\varphi) \bigr). \]
		\end{enumerate}
        \item If $\gamma \in \pi_1(\Sigma_{g,n}^1)$ is positively oriented, then $\mathfrak{i}_{\gamma^{-1}} = \mathfrak{i}_{\gamma} \circ S_{\mathcal{L}_{0, 1}(H)}$.
	\end{enumerate}
\end{proposition}

\medskip

For all $f \in \Gamma_{g,n}^1$, all simple elements $\gamma \in \pi_1(\Sigma_{g,n}^1)$, and all $\varphi \in \mathcal{L}_{0,1}(H)$ define:
\begin{equation}\label{eq-definition_action_Gammagn1_sur_Lgn}
\tilde{f} \bigl( \mathfrak{i}_\gamma(\varphi) \bigr) := \mathfrak{i}_{f(\gamma)}(\varphi).
\end{equation}

\begin{proposition}[{{\cite[(81) and §5.2.1]{faitg_derived_2026}}}]\label{prop-action_de_groupes_et_automorphisme_dalgebres}
	For all $f \in \Gamma_{g,n}^1$, the map $\tilde{f}$ is an algebra automorphism of $\mathcal{L}_{g, n}(H)$, and the map:
    \[ 
    \begin{array}{l|ccl}
		\widetilde{\bullet}: & \Gamma_{g,n}^1 & \longrightarrow & \Aut_{alg} \bigr( \mathcal{L}_{g, n}(H) \bigl) \\
        & f & \longmapsto & \tilde{f}
    \end{array} 
    \]
    is a group morphism.
\end{proposition}
\noindent
A proof of this result uses the isomorphism between $\mathcal{L}_{g, n}(H)$ and the stated skein algebra $\mathcal{S}^{\mathrm{st}}_H(\Sigma_{g,n}^{1, \bullet})$ (see \Cref{thm-isomorphisme_avec_stated_skein} and \cite[(81) and (89)]{faitg_derived_2026}).
Another proof, that does not rely on this isomorphism, involves tedious computations in $\mathcal{L}_{g, n}(H)$ (see \cite[Prop. 5.1]{faitg_projective_2020} for the case $n = 0$).

\medskip
\noindent
Recall that $v \in H$ is the ribbon element of $H$.
Since $H$ is assumed to be factorizable, $\Phi_{0,1}$ is invertible. 

\begin{proposition}\label{prop-action_twists_de_Dehn_non_separants_comme_conjugaison_et_proprietes}
    For any Dehn twist $\tau_\gamma$ along a non-separating simple closed curve $\gamma \subset \Sigma_{g,n}^1$, we have:
	\[ \forall x \in \mathcal{L}_{g, n}(H), ~ \tilde{\tau}_\gamma(x) = \hat{\tau}_\gamma \, x \, \hat{\tau}_\gamma^{-1}, \]
	where $\hat{\tau}_\gamma := \mathfrak{i}_{\gamma'}(\Phi_{0,1}^{-1}(v^{-1}))$ and $\gamma' \in \pi_1(\Sigma_{g,n}^1)$ is simple with the same free homotopy class as $\gamma$.
	Moreover, $\hat{\tau}_\gamma$ is independent of the orientation of $\gamma' \in \pi_1(\Sigma_{g,n}^1)$, and for any left integral $\lambda : H \to \mathbb{K}$ we have: 
    \[ \Phi_{0,1}^{-1}(v^{-1}) = \frac{v \triangleright \lambda}{\lambda(v)}. \]
\end{proposition}
\noindent
From now on, we will always write: 
\begin{equation}\label{eq:lambda^v_et_lambda^v-1}
	\lambda^v := \frac{v \triangleright \lambda}{\lambda(v)}, \qquad \lambda^{v^{-1}} := \frac{v^{-1} \triangleright \lambda}{\lambda(v^{-1})}.
\end{equation}
Before proving \Cref{prop-action_twists_de_Dehn_non_separants_comme_conjugaison_et_proprietes}, we deduce the following corollary: 
\begin{corollaire}\label{cor-action_mapping_class_comme_conjugaison}
	For all $f \in \Gamma_{g,n}^1$, the automorphism $\tilde{f}$ is interior. 
\end{corollaire}

\begin{proof}
	We know that $\Gamma_{g,n}^1$ is generated by the Dehn twists along non-separating simple closed curves of $\Sigma_{g,n}^1$ (see \Cref{thm-presentation_de_Gervais}).
	Let $f \in \Gamma_{g,n}^1$ and $w_f := \prod_{i = 1}^n \tau_{\gamma_i}^{\varepsilon_i}$ be a word representing $f$, where $\varepsilon_i \in \{\pm 1 \}$ and $\gamma_i \subset \Sigma_{g,n}^1$ is a non-separating simple closed curve, for all $i \in [\![ 1, n ]\!] $.
    Using \Cref{prop-action_de_groupes_et_automorphisme_dalgebres,prop-action_twists_de_Dehn_non_separants_comme_conjugaison_et_proprietes}, we have:
	\[ \tilde{f}(x) = \tilde{\tau}_{\gamma_1}^{\varepsilon_1} \circ \cdots \circ \tilde{\tau}_{\gamma_n}^{\varepsilon_n} (x) = \biggl( \prod_{i = 1}^n \hat{\tau}_{\gamma_i}^{\varepsilon_i} \biggr) x  \biggl( \prod_{i = 1}^n \hat{\tau}_{\gamma_i}^{\varepsilon_i} \biggr)^{-1} = \hat{f} x \hat{f}^{-1}, \]
 	where $\hat{f} := \prod_{i = 1}^n \hat{\tau}_{\gamma_i}^{\varepsilon_i}$.
\end{proof}

\begin{remarque}\label{not:representation_projective_vient_de_Skolem-Noether}
    When $n = 0$, it is not surprising that $\Gamma_g^1$ acts on $\mathcal{L}_{g, 0}(H)$ by conjugation.
	Indeed, we saw in \Cref{thm-isomorphisme_phign} the isomorphism between $\mathcal{L}_{g,0}(H)$ and $M_{\dim(H)}(\mathbb{K})^{\otimes g}$. 
	Consequently, the Skolem-Noether theorem (see \cite[Thm. 3.14]{farb_noncommutative_1993}) applies, so every automorphism of $\mathcal{L}_{g, 0}(H)$ is interior.  
\end{remarque}

The rest of this section is devoted to the proof of \Cref{prop-action_twists_de_Dehn_non_separants_comme_conjugaison_et_proprietes}.
The case $n = 0$ has been treated in \cite[§5.2.2]{faitg_derived_2026}. 
We recall arguments for completeness.

\smallskip
\noindent
Recall that a \emph{left integral} $\lambda$ of $H$ is an element of $H^*$ that is invariant under the regular left action of $H^*$, i.e. for all $\varphi \in H^*$, $\varphi \star \lambda = \varepsilon_{H^*}(\varphi)\lambda$, or equivalently: 
\begin{equation}\label{eq-definition_integrale_a_gauche}
	\forall h \in H, ~ h_{(1)} \lambda(h_{(2)}) = \lambda(h)1_H.
\end{equation} 
Similarly, we can define right integrals by the relation $\lambda \star \varphi = \varepsilon_{H^*}(\varphi)\lambda$, for all $\varphi \in H^*$.
We will use without restriction the following properties of integrals (see \cite[§2.1]{montgomery_hopf_1993} or \cite[Chap. 10]{radford_hopf_2012}):
\begin{enumerate}[label=(\alph*),ref=\alph*]
	\item\label{item-espace_des_integrales_dimension_1} Since $H$ is finite-dimensional, the space of left integrals (resp. right integrals) is a 1-dimensional vector space.
    \item\label{item-passage_integrale_a_gauche_a_integrale_a_droite} Let $\lambda$ be a left integral (resp. right integral). Then $\lambda \circ S^{\pm 1}$ is a right integral (resp. a left integral).
    \item\label{item-integrale_invariante_coad^r} Let $\lambda$ be a left integral. Then $\lambda \in \inv \coad^r$. 
    Since this fact is not standard, let us give details. 
	First, we compute $\Phi_{0,1} (v^{-1} \triangleright \lambda)$:
	\begin{equation}\label{eq-Phi01lambda^vinverse}
		\begin{aligned}
			\Phi_{0,1} (v^{-1} \triangleright \lambda) & = (\lambda \otimes \id_H ) \bigl( R R_{21} (v^{-1} \otimes 1) \bigr) \\
			& \overset{\text{\labelcref{eq-propriete_v}}}{=} (\lambda \otimes \id_H) \bigl( \Delta^{op}(v^{-1})(v \otimes v)(v^{-1} \otimes 1) \bigr) \\
			& ~~ = (\lambda \otimes \id_H) \bigl( \Delta^{op}(v^{-1})(1 \otimes v) \bigr) = (\lambda \otimes \id_H) \bigl( \Delta^{op}(v^{-1}) \bigr) v = \lambda(v^{-1}) \, v.
		\end{aligned}
	\end{equation}
	Note that it is using the ribbon structure of $H$.
	Since $v$ is central, $v$ is invariant under the right adjoint action ($\ad^r(h)(x) := S(h_{(1)}) x h_{(2)}$).
	Moreover, the factorizability of $H$ implies that $\Phi_{0,1}$ is an isomorphism of $H$-modules, thus $v^{-1} \triangleright \lambda \in \inv \coad^r$. 
	It follows from the centrality of $v$ and \labelcref{eq-reformulation_invariance_coad^r} that $v \triangleright (v^{-1} \triangleright \lambda) = \lambda \in \inv \coad^r$. 
	\item\label{item-lambda_circ_S^2_egale_lambda} Let $\lambda$ be a left integral. From \labelcref{item-espace_des_integrales_dimension_1,item-passage_integrale_a_gauche_a_integrale_a_droite,item-integrale_invariante_coad^r}, by using \labelcref{eq-propriete_g} it follows that $\lambda \circ S^2 = \lambda$. 
\end{enumerate}

\begin{remarque}\label{rem-lambda_est_integrale_a_droite}
    Since $\lambda \in \inv \coad^r$, \Cref{lem-produit_L01_si_un_est_invariant} implies that $\lambda$ is also a right integral. 
	This is not true in general; here the result is due to the factorizability of $H$. 
\end{remarque}

\begin{proof}[Proof (\Cref{prop-action_twists_de_Dehn_non_separants_comme_conjugaison_et_proprietes})]
    Let $\lambda$ be a left integral of $H$.
	Let us begin by showing that $\Phi_{0,1}(v \triangleright \lambda) = \lambda(v) \, v^{-1}$.
	The following computations are based on those in \cite[Prop. 2.3.4]{faitg_mapping_2019}. Since some points differ, we give details.
    Using \labelcref{item-integrale_invariante_coad^r} and \labelcref{eq-reformulation_invariance_coad^r} we see that $\lambda^r := \lambda \circ S^{-1}$ satisfies:
	\begin{equation}\label{eq-reformulation_invariance_coad^l}
        \forall x, y \in H, ~ \lambda^r(xy) = \lambda^r \bigl( y S^{-2}(x) \bigr).
    \end{equation}
	Since $v$ is central, the relation \labelcref{eq-reformulation_invariance_coad^l} holds for $v \triangleright \lambda^r$.
	These identities and \labelcref{eq-propriete_v} allow us to compute $\Phi_{0,1} (v \triangleright \lambda)$: 
    \begin{equation}\label{eq-Phi01lambda^v}
		\begin{aligned}
			\Phi_{0,1} (v \triangleright \lambda) = (v \triangleright \lambda \otimes \id_H)(R R_{21}) & = \sum_{i,j}\lambda (r_i r^j v) r^i r_j \\
			& = \sum_{i,j}\lambda^r \bigl( S(v) S(r^j) S(r_i) \bigr) r^i r_j \\
			& \overset{\text{\labelcref{eq-reformulation_invariance_coad^l}}}{=} \sum_{i,j}\lambda^r \bigl( v S(r_i) S^{-1}(r^j) \bigr) r^i r_j \\
			& = (\lambda^r \otimes \id_H) \bigl( R^{-1} R_{21}^{-1} (v \otimes 1_H) \bigr) \\
			& = (\lambda^r \otimes \id_H) \bigl( \Delta(v) (v^{-1} \otimes v^{-1}) (v \otimes 1_H) \bigr) \\
			& = (\lambda^r \otimes \id_H) \bigl( \Delta(v) (1_H \otimes v^{-1}) \bigr) \\
			& = (\lambda^r \otimes \id_H) \bigl( \Delta(v) \bigr) v^{-1} = \lambda^r(v) \, v^{-1} = \lambda(v) \, v^{-1}.
		\end{aligned}
	\end{equation}
	Thus we have  $\hat{\tau}_\gamma := \mathfrak{i}_{\gamma'}(\Phi_{0,1}^{-1}(v^{-1})) = \mathfrak{i}_{\gamma'} (\lambda^v)$.

	Let us now show that $\hat{\tau}_\gamma$ is independent of the orientation of $\gamma'$. 
    This relies on the equality $S_{\mathcal{L}_{0, 1}(H)}(\lambda^v) = \lambda^v$ (see \cite[Lem. 5.20]{faitg_derived_2026}).
	Indeed, if $\gamma'$ is positively oriented, we have $\mathfrak{i}_{\gamma'^{-1}}(\lambda^v) \overset{\text{\labelcref{prop-proprietes_igamma}}}{=} \mathfrak{i}_{\gamma'}(S_{\mathcal{L}_{0, 1}(H)}(\lambda^v)) = \mathfrak{i}_{\gamma'}(\lambda^v)$.
    Thus $\hat{\tau}_\gamma$ does not depend on the orientation of $\gamma'$.

	Finally, let us show that $\forall x \in \mathcal{L}_{g, n}(H), ~ \tilde{\tau}_\gamma(x) = \hat{\tau}_\gamma \, x \, \hat{\tau}_\gamma^{-1}$.  
    We begin by showing the result for $\gamma = a_1$. 
    By the definition of the action of $\Gamma_{g,n}^1$ on $\mathcal{L}_{g, n}(H)$ (see \labelcref{eq-definition_action_Gammagn1_sur_Lgn}) and the fact that $\tau_{a_1}(b_1) = b_1 a_1$, we must show: 
	\begin{alignat}{3}
        & \hat{\tau}_{a_1} \mathfrak{i}_{a_1}(\varphi) = \mathfrak{i}_{a_1}(\varphi) \hat{\tau}_{a_1}, \qquad & & \hat{\tau}_{a_1} \mathfrak{i}_{b_1}(\varphi) = \varphi_{(1)}(v^{-1}) \mathfrak{i}_{b_1}(\varphi_{(2)}) \mathfrak{i}_{a_1}(\varphi_{(3)}) \hat{\tau}_{a_1}, \qquad & & \label{eq-conjugaison_a1_et_b1}\\
        & \hat{\tau}_{a_1} \mathfrak{i}_{a_j}(\varphi) = \mathfrak{i}_{a_j}(\varphi) \hat{\tau}_{a_1}, \qquad & & \hat{\tau}_{a_1} \mathfrak{i}_{b_j}(\varphi)= \mathfrak{i}_{b_j}(\varphi) \hat{\tau}_{a_1}, \quad & & \forall j \geq 2, \label{eq-conjugaison_aj_et_bj}\\
		& \hat{\tau}_{a_1} \mathfrak{i}_{m_{g+k}}(\varphi) = \mathfrak{i}_{m_{g+k}}(\varphi) \hat{\tau}_{a_1}, \qquad & & & & \forall k \geq 1. \label{eq-conjugaison_mg+k}
    \end{alignat}
	A direct adaptation of the proof of \cite[Prop. 5.21]{faitg_derived_2026} allows us to show \labelcref{eq-conjugaison_a1_et_b1,eq-conjugaison_aj_et_bj}. 
	It remains to show \labelcref{eq-conjugaison_mg+k}.
	We have:
    \begin{equation*}
		\mathfrak{i}_{m_{g+k}}(\varphi) \hat{\tau}_{a_1} = \mathfrak{i}_{m_{g+k}}(\varphi) \, \mathfrak{i}_{a_1} (\lambda^v) = \sum_{a,b,c,d} \mathfrak{i}_{a_1} \bigl( S(r_c r_d) \triangleright \lambda^v \triangleleft r_a r_b \bigr) \, \mathfrak{i}_{m_{g+k}} \bigl( S (r^a r^c) \triangleright \varphi \triangleleft r^b r^d \bigr).
    \end{equation*} 
	To simplify the notations, we will only consider (with implicit summations):
	\[ \bigl( S(r_c r_d) \triangleright \lambda^v \triangleleft r_a r_b \bigr) \otimes \bigl( S (r^a r^c) \triangleright \varphi \triangleleft r^b r^d \bigr). \]
    In the following, $\langle \psi, h \bullet g \rangle$ means $g \triangleright \psi \triangleleft h$:
    \begin{align*}
		\bigl( S(r_c r_d) \triangleright \lambda^v \triangleleft r_a r_b \bigr) \otimes \bigl( S (r^a r^c) \triangleright \varphi \triangleleft r^b r^d \bigr) & = \bigl\langle \lambda^v, r_a r_b \bullet S(r_d) S(r_c) \bigr\rangle \otimes \bigl\langle \varphi, r^b r^d \bullet S(r^c) S(r^a) \bigr\rangle \\
		& = \bigl\langle \lambda^v, r_b \bullet S(r_d) S(r_c) S^{2}(r_a) \bigr\rangle \otimes \bigl\langle \varphi, r^b r^d \bullet S(r^c) S(r^a) \bigr\rangle \\ 
		& = \bigl\langle \lambda^v, r_b \bullet S(r_d) S \big( S(r_a) r_c \big) \bigr\rangle \otimes \bigl\langle \varphi, r^b r^d \bullet S(r^a r^c) \bigr\rangle \\ 
		& \overset{\text{\labelcref{eq-R-matrice_et_antipode}}}{=} \bigl\langle \lambda^v, r_b \bullet S(r_d) \bigr\rangle \otimes \langle \varphi, r^b r^d \bullet \rangle \\ 
		& = \bigl\langle \lambda^v, \bullet \, S(r_d) S^2(r_b) \bigr\rangle \otimes \langle \varphi, r^b r^d \bullet \rangle \\ 
		& = \bigl\langle \lambda^v, \bullet \, S \bigl( S(r_b) r_d \bigr) \bigr\rangle \otimes \langle \varphi, r^b r^d \bullet \rangle \overset{\text{\labelcref{eq-R-matrice_et_antipode}}}{=} \langle \lambda^v, \bullet \rangle \otimes \langle \varphi, \bullet \rangle = \lambda^v \otimes \varphi.
    \end{align*}
	In the second and fifth equality we used the fact that $\lambda^v \in \inv \coad^r$, which is an immediate consequence from the centrality of $v$ and the fact that $\lambda \in \inv \coad^r$. 
    Thus we obtain: 
	\[ \mathfrak{i}_{m_{g+k}}(\varphi) \hat{\tau}_{a_1} = \mathfrak{i}_{a_1} (\lambda^v) \mathfrak{i}_{m_{g+k}}(\varphi) = \hat{\tau}_{a_1} \mathfrak{i}_{m_{g+k}}(\varphi). \]
	To extend this result to any non-separating simple closed curve $\gamma \subset \Sigma_{g,n}^1$, we will use the following two properties of $\Gamma_{g,n}^1$. 
	Given a non-separating simple closed curve $\gamma \subset \Sigma_{g,n}^1$, we have:
    \begin{enumerate}
		\item There exists $f \in \Gamma_{g,n}^1$ such that $f(a_1) = \gamma$ (change of coordinates principle, see \cite[§1.3]{farb_primer_2012}),
		\item For any $f \in \Gamma_{g,n}^1$, we have $\tau_{f(\gamma)} = f \tau_\gamma f^{-1}$ (see \cite[Fact 3.7]{farb_primer_2012}).
    \end{enumerate}
    Therefore, we have $\tilde{\tau}_\gamma (x) = \tilde{\tau}_{f(a_1)} (x) = \widetilde{f \tau_{a_1} f^{-1}} (x) = \tilde{f} \circ \tilde{\tau}_{a_1} \circ \tilde{f}^{-1} (x)$.
	Thus, if we set $y = \tilde{f} (x)$ we obtain: 
	\[ \tilde{\tau}_\gamma (y) = \tilde{f} \circ \tilde{\tau}_{a_1} (x) = \tilde{f}(\hat{\tau}_{a_1} \, x \, \hat{\tau}_{a_1}^{-1}) = \tilde{f}(\hat{\tau}_{a_1})  \tilde{f}(x) \tilde{f}(\hat{\tau}_{a_1}^{-1}) = \hat{\tau}_\gamma \, y \, \hat{\tau}_\gamma^{-1}, \]
    which concludes the proof.
\end{proof}

The following lemma will be useful in \Cref{subsec-calcul_relations}.

\begin{lemme}\label{lem-lambda^v_star_lambda^v-1_egale_epsilon}
    We have the following equality: $\lambda^v \star \lambda^{v^{-1}} = \lambda^{v^{-1}} \star \lambda^v = \varepsilon$.
\end{lemme}

\begin{proof}
	Using \labelcref{eq-Phi01lambda^vinverse,eq-Phi01lambda^v} we get: 
	\[ 1_H = v v^{-1} = \Phi_{0,1}(\lambda^{v^{-1}}) \Phi_{0,1}(\lambda^v). \]
	Since $\Phi_{0,1}$ is an embedding of algebras, we have $\varepsilon = \lambda^{v^{-1}} \lambda^v$. 
	Furthermore, since $\lambda \in \inv \coad^r$ and $v$ is central we have $\lambda^v \in \inv \coad^r$, and hence \Cref{lem-produit_L01_si_un_est_invariant} implies $\lambda^{v^{-1}} \lambda^v = \lambda^{v^{-1}} \star \lambda^v = \lambda^v \lambda^{v^{-1}}$, which proves the result.
\end{proof}

\subsection{The projective representation \texorpdfstring{$\rho_{g,n}^1$}{rhogn1} of \texorpdfstring{$\Gamma_{g,n}^1$}{Gammagn1}}\label{subsec-representation_projective_rhogn1}

The goal of this subsection is to define a projective representation of the mapping class group $\Gamma_{g,n}^1$.
Let $F_\tau$ denote the free group generated by all Dehn twists along non-separating simple closed curves of $\Sigma_{g,n}^1$ if $g \geq 2$, and the free group generated by $\tau_a$ and $\tau_b$ if $g = 1$ and $n = 0$. 
For any collection $X := \{ X_i \}_{i = 1, \cdots, n}$ of $H$-modules, \Cref{prop-action_twists_de_Dehn_non_separants_comme_conjugaison_et_proprietes} allows us to define the following map:
\begin{equation}\label{eq-application_rho_chapeau}
    \begin{array}{l|ccccl}
        \hat{\rho}_{g,n}^1(X): & F_\tau & \longrightarrow & \mathcal{L}_{g, n}(H)^\times & \longrightarrow & \GL \bigl( (H^*)^{\otimes g} \otimes X_1 \otimes \cdots \otimes X_n \bigr) \\
        & \tau_\gamma & \longmapsto & \hat{\tau}_\gamma & \longmapsto & \mathcal{R}_{g,n}^1(X)(\hat{\tau}_\gamma),
    \end{array}
\end{equation}
where $\mathcal{R}_{g,n}^1(X)$ is the representation of $\mathcal{L}_{g, n}(H)$ defined in \Cref{prop-definition_representation_combinee}.
By \Cref{lem-f_chapeau_dans_Lgninv} below, we will see that the elements $\hat{\tau}_\gamma$ belong to a certain subalgebra $\mathcal{L}_{g, n}^{\mathrm{inv}}(H) \subset \mathcal{L}_{g, n}(H)$.
We will show that the map $\hat{\rho}_{g,n}^1(V)$, for $V := \{ V_i \}_{i = 1, \dots, n}$ a collection of simple $H$-modules, induces a projective representation of $\Gamma_{g,n}^1$.
For this, we will use the Gervais presentation of $\Gamma_{g,n}^1$ (see \Cref{thm-presentation_de_Gervais}) for all $g \geq 2$, and the presentation $\langle \tau_a, \tau_b ~ \vert ~ \tau_a \tau_b \tau_a =\tau_b \tau_a \tau_b \rangle$ of $\Gamma_1^1$ (see \labelcref{eq-presentation_Gamma11_et_Gamma1}) for $g = 1$ and $n = 0$.

\smallskip

We must verify that the images by $\hat{\rho}_{g,n}^1(X)$ of all the relators of these presentations are scalar multiples of the identity of $(H^*)^{\otimes g} \otimes X_1 \otimes \cdots \otimes X_n$. 
Let us take an arbitrary relator $r = \prod_{1}^{n} \tau_{\gamma_i}^{\varepsilon_i}$, where $\varepsilon_i = \pm 1$. 
Since $\widetilde{\bullet}$ defines a group action of $\Gamma_{g,n}^1$ on $\mathcal{L}_{g, n}(H)$ (see \Cref{prop-action_de_groupes_et_automorphisme_dalgebres}), for all $x \in \mathcal{L}_{g, n}(H)$ we have:
\[ x = \tilde{r}(x) = \tilde{\tau}_{\gamma_1}^{\varepsilon_1} \circ \cdots \circ \tilde{\tau}_{\gamma_n}^{\varepsilon_n}(x) \overset{\text{\labelcref{prop-action_twists_de_Dehn_non_separants_comme_conjugaison_et_proprietes}}}{=} (\hat{\tau}_{\gamma_1}^{\varepsilon_1} \cdots \hat{\tau}_{\gamma_n}^{\varepsilon_n}) \, x \, (\hat{\tau}_{\gamma_1}^{\varepsilon_1} \cdots \hat{\tau}_{\gamma_n}^{\varepsilon_n})^{-1}. \]
Therefore, if we denote by $Z(\mathcal{L}_{g, n}(H))$ the center of $\mathcal{L}_{g, n}(H)$, we obtain: 
\begin{equation}\label{eq-image_de_relation_est_dans_centre_de_Lgn}
	\hat{r} := \hat{\tau}_{\gamma_1}^{\varepsilon_1} \cdots \hat{\tau}_{\gamma_n}^{\varepsilon_n} \in Z \bigl( \mathcal{L}_{g, n}(H) \bigr).
\end{equation}
In general, $Z(\mathcal{L}_{g, n}(H)) \neq \mathbb{K} \id_{\mathcal{L}_{g, n}(H)}$.
However, the Schur lemma (see \cite[Prop. 2.3.9 and Cor. 2.3.10]{etingof_introduction_2011}) implies that every central element of $\mathcal{L}_{g, n}(H)$ acts by scalar multiplication on every irreducible representation of $\mathcal{L}_{g, n}(H)$. 
Consequently, since $\mathcal{R}_{g,n}^1(V)$ is irreducible for any collection $V := \{ V_i \}_{i = 1, \cdots, n}$ of simple $H$-modules (see \Cref{rem-isomorphisme_combinee_phign_et_irreductible}), for all relators $r$ in the Gervais presentation $\hat{\rho}_{g,n}^1(V)(r)$ acts on $(H^*)^{\otimes g} \otimes V_1 \otimes \cdots \otimes V_n$ by a scalar, so we obtain a projective representation of $\Gamma_{g,n}^1$: 

\begin{definition}\label{def-representation_projective_rhogn1}
    Let $V := \{ V_i \}_{i = 1, \cdots, n}$ be a collection of simple $H$-modules.
    We denote by $\rho_{g,n}^1(V) : \Gamma_{g,n}^1 \to \PGL((H^*)^{\otimes g} \otimes V_1 \otimes \cdots \otimes V_n)$ the projective representation defined by the following commutative diagram, where $\pi$ is the quotient map: 
    \[ 
    \begin{tikzcd}[column sep=large]
		F_\tau \arrow[r, "\hat{\rho}_{g,n}^1(V)"]  \arrow[d, "\pi"] & \GL \bigl( (H^*)^{\otimes g} \otimes V_1 \otimes \cdots \otimes V_n \bigr)  \arrow[d, "\pi"] \\
		\Gamma_{g,n}^1 \arrow[r, "\rho_{g,n}^1(V)"] & \PGL \bigl( (H^*)^{\otimes g} \otimes V_1 \otimes \cdots \otimes V_n \bigr) 
    \end{tikzcd}
    \] 
\end{definition}
\noindent
In the following, for all collections $ V := \{ V_i \}_{i = 1, \cdots, n}$ of simple $H$-modules we denote: 
\begin{equation}\label{eq-notation_representations_projectives_et_espaces_de_representations_Gammagn1}
	\rho_{g,n}^1 := \rho_{g,n}^1(V), \qquad V_{g,n}^1 := (H^*)^{\otimes g} \otimes V_1 \otimes \cdots \otimes V_n.
\end{equation}

\begin{remarque}\label{rem-ordre_tau_chapeau_gamma}
	\begin{enumerate}[leftmargin=*]
		\item Recall from \Cref{prop-action_twists_de_Dehn_non_separants_comme_conjugaison_et_proprietes} that for any non-separating simple closed curve $\gamma \subset \Sigma_{g,n}^1$, we have $\hat{\tau}_\gamma = \mathfrak{i}_\gamma(\Phi_{0,1}^{-1}(v^{-1}))$. 
		This, the fact that $\Phi_{0,1}$ is a morphism of algebras, the first statement of \Cref{prop-proprietes_igamma}, and \labelcref{eq-application_rho_chapeau} immediately imply that the order of $\hat{\rho}_{g,n}^1(V)(\tau_{\gamma})$ is given by the order of the ribbon element $v \in H$.
		\item When $n = 0$ the collection $V$ of simple $H$-modules is empty, and since $\mathcal{L}_{g, 0}(H)$ is isomorphic to a matrix algebra (see \Cref{thm-isomorphisme_phign}), $(H^*)^{\otimes g}$ is already an irreducible representation and $Z(\mathcal{L}_{g, 0}(H)) = \mathbb{K}\id_{\mathcal{L}_{g, 0}(H)}$.
	\end{enumerate}
\end{remarque}

In \Cref{sec-calcul_extension_associee} we construct a minimal central extension of $\Gamma_{g,n}^1$ that linearizes the projective representation $\rho_{g,n}^1$ of $\Gamma_{g,n}^1$.
Since the representation $\mathcal{R}_{g,n}^1(H)$ lifts $\mathcal{R}_{g,n}^1(V)$ for all collections $V$ of simple $H$-modules (see \Cref{prop-representation_relevee_a_H}), we will use $\mathcal{R}_{g,n}^1(H)$ to compute the action of any relation in the Gervais presentation of $\Gamma_{g,n}^s$. 

\begin{figure}[H]
    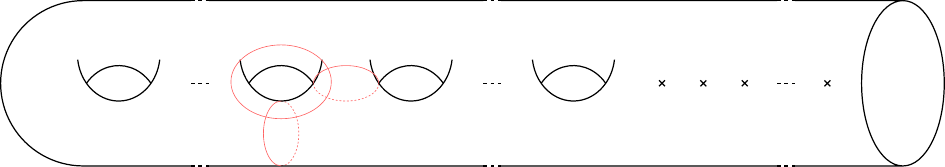
    \caption{Curves $a_j$, $b_j$, and $c_j$ of $\Sigma_{g,n}^1$.}
    \label{fig-courbes_ai_bi_ci}
\end{figure}

\begin{proposition}\label{prop-formules_ai_bi_ci}
	We have the following formulas: 
	\begin{align*}
        & \mathcal{R}_{g,n}^1(H)(\hat{\tau}_{a_j}) (\psi_1 \otimes \cdots \otimes \psi_g \otimes h_1 \otimes \cdots \otimes h_n) \\
		& \qquad \qquad \qquad \qquad \qquad \qquad \qquad \quad = \psi_1 \otimes \cdots \otimes \psi_{j-1} \otimes (v^{-1} \triangleright \psi_j) \otimes \psi_{j+1} \otimes \cdots \otimes \psi_g \otimes h_1 \otimes \cdots \otimes h_n, \\[.8em]
        & \mathcal{R}_{g,n}^1(H) (\hat{\tau}_{b_j})(\psi_1 \otimes \cdots \otimes \psi_g \otimes h_1 \otimes \cdots \otimes h_n) \\
        & \qquad \qquad \qquad \qquad \qquad \qquad \qquad \hspace{1.5em} = \psi_1 \otimes \cdots \otimes \psi_{j-1} \otimes (\psi_j \star \lambda^v) \otimes \psi_{j+1} \otimes \cdots \otimes \psi_g \otimes h_1 \otimes \cdots \otimes h_n, \\[.8em] 
        & \mathcal{R}_{g,n}^1(H)(\hat{\tau}_{c_j}) (\psi_1 \otimes \cdots \otimes \psi_g \otimes h_1 \otimes \cdots \otimes h_n) \\
        & \qquad \qquad \qquad \quad = \psi_1 \otimes \cdots \otimes \psi_{j-1} \otimes \bigl( S(v^{-1}_{(1)}) \triangleright \psi_j \bigr) \otimes (\psi_{j+1} \triangleleft v^{-1}_{(2)}) \otimes \psi_{j+2} \otimes \cdots \otimes \psi_g \otimes h_1 \otimes \cdots \otimes h_n,
	\end{align*}
	for the curves $a_j$, $b_j$, and $c_j$ represented in \Cref{fig-courbes_ai_bi_ci} above.
\end{proposition}

\begin{proof}
    See \cite[Prop. 5.22]{faitg_derived_2026} for the actions of $\hat{\tau}_{a_j}$ and $\hat{\tau}_{b_j}$. 
	The action of $\hat{\tau}_{c_j}$ is computed in \Cref{ann-action_ci} using the diagrammatic computation described in \Cref{ann-description_diagrammatique_Lgn}.
\end{proof}

\subsection{The projective representation \texorpdfstring{$\rho_{g,n}$}{rhogn} of \texorpdfstring{$\Gamma_{g,n}$}{Gammagn}}\label{subsec-representation_projective_rhogn}

The goal of this section is to extend the results of \Cref{subsec-representation_projective_rhogn1} to surfaces without boundary component $\Sigma_{g,n}$.

\medskip

We will need the based curve $\partial_{g,n} \subset \Sigma_{g,n}^1$ parallel to the boundary component, represented in the following figure: 
\begin{figure}[H]
    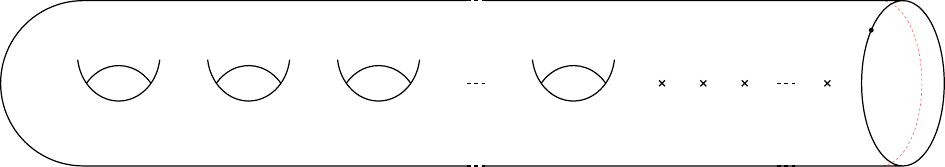
    \caption{Curve $\partial_{g,n}$ of $\Sigma_{g,n}^1$.}
    \label{fig-courbe_bord_sigma}
\end{figure}
\noindent
The capping morphism (\cite[Prop. 3.19]{farb_primer_2012}) and the Birman exact sequence (\cite[Thm. 4.6]{farb_primer_2012}) imply that $\Gamma_{g,n}$ is a quotient of $\Gamma_{g,n}^1$. 
More precisely, we have the following composition of morphisms: 
\begin{equation}\label{eq-application_Gammagn1_vers_Gammagn}
    \Gamma_{g,n}^1 \xrightarrow{\mathrm{Cap}} \Gamma_{g,n+1} \xrightarrow{~\mathrm{Forget}~} \Gamma_{g,n},
\end{equation}
where $\mathrm{Push} : \pi_1(\Sigma_{g,n+1}) \to \Gamma_{g,n+1}$ is the map from the Birman exact sequence and: 
\begin{itemize}
	\item $\ker (\mathrm{Cap}) = \langle \tau_{\partial_{g,n}} \rangle$, 
    \item $\ker (\mathrm{Forget}) = \mathrm{Push}( \pi_1(\Sigma_{g,n+1}))$,
    \item $\mathrm{Cap}$ and $\mathrm{Forget}$ are surjective maps. 
\end{itemize}
Consequently, to obtain a projective representation of $\Gamma_{g,n}$ from the projective representation $\rho_{g,n}^1$ of $\Gamma_{g,n}^1$, it is necessary for $\tau_{\partial_{g,n}}$ and $\mathrm{Cap}^{-1}(\mathrm{Push}(\pi_1(\Sigma_{g,n+1}))$) to be in the kernel of $\rho_{g,n}^1$. 
This is not the case (see \Cref{rem-restriction_a_Inv_necessaire}), but we will see that a subalgebra $\mathcal{L}_{g, n}^{\mathrm{inv}}(H) \subset \mathcal{L}_{g, n}(H)$ has a subrepresentation $V_{g,n} \subset V_{g,n}^1$ which satisfies this property, and thus will allow us to construct a projective representation of $\Gamma_{g,n}$.

\begin{definition}
    The \emph{moduli algebra} $\mathcal{L}_{g, n}^{\mathrm{inv}}(H)$ is the $\mathbb{K}$-vector subspace of $\mathcal{L}_{g, n}(H)$ invariant under the right coadjoint action, that is: 
	\[ \mathcal{L}_{g, n}^{\mathrm{inv}}(H) := \bigl\{ x \in \mathcal{L}_{g, n}(H) ~ \vert ~ \forall h \in H, \coad^r(h)(x) = \varepsilon(h)x \bigr\}. \] 
\end{definition}
\noindent
Since $\mathcal{L}_{g, n}(H)$ is an $H$-module-algebra, $\mathcal{L}_{g, n}^{\mathrm{inv}}(H)$ is a subalgebra of $\mathcal{L}_{g, n}(H)$.

\smallskip
Recall that the element $\partial_{g,n} \in \pi_1(\Sigma_{g,n}^{1})$ represented in \Cref{fig-courbe_bord_sigma} induces a map $\mathfrak{i}_{\partial_{g,n}}$ (see \Cref{def-injection_igamma}).
The following map:
\begin{equation}\label{eq-definition_application_moment}
    \mu_{g,n} := \mathfrak{i}_{\partial_{g,n}} \circ \Phi_{0,1}^{-1} : H \longrightarrow \mathcal{L}_{g, n}(H)
\end{equation}
is a \emph{quantum moment map}, i.e. $\forall h \in H, \, \forall x \in \mathcal{L}_{g, n}(H), ~ \mu_{g,n}(h)x = \coad^r \bigl( S^{-1}(h_{(2)}) \bigr)(x)\mu(h_{(1)})$ (see \cite[Thm. 7.14]{baseilhac_noetherian_2025}).
This identity implies: 
\begin{equation}\label{eq-caracterisation_Lgninv_avec_application_moment}
	\mathcal{L}_{g, n}^{\mathrm{inv}}(H) = \bigl\{ x \in \mathcal{L}_{g, n}(H) ~ \vert ~ \forall h \in H, \mu_{g,n}(h)x = x\mu_{g,n}(h) \bigr\}.
\end{equation} 
Let us now define:
\begin{equation}\label{eq-representation_Lgninv}
	V_{g,n} := \bigl\{ v \in V_{g,n}^1 ~ \vert ~ \forall h \in H, \mathcal{R}_{g,n}^1 \bigl( \mu_{g,n}(h) \bigr) (v) = \varepsilon(h)v \bigr\}.
\end{equation} 
The identity \labelcref{eq-caracterisation_Lgninv_avec_application_moment} implies that $V_{g,n}$ is a $\mathcal{L}_{g, n}^{\mathrm{inv}}(H)$-module. 
We denote by $\mathcal{R}_{g,n}(V) : \mathcal{L}_{g, n}^{\mathrm{inv}}(H) \to \End_\mathbb{K}(V_{g,n})$ the associated representation.

\smallskip
\noindent
The following lemma is standard: 

\begin{lemme}\label{lem-f_chapeau_dans_Lgninv}
    For all $f \in \Gamma_{g,n}^1$, we have $\hat{f} \in \mathcal{L}_{g, n}^{\inv}(H)$.
\end{lemme}

\begin{proof}
    First, note that the element $\hat{f}$ is defined up to a central and invertible element of $\mathcal{L}_{g, n}(H)$. 
	However, \labelcref{eq-caracterisation_Lgninv_avec_application_moment} implies that $Z(\mathcal{L}_{g, n}(H)) \subset \mathcal{L}_{g, n}^{\mathrm{inv}}(H)$. 
	Consequently, if $\hat{f} \in \mathcal{L}_{g, n}^{\inv}(H)$, then for all $z \in Z (\mathcal{L}_{g, n}(H))$ we have $\hat{f}z \in \mathcal{L}_{g, n}^{\mathrm{inv}}(H)$ since $\mathcal{L}_{g, n}^{\mathrm{inv}}(H)$ is a subalgebra of $\mathcal{L}_{g, n}(H)$. \\
	Let us now prove the lemma. 
	We have:
    \begin{equation*}
		\hat{f} \mu_{g,n}(h) = \tilde{f} \bigl( \mu_{g,n}(h) \bigr) \hat{f} = \tilde{f} \bigl( (\mathfrak{i}_{\partial_{g,n}} \circ \Phi_{0,1}^{-1})(h) \bigr) \hat{f} = \mathfrak{i}_{f(\partial_{g,n})} \bigl( \Phi_{0,1}^{-1}(h) \bigr) \hat{f} = \mathfrak{i}_{\partial_{g,n}} \bigl( \Phi_{0,1}^{-1}(h) \bigr) \hat{f} = \mu_{g,n}(h) \hat{f}.
	\end{equation*}
	By \labelcref{eq-caracterisation_Lgninv_avec_application_moment}, the conclusion follows.
\end{proof}

Denote by $\rho_{g,n} : \Gamma_{g,n}^1 \to \PGL(V_{g,n})$ the restriction of $\rho_{g,n}^1$, and by $\hat{\rho}_{g,n}: F_\tau \to \GL(V_{g,n})$ the restriction of $\hat{\rho}_{g,n}^1$. 
\begin{proposition}\label{prop-twist_le_long_du_bord_et_push_dans_noyau}
	For all $g \geq$ 2 and $n \geq 0$, the elements $\tau_{\partial_{g,n}}$ and $\mathrm{Cap}^{-1}(\mathrm{Push}(\pi_1(\Sigma_{g,n+1})))$ are in the kernel of $\rho_{g,n}$.
\end{proposition}
\noindent
From \Cref{lem-f_chapeau_dans_Lgninv} and the discussion after \labelcref{eq-application_Gammagn1_vers_Gammagn} it follows:

\begin{corollaire}\label{cor-representation_projective_rhogn}
	For all $g \geq 2$ and $n \geq 0$, the representation $\rho_{g,n}^1 : \Gamma_{g,n}^1 \to \PGL(V_{g,n}^1)$ factors into a projective representation of $\Gamma_{g,n}$ on $V_{g,n}$, denoted by $\rho_{g,n}$.
    Therefore, we have the following commutative diagram: 
	\[ 
    \begin{tikzcd}
        F_\tau \arrow[r, "\hat{\rho}_{g,n}"] \arrow[d, "\pi"] & \GL(V_{g,n}) \arrow[d, "\pi"] \\
		\Gamma_{g,n} \arrow[r, "\rho_{g,n}"] & \PGL(V_{g,n})
    \end{tikzcd}
	\] 
\end{corollaire}

\begin{proof}[Proof (\Cref{prop-twist_le_long_du_bord_et_push_dans_noyau})]
    Let us begin by showing that $\tau_{\partial_{g,n}}$ is in the kernel of $\rho_{g,n}$.
	We will use the star relation (see \cite[§5.2.3]{farb_primer_2012}) to express the Dehn twist $\tau_{\partial_{g,n}}$ as a product of Dehn twists along non-separating curves. 
	The star relation gives us $(\tau_{\alpha_3} \tau_{\alpha_2} \tau_{\alpha_1} \tau_\beta)^3 = \tau_{\delta_1} \tau_{\delta_2} \tau_{\partial_{g,n}}$ with the curves in the following figure.
	\begin{figure}[H]
        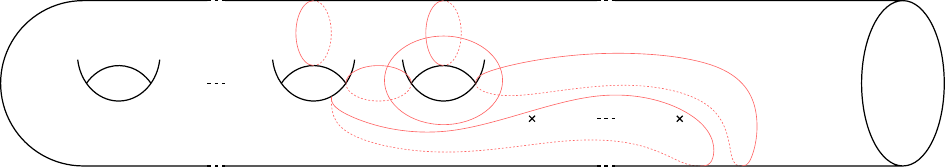
        \caption{Curves of the star relation}
        \label{fig-courbes_relation_etoile}
    \end{figure}
	\noindent
	Therefore, we have: 
    \[ \tau_{\partial_{g,n}} \in \ker \rho_{g,n} \iff \rho_{g,n} \bigl( (\tau_{\alpha_3} \tau_{\alpha_2} \tau_{\alpha_1} \tau_\beta)^3 \bigr) = \rho_{g,n}(\tau_{\delta_1} \tau_{\delta_2}). \]
	We will now prove the equality $\rho_{g,n} \bigl( (\tau_{\alpha_3} \tau_{\alpha_2} \tau_{\alpha_1} \tau_\beta)^3 \bigr) = \rho_{g,n}(\tau_{\delta_1} \tau_{\delta_2})$. 
	As based curves, we have ${\alpha_2}_- = {\alpha_3}_+ \partial_{g,n}^{-1}$, where ${\alpha_2}_- \in \pi_1(\Sigma_{g,n}^1)$ (resp. ${\alpha_3}_+ \in \pi_1(\Sigma_{g,n}^1)$) denotes a representative of $\alpha_2$ (resp. $\alpha_3$) that is negatively (resp. positively) oriented (\Cref{prop-action_twists_de_Dehn_non_separants_comme_conjugaison_et_proprietes} implies that the orientation does not matter since $S_{\mathcal{L}_{0, 1}(H)}(\lambda^v) = \lambda^v$).
	Therefore, for all $\varphi \in H^*$ we have (see \cite[Lem. 5.11]{faitg_derived_2026}):
	\[ \mathfrak{i}_{{\alpha_2}_-}(\varphi) = \mathfrak{i}_{{\alpha_3}_+}(\varphi_{(1)}) \mathfrak{i}_{\partial_{g,n}^{-1}} (\varphi_{(2)}). \]
	Consequently, for all $x \in V_{g,n}$, we have (where the representation $\mathcal{R}_{g,n}(V)$ is defined after \labelcref{eq-representation_Lgninv}):
	\begin{equation}\label{eq-egalite_action_de_courbes_qui_different_du_bord}
		\begin{aligned}
			\hat{\rho}_{g,n}(\tau_{\alpha_2})(x) = \mathcal{R}_{g,n}(V) (\hat{\tau}_{\alpha_2}) & (x) = \mathcal{R}_{g,n}(V) \bigl( \mathfrak{i}_{{\alpha_2}_-}(\lambda^v) \bigr)(x) \\ 
			& = \mathcal{R}_{g,n}(V) \bigl( \mathfrak{i}_{{\alpha_3}_+}(\lambda^v_{(1)}) \mathfrak{i}_{\partial_{g,n}^{-1}}(\lambda^v_{(2)}) \bigr) (x) \\ 
			& \overset{\text{\labelcref{prop-proprietes_igamma}}}{=} \mathcal{R}_{g,n}(V) \bigl( \mathfrak{i}_{{\alpha_3}_+}(\lambda^v_{(1)}) \bigr) \circ \mathcal{R}_{g,n}(V) \bigl( \mathfrak{i}_{\partial_{g,n}} \circ S_{\mathcal{L}_{0, 1}(H)} (\lambda^v_{(2)}) \bigr) (x) \\
			& = \mathcal{R}_{g,n}(V) \bigl( \mathfrak{i}_{{\alpha_3}_+}(\lambda^v_{(1)}) \bigr) \circ \mathcal{R}_{g,n}(V) \bigl( \mathfrak{i}_{\partial_{g,n}} \circ \Phi_{0,1}^{-1} \circ \Phi_{0,1} \circ S_{\mathcal{L}_{0, 1}(H)}(\lambda^v_{(2)}) \bigr) (x) \\
			& \overset{\text{\labelcref{eq-definition_application_moment}}}{=} \mathcal{R}_{g,n}(V) \bigl( \mathfrak{i}_{{\alpha_3}_+}(\lambda^v_{(1)}) \bigr) \circ \mathcal{R}_{g,n}(V) \bigl( \mu_{g,n} \circ \Phi_{0,1} \circ S_{\mathcal{L}_{0, 1}(H)} (\lambda^v_{(2)}) \bigr) (x) \\
			& = \mathcal{R}_{g,n}(V) \bigl( \mathfrak{i}_{{\alpha_3}_+}(\lambda^v_{(1)}) \bigr) \bigl( \varepsilon \circ \Phi_{0,1} \circ S_{\mathcal{L}_{0, 1}(H)}(\lambda^v_{(2)}) x \bigr) \\
			& \overset{\text{\labelcref{eq-representation_Lgninv}}}{=} \mathcal{R}_{g,n}(V) \bigl( \mathfrak{i}_{{\alpha_3}_+}(\lambda^v_{(1)}) \bigr) \bigl( \varepsilon_{H^*} \circ S_{\mathcal{L}_{0, 1}(H)}(\lambda^v_{(2)}) \, x \bigr) \\
			& = \mathcal{R}_{g,n}(V) \bigl( \mathfrak{i}_{{\alpha_3}_+}(\lambda^v_{(1)}) \bigr) \bigl( \varepsilon_{H^*}(\lambda^v_{(2)}) \, x \bigr) \\
			& = \mathcal{R}_{g,n}(V) \bigl( \mathfrak{i}_{{\alpha_3}_+}(\lambda^v) \bigr) (x) = \mathcal{R}_{g,n}(V) (\hat{\tau}_{\alpha_3})(x) = \hat{\rho}_{g,n} (\tau_{\alpha_3})(x). 
		\end{aligned}
	\end{equation}
	Thus we obtain: 
	\[ \rho_{g,n} \bigl( (\tau_{\alpha_3} \tau_{\alpha_2} \tau_{\alpha_1} \tau_\beta)^3 \bigr) = \rho_{g,n} \bigl( (\tau_{\alpha_2}^2 \tau_{\alpha_1} \tau_\beta)^3 \bigr). \]
	We know that all 0-braid and 1-braid relations in $\Gamma_{g,n}^1$ are in the kernel of $\rho_{g,n}^1$, so they are also in the kernel of $\rho_{g,n}$ since it is a restriction of $\rho_{g,n}^1$.
    Therefore, we compute: 
		\begin{align*}
			(\tau_{\alpha_2}^2 \tau_{\alpha_1} \tau_\beta)^3 & = \tau_{\alpha_2} \underline{\tau_{\alpha_2} \tau_{\alpha_1}} \tau_\beta \tau_{\alpha_2} \underline{\tau_{\alpha_2} \tau_{\alpha_1}} \tau_\beta \tau_{\alpha_2} \tau_{\alpha_2} \tau_{\alpha_1} \tau_\beta \\ 
			& = \tau_{\alpha_2} \tau_{\alpha_1} \underline{\tau_{\alpha_2} \tau_\beta \tau_{\alpha_2}} \tau_{\alpha_1} \underline{\tau_{\alpha_2} \tau_\beta \tau_{\alpha_2}} \tau_{\alpha_2} \tau_{\alpha_1} \tau_\beta \\ 
			& = \underline{\tau_{\alpha_2} \tau_{\alpha_1}} \tau_\beta \tau_{\alpha_2} \tau_\beta \tau_{\alpha_1} \tau_\beta \tau_{\alpha_2} \tau_\beta \tau_{\alpha_2} \tau_{\alpha_1} \tau_\beta \\ 
			& = \tau_{\alpha_1} \underline{\tau_{\alpha_2} \tau_\beta \tau_{\alpha_2}} \, \underline{\tau_\beta \tau_{\alpha_1} \tau_\beta} \tau_{\alpha_2} \tau_\beta \tau_{\alpha_2} \tau_{\alpha_1} \tau_\beta \\ 
			& = \tau_{\alpha_1} \tau_\beta \tau_{\alpha_2} \underline{\tau_\beta \tau_{\alpha_1} \tau_\beta} \tau_{\alpha_1} \tau_{\alpha_2} \tau_\beta \tau_{\alpha_2} \tau_{\alpha_1} \tau_\beta \\ 
			& = \tau_{\alpha_1} \tau_\beta \tau_{\alpha_2} \tau_{\alpha_1} \tau_\beta \tau_{\alpha_1} \underline{\tau_{\alpha_1} \tau_{\alpha_2}} \tau_\beta \underline{\tau_{\alpha_2} \tau_{\alpha_1}} \tau_\beta \\ 
			& = \tau_{\alpha_1} \tau_\beta \tau_{\alpha_2} \tau_{\alpha_1} \tau_\beta \underline{\tau_{\alpha_1} \tau_{\alpha_2}} \, \underline{\tau_{\alpha_1} \tau_\beta \tau_{\alpha_1}} \tau_{\alpha_2} \tau_\beta \\ 
			& = \tau_{\alpha_1} \tau_\beta \tau_{\alpha_2} \tau_{\alpha_1} \tau_\beta \tau_{\alpha_2} \tau_{\alpha_1} \tau_\beta \tau_{\alpha_1} \underline{\tau_\beta \tau_{\alpha_2} \tau_\beta} \\ 
			& = \tau_{\alpha_1} \tau_\beta \tau_{\alpha_2} \tau_{\alpha_1} \tau_\beta \tau_{\alpha_2} \tau_{\alpha_1} \tau_\beta \underline{\tau_{\alpha_1} \tau_{\alpha_2}} \tau_\beta \tau_{\alpha_2} \\ 
			& = \tau_{\alpha_1} \tau_\beta \tau_{\alpha_2} \tau_{\alpha_1} \tau_\beta \tau_{\alpha_2} \tau_{\alpha_1} \tau_\beta \tau_{\alpha_2} \tau_{\alpha_1} \tau_\beta \tau_{\alpha_2} = (\tau_{\alpha_1} \tau_\beta \tau_{\alpha_2})^4.
		\end{align*}
	We obtain: 
    \[ \rho_{g,n} \bigl( (\tau_{\alpha_2}^2 \tau_{\alpha_1} \tau_\beta)^3 \bigr) = \rho_{g,n} \bigl( (\tau_{\alpha_1} \tau_\beta \tau_{\alpha_2})^4 \bigr). \]
	Furthermore, in $\Gamma_{g,n}^1$ we also have the 3-chain relation $(\tau_{\alpha_1} \tau_\beta \tau_{\alpha_2})^4 = \tau_{\delta_1} \tau_{\gamma}$ where $\gamma$ is the curve represented in the following figure.
    We denote by $r_{c_0}$ its associated relator.
	\begin{center}
        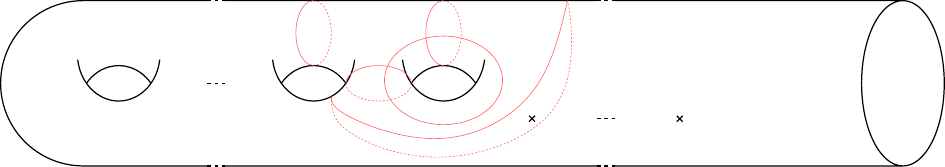
    \end{center}
    Since $r_{c_0}$ is in the kernel of $\rho_{g,n}^1$, it is also in the kernel of $\rho_{g,n}$ because it is a restriction of $\rho_{g,n}^1$. 
	Furthermore, ${\delta_2}_+\partial_{g,n}^{-1} = {\gamma}_-$, where $\delta_2$ is the curve represented in \Cref{fig-courbes_relation_etoile}. 
	Thus, by the same computations as for \labelcref{eq-egalite_action_de_courbes_qui_different_du_bord} we have $\rho_{g,n}(\tau_{\gamma}) = \rho_{g,n}(\tau_{\delta_2})$, and all the previous computations imply: 
    \begin{equation*}
		\rho_{g,n} \bigl( (\tau_{\alpha_3} \tau_{\alpha_2} \tau_{\alpha_1} \tau_\beta)^3 \bigr) = \rho_{g,n} \bigl( (\tau_{\alpha_2}^2 \tau_{\alpha_1} \tau_\beta)^3 \bigr) = \rho_{g,n} \bigl( (\tau_{\alpha_1} \tau_\beta \tau_{\alpha_2})^4 \bigr) = \rho_{g,n}(\tau_{\delta_1} \tau_{\gamma}) = \rho_{g,n}(\tau_{\delta_1} \tau_{\delta_2}),
    \end{equation*}
	hence the result. 

	\smallskip
	
    Let us now show that $\mathrm{Cap}^{-1}(\mathrm{Push}(\pi_1(\Sigma_{g,n+1})))$ is contained in the kernel of $\rho_{g,n}$. 
	Recall that $\pi_1(\Sigma_{g,n+1})$ is generated by homotopy classes of non-separating simple closed curves of $\Sigma_{g,n+1}$. 
	Therefore, it suffices to verify that $\mathrm{Cap}^{-1}(\mathrm{Push}(\gamma)) \in \ker \rho_{g,n}$ for all non-separating simple closed curves $\gamma \in \pi_1(\Sigma_{g,n+1})$. 
	Let $\gamma \in \pi_1(\Sigma_{g,n+1})$ be a non-separating simple curve. 
	We have $\mathrm{Push}(\gamma) = \tau_{\gamma_1} \tau_{\gamma_2}^{-1}$, where $\gamma_1$ and $\gamma_2$ are the boundary components of a tubular neighborhood of $\gamma$ (see \cite[§4.2.2]{farb_primer_2012}). 
	Furthermore, $\mathrm{Cap}^{-1}(\tau_{\gamma_1} \tau_{\gamma_2}^{-1}) = \tau_{\gamma_1} \tau_{\gamma_2}^{-1} \in \Gamma_{g,n}^1$ by the definition of the capping morphism.
	Also, ${\gamma_1}_- = {\gamma_2}_+ \partial_{g,n}^{-1}$, so by the same computations as for \labelcref{eq-egalite_action_de_courbes_qui_different_du_bord} we obtain $\rho_{g,n}(\tau_{\gamma_1}) = \rho_{g,n}(\tau_{\gamma_2})$.
    Consequently, we have $\mathrm{Cap}^{-1}(\mathrm{Push}(\gamma)) \in \ker \rho_{g,n}$. 
\end{proof}

\begin{remarque}
    For all $g \geq 2$ and $n = 0$, the Wajnryb presentation (see \cite[Thm. 5.3]{farb_primer_2012}) implies that $\Gamma_g$ is the quotient of $\Gamma_g^1$ by the hyper-elliptic relation.
	We could have used this relation to prove the result in this specific case, as it was done in \cite{faitg_derived_2026}.
    The advantage of our method is that it covers all cases where $n \geq 0$. 
\end{remarque}

We now need to consider the case where $g = 1$ and $n = 0$. 
Recall that $\Gamma_{1}$ is the quotient of $\Gamma_{1}^1$ by the relator $(\tau_a \tau_b)^6$. 
Note that in the group $\Gamma_{1}^1$ the word $(\tau_a \tau_b)^6$ represents the Dehn twist $\tau_{\partial_{1,0}}$.

\begin{lemme}\label{lem-ab_puissance6_dans_noyau}
    The element $(\tau_a \tau_b)^6$ is in the kernel of $\rho_{1,0}$.
    Consequently, the map $\rho_{1,0}$ factors into a projective representation of $\Gamma_1$ and satisfies the following commutative diagram: 
	\[ 
    \begin{tikzcd}
        F_\tau \arrow[r, "\hat{\rho}_{1,0}"] \arrow[d, "\pi"] & \GL (V_{1,0}) \arrow[d, "\pi"] \\
		\Gamma_1 \arrow[r, "\rho_{1,0}"]  & \PGL (V_{1,0})
    \end{tikzcd}
    \] 
\end{lemme}

\begin{proof}
	First, recall the positively oriented curve $b := b_1$ of $\Sigma_1^1$ (see \Cref{fig-courbes_standards_Lgn}). 
    We have $(\tau_{a} \tau_{b})^3(b) = b^{-1}$. 
	Consequently, for all $\varphi \in H^*$ we have (by \labelcref{eq-definition_action_Gammagn1_sur_Lgn} and \Cref{prop-action_twists_de_Dehn_non_separants_comme_conjugaison_et_proprietes}):
	\[ (\hat{\tau}_{a} \hat{\tau}_{b})^3 \, \mathfrak{i}_b(\varphi) = \mathfrak{i}_{b^{-1}}(\varphi) (\hat{\tau}_{a} \hat{\tau}_{b})^3. \]
	Furthermore, for all $x \in \mathcal{L}_{1, 0}(H), \, \mathcal{R}_{1,0} (x)(\varphi) \overset{\text{\labelcref{prop-definition_representation_combinee}}}{=} \mathcal{R}_{1,0} \bigl( x \, \mathfrak{i}_{b}(\varphi) \bigr) (\varepsilon)$.
	Also, we will see in \Cref{prop-calcul_relation_ab_puissance6} (which only uses the results up to \Cref{lem-lambda^v_star_lambda^v-1_egale_epsilon}) that $\hat{\rho}_{1,0} \bigl( (\tau_{a} \tau_{b})^3 \bigr) (\varepsilon) = \frac{\lambda(v^{-1})}{\lambda(v)} \, \varepsilon$.
    Thus, combining these three equalities, we obtain: $\forall \varphi \in H^*$,
	\begin{align*}
		\hat{\rho}_{1,0} \bigl( (\tau_{a} \tau_{b})^3 \bigr)(\varphi) = \mathcal{R}_{1,0} \bigl( (\hat{\tau}_{a} \hat{\tau}_{b})^3 \, \mathfrak{i}_{b}(\varphi) \bigr) (\varepsilon) & = \mathcal{R}_{1,0} \bigl( \mathfrak{i}_{b^{-1}}(\varphi) (\hat{\tau}_{a} \hat{\tau}_{b})^3 \bigr) (\varepsilon) \\
		& = \mathcal{R}_{1,0} \bigl( \mathfrak{i}_{b^{-1}}(\varphi) \bigr) \biggl( \frac{\lambda(v^{-1})}{\lambda(v)}\varepsilon \biggr) = \frac{\lambda(v^{-1})}{\lambda(v)} S_{\mathcal{L}_{0, 1}(H)}(\varphi).
	\end{align*}
	Therefore:
	\begin{equation}\label{eq-ab_puissance6_egale_SL01carre}
		\hat{\rho}_{1,0} \bigl( (\tau_{a} \tau_{b})^6 \bigr)(\varphi) = \biggl( \frac{\lambda(v^{-1})}{\lambda(v)} \biggr)^2 S_{\mathcal{L}_{0, 1}(H)}^2(\varphi). 
	\end{equation}
	Recall that (see \Cref{subsec-applications_igamma_et_action_de_Gammagn1}) for all $\varphi \in H^*, \, S_{\mathcal{L}_{0, 1}(H)}(\varphi) = S_{H^*} \bigl( S(r_a) \triangleright \varphi \triangleleft r^a u^{-1} \bigr)$.
    Then using the notation $\varphi(h \bullet g)$ for $g \triangleright \varphi \triangleleft h$, we have:
	\begin{align}
		S_{\mathcal{L}_{0, 1}(H)}^2(\varphi) & = S_{H^*} \Bigl( S(r_b) \triangleright \bigl( S^* \bigl( S(r_a) \triangleright \varphi \triangleleft r^a u^{-1} \bigr) \bigr) \triangleleft r^b u^{-1} \Bigr) \nonumber \\
		& = \Bigl\langle S_{H^*} \bigl( S(r_a) \triangleright \varphi \triangleleft r^a u^{-1} \bigr), r^b u^{-1} S(\bullet) S(r_b) \Bigr\rangle \nonumber \\
		& = \varphi \bigl( r^a u^{-1} S(r^b u^{-1} S(\bullet) S(r_b)) S(r_a) \bigr) \nonumber \\
		& = \varphi \bigl( r^a u^{-1} S^2(r_b) S^2(\bullet) S(u^{-1}) S(r^b) S(r_a) \bigr) \nonumber \\
		& \overset{\text{\labelcref{eq-propriete_u}}}{=} \varphi \bigl( r^a r_b u^{-1} S^2(\bullet) S(u^{-1}) S(r^b) S(r_a) \bigr) \nonumber \\
		& \overset{\text{\labelcref{eq-propriete_u}}}{=} \varphi \bigl( r^a r_b \bullet u^{-1} S(u^{-1}) S(r^b) S(r_a) \bigr) \nonumber \\
		& = \varphi \bigl( r^a r_b \bullet u^{-1} (g g^{-1}) S(u^{-1}) S(r^b) S(r_a) \bigr) \nonumber \\
		& = \varphi \bigl( r^a r_b \bullet (u^{-1} g) S(u^{-1}g) S(r^b) S(r_a) \bigr) \nonumber \\
		& \overset{\text{\labelcref{eq-propriete_g}}}{=} \varphi \bigl( r^a r_b \bullet v^{-1} S(v^{-1}) S(r^b) S(r_a) \bigr) \nonumber \\
		& = \varphi \bigl( r^a r_b v^{-1} \bullet S(r_a r^b v^{-1}) \bigr) \overset{\text{\labelcref{eq-propriete_v}}}{=} \varphi \bigl( v_{(1)}^{-1} \bullet S(v_{(2)}^{-1}) \bigr) = \coad^r(v^{-1})(\varphi). \label{eq-SL01carre_egale_coadrvinverse}
	\end{align}
	We know that for all $\psi \in \mathcal{L}_{0, 1}(H)$ and $\varphi \in H^*$ we have $\mathfrak{i}_{\partial_{1,0}}(\psi) \cdot \varphi = \coad^r \bigl( S^{-1}(\Phi_{0,1}(\psi)) \bigr)(\varphi)$ (see \cite[Prop. C.1]{faitg_derived_2026}).
    On the other hand, if $\varphi \in V_{1,0}$ we have: 
    \begin{equation}\label{eq-passage_ibord_a_mubord}
		\mathfrak{i}_{\partial_{1,0}}(\psi) \cdot \varphi \overset{\text{\labelcref{eq-definition_application_moment}}}{=} \mu_{1,0} \bigl( \Phi_{0,1}(\psi) \bigr) \cdot \varphi \overset{\text{\labelcref{eq-representation_Lgninv}}}{=} \varepsilon \bigl( \Phi_{0,1}(\psi) \bigr) \, \varphi \overset{\text{\labelcref{eq-R-matrice_et_co-unite}}}{=} \psi(1_H) \, \varphi.
    \end{equation}
    Recall that $\Phi_{0,1}(\lambda^v) = v^{-1}$ and $S(v^{-1}) = v^{-1}$. 
	Using \labelcref{eq-SL01carre_egale_coadrvinverse,eq-passage_ibord_a_mubord}, for all $\varphi \in V_{1,0}$ we obtain:   
    \begin{equation}\label{eq-SL01carre_egale_id_pour_Inv}
		S_{\mathcal{L}_{0, 1}(H)}^2(\varphi) = \coad^r(v^{-1})(\varphi) = \coad^r \bigl( S^{-1}(\Phi_{0,1}(\lambda^v)) \bigr)(\varphi) = \mathfrak{i}_{\partial_{1,0}}(\lambda^v) \cdot \varphi = \lambda^v(1_H) \, \varphi = \varphi.
    \end{equation}
	Thus, using \labelcref{eq-ab_puissance6_egale_SL01carre,eq-SL01carre_egale_id_pour_Inv}, it follows: 
	\[ \forall \varphi \in V_{1,0}, ~ \hat{\rho}_{1,0} \bigl( (\tau_{a} \tau_{b})^6 \bigr)(\varphi) = \biggl( \frac{\lambda(v^{-1})}{\lambda(v)} \biggr)^2 \varphi, \]
	and hence $\rho_{1,0}\bigl( (\tau_{a} \tau_{b})^6 \bigr) = \id_{V_{1,0}}$.  
\end{proof}

\begin{remarque}\label{rem-restriction_a_Inv_necessaire}
	The proof of \Cref{lem-ab_puissance6_dans_noyau} shows that the restriction of $\rho_{g,n}^1$ to $V_{g,n}$ in \Cref{cor-representation_projective_rhogn} is necessary, otherwise in the case of the torus we should have $S_{\mathcal{L}_{0, 1}(H)}^2(\varphi) = \varphi$, for all $\varphi \in H^*$, which is not true in general. 
\end{remarque}

\section{Computation of the central extensions associated with \texorpdfstring{$\rho_{g,n}^1$}{rhogn1} and \texorpdfstring{$\rho_{g,n}$}{rhogn}}\label{sec-calcul_extension_associee}

Let $V := \{ V_i \}_{i = 1, \cdots, n}$ be a fixed collection of simple $H$-modules. 
In the following, we write:
\[ V_{g,n}^0 := V_{g,n}, \qquad \mathcal{R}_{g,n}^0(V) := \mathcal{R}_{g,n}(V), \qquad  \rho_{g,n}^0 := \rho_{g,n}. \]
Recall that we have the projective representations (see \Cref{def-representation_projective_rhogn1} and \Cref{cor-representation_projective_rhogn}):
\begin{equation*} 
	\rho_{g,n}^s : \Gamma_{g,n}^s \longrightarrow \PGL(V_{g,n}^s),
\end{equation*}
using the representations $\mathcal{R}_{g,n}^s(V)$ of $\mathcal{L}_{g, n}^{\mathrm{inv}}(H)$ defined in \Cref{prop-definition_representation_combinee} and \labelcref{eq-representation_Lgninv} (by using \Cref{lem-f_chapeau_dans_Lgninv} to see that the elements $\hat{\tau}_\gamma$ belong to $\mathcal{L}_{g, n}^{\mathrm{inv}}(H)$).

\smallskip
\noindent
Also, let us recall quickly the method for computing a minimal central extension on which we can construct a linearization of $\rho_{g,n}^1$. 
First, we consider the commutative diagram of \Cref{def-representation_projective_rhogn1}:
\[ 
\begin{tikzcd}
    F_\tau \arrow[r, "\hat{\rho}_{g,n}^1"] \arrow[d, "\pi"] & \GL(V_{g,n}^1) \arrow[d, "\pi"] \\
    \Gamma_{g,n}^1 \arrow[r, "\rho_{g,n}^1"] & \PGL(V_{g,n}^1)
\end{tikzcd}
\] 
where $\hat{\rho}_{g,n}^1(\tau_\gamma) := \mathcal{R}_{g,n}^1(V)(\hat{\tau}_\gamma) \in \GL(V_{g,n}^1)$, with $\hat{\tau}_\gamma$ as in \Cref{prop-action_twists_de_Dehn_non_separants_comme_conjugaison_et_proprietes}.  
We will say that a relator is \emph{trivial} if its image by $\hat{\rho}_{g,n}^1$ is $\id_{V_{g,n}^1}$.

\medskip
\noindent
\textsc{Step 1:} We compute the image by $\hat{\rho}_{g,n}^1$ of all the relators $r$ that appear in the Gervais presentation of $\Gamma_{g,n}^1$ ($g \geq 2$), in the presentation of $\Gamma_1^1$, or in the presentation of $\Gamma_1$ (see \Cref{sec-groupe_modulaire}). 
These images are scalar multiples of the identity of $V_{g,n}^1$ since $\rho_{g,n}^1$ is a projective representation.

\medskip
\noindent
\textsc{Step 2:} We renormalize, if necessary and possible, the image by $\hat{\rho}_{g,n}^1$ of the generators $\tau_\gamma$ in order to have as many trivial relators as possible.  

\medskip
\noindent
\textsc{Step 3:} Denote by $E$ the quotient of $F_\tau$ by the normal subgroup generated by these trivial relators. 
Using the remaining (non trivial) relators, we construct the group $A$  and the central extension $0 \to A \to E \to \Gamma_{g,n}^1 \to 1$ we are looking for.

\begin{remarque}
    The equalities we will obtain hold in $\GL(V_{g,n}^1)$.
    When $n = 0$, we know that $\mathcal{L}_{g, 0}(H) \simeq \mathrm{End}_\mathbb{K}(V_{g,0}^1)$ (see \Cref{thm-isomorphisme_phign}), so these equalities hold in $\mathcal{L}_{g, 0}(H)$. 
\end{remarque}

\subsection{Computation of relations}\label{subsec-calcul_relations}

The Gervais presentation has an infinite number of 0-braid and 1-braid relations. 
We must therefore, a priori, compute the image by $\hat{\rho}_{g,n}^1$ of an infinite number of relations. 
The change of coordinates principle (see \cite[§1.3]{farb_primer_2012}) allows us to compute only a finite number of relations.

\medskip
For simplicity of notations, in the following we write $\hat{\tau}_\gamma$ for $\hat{\rho}_{g,n}^1(\tau_\gamma)$ and we omit the composition; for example, by $\hat{\tau}_\alpha \hat{\tau}_\beta$ we mean $\hat{\rho}_{g,n}^1(\tau_\alpha) \circ \hat{\rho}_{g,n}^1(\tau_\beta) \in \GL(V_{g,n}^1)$. 
Let us first consider the case of 0-braid relations.

\begin{lemme}\label{lem-relation_0-tresse}
	Let $\alpha$ and $\beta$ be two non-separating simple closed curves of $\Sigma_{g,n}^1$ that do not intersect, and let $\tilde{\tau}_\alpha$ (resp. $\tilde{\tau}_\beta$) be a lift of $\rho_{g,n}^1(\tau_\alpha)$ (resp. $\rho_{g,n}^1(\tau_\beta)$) in $\GL(V_{g,n}^1)$. 
	Then the value of $\tilde{\tau}_\alpha \, \tilde{\tau}_\beta \, \tilde{\tau}_\alpha^{-1} \, \tilde{\tau}_\beta^{-1}$ does not depend on the choices of the lifts $\tilde{\tau}_\alpha$ and $\tilde{\tau}_\beta$.
\end{lemme}

\begin{proof}
	Let $\tilde{\tau}_\alpha' := \lambda \tilde{\tau}_\alpha$ be another lift of $\rho_{g,n}^1(\tau_\alpha)$ in $\GL(V_{g,n}^1)$. 
	We have:
	\begin{equation*}
		\tilde{\tau}_\alpha' \, \tilde{\tau}_\beta \, \tilde{\tau}_\alpha'^{-1} \, \tilde{\tau}_\beta^{-1} = \lambda \tilde{\tau}_\alpha \, \tilde{\tau}_\beta \, \frac{1}{\lambda}\tilde{\tau}_\alpha^{-1} \, \tilde{\tau}_\beta^{-1} = \tilde{\tau}_\alpha \, \tilde{\tau}_\beta \, \tilde{\tau}_\alpha^{-1} \, \tilde{\tau}_\beta^{-1},
	\end{equation*} 
	and similarly for $\beta$. 
\end{proof}

We say that two pairs of non-separating simple closed curves $(\alpha_1, \beta_1)$ and $(\alpha_2, \beta_2)$ such that the number of geometric intersection is $i(\alpha_i, \beta_i) = 0$ are in the \emph{same position} if one of the following two cases is true: 
\begin{enumerate}
	\item $\alpha_1 \cup \beta_1$ does not separate the surface $\Sigma_{g,n}^1$, and $\alpha_2 \cup \beta_2$ does not separate the surface $\Sigma_{g,n}^1$.
	\item There exists $k \in [\![1, g-1 ]\!]$ and $l \in [\![ 0, n]\!]$ such that  $\alpha_1 \cup \beta_1$ separates $\Sigma_{g,n}^1$ into $\Sigma_{k,l}^2$ and $\Sigma_{g-k-1, n-l}^3$, and $\alpha_2 \cup \beta_2$ separates $\Sigma_{g,n}^1$ into $\Sigma_{k,l}^2$ and $\Sigma_{g-k-1,n-l}^3$.
\end{enumerate}

\begin{figure}[H]
	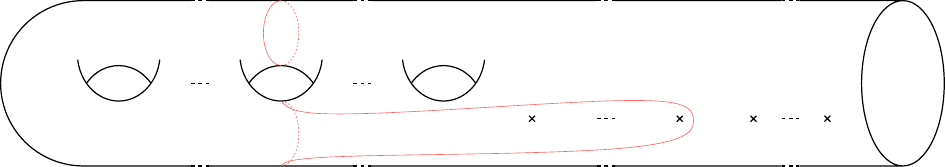
	\caption{Example of curves whose union separates $\Sigma_{g,n}^1$}
    \label{fig-courbes_qui_separent}
\end{figure}

\begin{lemme}
	Let $(\alpha_1, \beta_1)$ and $(\alpha_2, \beta_2)$ be two pairs of non-separating simple closed curves of $\Sigma_{g,n}^1$ that do not intersect and are in the same position. 
	Then we have:
	\[ \hat{\tau}_{\alpha_1} \, \hat{\tau}_{\beta_1} \, \hat{\tau}_{\alpha_1}^{-1} \, \hat{\tau}_{\beta_1}^{-1} = \hat{\tau}_{\alpha_2} \, \hat{\tau}_{\beta_2} \, \hat{\tau}_{\alpha_2}^{-1} \, \hat{\tau}_{\beta_2}^{-1}. \]
\end{lemme}

\begin{proof}
	Let $(\alpha_1, \beta_1)$ and $(\alpha_2, \beta_2)$ be two pairs of non-separating simple closed curves that do not intersect and are in the same position.
	By the change of coordinates principle there exists a direct homeomorphism $f$ of $\Sigma_{g,n}^1$ such that $f(\alpha_1)= \alpha_2$ and $f(\beta_1) = \beta_2$ (see \cite[§1.3.3]{farb_primer_2012}).
	Let $\hat{f}$ be a lift of $\rho_{g,n}^1(f)$ in $\GL(V_{g,n}^1)$. Set 
	\[ \tilde{\tau}_{\alpha_2} := \hat{f} \hat{\tau}_{\alpha_1} \hat{f}^{-1} \text{ and } \tilde{\tau}_{\beta_2} := \hat{f} \hat{\tau}_{\beta_1} \hat{f}^{-1}. \]
	Then $\tilde{\tau}_{\alpha_2}$ (resp. $\tilde{\tau}_{\beta_2}$) is a lift of $\rho_{g,n}^1(\tau_{\alpha_2})$ (resp. $\rho_{g,n}^1(\tau_{\beta_2})$) and we have:
	\[ \tilde{\tau}_{\alpha_2} \, \tilde{\tau}_{\beta_2} \, \tilde{\tau}_{\alpha_2}^{-1} \, \tilde{\tau}_{\beta_2}^{-1} = \hat{f} \underbrace{\hat{\tau}_{\alpha_1} \, \hat{\tau}_{\beta_1} \, \hat{\tau}_{\alpha_1}^{-1} \, \hat{\tau}_{\beta_1}^{-1}}_{\in \mathbb{K}^* \id_{V_{g,n}^1}} \hat{f}^{-1} = \hat{\tau}_{\alpha_1} \, \hat{\tau}_{\beta_1} \, \hat{\tau}_{\alpha_1}^{-1} \, \hat{\tau}_{\beta_1}^{-1}.
	\]
	Therefore, \Cref{lem-relation_0-tresse} implies $\hat{\tau}_{\alpha_1} \, \hat{\tau}_{\beta_1} \, \hat{\tau}_{\alpha_1}^{-1} \, \hat{\tau}_{\beta_1}^{-1} = \hat{\tau}_{\alpha_2} \, \hat{\tau}_{\beta_2} \, \hat{\tau}_{\alpha_2}^{-1} \, \hat{\tau}_{\beta_2}^{-1}$.
\end{proof}
\noindent
The lemma shows that we only need to compute the image by $\hat{\rho}_{g,n}^1$ of $(g-1)(n+1)+1$ 0-braid relations, one for the case where the union $\alpha_1 \cup \beta_1$ does not separate the surface and $(g-1)(n+1)$ for the cases where $\alpha_1 \cup \beta_1$ separates the surface.

\medskip
\noindent
Let us now consider the case of 1-braid relations.
Let $\gamma_1$ and $\gamma_2$ be two non-separating simple closed curves of $\Sigma_{g,n}^1$. 
By the change of coordinates principle there exists a direct homeomorphism $f$ of $\Sigma_{g,n}^1$ such that $f(\gamma_1) = \gamma_2$.  
Then we have: 
\begin{equation}\label{eq-conjugaison_entre_tau_chapeau}
	\hat{\tau}_{\gamma_2} \overset{\text{\labelcref{prop-action_twists_de_Dehn_non_separants_comme_conjugaison_et_proprietes}}}{=} \mathfrak{i}_{\gamma_2}(\lambda^v) = \mathfrak{i}_{f(\gamma_1)}(\lambda^v) \overset{\text{\labelcref{eq-definition_action_Gammagn1_sur_Lgn}}}{=} \tilde{f} \bigl( \mathfrak{i}_{\gamma_1}(\lambda^v) \bigr) \overset{\text{\labelcref{prop-action_twists_de_Dehn_non_separants_comme_conjugaison_et_proprietes}}}{=} \tilde{f}(\hat{\tau}_{\gamma_1}) \overset{\text{\labelcref{cor-action_mapping_class_comme_conjugaison}}}{=} \hat{f} \hat{\tau}_{\gamma_1} \hat{f}^{-1} \in \GL(V_{g,n}^1).
\end{equation}
where $\hat{f}$ is any lift of $\rho_{g,n}^1(f)$ and $\tilde{f}$ is the automorphism of $\mathcal{L}_{g, n}(H)$ induced by $f$. 
The choice of the lift of $\rho_{g,n}^1(f)$ is irrelevant because the scalar of $\hat{f}$ simplifies with the one of $\hat{f}^{-1}$.
Furthermore, since $\mathfrak{i}_{\gamma_2}(\lambda^v)$ does not depend on $f$, the choice of homeomorphism $f$ such that $f(\gamma_1) = \gamma_2$ is irrelevant.

\begin{lemme}\label{lem-egalite_entre_les_relations_1-tresse}
	Let $(\alpha_1, \beta_1)$ be a pair of non-separating simple closed curves of $\Sigma_{g,n}^1$ that intersect at a single point.
	For any other pair of non-separating simple closed curves $(\alpha_2, \beta_2)$ that intersect at a single point we have:
	\[ \hat{\tau}_{\alpha_1} \, \hat{\tau}_{\beta_1} \hat{\tau}_{\alpha_1} \hat{\tau}_{\beta_1}^{-1} \hat{\tau}_{\alpha_1}^{-1} \hat{\tau}_{\beta_1}^{-1} = \hat{\tau}_{\alpha_2} \, \hat{\tau}_{\beta_2} \hat{\tau}_{\alpha_2} \hat{\tau}_{\beta_2}^{-1} \hat{\tau}_{\alpha_2}^{-1} \hat{\tau}_{\beta_2}^{-1}. \]
\end{lemme}

\begin{proof}
	Let $(\alpha_1, \beta_1)$ and $(\alpha_2, \beta_2)$ be two pairs of non-separating simple closed curves, that intersect at a single point.
	By the change of coordinates principle there exists a direct homeomorphism $f$ of $\Sigma_{g,n}^1$ such that $f(\alpha_1) = \alpha_2$ and $f(\beta_1) = \beta_2$ (see \cite[§1.3.3]{farb_primer_2012}).
	If we denote $x^f := \hat{f} x \hat{f}^{-1}$, we have: 
	\begin{align*}
		\hat{\tau}_{\alpha_2} \, \hat{\tau}_{\beta_2} \hat{\tau}_{\alpha_2} \hat{\tau}_{\beta_2}^{-1} \hat{\tau}_{\alpha_2}^{-1} \hat{\tau}_{\beta_2}^{-1} & \overset{\text{\labelcref{eq-conjugaison_entre_tau_chapeau}}}{=} \hat{\tau}_{\alpha_1}^f \, \hat{\tau}_{\beta_1}^f \hat{\tau}_{\alpha_1}^f (\hat{\tau}_{\beta_1}^f)^{-1} (\hat{\tau}_{\alpha_1}^f)^{-1} (\hat{\tau}_{\beta_1}^f)^{-1} \\
		& ~ = \hat{f} \underbrace{\hat{\tau}_{\alpha_1} \hat{\tau}_{\beta_1} \hat{\tau}_{\alpha_1} \hat{\tau}_{\beta_1}^{-1} \hat{\tau}_{\alpha_1}^{-1} \hat{\tau}_{\beta_1}^{-1}}_{\in \mathbb{K}^* \id_{V_{g,n}^1}} \hat{f}^{-1} = \hat{\tau}_{\alpha_1} \, \hat{\tau}_{\beta_1} \hat{\tau}_{\alpha_1} \hat{\tau}_{\beta_1}^{-1} \hat{\tau}_{\alpha_1}^{-1} \hat{\tau}_{\beta_1}^{-1}. \qedhere
	\end{align*}
\end{proof}
\noindent
The lemma shows that we only need to compute the image by $\hat{\rho}_{g,n}^1$ of a single 1-braid relation.

\begin{lemme}
	Suppose that all 0-braid and 1-braid relators are trivial (i.e. their image by $\hat{\rho}_{g,n}^1$ is $\id_{V_{g,n}^1}$).
    We can then renormalize $\hat{\rho}_{g,n}^1$ such that the lantern relator, all the 0-braid relators, and all the 1-braid relators are trivial.
\end{lemme}

\begin{proof}
	Since $\rho_{g,n}^1$ is a projective representation, the image by $\hat{\rho}_{g,n}^1$ of the lantern relator (see \Cref{prop-relation_lanterne}) is a scalar multiple of the identity of $V_{g,n}^1$. 
	So we have $\hat{\tau}_a \hat{\tau}_b \hat{\tau}_c \hat{\tau}_d \hat{\tau}_z^{-1} \hat{\tau}_y^{-1} \hat{\tau}_x^{-1} = \eta \id_{V_{g,n}^1}$, where $\eta \in \mathbb{K}^*$. 
	Let us define $\tilde{\rho}_{g,n}^1 := \frac{\hat{\rho}_{g,n}^1}{\eta}$.
	Then we have:	
	\begin{align*}
		\tilde{\rho}_{g,n}^1 (\tau_{b_1} \tau_{b_2} \tau_{b_3} \tau_{b_4} \tau_z^{-1} \tau_y^{-1} \tau_x^{-1}) & = \frac{\hat{\tau}_{b_1}}{\eta} \, \frac{\hat{\tau}_{b_2}}{\eta} \, \frac{\hat{\tau}_{b_3}}{\eta} \, \frac{\hat{\tau}_{b_4}}{\eta} \, \eta \hat{\tau}_z^{-1} \, \eta \hat{\tau}_y^{-1} \, \eta \hat{\tau}_x^{-1} \\
		& = \frac{1}{\eta} \hat{\tau}_{b_1} \hat{\tau}_{b_2} \hat{\tau}_{b_3} \hat{\tau}_{b_4} \hat{\tau}_z^{-1} \hat{\tau}_y^{-1} \hat{\tau}_x^{-1} = \frac{1}{\eta} \eta \id_{V_{g,n}^1}= \id_{V_{g,n}^1}.
	\end{align*}
	Furthermore, since the sum of the exponents in the 0-braid and 1-braid relators is $0$, the images by $\tilde{\rho}_{g,n}^1$ and by $\hat{\rho}_{g,n}^1$ of these relators are equal, and therefore also equal to $\id_{V_{g,n}^1}$. 
\end{proof}

In the following, we compute the image by $\hat{\rho}_{g,n}^1$ of the $(g-1)(n+1)+1$ 0-braid relations, the 1-braid relation, the lantern relation, the relation $(\tau_a \tau_b)^6$, and the puncture relations.

\begin{proposition}[0-braid relation]\label{prop-calcul_relations_0-tresse}
	Consider the curves $a_1$, $e_k$, and $f_{k,l}$ of $\Sigma_{g,n}^1$, for all $k \in [\![ 2, g ]\!]$ and $l \in [\![ 0, n ]\!]$, represented in the following figure. 
	\begin{center}
		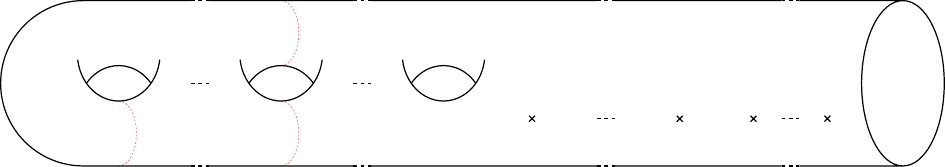
	\end{center}
	We have the following equalities in $\GL(V_{g,n}^1)$: 
	\begin{align*}
		\hat{\tau}_{a_1}\hat{\tau}_{f_{2,0}} = \hat{\tau}_{f_{2,0}}\hat{\tau}_{a_1} ~ \text{ and } ~ \hat{\tau}_{f_{k,l}} \hat{\tau}_{e_k} = \hat{\tau}_{e_k}\hat{\tau}_{f_{k,l}}.
	\end{align*}
\end{proposition}

\begin{proof}
	First, let us show that $\hat{\tau}_{a_1}\hat{\tau}_{f_{2,0}} = \hat{\tau}_{f_{2,0}}\hat{\tau}_{a_1}$. 
	Note that the curve $f_{2,0}$ is the same as the curve $a_2$ in \Cref{fig-courbes_ai_bi_ci}. 
	Since $\rho_{g,n}^1$ is a projective representation, we know that $\hat{\tau}_{a_1}\hat{\tau}_{a_2}$ and $\hat{\tau}_{a_2}\hat{\tau}_{a_1}$ are equal up to a scalar. 
	To find this scalar, by the discussion above \Cref{prop-formules_ai_bi_ci} we simply need to evaluate $\hat{\tau}_{a_1}\hat{\tau}_{a_2}$ and $\hat{\tau}_{a_2}\hat{\tau}_{a_1}$ against the element $\varepsilon^{\otimes g} \otimes 1_H^{\otimes n} \in (H^*)^{\otimes g} \otimes H^{\otimes n}$.
	Since the actions of $\mathfrak{i}_{a_i}(\varphi)$, $i \in \{1, 2 \}$, affect only the first two copies of $H^*$ in $(H^*)^{\otimes g} \otimes H^{\otimes n}$, it suffices to evaluate $\hat{\tau}_{a_1}\hat{\tau}_{a_2}$ and $\hat{\tau}_{a_2}\hat{\tau}_{a_1}$ against the element $\varepsilon^{\otimes 2}$. 
	On one hand we have:
	\begin{equation*}
		\hat{\tau}_{a_1}\hat{\tau}_{a_2} \cdot (\varepsilon^{\otimes 2}) \overset{\text{\labelcref{prop-formules_ai_bi_ci}}}{=} \hat{\tau}_{a_1} \cdot \bigl( \varepsilon \otimes (v^{-1} \triangleright \varepsilon) \bigr) = \hat{\tau}_{a_1} \cdot (\varepsilon^{\otimes 2}) \overset{\text{\labelcref{prop-formules_ai_bi_ci}}}{=} (v^{-1} \triangleright \varepsilon) \otimes \varepsilon = \varepsilon^{\otimes 2},
	\end{equation*}
	and on the other hand we have:
	\begin{equation*}
		\hat{\tau}_{a_2}\hat{\tau}_{a_1} \cdot (\varepsilon^{\otimes 2}) \overset{\text{\labelcref{prop-formules_ai_bi_ci}}}{=} \hat{\tau}_{a_2} \cdot \bigl( (v^{-1} \triangleright \varepsilon) \otimes \varepsilon \bigr) = \hat{\tau}_{a_2} \cdot (\varepsilon^{\otimes 2}) \overset{\text{\labelcref{prop-formules_ai_bi_ci}}}{=} \varepsilon \otimes (v^{-1} \triangleright \varepsilon) = \varepsilon^{\otimes 2},
	\end{equation*}
	where we used the fact that $\varepsilon$ is a morphism of algebras and $\varepsilon(v) = 1_\mathbb{K}$.
	It follows $\hat{\tau}_{a_1}\hat{\tau}_{a_2} = \hat{\tau}_{a_2}\hat{\tau}_{a_1}$ in $\GL(V_{g,n}^1)$.
	
	\smallskip

	Now let us show that $\hat{\tau}_{f_{k,l}}\hat{\tau}_{e_k} = \hat{\tau}_{e_k}\hat{\tau}_{f_{k,l}}$.
	Since $\rho_{g,n}^1$ is a projective representation, we know that $\hat{\tau}_{f_{k,l}}\hat{\tau}_{e_k}$ and $\hat{\tau}_{e_k}\hat{\tau}_{f_{k,l}}$ are equal up to a scalar. 
	To find this scalar, we simply need to evaluate $\hat{\tau}_{f_{k,l}}\hat{\tau}_{e_k}$ and $\hat{\tau}_{e_k}\hat{\tau}_{f_{k,l}}$ against the element $\varepsilon^{\otimes g} \otimes 1_H^{\otimes n} \in (H^*)^{\otimes g} \otimes H^{\otimes n}$.
	
	\noindent
	First, let us evaluate $\hat{\tau}_{e_k}$ against $\varepsilon^{\otimes g} \otimes 1_H^{\otimes n}$.
    The curve $e_k$ is represented by the word $\Bigl( \prod_{i=1}^{k-1} b_i a_i^{-1} b_i^{-1}a_i \Bigr) \\ 
	b_k a_k^{-1} b_k^{-1}$ (in the usual generators of $\pi_1(\Sigma_{g,n}^1)$ and their inverses), with normalization $N(e_k) = 2k$. 
	Since the actions of $\mathfrak{i}_{a_i}(\varphi)$, $\mathfrak{i}_{b_i}(\varphi)$, $i \in [\![ 1, k ]\!] $ affect only the first $k$ copies of $H^*$ in $(H^*)^{\otimes g} \otimes H^{\otimes n}$, it suffices to evaluate $\hat{\tau}_{e_k}$ against $\varepsilon^{\otimes k}$.
    We have:
	\begin{align*}
		\hat{\tau}_{e_k} = \mathfrak{i}_{e_k}(\lambda^v) \overset{\text{\labelcref{def-injection_igamma}}}{=} \lambda^v_{(1)}(v^{2k}) \biggl( \prod_{i=1}^{k-1} \mathfrak{i}_{b_i}(\lambda^v_{(4i-2)}) \mathfrak{i}_{a_i^{-1}}(\lambda^v_{(4i-1)}) \mathfrak{i}_{b_i^{-1}}(\lambda^v_{(4i)}) & \mathfrak{i}_{a_i}(\lambda^v_{(4i+1)}) \biggr) \\
		& \mathfrak{i}_{b_k}(\lambda^v_{(4k-2)}) \mathfrak{i}_{a_k^{-1}}(\lambda^v_{(4k-1)}) \mathfrak{i}_{b_k^{-1}}(\lambda^v_{(4k)}).
	\end{align*}
	If we denote by $d_i$ the curve represented by the word $b_i a_i^{-1}b_i^{-1}a_i$, since $v$ is central and $N(d_i) = 2$, we can rewrite this expression as: 
	\begin{equation*}
			\hat{\tau}_{e_k} = \mathfrak{i}_{d_1}(\lambda^v_{(1)}) \cdots \mathfrak{i}_{d_{k-1}}(\lambda^v_{(k-1)}) \, \lambda^v_{(k)}(v^2) \, \mathfrak{i}_{b_k}(\lambda^v_{(k+1)}) \mathfrak{i}_{a_k^{-1}}(\lambda^v_{(k+2)}) \mathfrak{i}_{b_k^{-1}}(\lambda^v_{(k+3)}).
	\end{equation*}
	Since the co-unit $\varepsilon$ is invariant under the right co-adjoint action, the formulas of \Cref{prop-definition_representation_combinee} for $\mathcal{R}_{g,n}^1(H)$ imply $\mathfrak{i}_{a_k}(\varphi) \cdot (\varepsilon^{\otimes g}) = \varphi(1_H) \, \varepsilon^{\otimes g}$.
	Furthermore, by the definition of the co-product of $H^*$ (i.e. $\Delta (\varphi) (x \otimes y) = \varphi (xy)$) we have $\Delta(\varphi)(x \otimes 1_H) = \varphi(x)$.
    Thus we deduce:
	\begin{align*}
		\hat{\tau}_{e_k} \cdot (\varepsilon^{\otimes k}) & = \mathfrak{i}_{d_1}(\lambda^v_{(1)}) \cdots \mathfrak{i}_{d_{k-1}}(\lambda^v_{(k-1)}) \lambda^v_{(k)}(v^2) \, \mathfrak{i}_{b_k}(\lambda^v_{(k+1)}) \mathfrak{i}_{a_k^{-1}}(\lambda^v_{(k+2)}) \mathfrak{i}_{b_k^{-1}}(\lambda^v_{(k+3)}) \cdot (\varepsilon^{\otimes k}) \\
		& = \mathfrak{i}_{d_1}(\lambda^v_{(1)}) \cdots \mathfrak{i}_{d_{k-1}}(\lambda^v_{(k-1)}) \lambda^v_{(k)}(v^2) \, \mathfrak{i}_{b_k}(\lambda^v_{(k+1)}) \mathfrak{i}_{a_k^{-1}}(\lambda^v_{(k+2)}) \mathfrak{i}_{b_k^{-1}}(\lambda^v_{(k+3)}) \lambda^v_{(k+4)}(1_H) \cdot (\varepsilon^{\otimes k}) \\
		& = \mathfrak{i}_{d_1}(\lambda^v_{(1)}) \cdots \mathfrak{i}_{d_{k-1}}(\lambda^v_{(k-1)}) \lambda^v_{(k)}(v^2) \, \mathfrak{i}_{b_k}(\lambda^v_{(k+1)}) \mathfrak{i}_{a_k^{-1}}(\lambda^v_{(k+2)}) \mathfrak{i}_{b_k^{-1}}(\lambda^v_{(k+3)}) \mathfrak{i}_{a_k}(\lambda^v_{(k+4)}) \cdot (\varepsilon^{\otimes k}) \\
		& = \mathfrak{i}_{d_1}(\lambda^v_{(1)}) \cdots \mathfrak{i}_{d_{k-1}}(\lambda^v_{(k-1)}) \mathfrak{i}_{d_{k}}(\lambda^v_{(k)}) \cdot (\varepsilon^{\otimes k}) \\
		& \overset{\text{\labelcref{prop-formule_di}}}{=} \lambda^v(1_H) \, \varepsilon^{\otimes k} = \varepsilon^{\otimes k}.
	\end{align*}
	Let us now evaluate $\hat{\tau}_{f_{k,l}}$ against $\varepsilon^{\otimes g} \otimes 1_H^{\otimes n}$. 
    The curve $f_{k,l}$ is represented by the word $b_g^{-1} \cdots b_{k+1}^{-1} a_k b_{k+1} \cdots \\
	b_g m_{g+1} \cdots m_{g+l}$ (in the usual generators of $\pi_1(\Sigma_{g,n}^1)$ and their inverses), with normalization $N(f_{k,l}) = 0$.
	Since the actions of $\mathfrak{i}_{a_k}(\varphi)$, $\mathfrak{i}_{b_i}(\varphi)$, and $\mathfrak{i}_{m_{g+j}}(\varphi)$, $i \in [\![ k+1, g ]\!]$, $j \in [\![ 1, l ]\!] $ affect only $(H^*)^{\otimes g} \otimes H^{\otimes l} \subset (H^*)^{\otimes g} \otimes H^{\otimes n}$, it suffices to evaluate $\hat{\tau}_{f_{k,l}}$ against $\varepsilon^{\otimes g} \otimes 1_H^{\otimes l}$.
	Thus we have: 
	\begin{equation*}
		\begin{split}
			\hat{\tau}_{f_{k,l}} = \mathfrak{i}_{f_{k,l}}(\lambda^v) = \mathfrak{i}_{b_g^{-1}}(\lambda^v_{(1)}) \cdots \mathfrak{i}_{b_{k+1}^{-1}}(\lambda^v_{(g-k)}) \, \mathfrak{i}_{a_k}(\lambda^v_{(g-k+1)}) \, & \mathfrak{i}_{b_{k+1}}(\lambda^v_{(g-k+2)}) \cdots \mathfrak{i}_{b_g}(\lambda^v_{(2g-2k+1)}) \\
			& \mathfrak{i}_{m_{g+1}}(\lambda^v_{(2g-2k+2)}) \cdots \mathfrak{i}_{m_{g+l}}(\lambda^v_{(2g-2k+l+1)}).
		\end{split}
	\end{equation*}
	Using the formulas of \Cref{prop-definition_representation_combinee} for $\mathcal{R}_{g,n}^1(H)$ and $\varepsilon \in \inv \coad^r$, we obtain: 
	\begin{align*}
		\hat{\tau}_{f_{k,l}} \cdot (\varepsilon^{\otimes g} & \otimes 1_H^{\otimes l}) = \cdots \mathfrak{i}_{m_{g+l}}(\lambda^v_{(2g-2k+l+1)}) \cdot (\varepsilon^{\otimes g} \otimes 1_H^{\otimes l}) \\ 
		& = \cdots \mathfrak{i}_{m_{g+l-1}}(\lambda^v_{(2g-2k+l)}) \cdot \Bigl( \varepsilon^{\otimes g} \otimes \Bigl(\bigotimes_{i=1}^{l-1}S({r_{a_l}}_{(i)}) \Bigr) \otimes \Phi_{0,1}\big(\coad^r(r^{a_l})(\lambda^v_{(2g-2k+l+1)})\big) \Bigr) \\ 
		& = \cdots \mathfrak{i}_{b_g}(\lambda^v_{(2g-2k+1)}) \cdot \Bigl( \varepsilon^{\otimes g} \otimes \Phi_{0,1}(\lambda^v_{(2g-2k+2)}) S(r_{a_2}) S({r_{a_3}}_{(1)}) \cdots S({r_{a_l}}_{(1)}) \\
		& \qquad \qquad \quad \otimes \Phi_{0,1}\bigl(\coad^r(r^{a_2})(\lambda^v_{(2g-2k+3)})\bigr) S({r_{a_3}}_{(2)}) \cdots S({r_{a_l}}_{(2)}) \otimes \cdots \\
		& \qquad \qquad \quad \otimes \Phi_{0,1} \bigl( \coad^r( r^{a_{l-1}})(\lambda^v_{(2g-2k+l)})\bigr) S({r_{a_l}}_{(l-1)}) \otimes \Phi_{0,1}\bigl(\coad^r(r^{a_l})(\lambda^v_{(2g-2k+l+1)})\bigr) \Bigr).
	\end{align*}
	In the rest of the computations, since the component in $H^{\otimes l}$ remains unchanged (up to the index of the co-product of $\lambda^v$), we denote it by $\clubsuit$.
	\begin{align*}
		\hat{\tau}_{f_{k,l}} \cdot (\varepsilon^{\otimes g} & \otimes 1_H^{\otimes l}) = \cdots \mathfrak{i}_{b_g}(\lambda^v_{(2g-2k+1)}) \cdot (\varepsilon^{\otimes g} \otimes \clubsuit) \\
		& = \cdots \mathfrak{i}_{a_k}(\lambda^v_{(g-k+1)}) \cdot \bigl( \varepsilon^{\otimes k} \otimes \lambda^v_{(g-k+2)} \otimes \cdots \otimes \lambda^v_{(2g-2k+1)} \otimes \clubsuit \bigr) \\ 
		& = \cdots \mathfrak{i}_{b_{k+1}^{-1}}(\lambda^v_{(g-k)}) \cdot \bigl( \varepsilon^{\otimes k-1} \otimes \bigl( \Phi_{0,1}(\lambda^v_{(g-k+1)}) \triangleright \varepsilon \bigr) \otimes \lambda^v_{(g-k+2)} \otimes \cdots \otimes \lambda^v_{(2g-2k+1)} \otimes \clubsuit \bigr) \\
		& = \cdots \mathfrak{i}_{b_{k+1}^{-1}}(\lambda^v_{(g-k)}) \cdot \bigl( \varepsilon^{\otimes k-1} \otimes \bigl( (\varepsilon \circ \Phi_{0,1})(\lambda^v_{(g-k+1)}) \, \varepsilon \bigr) \otimes \lambda^v_{(g-k+2)} \otimes \cdots \otimes \lambda^v_{(2g-2k+1)} \otimes \clubsuit \bigr) \\
		& = \cdots \mathfrak{i}_{b_{k+1}^{-1}}(\lambda^v_{(g-k)}) \cdot \bigl( \varepsilon^{\otimes k} \otimes \lambda^v_{(g-k+1)} \otimes \cdots \otimes \lambda^v_{(2g-2k)} \otimes \clubsuit \bigr) \\
		& = \cdots \mathfrak{i}_{b_{k+2}^{-1}}(\lambda^v_{(g-k-1)}) \cdot \bigl( \varepsilon^{\otimes k} \otimes \bigl( S_{\mathcal{L}_{0, 1}(H)}(\lambda^v_{(g-k)}) \lambda^v_{(g-k+1)} \bigr) \otimes \cdots \otimes \lambda^v_{(2g-2k)} \otimes \clubsuit \bigr) \\
		& = \cdots \mathfrak{i}_{b_{k+2}^{-1}}(\lambda^v_{(g-k-1)}) \cdot \bigl( \varepsilon^{\otimes k+1} \otimes \lambda^v_{(g-k)} \otimes \cdots \otimes \lambda^v_{(2g-2k-2)} \otimes \clubsuit \bigr) \\
		& = \varepsilon^{\otimes g} \otimes \clubsuit,
	\end{align*}
	where we use the formulas of \Cref{prop-definition_representation_combinee} for $\mathcal{R}_{g,n}^1(H)$, the fact that $\varepsilon \in \inv \coad^r$, \labelcref{eq-R-matrice_et_co-unite}, and the fact that $S_{\mathcal{L}_{0,1}(H)}$ is an antipode.
    Note finally that $\clubsuit$ is equal to:
	\begin{equation*}
		\begin{split}
			\Phi_{0,1}(\lambda^v_{(1)}) S(r_{a_2}) S({r_{a_3}}_{(1)}) \cdots & S({r_{a_l}}_{(1)}) \otimes \Phi_{0,1}\bigl(\coad^r(r^{a_2})(\lambda^v_{(2)})\bigr) S({r_{a_3}}_{(2)}) \cdots S({r_{a_l}}_{(2)}) \\
			& \otimes \cdots \otimes \Phi_{0,1} \bigl( \coad^r( r^{a_{l-1}})(\lambda^v_{(l-1)})\bigr) S({r_{a_l}}_{(l-1)}) \otimes \Phi_{0,1}\bigl(\coad^r(r^{a_l})(\lambda^v_{(l)})\bigr).
		\end{split}
	\end{equation*}
	So on one hand we get: 
	\begin{equation*}
		\hat{\tau}_{f_{k,l}} \hat{\tau}_{e_k} \cdot (\varepsilon^{\otimes g} \otimes 1_H^{\otimes l}) = \hat{\tau}_{f_{k,l}} \cdot (\varepsilon^{\otimes g} \otimes 1_H^{\otimes l}) = \varepsilon^{\otimes g} \otimes \clubsuit,
	\end{equation*}
	and on the other hand:	
	\begin{equation*}
		\hat{\tau}_{e_k} \hat{\tau}_{f_{k,l}} \cdot (\varepsilon^{\otimes g} \otimes 1_H^{\otimes l}) = \hat{\tau}_{e_k} \cdot (\varepsilon^{\otimes g} \otimes \clubsuit) = \varepsilon^{\otimes g} \otimes \clubsuit.
	\end{equation*}
	It follows $\hat{\tau}_{f_{k,l}} \hat{\tau}_{e_k} = \hat{\tau}_{e_k} \hat{\tau}_{f_{k,l}}$ in $\GL(V_{g,n}^1)$.
\end{proof}

\begin{proposition}[1-braid relation]\label{prop-calcul_relation_1-tresse}
	Consider the curves $a_1$ and $b_1$ of $\Sigma_{g,n}^1$ represented in the following figure. 
	\begin{center}
		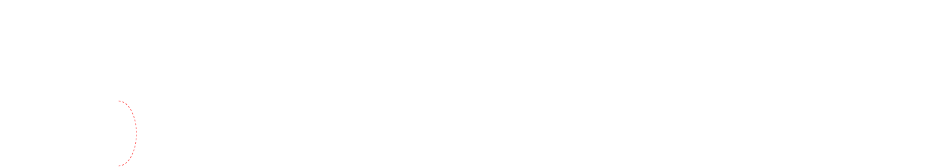
	\end{center}
	We have $\hat{\tau}_{a_1}\hat{\tau}_{b_1}\hat{\tau}_{a_1} = \hat{\tau}_{b_1}\hat{\tau}_{a_1}\hat{\tau}_{b_1}$ in $\GL(V_{g,n}^1)$.
\end{proposition}

\begin{proof}
	Since $\rho_{g,n}^1$ is a projective representation, we know that $\hat{\tau}_{a_1}\hat{\tau}_{b_1}\hat{\tau}_{a_1}$ and $\hat{\tau}_{b_1}\hat{\tau}_{a_1}\hat{\tau}_{b_1}$ are equal up to a scalar.
	To find this scalar, we simply need to evaluate $\hat{\tau}_{a_1}\hat{\tau}_{b_1}\hat{\tau}_{a_1}$ and $\hat{\tau}_{b_1}\hat{\tau}_{a_1}\hat{\tau}_{b_1}$ against the element $\varepsilon^{\otimes g} \otimes 1_H^{\otimes n} \in (H^*)^{\otimes g} \otimes H^{\otimes n}$.
	Since the actions of $\mathfrak{i}_{a_1}(\varphi)$ and $\mathfrak{i}_{b_1}(\varphi)$ affect only the first copy of $H^*$ in $(H^*)^{\otimes g} \otimes H^{\otimes n}$, it suffices to evaluate $\hat{\tau}_{a_1}\hat{\tau}_{b_1}\hat{\tau}_{a_1}$ and $\hat{\tau}_{b_1}\hat{\tau}_{a_1}\hat{\tau}_{b_1}$ against the co-unit $\varepsilon$. 
    On one hand, we compute:
	\begin{equation*}
		\hat{\tau}_{a_1}\hat{\tau}_{b_1}\hat{\tau}_{a_1} \cdot (\varepsilon) \overset{\text{\labelcref{prop-formules_ai_bi_ci}}}{=} \hat{\tau}_{a_1}\hat{\tau}_{b_1} \cdot (v^{-1} \triangleright \varepsilon) = \hat{\tau}_{a_1}\hat{\tau}_{b_1} \cdot (\varepsilon) \overset{\text{\labelcref{prop-formules_ai_bi_ci}}}{=} \hat{\tau}_{a_1} \cdot (\lambda^v) \overset{\text{\labelcref{prop-formules_ai_bi_ci}}}{=} v^{-1} \triangleright \lambda^v = \frac{1}{\lambda(v)} \lambda.
	\end{equation*}
	And on the other hand, we compute:
	\begin{equation*}
		\hat{\tau}_{b_1}\hat{\tau}_{a_1}\hat{\tau}_{b_1} \cdot (\varepsilon) \overset{\text{\labelcref{prop-formules_ai_bi_ci}}}{=} \hat{\tau}_{b_1}\hat{\tau}_{a_1} \cdot (\lambda^v) \overset{\text{\labelcref{prop-formules_ai_bi_ci}}}{=} \hat{\tau}_{b_1} \cdot (v^{-1} \triangleright \lambda^v) = \hat{\tau}_{b_1} \cdot \biggl( \frac{1}{\lambda(v)} \lambda \biggr) \overset{\text{\labelcref{prop-formules_ai_bi_ci}}}{=} \frac{1}{\lambda(v)} \, \lambda \star \lambda^v \overset{\text{\labelcref{rem-lambda_est_integrale_a_droite}}}{=} \frac{1}{\lambda(v)} \, \lambda^v(1_H) \, \lambda = \frac{1}{\lambda(v)} \, \lambda.
	\end{equation*}
	It follows $\hat{\tau}_{a_1}\hat{\tau}_{b_1}\hat{\tau}_{a_1} = \hat{\tau}_{b_1}\hat{\tau}_{a_1}\hat{\tau}_{b_1}$ in $\GL(V_{g,n}^1)$.
\end{proof}

\begin{proposition}[Lantern relation]\label{prop-calcul_relation_lanterne}
	Consider the curves $b_1$, $b_2$, $b_3$, $b_4$, $x$, $y$, and $z$ of $\Sigma_{g,n}^1$ represented in the following figure. 
	\begin{center}
		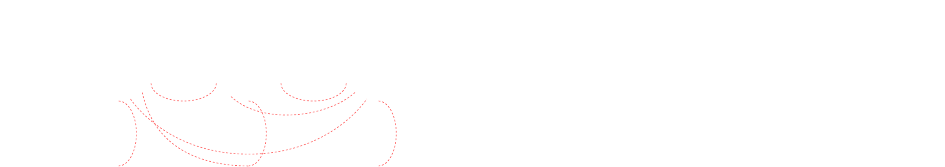
	\end{center}
	We have $\hat{\tau}_{b_1} \hat{\tau}_{b_2} \hat{\tau}_{b_3} \hat{\tau}_{b_4} = \hat{\tau}_x \hat{\tau}_y \hat{\tau}_z$ in $\GL(V_{g,n}^1)$.
\end{proposition}

\begin{proof}
	The formulas of the actions of $\tau_y$ and $\tau_z$ are given in \Cref{prop-formule_y,prop-formule_z}; they affect only the first three copies of $H^*$ in $(H^*)^{\otimes g} \otimes H^{\otimes n}$.
	Since the actions of $\mathfrak{i}_{a_j}(\varphi)$ and $\mathfrak{i}_{b_j}(\varphi)$, $j \in \{1, 2, 3\}$, also affect only the first three copies of $H^*$ in $(H^*)^{\otimes g} \otimes H^{\otimes n}$, it suffices to evaluate $\hat{\tau}_{b_1} \hat{\tau}_{b_2} \hat{\tau}_{b_3} \hat{\tau}_{b_4}$ and $\hat{\tau}_x \hat{\tau}_y \hat{\tau}_z$ against $\varepsilon^{\otimes 3}$ in order to compare them. 
	On one hand, we compute: 
	\begin{align*}
		\hat{\tau}_{b_1} \hat{\tau}_{b_2} \hat{\tau}_{b_3} \hat{\tau}_{b_4} \cdot (\varepsilon^{\otimes 3}) & \overset{\text{\labelcref{prop-formules_ai_bi_ci}}}{=} \hat{\tau}_{b_1} \hat{\tau}_{b_2} \hat{\tau}_{b_3} \cdot \bigl( \varepsilon^{\otimes 2} \otimes (v^{-1} \triangleright \varepsilon) \bigr) \\
		& = \hat{\tau}_{b_1} \hat{\tau}_{b_2} \hat{\tau}_{b_3} \cdot (\varepsilon^{\otimes 3}) \\
		& \overset{\text{\labelcref{prop-formules_ai_bi_ci}}}{=} \hat{\tau}_{b_1} \hat{\tau}_{b_2} \cdot \Bigl(  \varepsilon \otimes \bigl( S(v_{(1)}^{-1}) \triangleright \varepsilon \bigr) \otimes (\varepsilon \triangleleft v_{(2)}^{-1}) \Bigr) \\
		& = \hat{\tau}_{b_1} \hat{\tau}_{b_2} \cdot \bigl( \varepsilon \otimes (v^{-1} \triangleright \varepsilon) \otimes \varepsilon \bigr) \\
		& = \hat{\tau}_{b_1} \hat{\tau}_{b_2} \cdot (\varepsilon^{\otimes 3}) \\
		& \overset{\text{\labelcref{prop-formules_ai_bi_ci}}}{=} \hat{\tau}_{b_1} \cdot \Bigl(  \bigl( S(v_{(1)}^{-1}) \triangleright \varepsilon \bigr) \otimes (\varepsilon \triangleleft v_{(2)}^{-1}) \otimes \varepsilon \Bigr) \\
		& = \hat{\tau}_{b_1} \cdot (\varepsilon^{\otimes 3}) \overset{\text{\labelcref{prop-formules_ai_bi_ci}}}{=} (v^{-1} \triangleright \varepsilon) \otimes \varepsilon^{\otimes 2} = \varepsilon^{\otimes 3}.
	\end{align*}
	In the following, for all $\varphi \in H^*$, $\varphi (h \bullet g)$ denotes $g \triangleright \varphi \triangleleft h$, and $\varphi(\bullet_{(1)} \bullet_{(2)})$ denotes $\varphi \circ \mu \circ \Delta$.
	On the other hand, we compute:
	\begin{align*}
		\hat{\tau}_x \hat{\tau}_y \hat{\tau}_z \cdot (\varepsilon^{\otimes 3}) & \overset{\text{\labelcref{prop-formule_z}}}{=} \hat{\tau}_x \hat{\tau}_y \cdot \bigl(\lambda^v(1_H) \, \varepsilon^{\otimes 3} \bigr) \\
		& = \hat{\tau}_x \hat{\tau}_y \cdot (\varepsilon^{\otimes 3}) \\
		& \overset{\text{\labelcref{prop-formule_y}}}{=} \hat{\tau}_x \cdot \Bigl( \varepsilon \otimes \bigl( \lambda^v \bigl( S^{-1}(\bullet_{(2)}) \bullet_{(1)} \bigr) \star \varepsilon \bigr) \otimes \varepsilon \Bigr) \\
		& = \hat{\tau}_x \cdot \Bigl( \varepsilon \otimes \lambda^v \bigl( S^{-1}(\bullet_{(2)}) \bullet_{(1)} \bigr) \otimes \varepsilon \Bigr) \\
		& = \hat{\tau}_x \cdot \Bigl( \varepsilon \otimes \lambda^v \bigl( \varepsilon( \bullet ) \bigr) \otimes \varepsilon \Bigr) \\
		& = \hat{\tau}_x \cdot \bigl( \lambda^v(1_H) \, \varepsilon^{\otimes 3} \bigr) 
		\overset{\text{\labelcref{prop-formules_ai_bi_ci}}}{=} \varepsilon \otimes (v^{-1} \triangleright \varepsilon) \otimes \varepsilon = \varepsilon^{\otimes 3}.
	\end{align*}
	In these computations, we used the relations \labelcref{eq-R-matrice_et_co-unite} and the fact that $S^{-1}$ is an antipode for $\Delta^{op}$. 	
	It follows $\hat{\tau}_{b_1} \hat{\tau}_{b_2} \hat{\tau}_{b_3} \hat{\tau}_{b_4} = \hat{\tau}_x \hat{\tau}_y \hat{\tau}_z$ in $\GL(V_{g,n}^1)$.
\end{proof}

Since the 0-braid, 1-braid, and lantern relators are trivial, we do not need to renormalize $\hat{\rho}_{g,n}^1$ in order to determine the central extension associated with $\rho_{g,n}^1$.
All that remains is to compute the image by $\hat{\rho}_{g,n}^1$ of the 3-chain relation and the puncture relations, if $g \geq 2$, and of $(\tau_a \tau_b)^6$, if $g = 1$. 

\begin{proposition}[3-chain relation]\label{prop-calcul_relation_3-chaine}
	Consider the curves $a, b, c, d$, and $e$ of $\Sigma_{g,n}^1$ represented in the following figure. 
	\begin{center}
		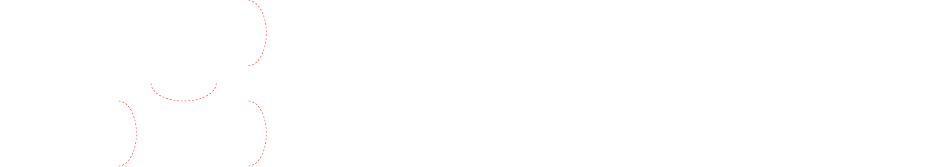
	\end{center}
	We have the following equality in $\GL(V_{g,n}^1)$:
	\[ (\hat{\tau}_a \hat{\tau}_b \hat{\tau}_c)^4 = \biggl( \frac{\lambda(v^{-1})}{\lambda(v)} \biggr)^2 \hat{\tau}_d \hat{\tau}_e. \]
\end{proposition}

\begin{proof}
	The formula of the action of $\tau_e$ is given in \Cref{prop-formule_e}; it affects only the first two copies of $H^*$ in $(H^*)^{\otimes g} \otimes H^{\otimes n}$.
	Since the actions of $\mathfrak{i}_{a_j}(\varphi)$ and $\mathfrak{i}_{b_j}(\varphi)$, $i \in \{ 1, 2\}$, also affect only the first two copies of $H^*$ in $(H^*)^{\otimes g} \otimes H^{\otimes n}$, it suffices to evaluate $( \hat{\tau}_a \hat{\tau}_b \hat{\tau}_c )^4$ and $\hat{\tau}_d \hat{\tau}_e$ against $\varepsilon^{\otimes 2}$ in order to compare them.
	First, let us look at $\hat{\tau}_d \hat{\tau}_e$ against $\varepsilon^{\otimes 2}$:
	\begin{align*}
		\hat{\tau}_d \hat{\tau}_e \cdot (\varepsilon^{\otimes 2}) & \overset{\text{\labelcref{prop-formule_e}}}{=} \hat{\tau}_d \cdot \biggl( \frac{1}{\lambda(v)} \, \lambda \bigl( S(r_f) r_e S(r_d) S(r^c) r^b S(r^a) v \bigr) ( r^e r_b \triangleright \varepsilon \triangleleft r_a r^f ) \otimes (\varepsilon \triangleleft r_c r^d) \biggr)\\
		& = \hat{\tau}_d \cdot \biggl( \frac{1}{\lambda(v)} \, \lambda (v) \, \varepsilon^{\otimes 2} \biggr) = \hat{\tau}_d \cdot (\varepsilon^{\otimes 2}) \overset{\text{\labelcref{prop-formules_ai_bi_ci}}}{=} \varepsilon \otimes (v^{-1} \triangleright \varepsilon) = \varepsilon^{\otimes 2}.
	\end{align*}
	In these computations and in what follows, we use the relations \labelcref{eq-R-matrice_et_antipode,eq-R-matrice_et_co-unite,eq-propriete_v}. 
	
	\smallskip
	\noindent
	Let us now look at $(\hat{\tau}_a \hat{\tau}_b \hat{\tau}_c)^4$ against $\varepsilon^{\otimes 2}$:
	\begin{align*}
		\hat{\tau}_a \hat{\tau}_b \hat{\tau}_c \cdot (\varepsilon^{\otimes 2}) & \overset{\text{\labelcref{prop-formules_ai_bi_ci}}}{=} \hat{\tau}_a \hat{\tau}_b \cdot \Bigl( \bigl( S(v_{(1)}^{-1}) \triangleright \varepsilon \bigr) \otimes (\varepsilon \triangleleft v_{(2)}^{-1}) \Bigr) \\
		& = \hat{\tau}_a \hat{\tau}_b \cdot \Bigl(  S \bigl( v_{(1)}^{-1} \, \varepsilon(v_{(2)}^{-1}) \bigr)  \triangleright \varepsilon \otimes \varepsilon \Bigr) \\
		& = \hat{\tau}_a \hat{\tau}_b \cdot \Bigl( \bigl( S(v^{-1}) \triangleright \varepsilon \bigr) \otimes \varepsilon \Bigr) \\
		& = \hat{\tau}_a \hat{\tau}_b \cdot \bigl( (v^{-1} \triangleright \varepsilon) \otimes \varepsilon \bigr) \\
		& = \hat{\tau}_a \hat{\tau}_b \cdot (\varepsilon^{\otimes 2}) \overset{\text{\labelcref{prop-formules_ai_bi_ci}}}{=} \hat{\tau}_a \cdot \bigl( \lambda^v \otimes \varepsilon \bigr) \overset{\text{\labelcref{prop-formules_ai_bi_ci}}}{=} (v^{-1} \triangleright \lambda^v) \otimes \varepsilon = \frac{1}{\lambda(v)} \, \lambda \otimes \varepsilon. \\
	\end{align*}
	Thus we get: 
	\begin{align*}
		(\hat{\tau}_a \hat{\tau}_b \hat{\tau}_c)^2 \cdot (\varepsilon^{\otimes 2}) & = \hat{\tau}_a \hat{\tau}_b \hat{\tau}_c \cdot \biggl( \frac{1}{\lambda(v)} \, \lambda \otimes \varepsilon \biggr) \\
		& \overset{\text{\labelcref{prop-formules_ai_bi_ci}}}{=} \hat{\tau}_a \hat{\tau}_b \cdot \biggl( \frac{1}{\lambda(v)} \, (v^{-1} \triangleright \lambda) \otimes \varepsilon \biggr) \\
		& \overset{\text{\labelcref{prop-formules_ai_bi_ci}}}{=} \hat{\tau}_a \cdot \biggl( \frac{1}{\lambda(v)}  \bigl( (v^{-1} \triangleright \lambda) \star \lambda^v \bigr) \otimes \varepsilon \biggr) \\
		& = \hat{\tau}_a \cdot \biggl( \frac{\lambda(v^{-1})}{\lambda(v)} \, (\lambda^{v^{-1}} \star \lambda^v) \otimes \varepsilon \biggr) \\
		& \overset{\text{\labelcref{lem-lambda^v_star_lambda^v-1_egale_epsilon}}}{=} \hat{\tau}_a \cdot \biggl( \frac{\lambda(v^{-1})}{\lambda(v)} \, \varepsilon^{\otimes 2} \biggr) \overset{\text{\labelcref{prop-formules_ai_bi_ci}}}{=} \frac{\lambda(v^{-1})}{\lambda(v)} \, (v^{-1} \triangleright \varepsilon) \otimes \varepsilon = \frac{\lambda(v^{-1})}{\lambda(v)} \, \varepsilon^{\otimes 2}.
	\end{align*}
	Therefore we have: 
	\begin{equation*}
		(\hat{\tau}_a \hat{\tau}_b \hat{\tau}_c)^4 \cdot (\varepsilon^{\otimes 2}) = (\hat{\tau}_a \hat{\tau}_b \hat{\tau}_c)^2 \cdot \biggl( \frac{\lambda(v^{-1})}{\lambda(v)} \, \varepsilon^{\otimes 2} \biggr) = \biggl( \frac{\lambda(v^{-1})}{\lambda(v)} \biggr)^{2} \varepsilon^{\otimes 2}.
	\end{equation*}
	It follows $( \hat{\tau}_a \hat{\tau}_b \hat{\tau}_c )^4 = \Bigl( \frac{\lambda(v^{-1})}{\lambda(v)} \Bigr)^2 \hat{\tau}_d \hat{\tau}_e$ in $\GL(V_{g,n}^1)$.
\end{proof}

The following relator is used when $g = 1$. 

\begin{proposition}[Relator $(ab)^6$]\label{prop-calcul_relation_ab_puissance6}
	Consider the curves $a$ and $b$ of $\Sigma_{1}^1$ represented in the following figure.
	\begin{center}
		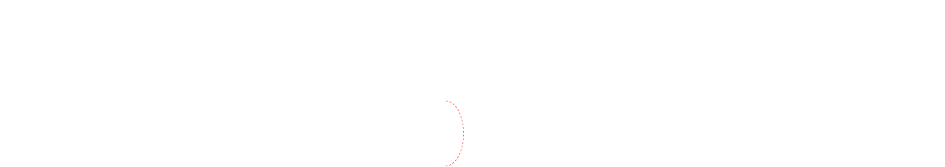
	\end{center}
	We have $(\hat{\tau}_a \hat{\tau}_b)^6 = \Bigl( \frac{\lambda(v^{-1})}{\lambda(v)} \Bigr)^2 \id_{V_{1,0}}$ in $\GL(V_{1,0})$.
\end{proposition}

\begin{proof}
	To find the scalar given by the relator $(\tau_a \tau_b)^6$, it suffices to evaluate $(\hat{\tau}_a \hat{\tau}_b)^6$ against the co-unit $\varepsilon \in V_{1,0}$: 
	\begin{equation*}
		\hat{\tau}_a \hat{\tau}_b \cdot (\varepsilon) \overset{\text{\labelcref{prop-formules_ai_bi_ci}}}{=} \hat{\tau}_a \cdot (\lambda^{v}) \overset{\text{\labelcref{prop-formules_ai_bi_ci}}}{=} v^{-1} \triangleright \lambda^v = \frac{1}{\lambda(v)} \, \lambda.
	\end{equation*}
	So we have: 
	\begin{equation*}
		(\hat{\tau}_a \hat{\tau}_b)^2 \cdot (\varepsilon) = \hat{\tau}_a \hat{\tau}_b \cdot \biggl( \frac{1}{\lambda(v)} \, \lambda \biggr) \overset{\text{\labelcref{prop-formules_ai_bi_ci}}}{=} \hat{\tau}_a \cdot \biggl( \frac{1}{\lambda(v)} \, (\lambda \star \lambda^{v}) \biggr) \overset{\text{\labelcref{rem-lambda_est_integrale_a_droite}}}{=} \hat{\tau}_a \cdot \biggl( \frac{1}{\lambda(v)} \, \lambda^{v}(1_H )\, \lambda \biggr) \overset{\text{\labelcref{prop-formules_ai_bi_ci}}}{=} \frac{1}{\lambda(v)} \, v^{-1} \triangleright \lambda.
	\end{equation*}
	Thus we get: 
	\begin{align*}
		(\hat{\tau}_a \hat{\tau}_b)^3 \cdot (\varepsilon) = \hat{\tau}_a \hat{\tau}_b \cdot \biggl( \frac{1}{\lambda(v)} \, v^{-1} \triangleright \lambda \biggr) & \overset{\text{\labelcref{prop-formules_ai_bi_ci}}}{=} \hat{\tau}_a \cdot \biggl( \frac{1}{\lambda(v)} \, (v^{-1} \triangleright \lambda) \star \lambda^v \biggr) \\
		& = \hat{\tau}_a \cdot \biggl( \frac{\lambda(v^{-1})}{\lambda(v)} \, \lambda^{v^{-1}} \star \lambda^v \biggr) \\
		& \overset{\text{\labelcref{lem-lambda^v_star_lambda^v-1_egale_epsilon}}}{=} \hat{\tau}_a \cdot \biggl( \frac{\lambda(v^{-1})}{\lambda(v)} \, \varepsilon \biggr) \overset{\text{\labelcref{prop-formules_ai_bi_ci}}}{=} \frac{\lambda(v^{-1})}{\lambda(v)} \, v^{-1} \triangleright \varepsilon = \frac{\lambda(v^{-1})}{\lambda(v)} \, \varepsilon.
	\end{align*}
	Therefore we have: 
	\begin{equation*}
			(\hat{\tau}_a \hat{\tau}_b)^6 \cdot (\varepsilon) = (\hat{\tau}_a \hat{\tau}_b)^3 \cdot \biggl( \frac{\lambda(v^{-1})}{\lambda(v)} \, \varepsilon \biggr) = \biggl( \frac{\lambda(v^{-1})}{\lambda(v)} \biggr)^2 \varepsilon.
	\end{equation*}
	It follows $(\hat{\tau}_a \hat{\tau}_b)^6 = \Bigl( \frac{\lambda(v^{-1})}{\lambda(v)} \Bigr)^2 \id_{V_{1,0}}$ in $\GL (V_{1,0})$.
\end{proof}

For surfaces with marked points, we need to compute the image by $\hat{\rho}_{g,n}^1$ of the puncture relations.

\begin{proposition}[$k$-th puncture relation]\label{prop-calcul_relation_piqure_k_avec_formules_Derived}
	Consider the curves $\gamma_1^k$, $\gamma_2^k$, $\gamma_3^k$, $x^k$, $y^k$, and $z^k$ of $\Sigma_{g,n}^1$ represented in the following figures.
	\begin{center}
		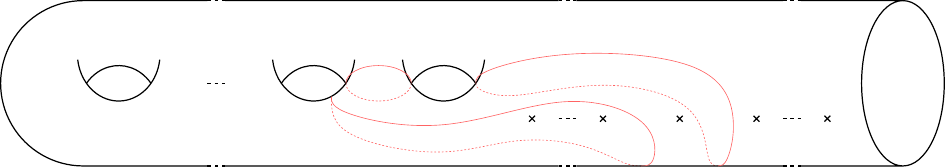
	\end{center}

	\smallskip

	\begin{center}
		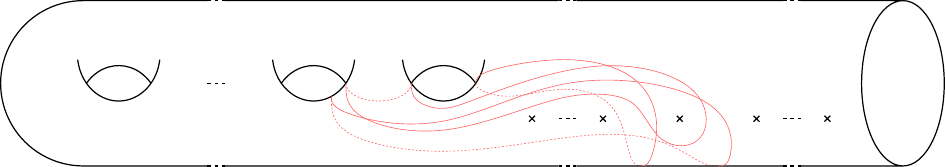
	\end{center}
	We have the following equality in $\GL(V_{g,n}^1)$:
	\[ \hat{\tau}_{x^k} \hat{\tau}_{y^k} \hat{\tau}_{z^k} = \mathcal{R}_{g,n}^1(V) \bigl( \mathfrak{i}_{m_{g+k}} (\lambda^v) \bigr) \hat{\tau}_{\gamma_1^k} \hat{\tau}_{\gamma_2^k} \hat{\tau}_{\gamma_3^k}, \]
	where $\mathfrak{i}_{m_{g+k}} (\lambda^v) \in Z(\mathcal{L}_{g, n}(H))^\times$.
\end{proposition}

\begin{proof}
	Let us first verify that the element $\mathfrak{i}_{m_{g+k}} (\lambda^v)$ is central and invertible. 
	To do this, we consider its image under the isomorphism of algebras $\Phi_{g,n} : \mathcal{L}_{g, n}(H) \to \mathcal{H}(H^*)^{\otimes g} \otimes H^{\otimes n}$. 
	Using the fact that $\lambda^v \in \inv \coad^r$, the definition of $\Phi_{g,n}$ (see \cite[Thm. 5.3]{baseilhac_noetherian_2025}) immediately implies: 
	\[ 
        \Phi_{g,n} \bigl( \mathfrak{i}_{m_{g+k}} (\lambda^v) \bigr) = 1_{\mathcal{H}(H^*)^{\otimes g}} \otimes 1_H^{\otimes k-1} \otimes v^{-1} \otimes 1_H^{\otimes n-k}.
	\]
    Since $v^{-1}$ is central and invertible, the element $\mathfrak{i}_{m_{g+k}} (\lambda^v)$ is also central and invertible.

	\smallskip
	\noindent
	We have: 
	\begin{itemize}
        \item The curve $\gamma_1^k$ is represented by the word $a_g m_{g+1} \cdots m_{g+k}$, with normalization $N(\gamma_1^k) = 0$,
		\item The curve $\gamma_3^k$ is represented by the word $b_{g}^{-1} a_{g-1} b_g m_{g+1} \cdots m_{g+k-1}$, with normalization $N(\gamma_3^k) \\ = 0$,
		\item The curve $x^k$ is represented by the word $b_{g}^{-1} a_{g-1} b_g m_{g+1} \cdots m_{g+k}$, with normalization $N(x^k) = 0$,
		\item The curve $y^k$ is represented by the word $a_{g} m_{g+1} \cdots m_{g+k-1}$, with normalization $N(y^k) = 0$,
		\item The curve $z^k$ is represented by the word $a_{g} b_{g}^{-1} a_{g-1}^{-1} b_g m_{g+k}$, with normalization $N(z^k) = 0$.
    \end{itemize}
	Since the actions of $\hat{\tau}_{\gamma_1^k}$, $\hat{\tau}_{\gamma_2^k}$, $\hat{\tau}_{\gamma_3^k}$, $\hat{\tau}_{x^k}$, $\hat{\tau}_{y^k}$, and $\hat{\tau}_{z^k}$ affect only $(H^*)^{\otimes g} \otimes H^{\otimes k} \subset (H^*)^{\otimes g} \otimes H^{\otimes n}$, it suffices to evaluate $\hat{\tau}_{\gamma_1^k} \hat{\tau}_{\gamma_2^k} \hat{\tau}_{\gamma_3^k}$ and $\hat{\tau}_{x^k} \hat{\tau}_{y^k} \hat{\tau}_{z^k}$ against the element $\varepsilon^{\otimes g} \otimes 1_H^{\otimes k}$ to find the central element given by the $k$-th puncture relation.
	Let us begin with the action of $\hat{\tau}_{\gamma_1^k} \hat{\tau}_{\gamma_2^k} \hat{\tau}_{\gamma_3^k}$ against $\varepsilon^{\otimes g} \otimes 1_H^{\otimes k}$; the summations will be implicit.
	First, let us evaluate $\hat{\tau}_{\gamma_3^k}$ against $\varepsilon^{\otimes g} \otimes 1_H^{\otimes k}$:	
	\begin{align*}
		\hat{\tau}_{\gamma_3^k} & \cdot (\varepsilon^{\otimes g} \otimes 1_H^{\otimes k}) = \mathfrak{i}_{b_g^{-1}}(\lambda^v_{(1)}) \mathfrak{i}_{a_{g-1}}(\lambda^v_{(2)}) \mathfrak{i}_{b_g}(\lambda^v_{(3)}) \mathfrak{i}_{m_{g+1}}(\lambda^v_{(4)}) \cdots \mathfrak{i}_{m_{g+k-1}}(\lambda^v_{(k+2)}) \cdot (\varepsilon^{\otimes g} \otimes 1_H^{\otimes k}) \\
		& = \cdots \mathfrak{i}_{m_{g+k-2}}(\lambda^v_{(k+2)}) \cdot \bigl( \varepsilon^{\otimes g} \otimes S({r_{a_{k-1}}}_{(1)}) \otimes \cdots \otimes S({r_{a_{k-1}}}_{(k-2)}) \otimes \Phi_{0,1}\big(\coad^r(r^{a_{k-1}})(\lambda^v_{(k+2)})\big) \otimes 1_H \bigr) \\
		& = \cdots \mathfrak{i}_{b_g}(\lambda^v_{(3)}) \cdot \bigl( \varepsilon^{\otimes g} \otimes \Phi_{0,1}(\lambda^v_{(4)}) S(r_{a_2}) S({r_{a_3}}_{(1)}) \cdots S({r_{a_{k-1}}}_{(1)}) \\
		& \qquad \qquad \qquad \qquad \quad \otimes \Phi_{0,1}\big(\coad^r(r^{a_2})(\lambda^v_{(5)})\big) S({r_{a_3}}_{(2)}) \cdots S({r_{a_{k-1}}}_{(2)}) \otimes \cdots \\
		& \qquad \qquad \qquad \qquad \quad \otimes \Phi_{0,1} \big( \coad^r( r^{a_{k-2}})(\lambda^v_{(k+1)})\big) S({r_{a_{k-1}}}_{(k-2)}) \otimes \Phi_{0,1}\big(\coad^r(r^{a_{k-1}})(\lambda^v_{(k+2)})\big) \otimes 1_H \bigr).
	\end{align*}
	In the rest of the computation, since the component in $H^{\otimes k}$ remains unchanged (up to the index of the co-product of $\lambda^v$), we denote it by $\clubsuit$.
	\begin{align*}
		\hat{\tau}_{\gamma_3^k} \cdot (\varepsilon^{\otimes g} \otimes 1_H^{\otimes k}) & = \cdots \mathfrak{i}_{b_g}(\lambda^v_{(3)}) \cdot (\varepsilon^{\otimes g} \otimes \clubsuit) \\
		& = \cdots \mathfrak{i}_{a_{g-1}}(\lambda^v_{(2)}) \cdot (\varepsilon^{\otimes g-1} \otimes \lambda^v_{(3)} \otimes \clubsuit) \\
		& = \mathfrak{i}_{b_g^{-1}}(\lambda^v_{(1)}) \cdot \bigl( \varepsilon^{\otimes g-2} \otimes \Phi_{0,1}(\lambda^v_{(2)}) \triangleright \varepsilon \otimes \lambda^v_{(3)} \otimes \clubsuit \bigr) \\ 
		& = \mathfrak{i}_{b_g^{-1}}(\lambda^v_{(1)}) \cdot \bigl( \varepsilon^{\otimes g-2} \otimes (\varepsilon \circ \Phi_{0,1})(\lambda^v_{(2)}) \, \varepsilon \otimes \lambda^v_{(3)} \otimes \clubsuit \bigr) \\ 
		& = \mathfrak{i}_{b_g^{-1}}(\lambda^v_{(1)}) \cdot \bigl( \varepsilon^{\otimes g-1} \otimes \lambda^v_{(2)} \otimes \clubsuit \bigr) = \varepsilon^{\otimes g-1} \otimes \bigl( S_{\mathcal{L}_{0, 1}(H)}(\lambda^v_{(1)}) \lambda^v_{(2)} \bigr) \otimes \clubsuit = \varepsilon^{\otimes g} \otimes \clubsuit,
	\end{align*}
	where $\clubsuit$ is equal to:
	\begin{equation}\label{eq-trefle_relation_piqure}
		\begin{split}
			\Phi_{0,1}(\lambda^v_{(1)}) & S(r_{a_2}) S({r_{a_3}}_{(1)}) \cdots S({r_{a_{k-1}}}_{(1)}) \otimes \Phi_{0,1} \bigl( \coad^r(r^{a_2})(\lambda^v_{(2)}) \bigr) S({r_{a_3}}_{(2)}) \cdots S({r_{a_{k-1}}}_{(2)}) \\
			& \otimes \cdots \otimes \Phi_{0,1} \bigl( \coad^r( r^{a_{k-2}})(\lambda^v_{(k-2)}) \bigr) S({r_{a_{k-1}}}_{(k-2)}) \otimes \Phi_{0,1} \bigl( \coad^r(r^{a_{k-1}})(\lambda^v_{(k-1)}) \bigr) \otimes 1_H. 
		\end{split}
	\end{equation}
	Then, let us evaluate $\hat{\tau}_{\gamma_2^k} \hat{\tau}_{\gamma_3^k}$ against $\varepsilon^{\otimes g} \otimes 1_H^{\otimes k}$: 
	\begin{equation*}
		\hat{\tau}_{\gamma_2^k} \hat{\tau}_{\gamma_3^k} \cdot (\varepsilon^{\otimes g} \otimes 1_H^{\otimes k}) = \hat{\tau}_{\gamma_2^k} \cdot (\varepsilon^{\otimes g} \otimes \clubsuit) \overset{\text{\labelcref{prop-formules_ai_bi_ci}}}{=} \varepsilon^{\otimes g-2} \otimes \bigl( S(v_{(1)}^{-1}) \triangleright \varepsilon \bigr) \otimes \bigl( \varepsilon \triangleleft v_{(2)}^{-1} \bigr) \otimes \clubsuit = \varepsilon^{\otimes g} \otimes \clubsuit.
	\end{equation*} 
	Let us now evaluate $\hat{\tau}_{\gamma_1^k} \hat{\tau}_{\gamma_2^k} \hat{\tau}_{\gamma_3^k}$ against $\varepsilon^{\otimes g} \otimes 1_H^{\otimes k}$:
	\begin{align*}
		\hat{\tau}_{\gamma_1^k} \hat{\tau}_{\gamma_2^k} \hat{\tau}_{\gamma_3^k} & \cdot (\varepsilon^{\otimes g} \otimes 1_H^{\otimes k}) = \hat{\tau}_{\gamma_1^k} \cdot (\varepsilon^{\otimes g} \otimes \clubsuit) \\
		& = \mathfrak{i}_{a_{g}}(\lambda^v_{(1)}) \mathfrak{i}_{m_{g+1}}(\lambda^v_{(2)}) \cdots \mathfrak{i}_{m_{g+k}}(\lambda^v_{(k+1)}) \cdot (\varepsilon^{\otimes g} \otimes \clubsuit) \\ 
		& = \mathfrak{i}_{a_{g}}(\lambda^v_{(1)}) \cdot \biggl( \Bigl( \varepsilon^{\otimes g} \otimes \Phi_{0,1}(\lambda^v_{(2)}) S(r_{b_2}) S({r_{b_3}}_{(1)}) \cdots S({r_{b_{k}}}_{(1)}) \\
		& \qquad \qquad \qquad \otimes \Phi_{0,1} \bigl( \coad^r(r^{b_2})(\lambda^v_{(3)}) \bigr) S({r_{b_3}}_{(2)}) \cdots S({r_{b_{k}}}_{(2)}) \otimes \cdots \\
		& \qquad \qquad \qquad \otimes \Phi_{0,1} \bigl( \coad^r( r^{b_{k-1}})(\lambda^v_{(k)}) \bigr) S({r_{b_{k}}}_{(k-1)}) \otimes \Phi_{0,1} \bigl( \coad^r(r^{b_{k}})(\lambda^v_{(k+1)}) \bigr) \Bigr) (\varepsilon^{\otimes g} \otimes \clubsuit) \biggr) \\
		& = \Bigl( \varepsilon^{\otimes g-1} \otimes \Phi_{0,1}(\lambda^v_{(1)}) \triangleright \varepsilon \otimes \Phi_{0,1}(\lambda^v_{(2)}) S(r_{b_2}) S({r_{b_3}}_{(1)}) \cdots S({r_{b_{k}}}_{(1)}) \\
		& \qquad \qquad \qquad \otimes \Phi_{0,1} \bigl( \coad^r(r^{b_2})(\lambda^v_{(3)}) \bigr) S({r_{b_3}}_{(2)}) \cdots S({r_{b_{k}}}_{(2)}) \otimes \cdots \\
		& \qquad \qquad \qquad \otimes \Phi_{0,1} \bigl( \coad^r( r^{a_{k-1}})(\lambda^v_{(k)}) \bigr) S({r_{b_{k}}}_{(k-1)}) \otimes \Phi_{0,1} \bigl( \coad^r(r^{b_{k}})(\lambda^v_{(k+1)}) \bigr) \Bigr) (\varepsilon^{\otimes g} \otimes \clubsuit) \\
		& = \Bigl( \varepsilon^{\otimes g} \otimes \Phi_{0,1}(\lambda^v_{(1)}) S(r_{b_2}) S({r_{b_3}}_{(1)}) \cdots S({r_{b_{k}}}_{(1)}) \otimes \Phi_{0,1} \bigl( \coad^r(r^{b_2})(\lambda^v_{(2)}) \bigr) S({r_{b_3}}_{(2)}) \cdots S({r_{b_{k}}}_{(2)}) \\
		& \qquad \qquad \qquad \otimes \cdots \otimes \Phi_{0,1} \bigl( \coad^r( r^{b_{k-1}})(\lambda^v_{(k)}) \bigr) S({r_{b_{k}}}_{(k-1)}) \otimes \Phi_{0,1} \bigl( \coad^r(r^{b_{k}})(\lambda^v_{(k)}) \bigr) \Bigr) (\varepsilon^{\otimes g} \otimes \clubsuit), 
	\end{align*}
	which we will denote more briefly by :
	\begin{equation*}
		\hat{\tau}_{\gamma_1^k} \hat{\tau}_{\gamma_2^k} \hat{\tau}_{\gamma_3^k} \cdot (\varepsilon^{\otimes g} \otimes 1_H^{\otimes k}) = (\varepsilon^{\otimes g} \otimes \spadesuit) (\varepsilon^{\otimes g} \otimes \clubsuit).
	\end{equation*}
	However, we have $\mathfrak{i}_{m_{g+k}} (\lambda^v) \cdot (\varepsilon^{\otimes g} \otimes 1_H^{\otimes l}) = \varepsilon^{\otimes g} \otimes 1_H^{\otimes l-1} \otimes v^{-1}$.
	Since $v$ is central, it follows:
	\begin{equation}\label{eq-membre_de_droite_relation_piqure}
		\mathfrak{i}_{m_{g+k}} (\lambda^v)\hat{\tau}_{\gamma_1^k} \hat{\tau}_{\gamma_2^k} \hat{\tau}_{\gamma_3^k} \cdot (\varepsilon^{\otimes g} \otimes 1_H^{\otimes k}) = (\varepsilon^{\otimes g} \otimes \spadesuit) (\varepsilon^{\otimes g} \otimes \clubsuit)(\varepsilon^{\otimes g} \otimes 1_H^{\otimes l-1} \otimes v^{-1}).
	\end{equation}
	Now let us compute the action of $\hat{\tau}_{x^k} \hat{\tau}_{y^k} \hat{\tau}_{z^k}$ against $\varepsilon^{\otimes g} \otimes 1_H^{\otimes k}$; the summations will be implicit. 
	First, let us evaluate $\hat{\tau}_{z^k}$ against $\varepsilon^{\otimes g} \otimes 1_H^{\otimes k}$: 
	\begin{align*}
		\hat{\tau}_{z^k} & \cdot (\varepsilon^{\otimes g} \otimes 1_H^{\otimes k}) = \mathfrak{i}_{a_{g}}(\lambda^v_{(1)}) \mathfrak{i}_{b_{g}^{-1}}(\lambda^v_{(2)}) \mathfrak{i}_{a_{g-1}^{-1}}(\lambda^v_{(3)}) \mathfrak{i}_{b_{g}}(\lambda^v_{(4)}) \mathfrak{i}_{m_{g+k}}(\lambda^v_{(5)}) \cdot (\varepsilon^{\otimes g} \otimes 1_H^{\otimes k}) \\
		& = \cdots \mathfrak{i}_{b_{g}}(\lambda^v_{(4)}) \cdot \Bigl( \varepsilon^{\otimes g} \otimes S({r_a}_{(1)}) \otimes \cdots \otimes S({r_a}_{(k-1)}) \otimes \Phi_{0,1} \bigl( \coad^r(r^a)(\lambda^v_{(5)}) \bigr) \Bigr) \\
		& = \cdots \mathfrak{i}_{a_{g-1}^{-1}}(\lambda^v_{(3)}) \cdot \Bigl( \varepsilon^{\otimes g-1} \otimes \lambda^v_{(4)} \otimes S({r_a}_{(1)}) \otimes \cdots \otimes S({r_a}_{(k-1)}) \otimes \Phi_{0,1} \bigl( \coad^r(r^a)(\lambda^v_{(5)}) \bigr) \Bigr) \\
		& = \cdots \mathfrak{i}_{b_{g}^{-1}}(\lambda^v_{(2)}) \cdot \Bigl( \varepsilon^{\otimes g-2} \otimes \Phi_{0,1} \bigl( S_{\mathcal{L}_{0, 1}(H)}(\lambda^v_{(3)}) \bigr) \triangleright \varepsilon \otimes \lambda^v_{(4)} \otimes S({r_a}_{(1)}) \otimes \cdots \otimes S({r_a}_{(k-1)}) \\
		& \qquad \qquad \qquad \qquad \qquad \qquad \qquad \qquad \qquad \qquad \qquad \qquad \qquad \qquad \qquad \otimes \Phi_{0,1} \bigl( \coad^r(r^a)(\lambda^v_{(5)}) \bigr) \Bigr) \\
		& = \cdots \mathfrak{i}_{b_{g}^{-1}}(\lambda^v_{(2)}) \cdot \Bigl( \varepsilon^{\otimes g-1} \otimes \lambda^v_{(3)} \otimes S({r_a}_{(1)}) \otimes \cdots \otimes S({r_a}_{(k-1)}) \otimes \Phi_{0,1} \bigl( \coad^r(r^a)(\lambda^v_{(4)}) \bigr) \Bigr) \\
		& = \mathfrak{i}_{a_{g}}(\lambda^v_{(1)}) \cdot \Bigl( \varepsilon^{\otimes g-1} \otimes S_{\mathcal{L}_{0, 1}(H)}(\lambda^v_{(2)}) \lambda^v_{(3)} \otimes S({r_a}_{(1)}) \otimes \cdots \otimes S({r_a}_{(k-1)}) \otimes \Phi_{0,1} \bigl( \coad^r(r^a)(\lambda^v_{(4)}) \bigr) \Bigr) \\
		& = \mathfrak{i}_{a_{g}}(\lambda^v_{(1)}) \cdot \Bigl( \varepsilon^{\otimes g} \otimes S({r_a}_{(1)}) \otimes \cdots \otimes S({r_a}_{(k-1)}) \otimes \Phi_{0,1} \bigl( \coad^r(r^a)(\lambda^v_{(2)}) \bigr) \Bigr) \\
		& = \varepsilon^{\otimes g-1} \otimes \Phi_{0,1}(\lambda^v_{(1)}) \triangleright \varepsilon \otimes S({r_a}_{(1)}) \otimes \cdots \otimes S({r_a}_{(k-1)}) \otimes \Phi_{0,1} \bigl( \coad^r(r^a)(\lambda^v_{(2)}) \bigr) \\
		& = \varepsilon^{\otimes g} \otimes S({r_a}_{(1)}) \otimes \cdots \otimes S({r_a}_{(k-1)}) \otimes \Phi_{0,1} \big(\coad^r(r^a)(\lambda^v) \big) \\
		& = \varepsilon^{\otimes g} \otimes 1_H^{\otimes k-1} \otimes \Phi_{0,1}(\lambda^v) = \varepsilon^{\otimes g} \otimes 1_H^{\otimes k-1} \otimes v^{-1}.
	\end{align*}
	Then, let us evaluate $\hat{\tau}_{y^k} \hat{\tau}_{z^k}$ against $\varepsilon^{\otimes g} \otimes 1_H^{\otimes k}$: 
	\begin{align*}
		\hat{\tau}_{y^k} \hat{\tau}_{z^k} & \cdot (\varepsilon^{\otimes g} \otimes 1_H^{\otimes k}) = \hat{\tau}_{y^k} \cdot (\varepsilon^{\otimes g} \otimes 1_H^{\otimes k-1} \otimes v^{-1}) \\
		& = \mathfrak{i}_{a_{g}}(\lambda^v_{(1)}) \mathfrak{i}_{m_{g+1}}(\lambda^v_{(2)}) \cdots \mathfrak{i}_{m_{g+k-1}}(\lambda^v_{(k)}) \cdot (\varepsilon^{\otimes g} \otimes 1_H^{\otimes k-1} \otimes v^{-1}) \\
		& = \mathfrak{i}_{a_{g}}(\lambda^v_{(1)}) \cdot \Bigl( \varepsilon^{\otimes g} \otimes \Phi_{0,1}(\lambda^v_{(2)}) S(r_{a_2}) S({r_{a_3}}_{(1)}) \cdots S({r_{a_{k-1}}}_{(1)}) \\
		& \qquad \qquad \quad \otimes \Phi_{0,1} \bigl( \coad^r(r^{a_2})(\lambda^v_{(3)}) \bigr) S({r_{a_3}}_{(2)}) \cdots S({r_{a_{k-1}}}_{(2)}) \otimes \cdots \\
		& \qquad \qquad \quad \otimes \Phi_{0,1} \bigl( \coad^r( r^{a_{k-2}})(\lambda^v_{(k-1)}) \bigr) S({r_{a_{k-1}}}_{(k-2)}) \otimes \Phi_{0,1} \bigl( \coad^r(r^{a_{k-1}})(\lambda^v_{(k)}) \bigr) \otimes v^{-1} \Bigr) \\ 
		& = \varepsilon^{\otimes g-1} \otimes \Phi_{0,1}(\lambda^v_{(1)}) \triangleright \varepsilon \otimes \Phi_{0,1}(\lambda^v_{(2)}) S(r_{a_2}) S({r_{a_3}}_{(1)}) \cdots S({r_{a_{k-1}}}_{(1)}) \\
		& \qquad \qquad \quad \otimes \Phi_{0,1} \bigl( \coad^r(r^{a_2})(\lambda^v_{(3)}) \bigr) S({r_{a_3}}_{(2)}) \cdots S({r_{a_{k-1}}}_{(2)}) \otimes \cdots \\
		& \qquad \qquad \quad \otimes \Phi_{0,1} \bigl( \coad^r( r^{a_{k-2}})(\lambda^v_{(k-1)}) \bigr) S({r_{a_{k-1}}}_{(k-2)}) \otimes \Phi_{0,1} \bigl( \coad^r(r^{a_{k-1}})(\lambda^v_{(k)}) \bigr) \otimes v^{-1} \\ 
		& = \varepsilon^{\otimes g} \otimes \Phi_{0,1}(\lambda^v_{(1)}) S(r_{a_2}) S({r_{a_3}}_{(1)}) \cdots S({r_{a_{k-1}}}_{(1)}) \\
		& \qquad \qquad \quad \otimes \Phi_{0,1} \bigl( \coad^r(r^{a_2})(\lambda^v_{(2)}) \bigr) S({r_{a_3}}_{(2)}) \cdots S({r_{a_{k-1}}}_{(2)}) \otimes \cdots \\
		& \qquad \qquad \quad \otimes \Phi_{0,1} \bigl( \coad^r( r^{a_{k-2}})(\lambda^v_{(k-2)}) \bigr) S({r_{a_{k-1}}}_{(k-2)}) \otimes \Phi_{0,1} \bigl( \coad^r(r^{a_{k-1}})(\lambda^v_{(k-1)}) \bigr) \otimes v^{-1} \\ 
		& = (\varepsilon^{\otimes g} \otimes \clubsuit) (\varepsilon^{\otimes g} \otimes 1_H^{\otimes k-1} \otimes v^{-1}),
	\end{align*}
	where $\clubsuit$ is defined in \labelcref{eq-trefle_relation_piqure}.
	Let us now evaluate $\hat{\tau}_{x^k} \hat{\tau}_{y^k} \hat{\tau}_{z^k}$ against $\varepsilon^{\otimes g} \otimes 1_H^{\otimes k}$: 
	\begin{align*}
		& \hat{\tau}_{x^k} \hat{\tau}_{y^k} \hat{\tau}_{z^k} \cdot (\varepsilon^{\otimes g} \otimes 1_H^{\otimes k}) = \hat{\tau}_{x^k} \cdot \bigl( (\varepsilon^{\otimes g} \otimes \clubsuit) (\varepsilon^{\otimes g} \otimes 1_H^{\otimes k-1} \otimes v^{-1}) \bigr) \\
		& = \mathfrak{i}_{b_g^{-1}}(\lambda^v_{(1)}) \mathfrak{i}_{a_{g-1}}(\lambda^v_{(2)}) \mathfrak{i}_{b_g}(\lambda^v_{(3)}) \mathfrak{i}_{m_{g+1}}(\lambda^v_{(4)}) \cdots \mathfrak{i}_{m_{g+k}}(\lambda^v_{(k+3)}) \cdot \bigl( (\varepsilon^{\otimes g} \otimes \clubsuit) (\varepsilon^{\otimes g} \otimes 1_H^{\otimes k-1} \otimes v^{-1}) \bigr) \\
		& = \cdots \mathfrak{i}_{b_g}(\lambda^v_{(3)}) \cdot \Bigl( \varepsilon^{\otimes g} \otimes \Phi_{0,1}(\lambda^v_{(4)}) S(r_{b_2}) S({r_{b_3}}_{(1)}) \cdots S({r_{b_{k}}}_{(1)}) \\
		& \qquad \otimes \Phi_{0,1} \bigl( \coad^r(r^{b_2})(\lambda^v_{(5)}) \bigr) S({r_{b_3}}_{(2)}) \cdots S({r_{b_{k}}}_{(2)}) \otimes \cdots \otimes \Phi_{0,1} \bigl( \coad^r(r^{b_{k-1}})(\lambda^v_{(k+2)}) \bigr) S({r_{b_{k}}}_{(k-1)}) \\
		& \qquad \qquad \qquad \qquad \qquad \qquad \qquad \qquad \otimes \Phi_{0,1} \bigl( \coad^r(r^{b_{k}})(\lambda^v_{(k+3)}) \bigr) (\varepsilon^{\otimes g} \otimes \clubsuit) (\varepsilon^{\otimes g} \otimes 1_H^{\otimes k-1} \otimes v^{-1}) \Bigr) \\
		& = \cdots \mathfrak{i}_{a_{g-1}}(\lambda^v_{(2)}) \cdot \Bigl( \varepsilon^{\otimes g-1} \otimes \lambda^v_{(3)} \otimes \Phi_{0,1}(\lambda^v_{(4)}) S(r_{b_2}) S({r_{b_3}}_{(1)}) \cdots S({r_{b_{k}}}_{(1)}) \\
		& \quad \otimes \Phi_{0,1} \bigl( \coad^r(r^{b_2})(\lambda^v_{(5)}) \bigr) S({r_{b_3}}_{(2)}) \cdots S({r_{b_{k}}}_{(2)}) \otimes \cdots \otimes \Phi_{0,1} \big( \coad^r( r^{b_{k-1}})(\lambda^v_{(k+2)})\big) S({r_{b_{k}}}_{(k-1)}) \\
		& \qquad \qquad \qquad \qquad \qquad \qquad \qquad \quad \otimes \Phi_{0,1}\big(\coad^r(r^{b_{k}})(\lambda^v_{(k+3)})\big) (\varepsilon^{\otimes g} \otimes \clubsuit) (\varepsilon^{\otimes g} \otimes 1_H^{\otimes k-1} \otimes v^{-1}) \Bigr) \\
		& = \mathfrak{i}_{b_g^{-1}}(\lambda^v_{(1)}) \cdot \Bigl( \varepsilon^{\otimes g-2} \otimes \Phi_{0,1}(\lambda^v_{(2)}) \triangleright \varepsilon \otimes \lambda^v_{(3)} \otimes \Phi_{0,1}(\lambda^v_{(4)}) S(r_{b_2}) S({r_{b_3}}_{(1)}) \cdots S({r_{b_{k}}}_{(1)}) \\
		& \quad \otimes \Phi_{0,1} \bigl( \coad^r(r^{b_2})(\lambda^v_{(5)}) \bigr) S({r_{b_3}}_{(2)}) \cdots S({r_{b_{k}}}_{(2)}) \otimes \cdots \otimes \Phi_{0,1} \bigl( \coad^r( r^{b_{k-1}})(\lambda^v_{(k+2)}) \bigr) S({r_{b_{k}}}_{(k-1)}) \\
		& \qquad \qquad \qquad \qquad \qquad \qquad \qquad \quad \otimes \Phi_{0,1} \bigl( \coad^r(r^{b_{k}})(\lambda^v_{(k+3)}) \bigr) (\varepsilon^{\otimes g} \otimes \clubsuit) (\varepsilon^{\otimes g} \otimes 1_H^{\otimes k-1} \otimes v^{-1}) \Bigr) \\
		& = \mathfrak{i}_{b_g^{-1}}(\lambda^v_{(1)}) \cdot \Bigl( \varepsilon^{\otimes g-1} \otimes \lambda^v_{(2)} \otimes \Phi_{0,1}(\lambda^v_{(3)}) S(r_{b_2}) S({r_{b_3}}_{(1)}) \cdots S({r_{b_{k}}}_{(1)}) \\
		& \quad \otimes \Phi_{0,1} \bigl( \coad^r(r^{b_2})(\lambda^v_{(4)}) \bigr) S({r_{b_3}}_{(2)}) \cdots S({r_{b_{k}}}_{(2)}) \otimes \cdots \otimes \Phi_{0,1} \bigl( \coad^r( r^{b_{k-1}})(\lambda^v_{(k+1)}) \bigr) S({r_{b_{k}}}_{(k-1)}) \\
		& \qquad \qquad \qquad \qquad \qquad \qquad \qquad \quad \otimes \Phi_{0,1} \bigl( \coad^r(r^{b_{k}})(\lambda^v_{(k+2)}) \bigr) (\varepsilon^{\otimes g} \otimes \clubsuit) (\varepsilon^{\otimes g} \otimes 1_H^{\otimes k-1} \otimes v^{-1}) \Bigr) \\
		& = \varepsilon^{\otimes g-1} \otimes S_{\mathcal{L}_{0, 1}(H)}(\lambda^v_{(1)}) \lambda^v_{(2)} \otimes \Phi_{0,1}(\lambda^v_{(3)}) S(r_{b_2}) S({r_{b_3}}_{(1)}) \cdots S({r_{b_{k}}}_{(1)}) \\
		& \quad \otimes \Phi_{0,1} \bigl( \coad^r(r^{b_2})(\lambda^v_{(4)}) \bigr) S({r_{b_3}}_{(2)}) \cdots S({r_{b_{k}}}_{(2)}) \otimes \cdots \otimes \Phi_{0,1} \bigl( \coad^r( r^{b_{k-1}})(\lambda^v_{(k+1)}) \bigr) S({r_{b_{k}}}_{(k-1)}) \\
		& \qquad \qquad \qquad \qquad \qquad \qquad \qquad \quad \otimes \Phi_{0,1} \bigl( \coad^r(r^{b_{k}})(\lambda^v_{(k+2)}) \bigr) (\varepsilon^{\otimes g} \otimes \clubsuit) (\varepsilon^{\otimes g} \otimes 1_H^{\otimes k-1} \otimes v^{-1}) \\
		& = \varepsilon^{\otimes g} \otimes \Phi_{0,1}(\lambda^v_{(1)}) S(r_{b_2}) S({r_{b_3}}_{(1)}) \cdots S({r_{b_{k}}}_{(1)}) \\
		& \quad \otimes \Phi_{0,1} \bigl( \coad^r(r^{b_2})(\lambda^v_{(2)}) \bigr) S({r_{b_3}}_{(2)}) \cdots S({r_{b_{k}}}_{(2)}) \otimes \cdots \otimes \Phi_{0,1} \bigl( \coad^r( r^{b_{k-1}})(\lambda^v_{(k-1)}) \bigr) S({r_{b_{k}}}_{(k-1)}) \\
		& \qquad \qquad \qquad \qquad \qquad \qquad \qquad \qquad \quad \otimes \Phi_{0,1} \bigl( \coad^r(r^{b_{k}})(\lambda^v_{(k)}) \bigr) (\varepsilon^{\otimes g} \otimes \clubsuit) (\varepsilon^{\otimes g} \otimes 1_H^{\otimes k-1} \otimes v^{-1}) \\
		& = (\varepsilon^{\otimes g} \otimes \spadesuit) (\varepsilon^{\otimes g} \otimes \clubsuit) (\varepsilon^{\otimes g} \otimes 1_H^{\otimes k-1} \otimes v^{-1}).
	\end{align*}
	where $\spadesuit$ is defined just above \labelcref{eq-membre_de_droite_relation_piqure}.
	Comparing this with \labelcref{eq-membre_de_droite_relation_piqure}, we obtain:
	\[ \hat{\tau}_{x^k} \hat{\tau}_{y^k} \hat{\tau}_{z^k} = \mathcal{R}_{g,n}^1(V) \bigl( \mathfrak{i}_{m_{g+k}} (\lambda^v) \bigr) \hat{\tau}_{\gamma_1^k} \hat{\tau}_{\gamma_2^k} \hat{\tau}_{\gamma_3^k} \text{ in } \GL(V_{g,n}^1). \qedhere \]
\end{proof}

\subsection{The intrinsic extension}\label{subsec-extension_intrinseque}

In this subsection, we construct a central extension $\widetilde{\Gamma}_{g,n}^s(2,l)$ of $\Gamma_{g,n}^s$ on which we can define a linearization of the projective representation $\rho_{g,n}^s$. 
In order to study $\widetilde{\Gamma}_{g,n}^s(2,l)$, we will compare it to another central extension $\widetilde{\Gamma}_{g,n}^s(l_0)$ which will be a generator of $H^2(\Gamma_{g}^s, \mathbb{Z})$ for $n = 0$ and $l_0 = 0$.

\medskip

Let $V := \{ V_i \}_{i = 1, \cdots, n}$ be a fixed collection of simple $H$-modules. 
Recall from \Cref{subsec-representation_projective_rhogn1,subsec-representation_projective_rhogn} (by using \Cref{lem-f_chapeau_dans_Lgninv} to see that the elements $\hat{\tau}_\gamma$ belong to $\mathcal{L}_{g, n}^{\mathrm{inv}}(H)$) that the projective representations $\rho_{g,n}^s : \Gamma_{g,n}^s \to \PGL(V_{g,n}^s)$ are quotients of the compositions of maps:
\[
\begin{array}{l|ccccl}
	\hat{\rho}_{g,n}^s : & F_\tau & \longrightarrow & \mathcal{L}_{g, n}^{\mathrm{inv}}(H)^\times & \longrightarrow & \GL(V_{g,n}^s)\\
    & \tau_\gamma & \longmapsto & \hat{\tau}_\gamma & \longmapsto & \mathcal{R}_{g,n}^s(V)(\hat{\tau}_\gamma).
\end{array}
\]  
The computations made in \Cref{subsec-calcul_relations} are also true for surfaces without boundary component, since the curves appearing in these computations are also non-separating simple closed curves of $\Sigma_{g,n}$ and the elements of $V_{g,n}^1$ on which we evaluate are elements of $V_{g,n}$. 
Therefore, in what follows we can deal with the two representations $\rho_{g,n}$ and $\rho_{g,n}^1$ simultaneously. 
We distinguish the case $g = 1$ from the cases $g \geq 2$ because of the different presentations of $\Gamma_1^s$. 

\smallskip
\noindent
Denote by $l_0 \in \mathbb{N}$ the order of the element $\Bigl( \frac{\lambda(v^{-1})}{\lambda(v)} \Bigr)^2 \in \mathbb{K}$.
We say that an element $x$ is of order $0$ if $x$ has infinite order.

\subsubsection{\textbf{Case \texorpdfstring{$g = 1$}{g=1}: the extension \texorpdfstring{$\widetilde{\Gamma}_{1}(2,l_0)$}{widetildeGamma1s(2)}}}\label{subsubsec-gammatilde12}

Recall that we have the following presentations (see \labelcref{eq-presentation_Gamma11_et_Gamma1}): 
\[ \Gamma_1^1 = \langle \tau_a, \tau_b ~ \vert ~ \tau_a \tau_b \tau_a = \tau_b \tau_a \tau_b \rangle ~ \text{ and } ~ \Gamma_1 = \langle \tau_a, \tau_b ~ \vert ~ \tau_a \tau_b \tau_a = \tau_b \tau_a \tau_b \text{ and } (\tau_a \tau_b)^6 = 1 \rangle. \]
\begin{definition}\label{def-gammatilde12}
	The group $\widetilde{\Gamma}_{1}(2, l_0)$ is defined by the following presentation: 
    \begin{equation*}
		\widetilde{\Gamma}_{1}(2,l_0) = \langle \tau_a, \tau_b, T ~ \vert ~ \tau_a \tau_b \tau_a = \tau_b \tau_a \tau_b, (\tau_a \tau_b)^6 = T^2, T^{l_0} = 1 \text{ and } T \text{ is central} \rangle,
	\end{equation*}
    where $T^{0}$ is by definition the empty word.
\end{definition}
\noindent
\Cref{prop-calcul_relation_1-tresse,prop-calcul_relation_ab_puissance6} allow us to define a linear representation of $\widetilde{\Gamma}_{1}(2,l_0)$ as follows:
\[ 
\begin{array}{l|ccl}
	\tilde{\rho}_{1}(2): & \widetilde{\Gamma}_{1}(2,l_0) & \longrightarrow & \GL (V_{1,0}) \\
	& \tau_\gamma & \longmapsto & \mathcal{R}_{1,0}(\hat{\tau}_\gamma), \\
	& T & \longmapsto & \frac{\lambda(v^{-1})}{\lambda(v)} \id_{V_{1,0}}.  
\end{array}
\]
Furthermore, we have the following commutative diagrams, where $\pi(\tau_\gamma) = \tau_\gamma$ for all the generators $\tau_\gamma$ and $\pi(T) = 1_{\Gamma_1}$ (the identity element of $\Gamma_1$):
\begin{equation}\label{eq-diagramme_commutatif_rhotildeg2_en_genre_1}
    \begin{tikzcd}
        \widetilde{\Gamma}_{1}(2,l_0) \arrow[r, "\tilde{\rho}_{1}(2)"] \arrow[d, "\pi"] & \GL (V_{1,0}) \arrow[d, "\pi"] \\
        \Gamma_1 \arrow[r, "\rho_{1}"] & \PGL (V_{1,0})
    \end{tikzcd}
\end{equation}

\begin{remarque}\label{rem-linearisation_de_rho_en_genre_1}
	\begin{enumerate}[leftmargin=*]
		\item In the case of $\Gamma_1^1$, $\rho_{1}^1$ is already a linear representation since the only relator is trivial (see \Cref{prop-calcul_relations_0-tresse}).
    	This makes sense since $H^2(\Gamma_1^1, \mathbb{Z}) = 0$.
		\item We can linearize the projective representation $\rho_{1}$ of $\Gamma_1$ if there exists a sixth root $\kappa \in \mathbb{K}$ of $\lambda(v^{-1})\lambda(v)^{-1}$ by setting $\tilde{\tau}_a := \kappa^{-1} \hat{\tau}_a$ and $\tilde{\tau}_b := \kappa^{-1} \hat{\tau}_b$.
		\item The linear representation $\tilde{\rho}_{1}(2)$ does not depend on the choice of the left integral $\lambda$.
		Indeed, since the space of left integrals is one dimensional (see \labelcref{item-espace_des_integrales_dimension_1}), the image of $T$ under $\tilde{\rho}_1(2)$ remains invariant because the scalar of $\lambda(v^{-1})$ will simplify to that of $\lambda(v)$.
	\end{enumerate}
\end{remarque}
\noindent
Let $\widetilde{\Gamma}_{1}(l_0)$ be the central extension of $\Gamma_1$ defined by the same presentation as $\widetilde{\Gamma}_1(2,l_0)$ in \Cref{def-gammatilde12} except for the relation $(\tau_a \tau_b)^6 = T^2$ which is replaced by $(\tau_a \tau_b)^6 = T$.
Using the presentations of $\widetilde{\Gamma}_{1}(l_0)$ and $\widetilde{\Gamma}_{1}(2,l_0)$, we obtain the following commutative diagram: 
\[
\begin{tikzcd}
0  \arrow[r] & \langle T  ~ \vert ~ T^{l_0} = 1 \rangle \arrow[r] \arrow[d,"\varphi"] & \widetilde{\Gamma}_{1}(l_0) \arrow[r, "\pi"] \arrow[d,"\psi"]& \Gamma_{1} \arrow[r] \arrow[d,"\id_{\Gamma_{1}}"]& 1\\
0  \arrow[r] & \langle T  ~ \vert ~ T^{l_0} = 1 \rangle  \arrow[r] & \widetilde{\Gamma}_{1}(2,l_0) \arrow[r, "\pi"] & \Gamma_{1} \arrow[r] & 1
\end{tikzcd}
\]
where $\varphi$, $\psi$, and $\pi$ are the group morphisms defined by:
\begin{itemize}
	\item $\varphi(T) = T^2$,
    \item $\psi (\tau_{\gamma}) = \tau_{\gamma}$ for all the generators $\tau_{\gamma}$ and $\psi(T) = T^2$,
    \item $\pi$ is the quotient map.
\end{itemize}

Denote by $c_{\widetilde{\Gamma}_{1}(l_0)}$ (resp. by $c_{\widetilde{\Gamma}_{1}(2,l_0)}$) the cohomology class in $H^2(\Gamma_{1}, \mathbb{Z}_{l_0})$ associated with the central extension $\widetilde{\Gamma}_{1}(l_0)$ (resp. $\widetilde{\Gamma}_{1}(2,l_0)$).

\begin{proposition}\label{prop-extension_ordre_2_en_genre_1}
	We have $c_{\widetilde{\Gamma}_{1}(2,l_0)} = 2 c_{\widetilde{\Gamma}_{1}(l_0)} \in H^2(\Gamma_1, \mathbb{Z}_{l_0})$, where $\mathbb{Z}_0 := \mathbb{Z}$ by definition.
\end{proposition}

\begin{proof}
    First, consider the case where $l_0 = 0$.
	Let us construct a cocycle $c_{\widetilde{\Gamma}_{1}(0)}$ associated with $\widetilde{\Gamma}_{1}(0)$ and a cocycle $c_{\widetilde{\Gamma}_{1}(2,0)}$ associated with $\widetilde{\Gamma}_{1}(2,0)$.
	Let $s$ be a set-theoretic section of $\widetilde{\Gamma}_{1}(0)$ such that $s(\tau_\gamma) = \tau_\gamma$, for all the generators $\tau_\gamma$.
    The map:
    \[
    \begin{array}{l|ccl}
		\mu_{\widetilde{\Gamma}_{1}(0)}: & \Gamma_{1} \times \Gamma_{1} & \longrightarrow & \langle T \rangle \\
        & (f, g) & \longmapsto & s(f) s(g) s(g \circ f)^{-1}
	\end{array}
    \]
    is a 2-cocycle associated with $\widetilde{\Gamma}_{1}(0)$. 
	In order to obtain a cocycle with values in $\mathbb{Z}$, we compose $\mu_{\widetilde{\Gamma}_{1}(0)}$ with the isomorphism $i_T : \langle T \rangle \to \mathbb{Z}$ defined by $i_T(T) = 1$ and we denote this cocycle by $c_{\widetilde{\Gamma}_{1}(0)}$.
    Similarly, let us consider the set-theoretic section $\psi \circ s$ of $\widetilde{\Gamma}_{1}(2,0)$. 	
	The map: 
    \[
    \begin{array}{l|ccl}
        \mu_{\widetilde{\Gamma}_{1}(2,0)}: & \Gamma_{1} \times \Gamma_{1} & \longrightarrow & \langle T \rangle \\
		& (f, g) & \longmapsto & (\psi \circ s)(f) (\psi \circ s)(g) (\psi \circ s)(g \circ f)^{-1}
    \end{array}
    \]
    induces a cocycle $c_{\widetilde{\Gamma}_{1}(2,0)} := i_T \circ \mu_{\widetilde{\Gamma}_{1}(2,0)}$ with values in $\mathbb{Z}$.
	Note that $i_T \circ \varphi = 2 \, i_T$. 
    Indeed, $(i_T \circ \varphi)(T) = i_T(T^2) = 2$.
	Thus we obtain: 
    \begin{align*}
        c_{\widetilde{\Gamma}_{1}(2,0)}(f,g) & = i_T \bigl( (\psi \circ s)(f) (\psi \circ s)(g) (\psi \circ s)(g \circ f)^{-1} \bigr) \\
		& = (i_T \circ \psi) \bigl( s(f) s(g) s(g \circ f)^{-1} \bigr) = 2 \, i_T \bigl( s(f) s(g) s(g \circ f)^{-1} \bigr) = 2 \, c_{\widetilde{\Gamma}_{1}(0)}(f,g).
    \end{align*}
    Hence the result. 

	\smallskip
	\noindent
	For the case $l_0 \geq 1$, it suffices to note that the extensions $\widetilde{\Gamma}_{1}(l_0)$ and $\widetilde{\Gamma}_{1}(2,l_0)$ are isomorphic to the pushouts of $\widetilde{\Gamma}_{1}(0)$ and $\widetilde{\Gamma}_{1}(2,0)$ induced by the map $\langle T \rangle \to \langle T ~ \vert ~ T^{l_0} = 1 \rangle$.
	Consequently, the equality $c_{\widetilde{\Gamma}_{1}(2,l_0)} = 2 c_{\widetilde{\Gamma}_{1}(l_0)}$ also holds.
\end{proof}
\noindent
We know that there is a group isomorphism $H^2(\Gamma_{1}, \mathbb{Z}) \simeq \mathbb{Z}_{12}$ (see \Cref{prop-H^2_petit_genre}). 
The following result allows us to identify $c_{\widetilde{\Gamma}_{1}(2,0)}$ more precisely:
\begin{proposition}\label{prop-generateur_H^2_genre_1}
	The cohomology class $c_{\widetilde{\Gamma}_1(0)}$ is a generator of $H^2(\Gamma_{1}, \mathbb{Z})$. 
\end{proposition}

\begin{proof} 
	We know that (see \Cref{prop-H^2_petit_genre} and \cite[Chap. III - Sec. 1 - Ex. 2]{brown_cohomology_1982}): 
	\[ H^2(\Gamma_1, \mathbb{Z}) \simeq \mathbb{Z}_{12}, \qquad H^2(\mathbb{Z}_{n}, \mathbb{Z}) \simeq \mathbb{Z}_{n}. \]
	Note that $\widetilde{\Gamma}_{1}(0)$ is isomorphic to $\Gamma_1^1$ and we denote by $c_{\Gamma_1^1}$ its associated cohomology class. 
	Then to prove the statement, we show that the restriction of $c_{\Gamma_1^1}$ to the subgroup $\langle \tau_a \tau_b \tau_a \rangle$, isomorphic to $\mathbb{Z}_4$, is a generator and thus of order 4, and that its restriction to $\langle \tau_a \tau_b \rangle$, isomorphic to $\mathbb{Z}_6$, is a generator and thus of order 6. \\
	Let us start with the restriction to $\mathbb{Z}_4$. 
	Note that in $\Gamma_1$ we have $(\tau_a \tau_b)^6 = (\tau_a \tau_b \tau_a)^4$, so we obtain a commutative diagram: 
	\[
	\begin{tikzcd}
		0  \arrow[r] & \langle (\tau_a \tau_b \tau_a)^4 \rangle_{\mathrm{free}} \arrow[r]  & \Gamma_{1}^1 \arrow[r, "\pi"] & \Gamma_{1} \arrow[r] & 1\\
		0 \arrow[r] &  \langle (\tau_a \tau_b \tau_a)^4 \rangle_{\mathrm{free}} \arrow[u, "\id"'] \arrow[r] &  (i_4)^*\Gamma_1^1 \arrow[u] \arrow[r] & \langle \tau_a \tau_b \tau_a \rangle \arrow[r] \arrow[u, hook, "i_4"'] & 1.
	\end{tikzcd}
	\]
	where $i_4$ is the injection and $(i_4)^*\Gamma_1^1$ is the pullback of $\Gamma_1^1$. 
	Let us define the following set-theoretic section: 
	\[ 
	\begin{array}{l|ccl}
		s: & \langle \tau_a \tau_b \tau_a \rangle & \longrightarrow & (i_4)^*\Gamma_1^1 \\
		& (\tau_a \tau_b \tau_a)^k & \longmapsto & \big(1_{\langle \tau_a \tau_b \tau_a \rangle} , (\tau_a \tau_b \tau_a)^k \big) \in \langle \tau_a \tau_b \tau_a \rangle \times \Gamma_1^1.
	\end{array}
	\]
	We have $s((\tau_a \tau_b \tau_a)^4) = (\tau_a \tau_b \tau_a)^4$, which is the generator of $\langle (\tau_a \tau_b \tau_a)^4 \rangle_{\mathrm{free}}$.
	Thus, the class associated with the extension $(i_4)^*\Gamma_1^1$ generates $H^2(\langle \tau_a \tau_b \tau_a \rangle, \mathbb{Z}) \simeq H^2(\mathbb{Z}_{4}, \mathbb{Z}) \simeq \mathbb{Z}_{4}$. 
	Doing the same with $i_6 : \langle \tau_a \tau_b \rangle \to \Gamma_1$, we obtain that the class associated with the extension $(i_6)^*\Gamma_1^1$ is a generator of $H^2(\langle \tau_a \tau_b \rangle, \mathbb{Z}) \simeq H^2(\mathbb{Z}_{6}, \mathbb{Z}) \simeq \mathbb{Z}_{6}$. 
	Consequently, $c_{\Gamma_{1}^1}$ is of order a multiple of 4 and 6 so at least 12. 
	However, $H^2(\Gamma_1, \mathbb{Z}) \simeq \mathbb{Z}_{12}$ so $c_{\Gamma_{1}^1}$ is of order 12 and is therefore a generator.
\end{proof}

\subsubsection{\textbf{Case \texorpdfstring{$g \geq 2$}{g superieur a 2} : the extension \texorpdfstring{$\widetilde{\Gamma}_{g,n}^s(2,l)$}{widetildeGammagns(2)}}}\label{subsubsec-gammatildegns2}

Recall that for all $g \geq 2$, $s \geq 0$, and $n \geq 1$, $\Gamma_{g,n}^s$ has the Gervais presentation (see \Cref{thm-presentation_de_Gervais}).

\noindent
Also, recall that $l_0 \in \mathbb{N}$ denotes the order of the element $\Bigl( \frac{\lambda(v^{-1})}{\lambda(v)} \Bigr)^2 \in \mathbb{K}$.
Recall that the element $\mathfrak{i}_{m_{g+i}}(\lambda^v)$ acts as a scalar on $V_{g,n}^s$ (see \Cref{prop-calcul_relation_piqure_k_avec_formules_Derived}).
Denote by $l_i \in \mathbb{N}$ the order of this element. 
We say that an element $x$ is of order $0$ if $x$ has infinite order. 
Let $l := (l_0, \cdots, l_n)$.

\begin{definition}\label{def-gammatildegns2}
	For all $g \geq 2$, $n \geq 0$, and $s \in \{ 0, 1 \}$, the group $\widetilde{\Gamma}_{g,n}^s(2,l)$ is defined by the following presentation: 
    \begin{itemize}
        \item Generators: $\{ \tau_{\gamma} ~ \vert ~ \text{non-separating simple closed curve } \gamma \subset \Sigma_{g,n}^s \} \cup \{ T, E_1, \cdots, E_n \}$.
        \item Relations: 
        \begin{enumerate}
            \item All 0 and 1-braid relations; \label{item-0_et_1-tresse}
            \item One lantern relation, if $g \geq 3$; \label{item-lanterne}
            \item $r_c = T^2$, where $r_c$ is the 3-chain relator; \label{item-rc_egale_Tcarre}
            \item $T^{l_0} = 1$; \label{item-T_puissance_l0}
            \item The element $T$ is central; \label{item-T_est_central}
            \item $r_{p_i} = E_i$, where $r_{p_i}$ is the $i$-th puncture relator for each puncture $p_i$; \label{item-rpi_egale_Ei}
            \item $E_i^{l_i} = 1$, for all $i \in [\![ 1, n]\!]$; \label{item-Ei_puissance_li}
            \item $E_i$ is central, for all $i \in [\![ 1, n]\!]$; \label{item-Ei_est_central}
        \end{enumerate}
    \end{itemize}
	where $T^0$ and $E_i^0$ are by definition the empty words. 
	When $n = 0$, we denote the group by $\widetilde{\Gamma}_{g}^s(2,l_0)$.
\end{definition}
\noindent
Let $\xi_i$ be the element of $\mathbb{K}$ such that $\mathcal{R}_{g,n}^s(V)(\mathfrak{i}_{m_{g+i}}(\lambda^v)) = \xi_i \id_{V_{g,n}^s}$.
The results of \Cref{subsec-calcul_relations} and the relations of the presentation in \Cref{def-gammatildegns2} allow us to define a linear representation of $\widetilde{\Gamma}_{g,n}^s(2,l)$ as follows: 
\[ 
\begin{array}{l|ccl}
	\tilde{\rho}_{g,n}^s(2): & \widetilde{\Gamma}_{g,n}^s(2,l) & \longrightarrow & \GL(V_{g,n}^s) \\
	& \tau_\gamma & \longmapsto & \mathcal{R}_{g,n}^s(V)(\hat{\tau}_\gamma), \\
	& T & \longmapsto & \frac{\lambda(v^{-1})}{\lambda(v)} \id_{V_{g,n}^s}, \\
	& E_i & \longmapsto & \xi_i \id_{V_{g,n}^s}.
\end{array}
\]
Furthermore, we have the following commutative diagram, where $\pi(\tau_\gamma) = \tau_\gamma$, $\pi(T) = 1_{\Gamma_{g,n}^s}$, and $\pi(E_i) = 1_{\Gamma_{g,n}^s}$, for all $i \in [\![ 1, n ]\!]$: 
\begin{equation}\label{eq-diagramme_commutatif_rhotildegnV2}
	\begin{tikzcd}
        \widetilde{\Gamma}_{g,n}^s(2,l) \arrow[r, "\tilde{\rho}_{g,n}^s(2)"] \arrow[d, "\pi"] & \GL(V_{g,n}^s) \arrow[d, "\pi"] \\
        \Gamma_{g,n}^s \arrow[r, "\rho_{g,n}^s"] & \PGL(V_{g,n}^s)
    \end{tikzcd}
\end{equation}

\begin{remarque}
	\begin{enumerate}[leftmargin=*]
		\item When $n = 0$, the same arguments as in the third statement of \Cref{rem-linearisation_de_rho_en_genre_1} imply that the linearization $\tilde{\rho}_{g}^s(2)$ does not depend on the choice of the left integral $\lambda$.
		\item When $n \geq 1$, the linear representation $\tilde{\rho}_{g,n}^s(2)$ depends on the choice of the left integral $\lambda$.
		Choosing a different integral simultaneously renormalizes $\tilde{\rho}_{g,n}^s(2)(E_i)$, for all $i \in [\![1, n]\!]$.
	\end{enumerate}
\end{remarque}

For all $g \geq 2$, $n \geq 0$, and $s \in \{ 0,1 \}$, let $\widetilde{\Gamma}_{g,n}^s(l_0)$ be the central extension of $\Gamma_{g,n}^s$ defined by the presentation $\langle F \vert R \rangle$, where the set of generators $F$ consists of all Dehn twists $\tau_\gamma$ along non-separating simple closed curves $\gamma \subset \Sigma_{g,n}^s$ and $T$, and the set of relations $R$ consists of all the relations in the Gervais presentation of $\Gamma_{g,n}^s$ (see \Cref{thm-presentation_de_Gervais}), except for the relation $r_{c} = 1$ which is replaced by $r_{c} = T$, and the relations \labelcref{item-T_puissance_l0} and \labelcref{item-T_est_central} from \Cref{def-gammatildegns2}. 
Denote by $c_{\widetilde{\Gamma}_{g,n}^s(l_0)} \in H^2(\Gamma_{g,n}^s, \mathbb{Z}_{l_0})$ the cohomology class associated with $\widetilde{\Gamma}_{g,n}^s(l_0)$.
Let $c_{\widetilde{\Gamma}_{g,n}^s(2,l)}$ be the cohomology class in $\bigoplus_{i=0}^n H^2(\Gamma_{g,n}^s, \mathbb{Z}_{l_i})$ associated with the central extension $\widetilde{\Gamma}_{g,n}^s(2,l)$.

\begin{theoreme}\label{thm-cohomologie_Gammagns2l}
    For all $g \geq 2$, $n \geq 0$, and $s \in \{ 0, 1 \}$, we have:
	\[ c_{\widetilde{\Gamma}_{g,n}^s(2,l)} = 2 c_{\widetilde{\Gamma}_{g,n}^s(l_0)} \oplus [e_1](l_1) \oplus \cdots \oplus [e_n](l_n) \in \bigoplus_{i=0}^n H^2(\Gamma_{g,n}^s, \mathbb{Z}_{l_i}), \]
	where $[e_i](l_i)$ is the pushout of the $i$-th Euler class $[e_i]$ induced by the map $\mathbb{Z} \to \mathbb{Z}_{l_i}$.
\end{theoreme}

\begin{proof}
    Let us first consider the case where $l = (0, \cdots, 0)$.
	For all $i \in [\![1, n]\!]$, we define the group $\widetilde{\Gamma}_{g,n}^s(e_i,0)$ by the presentation $\langle F_i \vert R_i \rangle$, where the set of generators $F_i$ consists of all Dehn twists $\tau_\gamma$ along non-separating simple closed curves $\gamma \subset \Sigma_{g,n}^s$ and $E_i$, and the set of relations $R_i$ consists of all the relations in the Gervais presentation of $\Gamma_{g,n}^s$ (see \Cref{thm-presentation_de_Gervais}), except for the relation $r_{p_i} = 1$ which is replaced by $r_{p_i} = E_i$, and the relations \labelcref{item-Ei_puissance_li,item-Ei_est_central} from \Cref{def-gammatildegns2} only for $E_i$. 
	
	\smallskip
	\noindent	
	We claim that the cohomology class $c_{\widetilde{\Gamma}_{g,n}^s(e_i,0)}$ associated with $\widetilde{\Gamma}_{g,n}^s(e_i,0)$ is the $i$-th Euler class $[e_i] \in H^2(\Gamma_{g,n}^s, \mathbb{Z}$).
	Indeed, recall that $[e_i]$ is the cohomology class associated with the central extension $0 \to \langle \tau_{\partial_i} \rangle \to \Gamma_{g,n-1}^{s+1} \xrightarrow{~\mathrm{Cap}~} \Gamma_{g,n}^{s} \to 1$, where $\mathrm{Cap} : \Gamma_{g,n-1}^{s+1} \to \Gamma_{g,n}^s$ is the capping morphism.
	Let us construct an isomorphism between the central extensions $\widetilde{\Gamma}_{g,n}^s(e_i,0)$ and $\Gamma_{g,n-1}^{s+1}$. 
    We define the following map: 
    \[ \begin{array}{l|ccl}
		\psi: & \langle F_i \rangle & \longrightarrow & \Gamma_{g,n-1}^{s+1} \\
        & \tau_\gamma & \longmapsto & \tau_\gamma \\ 
        & E_i & \longmapsto & \tau_{\partial_i}.
    \end{array} \] 
	First, it is clear that $\psi$ is surjective because $\Gamma_{g,n-1}^{s+1}$ is generated by Dehn twists along non-separating simple closed curves of $\Sigma_{g,n-1}^{s+1}$ (which are also non-separating simple closed curves of $\Sigma_{g,n}^s$). 
    Next, the 0-braid and 1-braid relators, the lantern, the chain, the $j$-th puncture relators for all $j \neq i$, and $(r_{p_i}E_i^{-1})$ are in the kernel of $\psi$.
	Consequently, $\psi$ induces an isomorphism $\psi : \widetilde{\Gamma}_{g,n}^s(e_i,0) \to \Gamma_{g,n-1}^{s+1}$. 
	Finally, it is clear that $\psi$ satisfies the following commutative diagram, where $\varphi(E_i) = \tau_{\partial_i}$, $\pi_i(\tau_\gamma) = \tau_\gamma$, and $\pi_i(E_i) = 1_{\Gamma_{g,n}^s}$: 
	\begin{equation}\label{eq-diagramme_commutatif_extension_ieme_classe_Euler}
		\begin{tikzcd}
        	0 \arrow[r] & \langle \tau_{\partial_i} \rangle \arrow[r] & \Gamma_{g,n-1}^{s+1} \arrow[r, "\pi"] & \Gamma_{g,n}^s \arrow[r] & 1 \\
        	0 \arrow[r] & \langle E_i \rangle \arrow[r, "\iota_i"] \arrow[u, "\varphi"'] & \widetilde{\Gamma}_{g,n}^s(e_i,0) \arrow[r, "\pi_i"] \arrow[u, "\psi"'] & \Gamma_{g,n}^s \arrow[r] \arrow[u, "\id_{\Gamma_{g,n}^s}"'] & 1,
    	\end{tikzcd} 
	\end{equation}
	hence $c_{\widetilde{\Gamma}_{g,n}^s(e_i,0)} = [e_i] \in H^2(\Gamma_{g,n}^s, \mathbb{Z})$, as claimed.

	\smallskip
	\noindent
	Denote by $\widetilde{\Gamma}_{g,n}^s(2,0)$ the group defined by the presentation $\langle F_0 \vert R_0 \rangle$, where the set of generators $F_0$ consists of all Dehn twists $\tau_\gamma$ along non-separating simple closed curves $\gamma \subset \Sigma_{g,n}^s$ and $T$, and the set of relations $R_0$ consists of the relations \labelcref{item-0_et_1-tresse,item-lanterne,item-rc_egale_Tcarre,item-T_est_central,item-T_puissance_l0} from \Cref{def-gammatildegns2} and the relation $r_{p_i} = 1$ for all $i \in [\![ 1, n]\!]$. 
	It is a central extension of $\Gamma_{g,n}^s$ with the maps: 
	\begin{equation}\label{eq-applications_extension_centrale_Gammagns2_un_seul_0}
		\begin{array}{l|ccl}
			\iota_0: &  \langle T \rangle & \longrightarrow & \widetilde{\Gamma}_{g,n}^s(2,0) \\
			& T & \longmapsto & T,
		\end{array}
		\qquad 
		\begin{array}{l|ccl}
			\pi_0: & \widetilde{\Gamma}_{g,n}^s(2,0) & \longrightarrow & \Gamma_{g,n}^s \\
			& \tau_\gamma & \longmapsto & \tau_\gamma \\
			& T & \longmapsto & 1_{\Gamma_{g,n}^s}.
		\end{array}
	\end{equation}
	Using the same arguments as in the proof of \Cref{prop-extension_ordre_2_en_genre_1}, the cohomology class $c_{\widetilde{\Gamma}_{g,n}^s(2,0)}$ of $\widetilde{\Gamma}_{g,n}^s(2,0)$ is equal to $2 c_{\widetilde{\Gamma}_{g,n}^s(0)} \in H^2(\Gamma_{g,n}^s, \mathbb{Z})$.
	
	\smallskip
	\noindent
	Let $E$ be the fiber product of the central extensions $\widetilde{\Gamma}_{g,n}^s(2,0)$ and $\widetilde{\Gamma}_{g,n}^s(e_i,0)$ for all $i \in [\![1, n]\!]$. 
    We obtain the following exact sequence: 
    \begin{equation}\label{eq-somme_directe_des_extensions}
	0 \longrightarrow \langle T \rangle \oplus \langle E_1 \rangle \oplus \cdots \oplus \langle E_n \rangle \xrightarrow{~~\iota~~} E \xrightarrow{~~\pi~~} \Gamma_{g,n}^s \longrightarrow 1,
    \end{equation}
	where $\iota := \iota_0 \oplus \cdots \oplus \iota_n$ with $\iota_k$ from \labelcref{eq-diagramme_commutatif_extension_ieme_classe_Euler,eq-applications_extension_centrale_Gammagns2_un_seul_0}, and $\pi$ is defined by $\pi(x_0, \cdots, x_n) := \pi_0(x_0)$ with $\pi_0$ from \labelcref{eq-applications_extension_centrale_Gammagns2_un_seul_0}.
	Let $c \in \bigoplus_{i=0}^n H^2(\Gamma_{g,n}^s, \mathbb{Z})$ denote the cohomology class of the central extension \labelcref{eq-somme_directe_des_extensions}.
    We have: 
    \begin{equation}\label{eq-cohomologie_somme_directe}
		c = c_{\widetilde{\Gamma}_{g,n}^s(2,0)} \oplus c_{\widetilde{\Gamma}_{g,n}^s(e_1,0)} \oplus \cdots \oplus c_{\widetilde{\Gamma}_{g,n}^s(e_n,0)} = 2 c_{\widetilde{\Gamma}_{g,n}^s(0)} \oplus [e_1] \oplus \cdots \oplus [e_n] \in \bigoplus_{i=0}^n H^2(\Gamma_{g,n}^s, \mathbb{Z}).
    \end{equation}
	The equality \labelcref{eq-cohomologie_somme_directe} follows from the fact that a set-theoretic section $s : \Gamma_{g,n}^s \to E$ of the central extension \labelcref{eq-somme_directe_des_extensions} can be defined by $s_0 \oplus \cdots \oplus s_n$, where $s_0 : \Gamma_{g,n}^s \to \widetilde{\Gamma}_{g,n}^s(2,0)$ and $s_i : \Gamma_{g,n}^s \to \widetilde{\Gamma}_{g,n}^s(e_i, 0)$ are arbitrary set-theoretic sections. 
	
	\smallskip
	\noindent
	Finally, to conclude we only need to show that the extension \labelcref{eq-somme_directe_des_extensions} is isomorphic to $\widetilde{\Gamma}_{g,n}^s(2,(0, \cdots, 0))$.
	This follows from the isomorphism:
    \begin{equation}\label{eq-isomorphisme_E_tildeGammagns20}
		\begin{array}{l|ccl}
            \varphi: & \widetilde{\Gamma}_{g,n}^s(2,l) & \longrightarrow & E \\
            & \tau_\gamma & \longmapsto & \tau_\gamma \\ 
            & T & \longmapsto & (T, 1, \cdots, 1) \\
			& E_i & \longmapsto & ( \underbrace{1, \cdots, 1}_{i}, E_i, 1, \cdots, 1).
        \end{array}
    \end{equation}

	For the case $l \neq (0, \cdots, 0)$, it suffices to note that the extensions $\widetilde{\Gamma}_{g,n}^s(e_i, l_i)$, for all $i \in [\![ 1, n ]\!] $, are isomorphic to the pushouts of $\widetilde{\Gamma}_{g,n}^s(e_i, 0)$ by the maps $\pi_{l_i} : \mathbb{Z} \to \mathbb{Z}_{l_i}$, via the isomorphisms:
	\[ 
    \begin{array}{l|ccl}
        & \langle E_i ~ \vert ~ E_i^{l_i} \rangle \times \widetilde{\Gamma}_{g,n}^s(e_i, 0) / \sim & \longrightarrow & \widetilde{\Gamma}_{g,n}^s(e_i, l_i) \\
		& (\bar{e},g) & \longmapsto & eg,
    \end{array}
    \]
    where $\bar{e}$ is the equivalence class in $\langle E_i ~ \vert ~ E_i^{l_i} \rangle$ of $e \in \langle E_i \rangle_{\mathrm{free}}$.
	Similarly, $\widetilde{\Gamma}_{g,n}^s(2,l_0)$ and $\widetilde{\Gamma}_{g,n}^s(l_0)$ are isomorphic to the pushouts of $\widetilde{\Gamma}_{g,n}^s(2,0)$ and $\widetilde{\Gamma}_{g,n}^s(0)$ induced by the map $\mathbb{Z} \to \mathbb{Z}_{l_0}$.
	Consequently, the equality $c = 2 c_{\widetilde{\Gamma}_{g,n}^s(l_0)} \oplus [e_1](l_1) \oplus \cdots \oplus [e_n](l_n)$ also holds, and the same isomorphism $\varphi$ (see \labelcref{eq-isomorphisme_E_tildeGammagns20}) allows us to conclude.
\end{proof}

When $n = 0$, recall that the group $H^2(\Gamma_{g}^s, \mathbb{Z})$ is a cyclic group (see \Cref{thm-H^2_grand_genre} and \Cref{prop-H^2_petit_genre}). 
The following result allows us to identify $c_{\widetilde{\Gamma}_{g}^s(2,l_0)}$ more precisely:

\begin{proposition}\label{prop-cohomologie_tildeGammags_genere_H^2}
	The cohomology class $c_{\widetilde{\Gamma}_{g}^s(0)}$ is a generator of $H^2(\Gamma_{g}^s, \mathbb{Z})$. 
\end{proposition}

\begin{proof}
	In the case $g \geq 4$, the group $\widetilde{\Gamma}_{g}^s(0)$ is exactly the universal central extension $\widetilde{\Gamma}_{g}^{s,\mathrm{univ}}$ of $\Gamma_g^s$ (see \Cref{thm-presentation_ECU}).
	Since $H_2(\Gamma_g^s, \mathbb{Z}) \simeq H^2(\Gamma_g^s, \mathbb{Z}) \simeq \mathbb{Z}$ (see \Cref{thm-H_2_grand_genre,thm-H^2_grand_genre}), the cohomology class associated with $\widetilde{\Gamma}_{g}^{s,\mathrm{univ}}$ is a generator of $H^2(\Gamma_g^s, \mathbb{Z})$ and therefore $c_{\widetilde{\Gamma}_g^s(0)}$ is as well.
	For the cases $g = 2$ and $3$, see \cite{moulai_anomalie_nodate}.
\end{proof}

\appendix

\section{Diagrammatic description of \texorpdfstring{$\mathcal{L}_{g, n}(H)$}{LgnH}}\label{ann-description_diagrammatique_Lgn}

In this appendix, we recall a correspondence between $\mathcal{L}_{g,n}(H)$ and an algebra of diagrams (see \Cref{lem-correspondance_Lgn_diagrammes}), introduced in \cite[Sec. 3]{faitg_holonomy_2024}. 
A difference is that here we do not use any matrix formalism.
First, we recall the evaluation \enquote{à la Hennings} introduced in \cite[§4.3]{faitg_holonomy_2024} (see \cite{hennings_invariants_1996} for the original method).

\subsection{Diagrams and the Hennings evaluation}\label{ann-diagrammes_et_Hennings}

We consider ribbon graph diagrams as in \Cref{fig-exemple_diagramme}, where:
\begin{itemize}
	\item Each handle labeled by $a_{j}$ (resp. $b_{j}$ or $m_{g+k}$) contains a number of cups oriented in any direction;
	\item $T$ is a ribbon graph diagram without coupons neither colors, as defined in \cite[Chap. I - §2.1]{turaev_quantum_2010}, whose orientations at the endpoints are compatible with those of the arcs coming from the handles.
\end{itemize}
\begin{figure}[ht!]
	\centering
	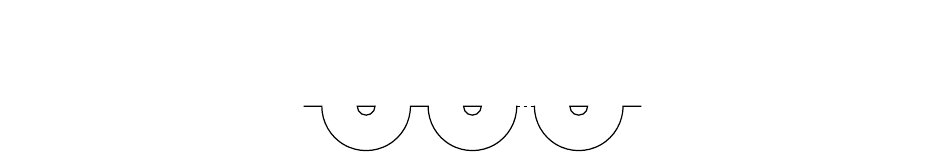
	\caption{Example of ribbon graph diagram}
	\label{fig-exemple_diagramme}
\end{figure}
\noindent
Throughout this \Cref{ann-diagrammes_et_Hennings}, $x_r$ is either the curves $a_j$, $b_j$, or $m_{g+k}$ (see \Cref{fig-courbes_standards_Lgn}).
Note that a diagram can pass through an arbitrarily large number of handles. 
Moreover, if a diagram has $t$ strands, it must have $2t$ arcs above the ribbon graph diagram $T$, one going down and the other one going up for each strand. 
These diagrams will allow us to represent elements of $\mathcal{L}_{g, n}(H)$ but should not be considered as projections of ribbon graphs in the thickened surface $\Sigma_{g,n}^1 \times [0,1]$. 

\begin{figure}[H]
	\centering
	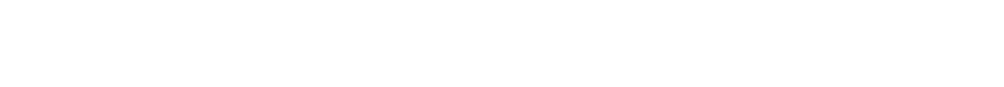
	\caption{Diagrams $D_{x_i}$, $\bar{D}_{x_i}$, and $\overset{?}{D}_{x_{i_1}} \otimes \cdots \otimes \overset{?}{D}_{x_{i_k}}$}
	\label{fig-diagrammes_Dxi_Dxi_bar_et_concatenation}
\end{figure}
\noindent
Denote by $\overset{?}{D}_{x_{i_1}} \otimes \cdots \otimes \overset{?}{D}_{x_{i_k}}$ the horizontal concatenation of the $k$ diagrams $\overset{?}{D}_{x_{i_1}}, \cdots, \overset{?}{D}_{x_{i_k}}$, where $\overset{?}{D}_{x_{i_j}}$ is either $D_{x_{i_j}}$ or $\bar{D}_{x_{i_j}}$ (see \Cref{fig-diagrammes_Dxi_Dxi_bar_et_concatenation}).

\medskip

We now describe a procedure \enquote{à la Hennings}, introduced in \cite{faitg_holonomy_2024}, which to each diagram $D$ with $t$ strands returns an element of $\mathcal{L}_{g, n}(H) \otimes H^{\otimes t}$.
This procedure can be summarized in two steps : label and multiply. 
Let us describe these steps.

\smallskip
\noindent
Let $D$ be a diagram as in \Cref{fig-exemple_diagramme} with $t$ strands. 
Let us recall that we use Sweedler's notation for the co-product $\Delta^{(r)}(x) = x_{(1)} \otimes \cdots \otimes x_{(r)}$ and the notation $R = \sum_{a}r_a \otimes r^a$ for the R-matrix.

\medskip
\textbf{STEP 1 - Label.}
As in \cite[Chap. I - §2.1]{turaev_quantum_2010}, the diagram $D$ can be isotoped so that it is only composed of the elementary (local) diagrams represented in \Cref{fig-diagrammes_locaux_elementaires}.

\begin{figure}[H]
	\centering
	\input{diagramme_local_elementaire.pdf_tex}
	\caption{Elementary local diagrams.}
	\label{fig-diagrammes_locaux_elementaires}
\end{figure}
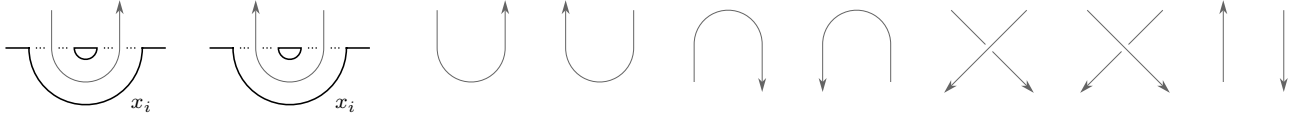
\noindent
In the first two diagrams in \Cref{fig-diagrammes_locaux_elementaires}, we have $r$ parallel cups in the handle $x_i$. 
Next, we label the arcs of each elementary diagram in $D$ according to the rules in \Cref{fig-regles_etiquetage_Hennings}.
 
\begin{figure}[H]
	\centering
	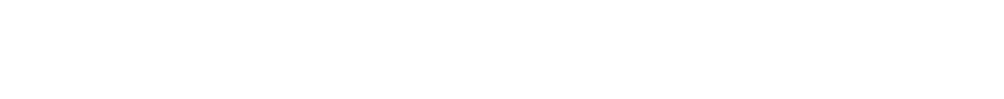
	\caption{Labeling rules \enquote{à la Hennings}.}
	\label{fig-regles_etiquetage_Hennings}
\end{figure}
\noindent
Let us describe \Cref{fig-regles_etiquetage_Hennings}. 
Some labels are components of R-matrices or iterated co-products $\Delta^{(r)}$. 
When these labels are present, we have implicit summations.
In the first two diagrams in \Cref{fig-regles_etiquetage_Hennings}, the cup shown is the $k$-th from the left among the $r$ parallel cups in the handle.
The label $x_i$ denotes either $a_i$, $b_i$, or $m_i$ (the usual generators of $\pi_1(\Sigma_{g,n}^1)$). 
The label $\mathcal{X}_1(i)_{(k)}$ denotes the $(k+1)$-th component of the element $(\id_{\mathcal{L}_{g, n}(H)} \otimes \Delta^{(r)})(\mathcal{X}(i))$, where $\mathcal{X}(i) = \mathcal{X}_0(i) \otimes \mathcal{X}_1(i) \in \mathcal{L}_{g, n}(H) \otimes H$ is the element defined by:
\begin{equation}\label{eq-propriete_element_universel}
	\forall \varphi \in H^*, ~ (\id_{\mathcal{L}_{g, n}(H)} \otimes \, \varphi) \bigl(\mathcal{X}(i) \bigr) = \mathfrak{i}_{x_i}(\varphi),
\end{equation}
where $\mathfrak{i}_{x_i}$ is any embedding defined in \labelcref{eq-definition_injections_canoniques}.
Since $H$ is finite-dimensional, we have $\mathcal{X}(i) = \sum_k \mathfrak{i}_{x_i}(h^k) \otimes h_k$, 
where $\{h_k\}_k$ is a basis of $H$ and $\{h^k\}_k$ is its dual basis.

\medskip
\textbf{STEP 2 - Multiply.}
Now that the arcs of the diagram $D$ are labeled with elements in $H$, we traverse each strand from its origin in the direction of its orientation and multiply from left to right the labels in the order in which we meet them. 
For each strand we obtain an element in $H$ labelling the strand.

\smallskip
\noindent
Let us look at an example: the first chain of diagrammatic equalities in \Cref{fig-regles_etiquetage_consequences}. 
The first two equalities consist of isotoping the diagram so that it is composed of the elementary diagrams in \Cref{fig-diagrammes_locaux_elementaires}. 
The third equality consists of labeling (see step 1), multiplying the labels and relabeling (see step 2), and isotoping the diagram to obtain the initial diagram. 
The last equality uses the identities \labelcref{eq-R-matrice_et_antipode,eq-propriete_g}.

\smallskip
\noindent
In the following and in \Cref{ann-action_ci}, we use the following notations: 
\begin{itemize}[itemsep=0.1em]
	\item The equalities $\overset{\text{top.}}{=}$ are only topological simplifications;
	\item The parts in red indicate the modifications at time \enquote{t};
	\item An integer next to a twist indicates the number of iterations of that twist. By default, this number is 1.
\end{itemize}
Thus, the labeling rules in \Cref{fig-regles_etiquetage_Hennings} imply the rules in \Cref{fig-regles_etiquetage_consequences}: 
\begin{figure}[H]
	\centering
	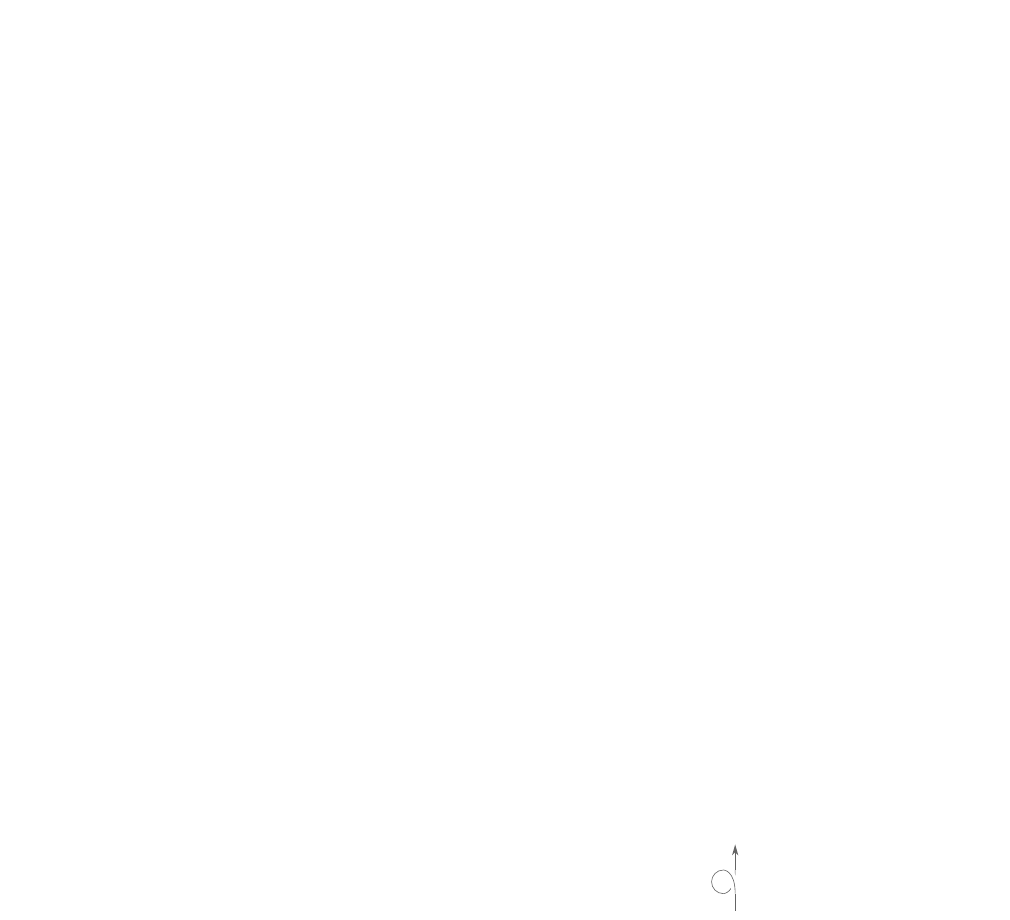
	\caption{Consequences of labeling rules \enquote{à la Hennings}}
	\label{fig-regles_etiquetage_consequences}
\end{figure}

\begin{definition}\label{def-evaluation_diagramme}
	Let $D$ be a diagram as in \Cref{fig-exemple_diagramme}, that passes through $s$ handles with $t$ strands. The \emph{evaluation} of the diagram $D$ is defined by :
	\[ \mathrm{Hen}(D) := \mathcal{X}_0(i_1) \cdots \mathcal{X}_0(i_s) \otimes h_1 \otimes \cdots \otimes h_t \in \mathcal{L}_{g, n}(H) \otimes H^{\otimes t}, \]
	where the element $h_k$ is associated with the $k$-th strand of $D$, as in step 2.
\end{definition}
 
We define a \emph{state} $s$ of $D$ by a map which to each strand of $D$ associates an element of $H^*$. 
In practice, we indicate the state at the start of each strand, as in \Cref{fig-exemple_diagramme_a_etats}; thus we will not always draw the end of the strands (in stated skein theory, we also have to label these ends with elements in $H$ (see \cite[Sec. 6]{baseilhac_noetherian_2025}); here it is $1_H$). 
We denote by $D^s$ the \emph{stated diagram} associated with $(D,s)$. 

\begin{figure}[ht!]
	\centering
	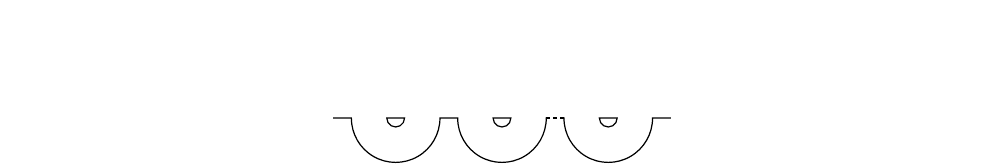
	\caption{A stated diagram}
	\label{fig-exemple_diagramme_a_etats}
\end{figure}
\noindent

\smallskip
\noindent
As with state-free diagrams, let $\overset{?}{D}_{x_{i_1}}^{\varphi_i}$ be the diagram $\overset{?}{D}_{x_i}$ with the state $\varphi_i$ and $\overset{?}{D}_{x_{i_1}}^{\varphi_1} \otimes \cdots \otimes \overset{?}{D}_{x_{i_k}}^{\varphi_k}$ the horizontal concatenation of the $k$ stated diagrams $\overset{?}{D}_{x_{i_1}}^{\varphi_1}, \cdots, \overset{?}{D}_{x_{i_k}}^{\varphi_k}$.

\begin{definition}\label{def-evaluation_a_etats}
	The \emph{stated evaluation} of a stated diagram $D^s$ is defined by: 
	\begin{equation*}
		\mathrm{ev^{st}} (D^s) := \bigl( \id_{\mathcal{L}_{g, n}(H)} \otimes \langle s, \bullet \rangle \bigr) \circ \mathrm{Hen}(D) \in \mathcal{L}_{g, n}(H).
	\end{equation*}
\end{definition}

\noindent
Let $\varphi \in \mathcal{L}_{0, 1}(H)$ and $h \in H$. 
Note that by definition of $\mathcal{X}(i)$ (see \labelcref{eq-propriete_element_universel}) we have $\mathrm{ev^{st}}(D_{x_i}^\varphi) = \mathfrak{i}_{x_i}(\varphi)$.
The element $\mathfrak{i}_{x_i}(\varphi)$ is then represented by the stated diagram in the following \Cref{fig-diagramme_a_etats_Dxi_phi}: 

\begin{figure}[H]
	\centering
	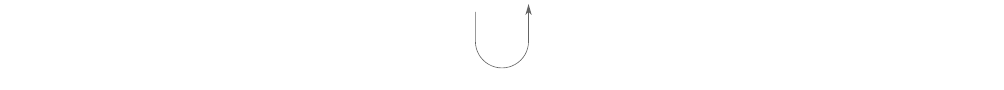
	\caption{Stated diagram $D_{x_i}^{\varphi}$}
	\label{fig-diagramme_a_etats_Dxi_phi}
\end{figure}
\noindent
For example, the image by $\mathrm{ev^{st}}$ of the diagram represented in \Cref{fig-diagramme_a_etats_co-produit} is $\mathfrak{i}_x(\varphi_{(1)}) \, \mathfrak{i}_y(\varphi_{(2)})$, where $\varphi_{(1)} \otimes \varphi_{(2)} = \Delta(\varphi)$. 

\begin{figure}[H]
	\centering
	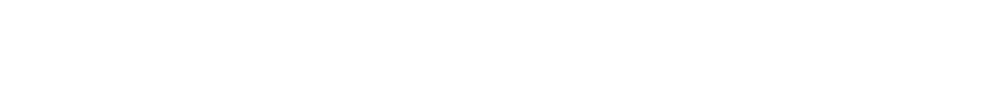
	\caption{Stated diagram of co-product}
	\label{fig-diagramme_a_etats_co-produit}
\end{figure}

Thus, we can define the following linear application:
\begin{equation}\label{eq-correspondance_partielle_Lgn_diagrammes}
	\begin{array}{l|ccl}
		& \mathcal{A}(H) & \longrightarrow & \mathcal{D} \\
		& \mathfrak{i}_{x_{i_1}}(\varphi_1) \cdots \mathfrak{i}_{x_{i_k}}(\varphi_k) & \longmapsto & D_{x_{i_1}}^{\varphi_1} \otimes \cdots \otimes D_{x_{i_k}}^{\varphi_k},
	\end{array}
\end{equation}
where $\mathcal{A}(H)$ is the free $\mathbb{K}$-algebra generated by the elements $\mathfrak{i}_{x_i}(\varphi)$ and $\mathcal{D}$ is the free $\mathbb{K}$-algebra generated by the set of stated ribbon graph diagrams without coupons neither colors, up to isotopy, whose product is given by the horizontal concatenation. 

\subsection{Diagrammatic calculus for \texorpdfstring{$\mathcal{L}_{g, n}(H)$}{LgnH} and the induced representation}

\subsubsection{\textbf{Fusion and exchange relations.}}\label{subsubsec-relations_fusion_et_echange}
The relations \labelcref{eq-injections_canoniques_morphismes_algebres,eq-commutation_ai_bi,eq-commutation_yj_xi} translate into the following diagrammatic equalities, where $x_i$ is $a_i$ or $b_i$, if $i \in [\![1, g]\!]$, and $m_i$, if $i \in [\![g+1, g+n]\!]$; the same applies to $y_j$, and $i < j$:
\begin{figure}[H]
	\centering
	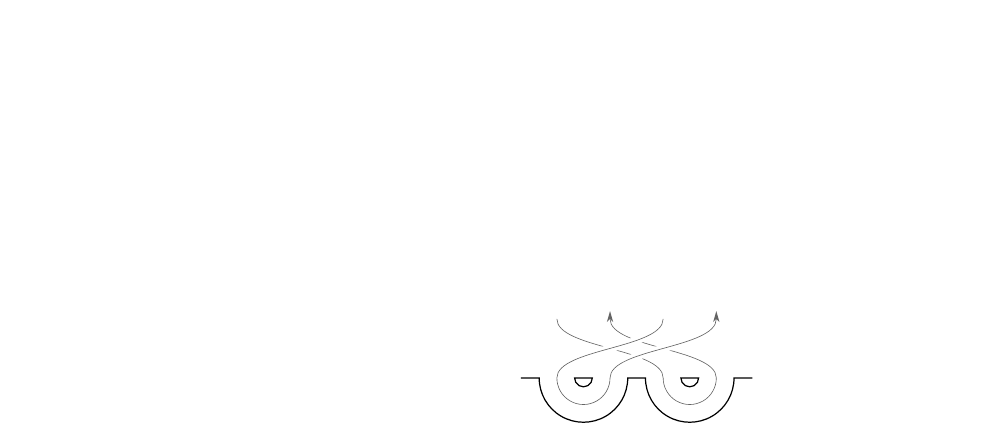
	\caption{Stated diagram of fusion and exchange relations}
	\label{fig-diagrammes_a_etats_relations}
\end{figure}
\noindent
For example, the first diagrammatic equality in \Cref{fig-diagrammes_a_etats_relations} decribes the product of $\mathcal{L}_{0, 1}(H)$. 
Indeed, the stated diagram on the left of the equality denotes the product $\mathfrak{i}_{x_i}(\varphi) \, \mathfrak{i}_{x_i}(\psi)$ by definition of the map \labelcref{eq-correspondance_partielle_Lgn_diagrammes}. 
Let us denote the stated diagram on the right by $D^s$.
Applying step 1 and step 2 to $D$ we compute (with implicit summations):
\begin{equation*}
	\mathrm{Hen}(D) = \mathcal{X}_0(i) \otimes \bigl( \mathcal{X}_1(i)_{(1)}r^aS(r^b) \bigr) \otimes \bigl( r_b \mathcal{X}_1(i)_{(2)}r_a \bigr).
\end{equation*}
Then we get (with implicit summations):
\begin{align*}
	\mathrm{ev^{st}}(D^s) & = \Bigl\langle \id_{\mathcal{L}_{g, n}(H)} \otimes \, \varphi \otimes \psi, \mathcal{X}_0(i) \otimes \bigl( \mathcal{X}_1(i)_{(1)}r^aS(r^b) \bigr) \otimes \bigl( r_b \mathcal{X}_1(i)_{(2)}r_a \bigr) \Bigr\rangle \\
	& = \mathfrak{i}_{x_i}(h^k) \, \bigl\langle \varphi, {h_k}_{(1)}r^aS(r^b) \bigr\rangle \bigl\langle \psi, r_b{h_k}_{(2)}r_a \bigr\rangle = \mathfrak{i}_{x_i}(h^k) \langle \varphi \psi, h_k \rangle = \mathfrak{i}_{x_i}(\varphi \psi).
\end{align*}

\smallskip

Denote by $\bar{\mathcal{D}}$ the quotient algebra of $\mathcal{D}$ by the relations in \Cref{fig-diagrammes_a_etats_relations}.
From the discussion above it follows:

\begin{lemme}\label{lem-correspondance_Lgn_diagrammes}
	The map \labelcref{eq-correspondance_partielle_Lgn_diagrammes} induces an algebra isomorphism $\mathcal{L}_{g, n}(H) \to \bar{\mathcal{D}}$, whose inverse is induced by the stated evaluation $\mathrm{ev^{st}}$ (see \Cref{def-evaluation_a_etats}).
\end{lemme}
\noindent
By this isomorphism and the non-degeneracy of the evaluation pairing $\langle -,-\rangle : H^* \otimes H \to \mathbb{K}$, all the computations in $\mathcal{L}_{g, n}(H)$ are reduced to computations with state-free diagrams, which is what we will do below. 

\smallskip
\noindent
In particular recall the definition of the embeddings $\mathfrak{i}_{x_i^{-1}}$ (see \Cref{subsec-applications_igamma_et_action_de_Gammagn1}): 
\[ \forall \varphi \in H^*, ~ \mathfrak{i}_{x_i^{-1}} (\varphi) = \mathfrak{i}_{x_i}  \circ S_{\mathcal{L}_{0, 1}(H)} (\varphi). \]
The element $\mathfrak{i}_{x_i^{-1}} (\varphi)$ is represented by the stated diagram in the following \Cref{fig-diagramme_a_etats_Dxi_inverse_phi}: 

\begin{figure}[H]
	\centering
	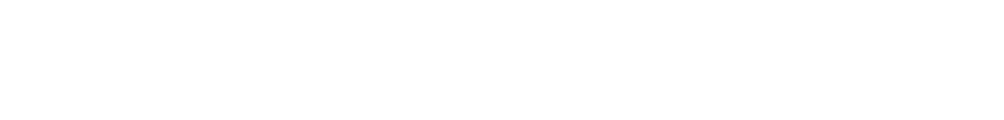
	\caption{Stated diagram of $\mathfrak{i}_{x_i^{-1}}(\varphi)$}
	\label{fig-diagramme_a_etats_Dxi_inverse_phi}
\end{figure}

\noindent
Indeed, let us denote this diagram by $D^{\varphi}$.
Applying step 1 and step 2 to $D$ we compute (with implicit summations):
\begin{equation*}
	\mathrm{Hen}(D) = \mathcal{X}_0(i) \otimes v^{-1} r^a g^{-1}S(\mathcal{X}_1(i))S(r_a).
\end{equation*}
Then we get (with implicit summations):
\begin{align*}
	\mathrm{ev^{st}}(D^s) & = \Bigl\langle \id_{\mathcal{L}_{g, n}(H)} \otimes \, \varphi, \mathcal{X}_0(i) \otimes v^{-1} r^a g^{-1}S \bigl( \mathcal{X}_1(i) \bigr) S(r_a) \Bigr\rangle \\
	& = \mathfrak{i}_{x_i}(h^k) \, \bigl\langle \varphi, r^a u^{-1} S(h_k)S(r_a) \bigr\rangle  = \mathfrak{i}_{x_i} \Bigl( S_{H^*} \bigl( S(r_a) \triangleright \varphi \triangleleft r^a u^{-1} \bigr) \Bigr) = \mathfrak{i}_{x_i} \bigl( S_{\mathcal{L}_{0, 1}(H)}(\varphi) \bigr) = \mathfrak{i}_{x_i^{-1}}(\varphi).
\end{align*}

More generally, let $\gamma \in \pi_1(\Sigma_{g,n}^1)$ be a positively oriented, simple curve represented by a word $x_1 \cdots x_k$ in the usual generators of $\pi_1(\Sigma_{g,n}^1)$ and their inverses.
We defined the map $\mathfrak{i}_\gamma$ in \Cref{def-injection_igamma}.
The element $\mathfrak{i}_\gamma(\varphi)$ will be represented more concisely by a stated diagram formed by a cup passing through a handle labeled by $\gamma$.
This diagram should be interpreted as a single strand with a certain number of twists depending on the normalization of $\gamma$ and passing through the horizontal concatenation of the handles labeled by $x_1, \cdots, x_k$. 
For an example, see \Cref{fig-diagramme_ci}.

\subsubsection{\textbf{Induced representation of $\mathcal{L}_{g, n}(H)$} (see \Cref{def-definition_representation_induite})\textbf{.}}\label{subsubsec-representation_induite}

Recall that $\mathcal{L}_{g, n}(H)$ acts by left multiplication on $\mathcal{L}_{g, n}(H) \otimes_\mathcal{A} \mathbb{K} \simeq \mathcal{L}_{0, g+n}(H) \simeq (H^*)^{\otimes (g+n)}$, where:
\[ \mathcal{A} := \mathrm{Vect}_\mathbb{K} \bigl\{ \mathfrak{i}_{a_1}(\alpha_1)\cdots \mathfrak{i}_{a_g}(\alpha_g) ~ \vert ~ \alpha_j \in H^* \bigr\}\]
and the representation of $\mathcal{A}$ on $\mathbb{K}$ is given by the formula:
\[ \forall \alpha_j \in H^*, ~ \mathfrak{i}_{a_j}(\alpha_j) \cdot 1_\mathbb{K} := \alpha_j(1_H) \, 1_\mathbb{K}. \]
Recall that the basis element of $\mathcal{L}_{g, n}(H)$ are represented diagrammatically by the horizontal concatenation of stated diagrams $\overset{?}{D}_{x_i}^\varphi$ (see \Cref{fig-diagrammes_Dxi_Dxi_bar_et_concatenation,fig-diagramme_a_etats_Dxi_phi} and \labelcref{eq-correspondance_partielle_Lgn_diagrammes}). 
Thus, the representation of $\mathcal{A}$ on $\mathbb{K}$ is represented diagrammatically as follows:
\begin{figure}[H]
	\centering
	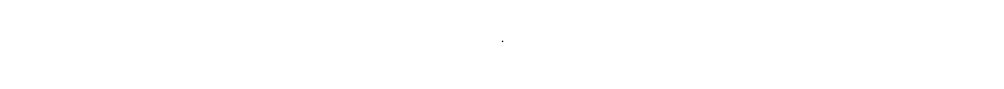
	\caption{Stated diagram of the representation of $\mathcal{A}$}
	\label{fig-diagramme_a_etats_representation_de_A}
\end{figure}
\noindent
For the diagrammatic representation of the induced representation it will be more convenient to consider it as a representation on $\mathcal{L}_{0, g+n}(H)$.
If we denote by $\Theta$ the isomorphism of vector spaces $\mathcal{L}_{g, n}(H) \otimes_\mathcal{A} \mathbb{K} \simeq \mathcal{L}_{0, g+n}(H)$ of \labelcref{eq-isomorphisme_representation_induite_L0g+n}, then by the very definition of the induced representation we have:
\[ \mathfrak{i}_{x_r}(\varphi) \cdot \bigl( \mathfrak{i}_{m_1}(\psi_1) \cdots \mathfrak{i}_{m_{g+n}}(\psi_{g+n}) \bigr) = \Theta \bigl( \mathfrak{i}_{x_r}(\varphi) \, \mathfrak{i}_{b_1}(\psi_1) \cdots \mathfrak{i}_{b_g}(\psi_g) \mathfrak{i}_{m_{g+1}}(\psi_{g+1}) \cdots \mathfrak{i}_{m_{g+n}}(\psi_{g+n}) \otimes 1_{\mathbb{K}} \bigr), \]
where $x_r$ is either $a_j$, $b_j$, or $m_{g+k}$.
The formulas in \Cref{ann-calcul_diagrammatique} use the following statement:

\begin{proposition}\label{prop-diagrammes_a_etats_representation_induite}
	The formulas of the induced representation in \labelcref{eq-formules_representation_induite} are represented by the following diagrams:
	\begin{figure}[H]
		\centering
		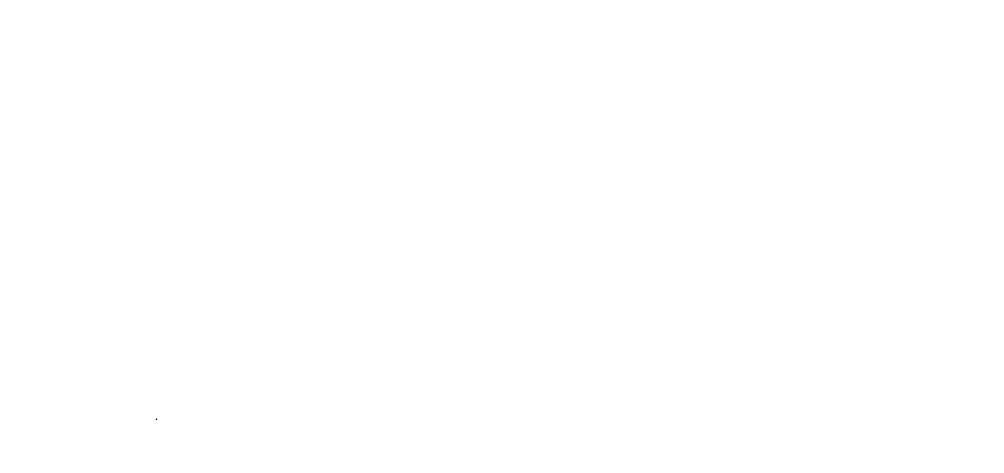
		\caption{Stated diagrams of the induced representation}
		\label{fig-diagrammes_a_etats_representation_induite}
	\end{figure}
\end{proposition}
\noindent
We give the details of the proof for the curves $a_i$ in the following figure, where the equalities $\overset{\text{top.}}{=}$ denote only topological simplifications:

\begin{figure}[H]
	\centering
	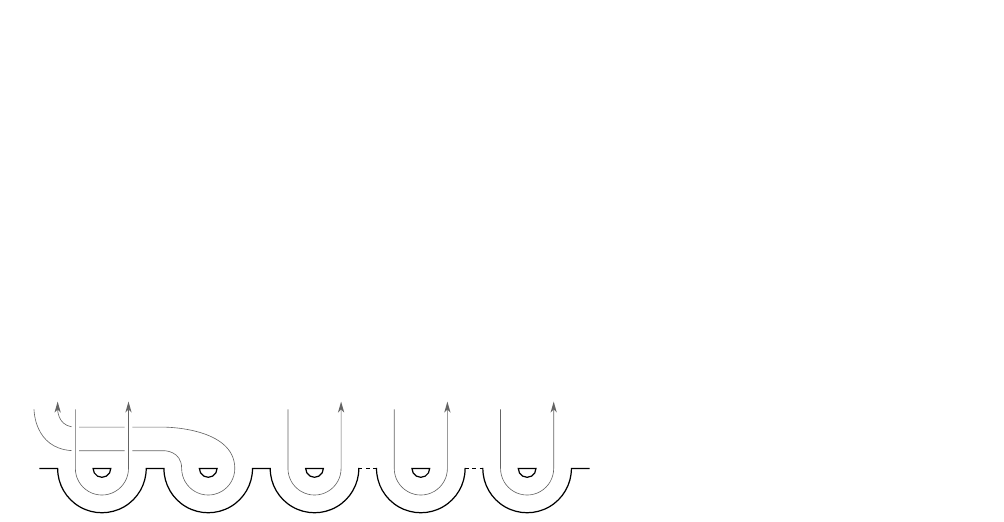
	\caption{Action of $\mathfrak{i}_{a_i}$ via the induced representation part 1}
\end{figure}

\begin{figure}[H]
	\centering
	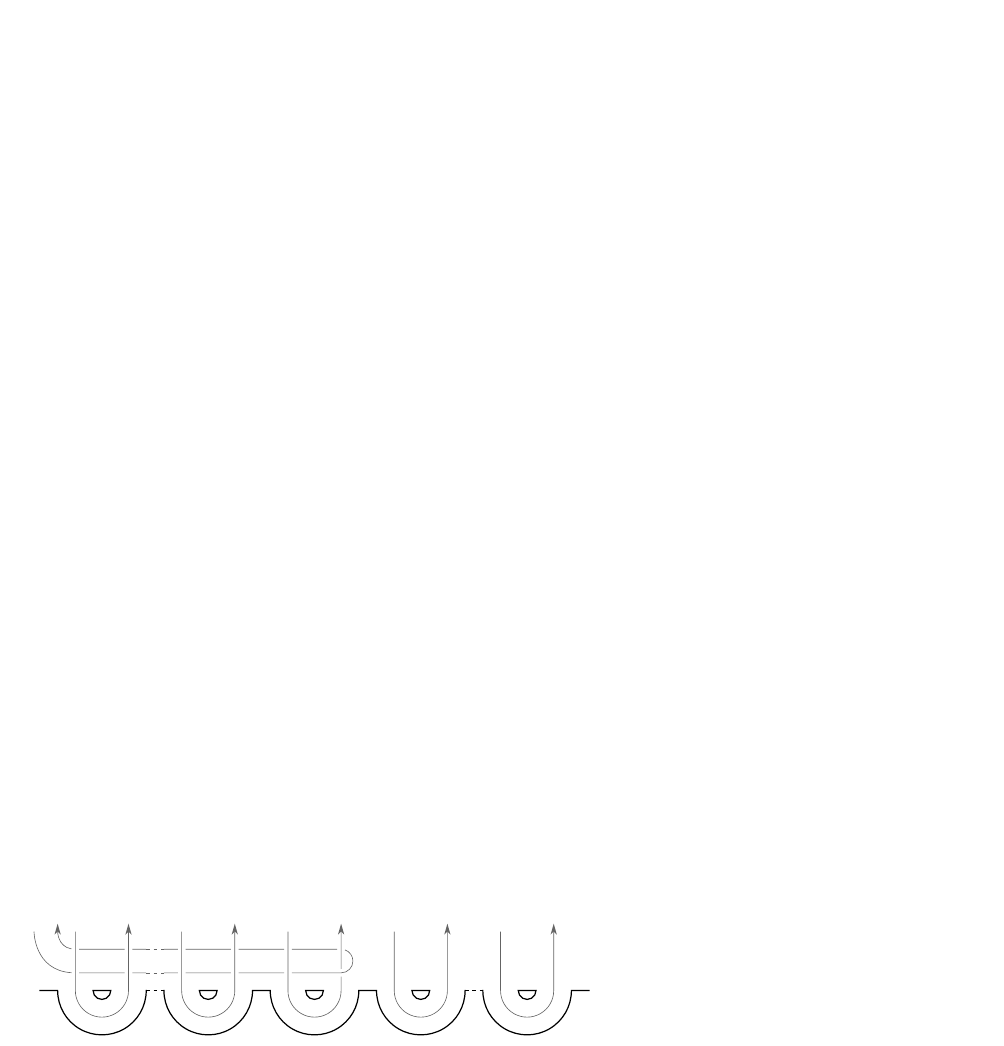
	\caption{Action of $\mathfrak{i}_{a_i}$ via the induced representation part 2}
\end{figure}

\newpage
\section{Diagrammatic computation of the action of some Dehn twists}\label{ann-calcul_diagrammatique}

In this appendix we compute the action of the Dehn twist $\tau_{c_i}$ and give formulas for other Dehn twists by using the diagrammatic correspondence in \Cref{lem-correspondance_Lgn_diagrammes} and \Cref{prop-diagrammes_a_etats_representation_induite}. 

\subsection{Action of \texorpdfstring{$c_i$}{ci}}\label{ann-action_ci}

The following result proves the third formula in \Cref{prop-formules_ai_bi_ci}, and uses \Cref{prop-diagrammes_a_etats_representation_induite}.

\begin{proposition}\label{prop-formule_ci}
	The action of $\hat{\tau}_{c_i}$ on $\psi_1 \otimes \cdots \otimes \psi_{g+n}$ is given by: 
	\[ \hat{\tau}_{c_i} \cdot (\psi_1 \otimes \cdots \otimes \psi_{g+n}) = \psi_1 \otimes \cdots \otimes \psi_{i-1} \otimes \bigl( S(v^{-1}_{(1)}) \triangleright \psi_i \bigr) \otimes (\psi_{i+1} \triangleleft v^{-1}_{(2)}) \otimes \psi_{i+2} \otimes \cdots \otimes \psi_{g+n}. \]
	\begin{center}
		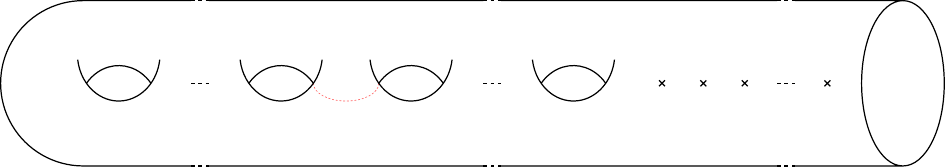
	\end{center}
\end{proposition}

\begin{proof}
	The curve $c_i$ is represented by the word $a_i b_{i+1} a_{i+1}^{-1} b_{i+1}^{-1}$ (in the usual generators of $\pi_1(\Sigma_{g,n}^1)$ and their inverses). 
	So, we can draw the curve $c_i$ on the ribbon surface as follow: 
	\begin{center}
	
		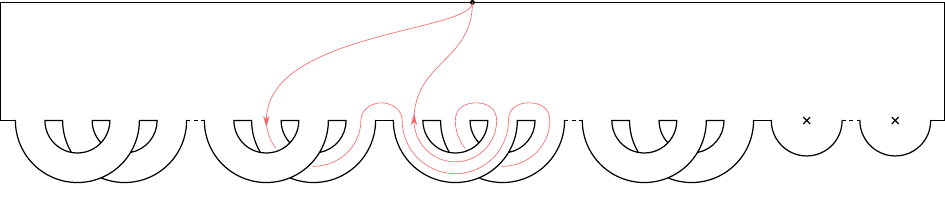
	\end{center}
	\noindent
	Let us compute the normalization of the curve $c_i$: $N(c_i) = N_\cup(c_i) - N_\cap(c_i) = 2$. 
	We get the following diagram:

	\begin{figure}[H]
		\centering
		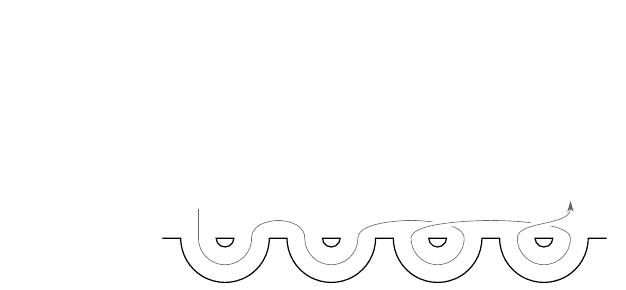
		\caption{Diagram of $\mathfrak{i}_{c_i}$}
		\label{fig-diagramme_ci}
	\end{figure}

	\noindent
	The first equality follows from \Cref{def-injection_igamma}, and the second follows from \Cref{fig-diagramme_a_etats_Dxi_inverse_phi}. 

	\smallskip
	The actions of $\mathfrak{i}_{a_j}(\alpha_j)$ and $\mathfrak{i}_{b_j}(\beta_j)$, $j \in [\![ 1, i+1 ]\!]$ (see \Cref{prop-diagrammes_a_etats_representation_induite}), allow us to reduce the evaluation of the action of $\mathfrak{i}_{c_i}$ to only $\psi_1 \otimes \cdots \otimes \psi_{i+1} \in (H^*)^{\otimes i+1} \subset (H^*)^{\otimes (g+n)}$.
	We have:

	\begin{figure}[H]
		\centering
		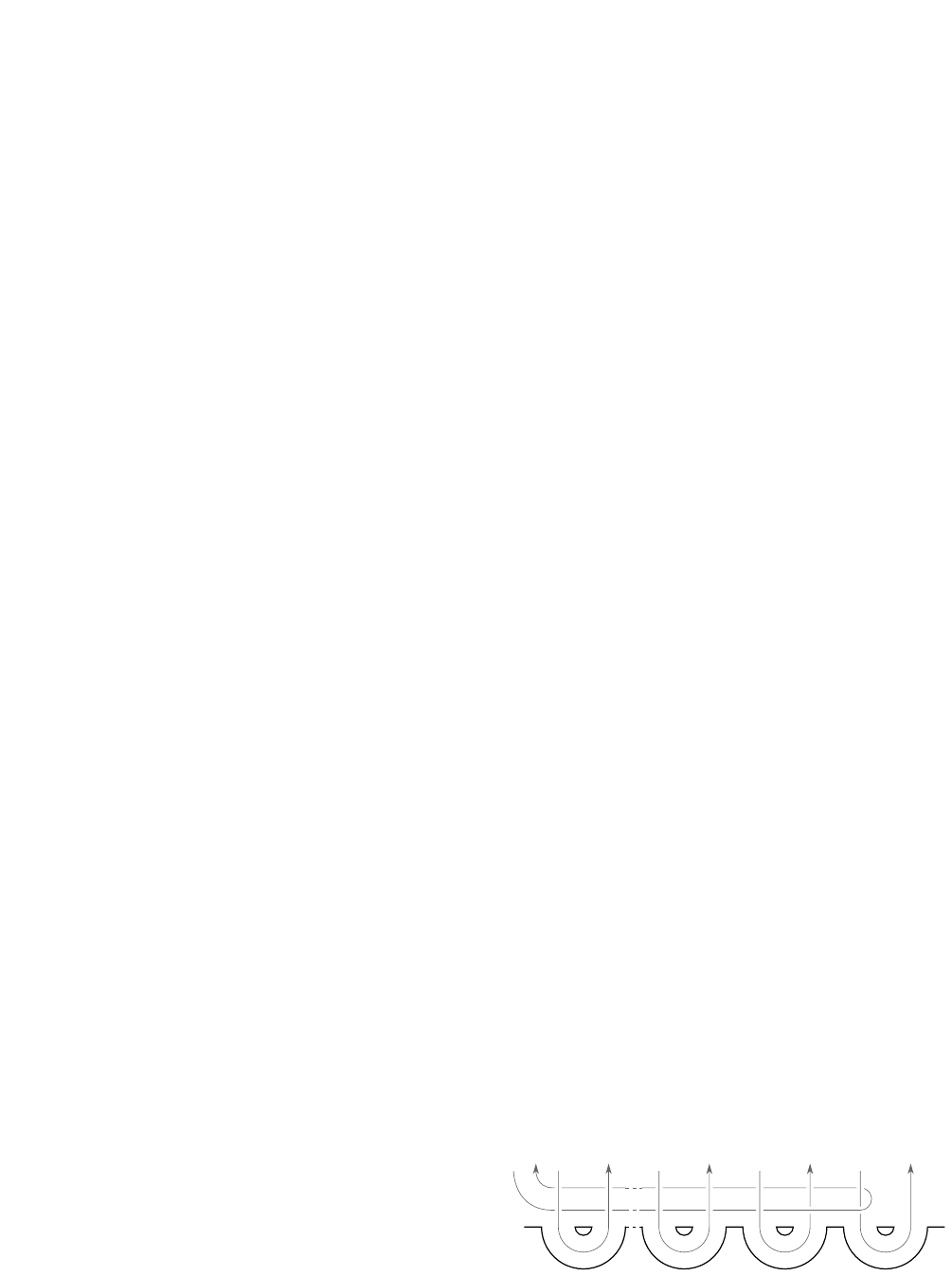
		\caption{Action of $\mathfrak{i}_{c_i}$}
		\label{fig-action_ci}
	\end{figure}

	\newpage

	\begin{figure}[ht!]
		\centering
		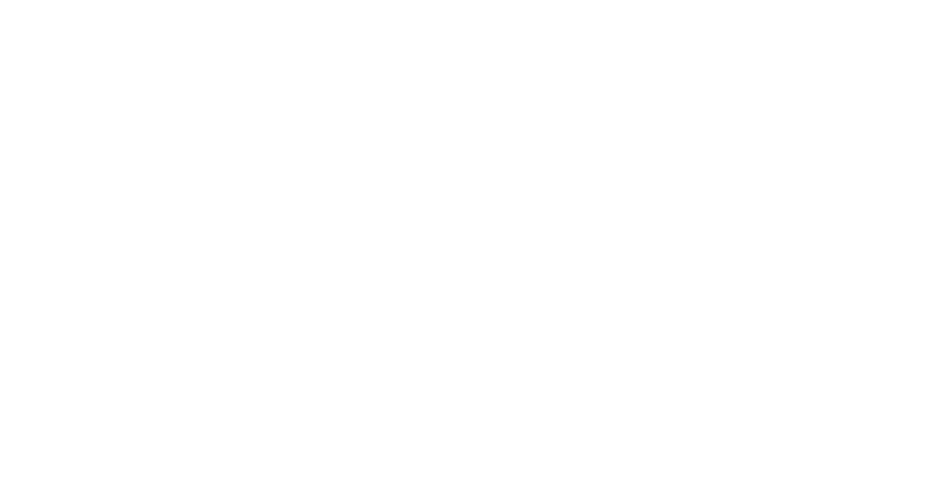
		\caption{Evaluation of $\mathfrak{i}_{c_i}$}
		\label{fig-evaluation_ci}
	\end{figure}
	\noindent
	Let us explain the third diagram in \Cref{fig-evaluation_ci}. 
	The labels indicate the indices of copies of R-matrices. 
	The labeling of this diagram is:
	\begin{equation}\label{eq-etiquetage_ci}
		\begin{split}
			\mathcal{M}_0(1) \cdots \mathcal{M}_0(i+1) & \otimes \bigl( r^{a_1} r^{b_1} \cdots r^{a_{i-1}} r^{b_{i-1}} r^{a_i} r_c S(r_d) S(r^{z_{i+1}}) r^{y_i} S(r^{z_i}) \cdots r^{y_1} S(r^{z_1}) \bigr) \\
			& \otimes \bigl( r_{z_1} r_{a_1} \mathcal{M}_1(1) S(r_{b_1}) r_{y_1} \bigr) \otimes \cdots \otimes \bigl( r_{z_{i-1}} r_{a_{i-1}} \mathcal{M}_1(i-1) S(r_{b_{i-1}}) r_{y_{i-1}} \bigr)\\
			& \otimes \bigl( r_{z_i} r_{a_i} \mathcal{M}_1(i) r^c r_{y_i} \bigr) \otimes \bigl( r_{z_{i+1}} r^d \mathcal{M}_1(i+1) \bigr).
		\end{split}
	\end{equation}
	Recall that for any Dehn twist $\tau_\gamma$, we have $\hat{\tau}_\gamma = \mathfrak{i}_\gamma(\lambda^v)$ (see \Cref{prop-action_twists_de_Dehn_non_separants_comme_conjugaison_et_proprietes}).
	From \Cref{def-evaluation_a_etats}, we need to evaluate the tensor $\id_{\mathcal{L}_{0, i+1}(H)} \otimes \, \lambda^v \otimes \psi_1 \otimes \cdots \otimes \psi_{i+1}$, where $\psi_i \in H^*$, against \labelcref{eq-etiquetage_ci}.
	We get: 
	\begin{equation}\label{eq-action_ci_non_simplifiee}
		\begin{split}
			\lambda^v \bigl( r^{a_1} r^{b_1} \cdots & r^{a_{i-1}} r^{b_{i-1}} r^{a_i} r_c S(r_d) S(r^{z_{i+1}}) r^{y_i} S(r^{z_i}) \cdots r^{y_1} S(r^{z_1}) \bigr) \, \bigl( S(r_{b_1}) r_{y_1} \triangleright \psi_1 \triangleleft r_{z_1} r_{a_1} \bigr) \\
			& \otimes \cdots \otimes \bigl( S(r_{b_{i-1}}) r_{y_{i-1}} \triangleright \psi_{i-1} \triangleleft r_{z_{i-1}} r_{a_{i-1}} \bigr) \otimes \bigl( r^c r_{y_i} \triangleright \psi_i \triangleleft r_{z_i} r_{a_i} \bigr) \otimes \bigl( \psi_{i+1} \triangleleft r_{z_{i+1}} r^d \bigr).
		\end{split}
	\end{equation}
	Now let us simplify the expression \labelcref{eq-action_ci_non_simplifiee}. The R-matrices indexed by $a_1$ and $z_1$ give:
	\begin{equation*}
		\lambda^v \bigl( r^{a_1} \cdots S(r^{z_1}) \bigr) (\cdots \triangleright \psi_1 \triangleleft r_{z_1} r_{a_1}) = \lambda^v \bigl( S^{-1}(r^{z_1}) r^{a_1} \cdots \bigr) (\cdots \triangleright \psi_1 \triangleleft r_{z_1} r_{a_1}) = \lambda^v(\cdots)(\cdots \triangleright \psi_1).
	\end{equation*}  
	The R-matrices indexed by $b_1$ and $y_1$ give:
	\begin{align*}
		\lambda^v(r^{b_1} \cdots r^{y_1}) \bigl( S(r_{b_1}) r_{y_1} \triangleright \psi_1 \bigr) & = \lambda^v \bigl( r^{b_1} \cdots S(r^{y_1}) \bigr) \bigl( S(r_{b_1}) S(r_{y_1}) \triangleright \psi_1 \bigr) \\
		& = \lambda^v \bigl( r^{b_1} \cdots S(r^{y_1}) \bigr) \bigl( S(r_{y_1} r_{b_1}) \triangleright \psi_1 \bigr) \\	
		& = \lambda^v \bigl( S^{-1}(r^{y_1}) r^{b_1} \cdots \bigr) \bigl( S(r_{y_1} r_{b_1}) \triangleright \psi_1 \bigr) = \lambda^v \bigl( \cdots \bigr) \psi_1.
	\end{align*}
	By doing exactly the same thing for the R-matrices indexed by $a_j$, $z_j$ and $b_k$, $y_k$, for all $j \in [\![ 2, i ]\!]$ and $k \in [\![ 2, i-1 ]\!]$, we obtain: 
	\begin{equation}\label{eq-action_ci_simplifiee_partiellement}
		\lambda^v \bigl( r_c S(r_d) S(r^{z_{i+1}}) r^{y_i}  \bigr) \, \psi_1 \otimes \cdots \otimes \psi_{i-1} \otimes \bigl( r^c r_{y_i} \triangleright \psi_i \bigr) \otimes \bigl( \psi_{i+1} \triangleleft r_{z_{i+1}} r^d \bigr).
	\end{equation}
	We can further simplify the expression \labelcref{eq-action_ci_simplifiee_partiellement} by using the following identities: 
	\begin{align*}
		\sum_{c,d,y_i,z_{i+1}} r_c S(r_d) S(r^{z_{i+1}}) r^{y_i} \otimes r^c r_{y_i} \otimes r_{z_{i+1}} r^d & = \sum_{c,d,y_i,z_{i+1}} r_c r_d r^{z_{i+1}} r^{y_i} \otimes r^c r_{y_i} \otimes S^{-1}(r^d r_{z_{i+1}}) \\ 
		& = \sum_{b,y} r_b r^y \otimes r^b_{(2)} {r_y}_{(2)} \otimes S^{-1}(r^b_{(1)} {r_y}_{(1)}).
	\end{align*}
	The first equality uses the identity $\labelcref{eq-R-matrice_et_antipode}$, and the second equality uses the identity $(\id_H \otimes \, \Delta)(R) = R_{13} R_{12}$ on the R-matrices $c$ and $d$, and $(\Delta \otimes \id_H)(R) = R_{13} R_{23}$ on the R-matrices $y_i$ and $z_{i+1}$. 
	We finally arrive at:
	\begin{align*}
		 \hat{\tau}_{c_i} \cdot (\psi_1 \otimes \cdots \otimes \psi_i \otimes \psi_{i+1}) & = \lambda^v(r_b r^y) \, \psi_1 \otimes \cdots \otimes \psi_{i-1} \otimes (r^b_{(2)} {r_y}_{(2)} \triangleright \psi_i) \otimes \bigl( \psi_{i+1} \triangleleft S^{-1}(r^b_{(1)} {r_y}_{(1)}) \bigr) \\ 
		& \overset{\text{\labelcref{eq-phi01}}}{=} \psi_1 \otimes \cdots \otimes \psi_{i-1} \otimes \bigl( \Phi_{0,1} (\lambda^v)_{(2)} \triangleright \psi_i \bigr) \otimes \bigl( \psi_{i+1} \triangleleft S^{-1}(\Phi_{0,1}(\lambda^v)_{(1)}) \bigr) \\ 
		& \overset{\text{\labelcref{eq-Phi01lambda^v}}}{=} \frac{\lambda(v)}{\lambda(v)} \, \psi_1 \otimes \cdots \otimes \psi_{i-1} \otimes (v^{-1}_{(2)} \triangleright \psi_i) \otimes \bigl( \psi_{i+1} \triangleleft S^{-1} (v^{-1}_{(1)}) \bigr) \\
		& = \psi_1 \otimes \cdots \otimes \psi_{i-1} \otimes \bigl( S(v^{-1}_{(1)}) \triangleright \psi_i \bigr) \otimes (\psi_{i+1} \triangleleft v^{-1}_{(2)}),
	\end{align*}
	where the fourth equality uses $(S \otimes S) \circ \Delta^{op} = \Delta \circ S$ and $S(v)=v$. Hence the result.
\end{proof}

\subsection{Actions of \texorpdfstring{$d_i, y, z$}{di,y,z}, and \texorpdfstring{$e$}{e}}\label{ann-action_di_y_z_et_e}

In this subsection, we give the formulas for the actions via the induced representation of the Dehn twists along the curves $d_i$, $y$, $z$, and $e$ represented in the figures below.
The proofs are similar to the one in \Cref{ann-action_ci} (see \cite{moulai_anomalie_nodate} for details). 
To keep on the safe side, we give some details for the curve $z$ in \Cref{prop-formule_z}.

\begin{proposition}\label{prop-formule_di}
	The action of $i_{d_i}(\varphi)$ on $\varepsilon^{\otimes (g+n)}$ is given by: 
	\[ \mathfrak{i}_{d_i}(\varphi) \cdot (\varepsilon^{\otimes (g+n)}) = \varphi(1_H) \, \varepsilon^{\otimes (g+n)}. \]
	\begin{center}
			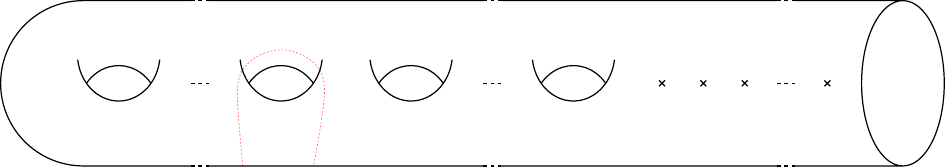
	\end{center}
\end{proposition}

\smallskip

\begin{proposition}\label{prop-formule_y}
	The action of $\hat{\tau}_y$ on $\psi_1 \otimes \cdots \otimes \psi_{g+n}$ is given by: 
	\begin{align*}
		\hat{\tau}_y \cdot (\psi_1 \otimes \cdots \otimes \psi_{g+n}) = \bigl( r^f S^2(r_e) \triangleright \psi_1 \bigr) \otimes \Bigl( \lambda^v \bigl( S^{-1}(\bullet^2_{(2)}) r^e r_f & \bullet^2_{(1)} S(r_g) S(r^h) \bigr) \star \psi_2 \Bigr) \\
		& \otimes (\psi_3 \triangleleft r_h r^g) \otimes \psi_4 \otimes \cdots \otimes \psi_{g+n}.
	\end{align*}
	\begin{center}
		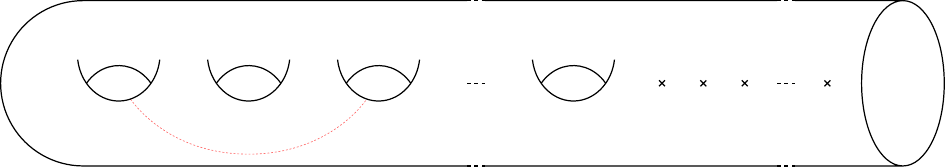
	\end{center}
\end{proposition}

\smallskip

\begin{proposition}\label{prop-formule_z}
	The action of $\hat{\tau}_z$ on $\psi_1 \otimes \cdots \otimes \psi_{g+n}$ is given by: 
	\begin{equation*}
		\hat{\tau}_z \cdot (\psi_1 \otimes \cdots \otimes \psi_{g+n}) = \bigl\langle \lambda^v, r^c S^2(r_d) S(r_e) S(r^f) \bigr\rangle \, \psi_1 \otimes ( r_f r^e \triangleright \psi_2 ) \otimes (\psi_3 \triangleleft r^d r_c) \otimes \psi_4 \otimes \cdots \otimes \psi_{g+n}.
	\end{equation*}
	\begin{center}
		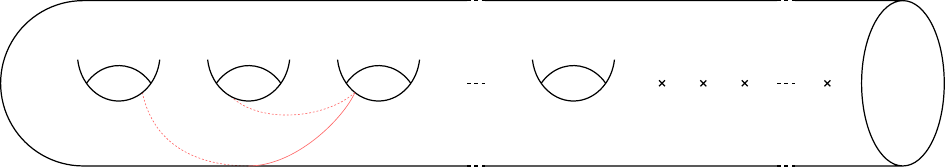
	\end{center}
\end{proposition}
\noindent
Let us give some details. 
The curve $z$ is represented by the word $a_1b_2b_3a_3b_3^{-1}a_2^{-1}b_2^{-1}$ (in the usual generators of $\pi_1(\Sigma_{g,n}^1)$ and their inverses), with normalization $N(z) = 3 - 3 = 0$.
After computing diagrammaticaly the action of $\mathfrak{i}_z$ as in \Cref{fig-action_ci} we get:
\begin{figure}[H]
	\centering
	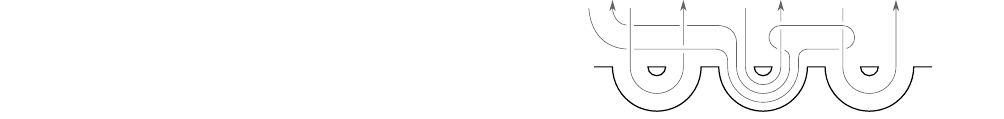
	\caption{Action of $\mathfrak{i}_{z}$}
	\label{fig-action_z}
\end{figure}
\begin{figure}[H]
	\centering
	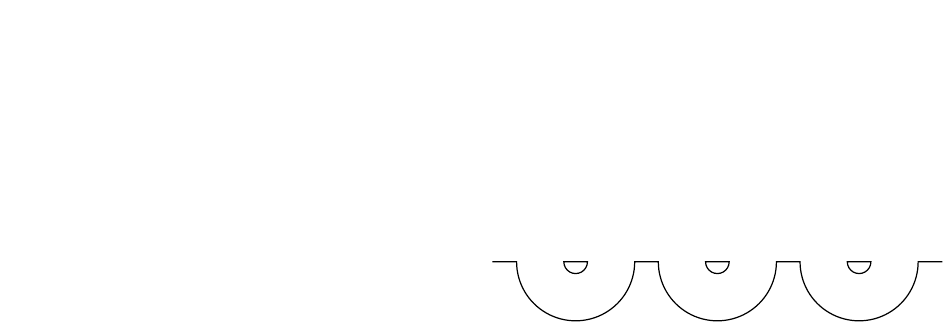
	\caption{Evaluation of $\mathfrak{i}_{z}$}
	\label{fig-evaluation_z}
\end{figure}
\noindent
Labeling the diagram on the bottom side in \Cref{fig-evaluation_z} as for \Cref{fig-evaluation_ci}, we obtain:
\begin{equation}\label{eq-etiquetage_z}
	\begin{split}
		\mathcal{M}_0(1) \mathcal{M}_0(2) \mathcal{M}_0(3) & \otimes \bigl( r^a r^b {\mathcal{M}_1(2)}_{(1)} r^c S^2(r_d) g r_e r^f g^{-1} S({\mathcal{M}_1(2)}_{(2)}) r^g S(r^h) \bigr)\\
		& \otimes \bigl( r_h r_a \mathcal{M}_1(1) S(r_b) r_g \bigr) \otimes \bigl( {\mathcal{M}_1(2)}_{(3)} S(r_f) S(r^e) \bigr) \otimes \bigl( r^d r_c \mathcal{M}_1(3) \bigr).
	\end{split}  
\end{equation}
From \Cref{def-evaluation_a_etats}, we need to evaluate the tensor $\id_{\mathcal{L}_{0, 3}(H)} \otimes \lambda^v \otimes \psi_1 \otimes \psi_2 \otimes \psi_3$, where $\psi_i \in H^*$, against \labelcref{eq-etiquetage_z}.
We get:
\begin{equation}\label{eq-action_z_non_simplifiee}
	\begin{aligned}
		\bigl( S(r_b) r_g \triangleright \psi_1 \triangleleft r_h r_a \bigr) \otimes \lambda^v \bigl( r^a r^b \bullet^2_{(1)} r^c S^2(r_d) g r_e r^f g^{-1} S( & \bullet^2_{(2)}) r^g S(r^h) \bigr) \\
		& \star \bigl( S(r_f) S(r^e) \triangleright \psi_2 \bigr) \otimes (\psi_3 \triangleleft r^d r_c).
	\end{aligned} 
\end{equation}
Finally, the result follows from simplifications of \labelcref{eq-action_z_non_simplifiee} by using \labelcref{eq-R-matrice_et_antipode,eq-propriete_g,eq-reformulation_invariance_coad^r} as in the proof of \Cref{prop-formule_ci}.

\begin{proposition}\label{prop-formule_e}
	The action of $\hat{\tau}_e$ on $\psi_1 \otimes \cdots \otimes \psi_{g+n}$ is given by: 
	\begin{equation*}
		\hat{\tau}_e \cdot (\psi_1 \otimes \cdots \otimes \psi_{g+n}) = \bigl\langle \lambda^v, S(r_a) r_b S(r_c) S(r^d) r^e S(r^f) \bigr\rangle \, (r^b r_e \triangleright \psi_1 \triangleleft r_f r^a) \otimes (\psi_2 \triangleleft r_d r^c) \otimes \psi_3 \otimes \cdots \otimes \psi_{g+n}.
	\end{equation*}
	\begin{center}
		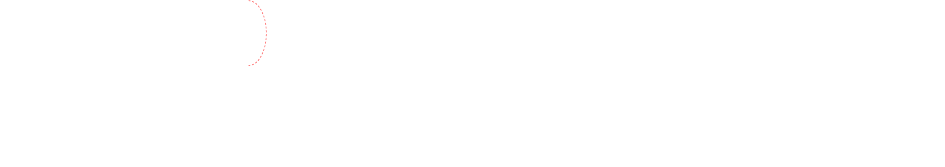
	\end{center}
\end{proposition}

\printbibliography
\end{document}

%% file: Courbes_relation_lanterne_surface_minimale.pdf_tex
\begingroup%
  \makeatletter%
  \providecommand\color[2][]{%
    \errmessage{(Inkscape) Color is used for the text in Inkscape, but the package 'color.sty' is not loaded}%
    \renewcommand\color[2][]{}%
  }%
  \providecommand\transparent[1]{%
    \errmessage{(Inkscape) Transparency is used (non-zero) for the text in Inkscape, but the package 'transparent.sty' is not loaded}%
    \renewcommand\transparent[1]{}%
  }%
  \providecommand\rotatebox[2]{#2}%
  \newcommand*\fsize{\dimexpr\f@size pt\relax}%
  \newcommand*\lineheight[1]{\fontsize{\fsize}{#1\fsize}\selectfont}%
  \ifx\svgwidth\undefined%
    \setlength{\unitlength}{453.54330709bp}%
    \ifx\svgscale\undefined%
      \relax%
    \else%
      \setlength{\unitlength}{\unitlength * \real{\svgscale}}%
    \fi%
  \else%
    \setlength{\unitlength}{\svgwidth}%
  \fi%
  \global\let\svgwidth\undefined%
  \global\let\svgscale\undefined%
  \makeatother%
  \begin{picture}(1,0.25129032)%
    \lineheight{1}%
    \setlength\tabcolsep{0pt}%
    \put(0,0){\includegraphics[width=\unitlength,page=1]{Courbes_relation_lanterne_surface_minimale.pdf}}%
    \put(0.62473607,0.12427027){\color[rgb]{0,0,0}\makebox(0,0)[lt]{\lineheight{0}\smash{\begin{tabular}[t]{l}$\scriptstyle b_1$\end{tabular}}}}%
    \put(0.53125006,0.22427029){\color[rgb]{0,0,0}\makebox(0,0)[lt]{\lineheight{0}\smash{\begin{tabular}[t]{l}$\scriptstyle b_2$\end{tabular}}}}%
    \put(0.34375006,0.12427026){\color[rgb]{0,0,0}\makebox(0,0)[lt]{\lineheight{0}\smash{\begin{tabular}[t]{l}$\scriptstyle b_3$\end{tabular}}}}%
    \put(0.53125004,0.02427027){\color[rgb]{0,0,0}\makebox(0,0)[lt]{\lineheight{0}\smash{\begin{tabular}[t]{l}$\scriptstyle b_4$\end{tabular}}}}%
    \put(0.42500005,0.14927027){\color[rgb]{0,0,0}\makebox(0,0)[lt]{\lineheight{0}\smash{\begin{tabular}[t]{l}$\scriptstyle x$\end{tabular}}}}%
    \put(0.54375006,0.14927027){\color[rgb]{0,0,0}\makebox(0,0)[lt]{\lineheight{0}\smash{\begin{tabular}[t]{l}$\scriptstyle y$\end{tabular}}}}%
    \put(0.50000007,0.06177026){\color[rgb]{0,0,0}\makebox(0,0)[lt]{\lineheight{0}\smash{\begin{tabular}[t]{l}$\scriptstyle z$\end{tabular}}}}%
    \put(0,0){\includegraphics[width=\unitlength,page=2]{Courbes_relation_lanterne_surface_minimale.pdf}}%
  \end{picture}%
\endgroup%

%% file: Courbes_relation_piqure_surface_minimale.pdf_tex
\begingroup%
  \makeatletter%
  \providecommand\color[2][]{%
    \errmessage{(Inkscape) Color is used for the text in Inkscape, but the package 'color.sty' is not loaded}%
    \renewcommand\color[2][]{}%
  }%
  \providecommand\transparent[1]{%
    \errmessage{(Inkscape) Transparency is used (non-zero) for the text in Inkscape, but the package 'transparent.sty' is not loaded}%
    \renewcommand\transparent[1]{}%
  }%
  \providecommand\rotatebox[2]{#2}%
  \newcommand*\fsize{\dimexpr\f@size pt\relax}%
  \newcommand*\lineheight[1]{\fontsize{\fsize}{#1\fsize}\selectfont}%
  \ifx\svgwidth\undefined%
    \setlength{\unitlength}{453.54330709bp}%
    \ifx\svgscale\undefined%
      \relax%
    \else%
      \setlength{\unitlength}{\unitlength * \real{\svgscale}}%
    \fi%
  \else%
    \setlength{\unitlength}{\svgwidth}%
  \fi%
  \global\let\svgwidth\undefined%
  \global\let\svgscale\undefined%
  \makeatother%
  \begin{picture}(1,0.25129032)%
    \lineheight{1}%
    \setlength\tabcolsep{0pt}%
    \put(0,0){\includegraphics[width=\unitlength,page=1]{Courbes_relation_piqure_surface_minimale.pdf}}%
    \put(0.34375006,0.12427026){\color[rgb]{0,0,0}\makebox(0,0)[lt]{\lineheight{0}\smash{\begin{tabular}[t]{l}$\scriptstyle b_3$\end{tabular}}}}%
    \put(0.53125006,0.22427029){\color[rgb]{0,0,0}\makebox(0,0)[lt]{\lineheight{0}\smash{\begin{tabular}[t]{l}$\scriptstyle b_2$\end{tabular}}}}%
    \put(0.62473607,0.12427027){\color[rgb]{0,0,0}\makebox(0,0)[lt]{\lineheight{0}\smash{\begin{tabular}[t]{l}$\scriptstyle b_1$\end{tabular}}}}%
    \put(0.42500005,0.14927027){\color[rgb]{0,0,0}\makebox(0,0)[lt]{\lineheight{0}\smash{\begin{tabular}[t]{l}$\scriptstyle x$\end{tabular}}}}%
    \put(0.54375006,0.14927027){\color[rgb]{0,0,0}\makebox(0,0)[lt]{\lineheight{0}\smash{\begin{tabular}[t]{l}$\scriptstyle y$\end{tabular}}}}%
    \put(0.50000007,0.06177026){\color[rgb]{0,0,0}\makebox(0,0)[lt]{\lineheight{0}\smash{\begin{tabular}[t]{l}$\scriptstyle z$\end{tabular}}}}%
    \put(0,0){\includegraphics[width=\unitlength,page=2]{Courbes_relation_piqure_surface_minimale.pdf}}%
  \end{picture}%
\endgroup%

%% file: Courbes_relation_3-chaine_surface_minimale.pdf_tex
\begingroup%
  \makeatletter%
  \providecommand\color[2][]{%
    \errmessage{(Inkscape) Color is used for the text in Inkscape, but the package 'color.sty' is not loaded}%
    \renewcommand\color[2][]{}%
  }%
  \providecommand\transparent[1]{%
    \errmessage{(Inkscape) Transparency is used (non-zero) for the text in Inkscape, but the package 'transparent.sty' is not loaded}%
    \renewcommand\transparent[1]{}%
  }%
  \providecommand\rotatebox[2]{#2}%
  \newcommand*\fsize{\dimexpr\f@size pt\relax}%
  \newcommand*\lineheight[1]{\fontsize{\fsize}{#1\fsize}\selectfont}%
  \ifx\svgwidth\undefined%
    \setlength{\unitlength}{453.54330709bp}%
    \ifx\svgscale\undefined%
      \relax%
    \else%
      \setlength{\unitlength}{\unitlength * \real{\svgscale}}%
    \fi%
  \else%
    \setlength{\unitlength}{\svgwidth}%
  \fi%
  \global\let\svgwidth\undefined%
  \global\let\svgscale\undefined%
  \makeatother%
  \begin{picture}(1,0.176299)%
    \lineheight{1}%
    \setlength\tabcolsep{0pt}%
    \put(0,0){\includegraphics[width=\unitlength,page=1]{Courbes_relation_3-chaine_surface_minimale.pdf}}%
    \put(0.5468748,0.11315525){\color[rgb]{0,0,0}\makebox(0,0)[lt]{\lineheight{0}\smash{\begin{tabular}[t]{l}$\scriptstyle c$\end{tabular}}}}%
    \put(0.47812481,0.14440525){\color[rgb]{0,0,0}\makebox(0,0)[lt]{\lineheight{0}\smash{\begin{tabular}[t]{l}$\scriptstyle b$\end{tabular}}}}%
    \put(0.47187481,0.02565525){\color[rgb]{0,0,0}\makebox(0,0)[lt]{\lineheight{0}\smash{\begin{tabular}[t]{l}$\scriptstyle a$\end{tabular}}}}%
    \put(0.57187479,0.13815525){\color[rgb]{0,0,0}\makebox(0,0)[lt]{\lineheight{0}\smash{\begin{tabular}[t]{l}$\scriptstyle d$\end{tabular}}}}%
    \put(0.57187481,0.03190525){\color[rgb]{0,0,0}\makebox(0,0)[lt]{\lineheight{0}\smash{\begin{tabular}[t]{l}$\scriptstyle e$\end{tabular}}}}%
    \put(0,0){\includegraphics[width=\unitlength,page=2]{Courbes_relation_3-chaine_surface_minimale.pdf}}%
  \end{picture}%
\endgroup%

%% file: Courbes_standards_surface.pdf_tex
\begingroup%
  \makeatletter%
  \providecommand\color[2][]{%
    \errmessage{(Inkscape) Color is used for the text in Inkscape, but the package 'color.sty' is not loaded}%
    \renewcommand\color[2][]{}%
  }%
  \providecommand\transparent[1]{%
    \errmessage{(Inkscape) Transparency is used (non-zero) for the text in Inkscape, but the package 'transparent.sty' is not loaded}%
    \renewcommand\transparent[1]{}%
  }%
  \providecommand\rotatebox[2]{#2}%
  \newcommand*\fsize{\dimexpr\f@size pt\relax}%
  \newcommand*\lineheight[1]{\fontsize{\fsize}{#1\fsize}\selectfont}%
  \ifx\svgwidth\undefined%
    \setlength{\unitlength}{453.54313407bp}%
    \ifx\svgscale\undefined%
      \relax%
    \else%
      \setlength{\unitlength}{\unitlength * \real{\svgscale}}%
    \fi%
  \else%
    \setlength{\unitlength}{\svgwidth}%
  \fi%
  \global\let\svgwidth\undefined%
  \global\let\svgscale\undefined%
  \makeatother%
  \begin{picture}(1,0.17628758)%
    \lineheight{1}%
    \setlength\tabcolsep{0pt}%
    \put(0,0){\includegraphics[width=\unitlength,page=1]{Courbes_standards_surface.pdf}}%
    \put(0.78189368,0.06001877){\color[rgb]{0,0,0}\makebox(0,0)[lt]{\lineheight{1.25}\smash{\begin{tabular}[t]{l}$\scriptstyle m_{g+k}$\end{tabular}}}}%
    \put(0.2850172,0.03189377){\color[rgb]{0,0,0}\makebox(0,0)[lt]{\lineheight{1.25}\smash{\begin{tabular}[t]{l}$\scriptstyle a_j$\end{tabular}}}}%
    \put(0.28501895,0.12876882){\color[rgb]{0,0,0}\makebox(0,0)[lt]{\lineheight{1.25}\smash{\begin{tabular}[t]{l}$\scriptstyle b_j$\end{tabular}}}}%
    \put(0,0){\includegraphics[width=\unitlength,page=2]{Courbes_standards_surface.pdf}}%
  \end{picture}%
\endgroup%

%% file: Courbes_standards_ruban.pdf_tex
\begingroup%
  \makeatletter%
  \providecommand\color[2][]{%
    \errmessage{(Inkscape) Color is used for the text in Inkscape, but the package 'color.sty' is not loaded}%
    \renewcommand\color[2][]{}%
  }%
  \providecommand\transparent[1]{%
    \errmessage{(Inkscape) Transparency is used (non-zero) for the text in Inkscape, but the package 'transparent.sty' is not loaded}%
    \renewcommand\transparent[1]{}%
  }%
  \providecommand\rotatebox[2]{#2}%
  \newcommand*\fsize{\dimexpr\f@size pt\relax}%
  \newcommand*\lineheight[1]{\fontsize{\fsize}{#1\fsize}\selectfont}%
  \ifx\svgwidth\undefined%
    \setlength{\unitlength}{453.54330709bp}%
    \ifx\svgscale\undefined%
      \relax%
    \else%
      \setlength{\unitlength}{\unitlength * \real{\svgscale}}%
    \fi%
  \else%
    \setlength{\unitlength}{\svgwidth}%
  \fi%
  \global\let\svgwidth\undefined%
  \global\let\svgscale\undefined%
  \makeatother%
  \begin{picture}(1,0.19383199)%
    \lineheight{1}%
    \setlength\tabcolsep{0pt}%
    \put(0,0){\includegraphics[width=\unitlength,page=1]{Courbes_standards_ruban.pdf}}%
    \put(0.33501576,0.12876882){\color[rgb]{0,0,0}\makebox(0,0)[lt]{\lineheight{0}\smash{\begin{tabular}[t]{l}\scriptsize $b_j$\end{tabular}}}}%
    \put(0.47877171,0.128771){\color[rgb]{0,0,0}\makebox(0,0)[lt]{\lineheight{0}\smash{\begin{tabular}[t]{l}\scriptsize $a_j$\end{tabular}}}}%
    \put(0.7756438,0.12876882){\color[rgb]{0,0,0}\makebox(0,0)[lt]{\lineheight{0}\smash{\begin{tabular}[t]{l}\scriptsize $m_{g+k}$\end{tabular}}}}%
    \put(0,0){\includegraphics[width=\unitlength,page=2]{Courbes_standards_ruban.pdf}}%
  \end{picture}%
\endgroup%

%% file: Courbe_z_ruban.pdf_tex
\begingroup%
  \makeatletter%
  \providecommand\color[2][]{%
    \errmessage{(Inkscape) Color is used for the text in Inkscape, but the package 'color.sty' is not loaded}%
    \renewcommand\color[2][]{}%
  }%
  \providecommand\transparent[1]{%
    \errmessage{(Inkscape) Transparency is used (non-zero) for the text in Inkscape, but the package 'transparent.sty' is not loaded}%
    \renewcommand\transparent[1]{}%
  }%
  \providecommand\rotatebox[2]{#2}%
  \newcommand*\fsize{\dimexpr\f@size pt\relax}%
  \newcommand*\lineheight[1]{\fontsize{\fsize}{#1\fsize}\selectfont}%
  \ifx\svgwidth\undefined%
    \setlength{\unitlength}{453.54330709bp}%
    \ifx\svgscale\undefined%
      \relax%
    \else%
      \setlength{\unitlength}{\unitlength * \real{\svgscale}}%
    \fi%
  \else%
    \setlength{\unitlength}{\svgwidth}%
  \fi%
  \global\let\svgwidth\undefined%
  \global\let\svgscale\undefined%
  \makeatother%
  \begin{picture}(1,0.21291177)%
    \lineheight{1}%
    \setlength\tabcolsep{0pt}%
    \put(0,0){\includegraphics[width=\unitlength,page=1]{Courbe_z_ruban.pdf}}%
    \put(0.07251876,0.00409859){\color[rgb]{0,0,0}\makebox(0,0)[lt]{\lineheight{0}\smash{\begin{tabular}[t]{l}\scriptsize $b_1$\end{tabular}}}}%
    \put(0.12251875,0.00409859){\color[rgb]{0,0,0}\makebox(0,0)[lt]{\lineheight{0}\smash{\begin{tabular}[t]{l}\scriptsize $a_1$\end{tabular}}}}%
    \put(0.27251875,0.00409859){\color[rgb]{0,0,0}\makebox(0,0)[lt]{\lineheight{0}\smash{\begin{tabular}[t]{l}\scriptsize $b_2$\end{tabular}}}}%
    \put(0.32251875,0.00409859){\color[rgb]{0,0,0}\makebox(0,0)[lt]{\lineheight{0}\smash{\begin{tabular}[t]{l}\scriptsize $a_2$\end{tabular}}}}%
    \put(0.47251881,0.00409859){\color[rgb]{0,0,0}\makebox(0,0)[lt]{\lineheight{0}\smash{\begin{tabular}[t]{l}\scriptsize $b_3$\end{tabular}}}}%
    \put(0.52251878,0.00409859){\color[rgb]{0,0,0}\makebox(0,0)[lt]{\lineheight{0}\smash{\begin{tabular}[t]{l}\scriptsize $a_3$\end{tabular}}}}%
    \put(0.6725188,0.00409859){\color[rgb]{0,0,0}\makebox(0,0)[lt]{\lineheight{0}\smash{\begin{tabular}[t]{l}\scriptsize $b_g$\end{tabular}}}}%
    \put(0.72251869,0.00409859){\color[rgb]{0,0,0}\makebox(0,0)[lt]{\lineheight{0}\smash{\begin{tabular}[t]{l}\scriptsize $a_g$\end{tabular}}}}%
    \put(0.83814331,0.03222361){\color[rgb]{0,0,0}\makebox(0,0)[lt]{\lineheight{0}\smash{\begin{tabular}[t]{l}\scriptsize $m_{g+1}$\end{tabular}}}}%
    \put(0.93189327,0.03222361){\color[rgb]{0,0,0}\makebox(0,0)[lt]{\lineheight{0}\smash{\begin{tabular}[t]{l}\scriptsize $m_{g+n}$\end{tabular}}}}%
    \put(0,0){\includegraphics[width=\unitlength,page=2]{Courbe_z_ruban.pdf}}%
    \put(0.49439378,0.1790986){\color[rgb]{0,0,0}\makebox(0,0)[lt]{\lineheight{1.25}\smash{\begin{tabular}[t]{l}\scriptsize $z$\end{tabular}}}}%
  \end{picture}%
\endgroup%

%% file: Courbes_ai_bi_ci.pdf_tex
\begingroup%
  \makeatletter%
  \providecommand\color[2][]{%
    \errmessage{(Inkscape) Color is used for the text in Inkscape, but the package 'color.sty' is not loaded}%
    \renewcommand\color[2][]{}%
  }%
  \providecommand\transparent[1]{%
    \errmessage{(Inkscape) Transparency is used (non-zero) for the text in Inkscape, but the package 'transparent.sty' is not loaded}%
    \renewcommand\transparent[1]{}%
  }%
  \providecommand\rotatebox[2]{#2}%
  \newcommand*\fsize{\dimexpr\f@size pt\relax}%
  \newcommand*\lineheight[1]{\fontsize{\fsize}{#1\fsize}\selectfont}%
  \ifx\svgwidth\undefined%
    \setlength{\unitlength}{453.54313407bp}%
    \ifx\svgscale\undefined%
      \relax%
    \else%
      \setlength{\unitlength}{\unitlength * \real{\svgscale}}%
    \fi%
  \else%
    \setlength{\unitlength}{\svgwidth}%
  \fi%
  \global\let\svgwidth\undefined%
  \global\let\svgscale\undefined%
  \makeatother%
  \begin{picture}(1,0.17628758)%
    \lineheight{1}%
    \setlength\tabcolsep{0pt}%
    \put(0,0){\includegraphics[width=\unitlength,page=1]{Courbes_ai_bi_ci.pdf}}%
    \put(0.36314348,0.1131438){\color[rgb]{0,0,0}\makebox(0,0)[lt]{\lineheight{1.25}\smash{\begin{tabular}[t]{l}$\scriptstyle c_j$\end{tabular}}}}%
    \put(0.29751848,0.13501608){\color[rgb]{0,0,0}\makebox(0,0)[lt]{\lineheight{1.25}\smash{\begin{tabular}[t]{l}$\scriptstyle b_j$\end{tabular}}}}%
    \put(0.28189349,0.03189377){\color[rgb]{0,0,0}\makebox(0,0)[lt]{\lineheight{1.25}\smash{\begin{tabular}[t]{l}$\scriptstyle a_j$\end{tabular}}}}%
  \end{picture}%
\endgroup%

%% file: Courbe_bord_sigma.pdf_tex
\begingroup%
  \makeatletter%
  \providecommand\color[2][]{%
    \errmessage{(Inkscape) Color is used for the text in Inkscape, but the package 'color.sty' is not loaded}%
    \renewcommand\color[2][]{}%
  }%
  \providecommand\transparent[1]{%
    \errmessage{(Inkscape) Transparency is used (non-zero) for the text in Inkscape, but the package 'transparent.sty' is not loaded}%
    \renewcommand\transparent[1]{}%
  }%
  \providecommand\rotatebox[2]{#2}%
  \newcommand*\fsize{\dimexpr\f@size pt\relax}%
  \newcommand*\lineheight[1]{\fontsize{\fsize}{#1\fsize}\selectfont}%
  \ifx\svgwidth\undefined%
    \setlength{\unitlength}{453.54313407bp}%
    \ifx\svgscale\undefined%
      \relax%
    \else%
      \setlength{\unitlength}{\unitlength * \real{\svgscale}}%
    \fi%
  \else%
    \setlength{\unitlength}{\svgwidth}%
  \fi%
  \global\let\svgwidth\undefined%
  \global\let\svgscale\undefined%
  \makeatother%
  \begin{picture}(1,0.17628758)%
    \lineheight{1}%
    \setlength\tabcolsep{0pt}%
    \put(0,0){\includegraphics[width=\unitlength,page=1]{Courbe_bord_sigma.pdf}}%
    \put(0.86314374,0.14126921){\color[rgb]{0,0,0}\makebox(0,0)[lt]{\lineheight{1.25}\smash{\begin{tabular}[t]{l}$\scriptstyle \partial_{g,n}$\end{tabular}}}}%
    \put(0,0){\includegraphics[width=\unitlength,page=2]{Courbe_bord_sigma.pdf}}%
  \end{picture}%
\endgroup%

%% file: Courbes_etoile.pdf_tex
\begingroup%
  \makeatletter%
  \providecommand\color[2][]{%
    \errmessage{(Inkscape) Color is used for the text in Inkscape, but the package 'color.sty' is not loaded}%
    \renewcommand\color[2][]{}%
  }%
  \providecommand\transparent[1]{%
    \errmessage{(Inkscape) Transparency is used (non-zero) for the text in Inkscape, but the package 'transparent.sty' is not loaded}%
    \renewcommand\transparent[1]{}%
  }%
  \providecommand\rotatebox[2]{#2}%
  \newcommand*\fsize{\dimexpr\f@size pt\relax}%
  \newcommand*\lineheight[1]{\fontsize{\fsize}{#1\fsize}\selectfont}%
  \ifx\svgwidth\undefined%
    \setlength{\unitlength}{453.54313407bp}%
    \ifx\svgscale\undefined%
      \relax%
    \else%
      \setlength{\unitlength}{\unitlength * \real{\svgscale}}%
    \fi%
  \else%
    \setlength{\unitlength}{\svgwidth}%
  \fi%
  \global\let\svgwidth\undefined%
  \global\let\svgscale\undefined%
  \makeatother%
  \begin{picture}(1,0.17628758)%
    \lineheight{1}%
    \setlength\tabcolsep{0pt}%
    \put(0.63814363,0.0600188){\color[rgb]{0,0,0}\makebox(0,0)[lt]{\lineheight{1.25}\smash{\begin{tabular}[t]{l}$\scriptstyle \delta_2$\end{tabular}}}}%
    \put(0.31939349,0.14126883){\color[rgb]{0,0,0}\makebox(0,0)[lt]{\lineheight{1.25}\smash{\begin{tabular}[t]{l}$\scriptstyle \delta_1$\end{tabular}}}}%
    \put(0.63814363,0.1256438){\color[rgb]{0,0,0}\makebox(0,0)[lt]{\lineheight{1.25}\smash{\begin{tabular}[t]{l}$\scriptstyle \alpha_3$\end{tabular}}}}%
    \put(0.45689356,0.15376883){\color[rgb]{0,0,0}\makebox(0,0)[lt]{\lineheight{1.25}\smash{\begin{tabular}[t]{l}$\scriptstyle \alpha_2$\end{tabular}}}}%
    \put(0.39126604,0.1131438){\color[rgb]{0,0,0}\makebox(0,0)[lt]{\lineheight{1.25}\smash{\begin{tabular}[t]{l}$\scriptstyle \alpha_1$\end{tabular}}}}%
    \put(0.51939359,0.1256438){\color[rgb]{0,0,0}\makebox(0,0)[lt]{\lineheight{1.25}\smash{\begin{tabular}[t]{l}$\scriptstyle \beta$\end{tabular}}}}%
    \put(0.7147037,0.03501878){\color[rgb]{0,0,0}\makebox(0,0)[lt]{\lineheight{1.25}\smash{\begin{tabular}[t]{l}$\scriptscriptstyle n$\end{tabular}}}}%
    \put(0,0){\includegraphics[width=\unitlength,page=1]{Courbes_etoile.pdf}}%
  \end{picture}%
\endgroup%

%% file: Courbes_3-chaine_rc0.pdf_tex
\begingroup%
  \makeatletter%
  \providecommand\color[2][]{%
    \errmessage{(Inkscape) Color is used for the text in Inkscape, but the package 'color.sty' is not loaded}%
    \renewcommand\color[2][]{}%
  }%
  \providecommand\transparent[1]{%
    \errmessage{(Inkscape) Transparency is used (non-zero) for the text in Inkscape, but the package 'transparent.sty' is not loaded}%
    \renewcommand\transparent[1]{}%
  }%
  \providecommand\rotatebox[2]{#2}%
  \newcommand*\fsize{\dimexpr\f@size pt\relax}%
  \newcommand*\lineheight[1]{\fontsize{\fsize}{#1\fsize}\selectfont}%
  \ifx\svgwidth\undefined%
    \setlength{\unitlength}{453.54313407bp}%
    \ifx\svgscale\undefined%
      \relax%
    \else%
      \setlength{\unitlength}{\unitlength * \real{\svgscale}}%
    \fi%
  \else%
    \setlength{\unitlength}{\svgwidth}%
  \fi%
  \global\let\svgwidth\undefined%
  \global\let\svgscale\undefined%
  \makeatother%
  \begin{picture}(1,0.17628758)%
    \lineheight{1}%
    \setlength\tabcolsep{0pt}%
    \put(0.61314361,0.1256438){\color[rgb]{0,0,0}\makebox(0,0)[lt]{\lineheight{1.25}\smash{\begin{tabular}[t]{l}$\scriptstyle \gamma$\end{tabular}}}}%
    \put(0.31939349,0.14126883){\color[rgb]{0,0,0}\makebox(0,0)[lt]{\lineheight{1.25}\smash{\begin{tabular}[t]{l}$\scriptstyle \delta_1$\end{tabular}}}}%
    \put(0.45689356,0.15376883){\color[rgb]{0,0,0}\makebox(0,0)[lt]{\lineheight{1.25}\smash{\begin{tabular}[t]{l}$\scriptstyle \alpha_2$\end{tabular}}}}%
    \put(0.39126604,0.1131438){\color[rgb]{0,0,0}\makebox(0,0)[lt]{\lineheight{1.25}\smash{\begin{tabular}[t]{l}$\scriptstyle \alpha_1$\end{tabular}}}}%
    \put(0.51939359,0.1256438){\color[rgb]{0,0,0}\makebox(0,0)[lt]{\lineheight{1.25}\smash{\begin{tabular}[t]{l}$\scriptstyle \beta$\end{tabular}}}}%
    \put(0.7147037,0.03501878){\color[rgb]{0,0,0}\makebox(0,0)[lt]{\lineheight{1.25}\smash{\begin{tabular}[t]{l}$\scriptscriptstyle n$\end{tabular}}}}%
    \put(0,0){\includegraphics[width=\unitlength,page=1]{Courbes_3-chaine_rc0.pdf}}%
  \end{picture}%
\endgroup%

%% file: Courbes_union_separante.pdf_tex
\begingroup%
  \makeatletter%
  \providecommand\color[2][]{%
    \errmessage{(Inkscape) Color is used for the text in Inkscape, but the package 'color.sty' is not loaded}%
    \renewcommand\color[2][]{}%
  }%
  \providecommand\transparent[1]{%
    \errmessage{(Inkscape) Transparency is used (non-zero) for the text in Inkscape, but the package 'transparent.sty' is not loaded}%
    \renewcommand\transparent[1]{}%
  }%
  \providecommand\rotatebox[2]{#2}%
  \newcommand*\fsize{\dimexpr\f@size pt\relax}%
  \newcommand*\lineheight[1]{\fontsize{\fsize}{#1\fsize}\selectfont}%
  \ifx\svgwidth\undefined%
    \setlength{\unitlength}{453.54313407bp}%
    \ifx\svgscale\undefined%
      \relax%
    \else%
      \setlength{\unitlength}{\unitlength * \real{\svgscale}}%
    \fi%
  \else%
    \setlength{\unitlength}{\svgwidth}%
  \fi%
  \global\let\svgwidth\undefined%
  \global\let\svgscale\undefined%
  \makeatother%
  \begin{picture}(1,0.17628758)%
    \lineheight{1}%
    \setlength\tabcolsep{0pt}%
    \put(0,0){\includegraphics[width=\unitlength,page=1]{Courbes_union_separante.pdf}}%
  \end{picture}%
\endgroup%

%% file: Courbes_0-tresse.pdf_tex
\begingroup%
  \makeatletter%
  \providecommand\color[2][]{%
    \errmessage{(Inkscape) Color is used for the text in Inkscape, but the package 'color.sty' is not loaded}%
    \renewcommand\color[2][]{}%
  }%
  \providecommand\transparent[1]{%
    \errmessage{(Inkscape) Transparency is used (non-zero) for the text in Inkscape, but the package 'transparent.sty' is not loaded}%
    \renewcommand\transparent[1]{}%
  }%
  \providecommand\rotatebox[2]{#2}%
  \newcommand*\fsize{\dimexpr\f@size pt\relax}%
  \newcommand*\lineheight[1]{\fontsize{\fsize}{#1\fsize}\selectfont}%
  \ifx\svgwidth\undefined%
    \setlength{\unitlength}{453.54313407bp}%
    \ifx\svgscale\undefined%
      \relax%
    \else%
      \setlength{\unitlength}{\unitlength * \real{\svgscale}}%
    \fi%
  \else%
    \setlength{\unitlength}{\svgwidth}%
  \fi%
  \global\let\svgwidth\undefined%
  \global\let\svgscale\undefined%
  \makeatother%
  \begin{picture}(1,0.17628758)%
    \lineheight{1}%
    \setlength\tabcolsep{0pt}%
    \put(0.28501601,0.03189376){\color[rgb]{0,0,0}\makebox(0,0)[lt]{\lineheight{1.25}\smash{\begin{tabular}[t]{l}$\scriptstyle f_{k,l}$\end{tabular}}}}%
    \put(0.7162662,0.03501878){\color[rgb]{0,0,0}\makebox(0,0)[lt]{\lineheight{1.25}\smash{\begin{tabular}[t]{l}$\scriptscriptstyle l$\end{tabular}}}}%
    \put(0.29126602,0.08501881){\color[rgb]{0,0,0}\makebox(0,0)[lt]{\lineheight{1.25}\smash{\begin{tabular}[t]{l}$\scriptscriptstyle k$\end{tabular}}}}%
    \put(0,0){\includegraphics[width=\unitlength,page=1]{Courbes_0-tresse.pdf}}%
    \put(0.28509799,0.13815534){\color[rgb]{0,0,0}\makebox(0,0)[lt]{\lineheight{1.25}\smash{\begin{tabular}[t]{l}$\scriptstyle e_k$\end{tabular}}}}%
    \put(0.11314341,0.03190528){\color[rgb]{0,0,0}\makebox(0,0)[lt]{\lineheight{1.25}\smash{\begin{tabular}[t]{l}$\scriptstyle a_1$\end{tabular}}}}%
    \put(0,0){\includegraphics[width=\unitlength,page=2]{Courbes_0-tresse.pdf}}%
  \end{picture}%
\endgroup%

%% file: Courbes_1-tresse.pdf_tex
\begingroup%
  \makeatletter%
  \providecommand\color[2][]{%
    \errmessage{(Inkscape) Color is used for the text in Inkscape, but the package 'color.sty' is not loaded}%
    \renewcommand\color[2][]{}%
  }%
  \providecommand\transparent[1]{%
    \errmessage{(Inkscape) Transparency is used (non-zero) for the text in Inkscape, but the package 'transparent.sty' is not loaded}%
    \renewcommand\transparent[1]{}%
  }%
  \providecommand\rotatebox[2]{#2}%
  \newcommand*\fsize{\dimexpr\f@size pt\relax}%
  \newcommand*\lineheight[1]{\fontsize{\fsize}{#1\fsize}\selectfont}%
  \ifx\svgwidth\undefined%
    \setlength{\unitlength}{453.54313407bp}%
    \ifx\svgscale\undefined%
      \relax%
    \else%
      \setlength{\unitlength}{\unitlength * \real{\svgscale}}%
    \fi%
  \else%
    \setlength{\unitlength}{\svgwidth}%
  \fi%
  \global\let\svgwidth\undefined%
  \global\let\svgscale\undefined%
  \makeatother%
  \begin{picture}(1,0.17628758)%
    \lineheight{1}%
    \setlength\tabcolsep{0pt}%
    \put(0,0){\includegraphics[width=\unitlength,page=1]{Courbes_1-tresse.pdf}}%
    \put(0.1131434,0.03189376){\color[rgb]{0,0,0}\makebox(0,0)[lt]{\lineheight{1.25}\smash{\begin{tabular}[t]{l}$\scriptstyle a_1$\end{tabular}}}}%
    \put(0.1225184,0.14439381){\color[rgb]{0,0,0}\makebox(0,0)[lt]{\lineheight{1.25}\smash{\begin{tabular}[t]{l}$\scriptstyle b_1$\end{tabular}}}}%
    \put(0,0){\includegraphics[width=\unitlength,page=2]{Courbes_1-tresse.pdf}}%
  \end{picture}%
\endgroup%

%% file: Courbes_lanterne.pdf_tex
\begingroup%
  \makeatletter%
  \providecommand\color[2][]{%
    \errmessage{(Inkscape) Color is used for the text in Inkscape, but the package 'color.sty' is not loaded}%
    \renewcommand\color[2][]{}%
  }%
  \providecommand\transparent[1]{%
    \errmessage{(Inkscape) Transparency is used (non-zero) for the text in Inkscape, but the package 'transparent.sty' is not loaded}%
    \renewcommand\transparent[1]{}%
  }%
  \providecommand\rotatebox[2]{#2}%
  \newcommand*\fsize{\dimexpr\f@size pt\relax}%
  \newcommand*\lineheight[1]{\fontsize{\fsize}{#1\fsize}\selectfont}%
  \ifx\svgwidth\undefined%
    \setlength{\unitlength}{453.54313407bp}%
    \ifx\svgscale\undefined%
      \relax%
    \else%
      \setlength{\unitlength}{\unitlength * \real{\svgscale}}%
    \fi%
  \else%
    \setlength{\unitlength}{\svgwidth}%
  \fi%
  \global\let\svgwidth\undefined%
  \global\let\svgscale\undefined%
  \makeatother%
  \begin{picture}(1,0.17628758)%
    \lineheight{1}%
    \setlength\tabcolsep{0pt}%
    \put(0,0){\includegraphics[width=\unitlength,page=1]{Courbes_lanterne.pdf}}%
    \put(0.19126843,0.07251878){\color[rgb]{0,0,0}\makebox(0,0)[lt]{\lineheight{1.25}\smash{\begin{tabular}[t]{l}$\scriptstyle z$\end{tabular}}}}%
    \put(0.19126842,0.03501876){\color[rgb]{0,0,0}\makebox(0,0)[lt]{\lineheight{1.25}\smash{\begin{tabular}[t]{l}$\scriptstyle y$\end{tabular}}}}%
    \put(0.25064346,0.02564377){\color[rgb]{0,0,0}\makebox(0,0)[lt]{\lineheight{1.25}\smash{\begin{tabular}[t]{l}$\scriptstyle x$\end{tabular}}}}%
    \put(0.38814354,0.02564414){\color[rgb]{0,0,0}\makebox(0,0)[lt]{\lineheight{1.25}\smash{\begin{tabular}[t]{l}$\scriptstyle b_4$\end{tabular}}}}%
    \put(0.32876848,0.11314381){\color[rgb]{0,0,0}\makebox(0,0)[lt]{\lineheight{1.25}\smash{\begin{tabular}[t]{l}$\scriptstyle b_3$\end{tabular}}}}%
    \put(0.19126843,0.1131438){\color[rgb]{0,0,0}\makebox(0,0)[lt]{\lineheight{1.25}\smash{\begin{tabular}[t]{l}$\scriptstyle b_2$\end{tabular}}}}%
    \put(0.1131434,0.02564415){\color[rgb]{0,0,0}\makebox(0,0)[lt]{\lineheight{1.25}\smash{\begin{tabular}[t]{l}$\scriptstyle b_1$\end{tabular}}}}%
    \put(0,0){\includegraphics[width=\unitlength,page=2]{Courbes_lanterne.pdf}}%
  \end{picture}%
\endgroup%

%% file: Courbes_3-chaine.pdf_tex
\begingroup%
  \makeatletter%
  \providecommand\color[2][]{%
    \errmessage{(Inkscape) Color is used for the text in Inkscape, but the package 'color.sty' is not loaded}%
    \renewcommand\color[2][]{}%
  }%
  \providecommand\transparent[1]{%
    \errmessage{(Inkscape) Transparency is used (non-zero) for the text in Inkscape, but the package 'transparent.sty' is not loaded}%
    \renewcommand\transparent[1]{}%
  }%
  \providecommand\rotatebox[2]{#2}%
  \newcommand*\fsize{\dimexpr\f@size pt\relax}%
  \newcommand*\lineheight[1]{\fontsize{\fsize}{#1\fsize}\selectfont}%
  \ifx\svgwidth\undefined%
    \setlength{\unitlength}{453.54313407bp}%
    \ifx\svgscale\undefined%
      \relax%
    \else%
      \setlength{\unitlength}{\unitlength * \real{\svgscale}}%
    \fi%
  \else%
    \setlength{\unitlength}{\svgwidth}%
  \fi%
  \global\let\svgwidth\undefined%
  \global\let\svgscale\undefined%
  \makeatother%
  \begin{picture}(1,0.17628758)%
    \lineheight{1}%
    \setlength\tabcolsep{0pt}%
    \put(0,0){\includegraphics[width=\unitlength,page=1]{Courbes_3-chaine.pdf}}%
    \put(0.25064346,0.13814381){\color[rgb]{0,0,0}\makebox(0,0)[lt]{\lineheight{1.25}\smash{\begin{tabular}[t]{l}$\scriptstyle e$\end{tabular}}}}%
    \put(0.25064346,0.03189376){\color[rgb]{0,0,0}\makebox(0,0)[lt]{\lineheight{1.25}\smash{\begin{tabular}[t]{l}$\scriptstyle d$\end{tabular}}}}%
    \put(0.19126843,0.1131438){\color[rgb]{0,0,0}\makebox(0,0)[lt]{\lineheight{1.25}\smash{\begin{tabular}[t]{l}$\scriptstyle c$\end{tabular}}}}%
    \put(0.1225184,0.14439381){\color[rgb]{0,0,0}\makebox(0,0)[lt]{\lineheight{1.25}\smash{\begin{tabular}[t]{l}$\scriptstyle b$\end{tabular}}}}%
    \put(0.1131434,0.03189376){\color[rgb]{0,0,0}\makebox(0,0)[lt]{\lineheight{1.25}\smash{\begin{tabular}[t]{l}$\scriptstyle a$\end{tabular}}}}%
    \put(0,0){\includegraphics[width=\unitlength,page=2]{Courbes_3-chaine.pdf}}%
  \end{picture}%
\endgroup%

%% file: Courbes_ab_puissance6.pdf_tex
\begingroup%
  \makeatletter%
  \providecommand\color[2][]{%
    \errmessage{(Inkscape) Color is used for the text in Inkscape, but the package 'color.sty' is not loaded}%
    \renewcommand\color[2][]{}%
  }%
  \providecommand\transparent[1]{%
    \errmessage{(Inkscape) Transparency is used (non-zero) for the text in Inkscape, but the package 'transparent.sty' is not loaded}%
    \renewcommand\transparent[1]{}%
  }%
  \providecommand\rotatebox[2]{#2}%
  \newcommand*\fsize{\dimexpr\f@size pt\relax}%
  \newcommand*\lineheight[1]{\fontsize{\fsize}{#1\fsize}\selectfont}%
  \ifx\svgwidth\undefined%
    \setlength{\unitlength}{453.54313407bp}%
    \ifx\svgscale\undefined%
      \relax%
    \else%
      \setlength{\unitlength}{\unitlength * \real{\svgscale}}%
    \fi%
  \else%
    \setlength{\unitlength}{\svgwidth}%
  \fi%
  \global\let\svgwidth\undefined%
  \global\let\svgscale\undefined%
  \makeatother%
  \begin{picture}(1,0.17628758)%
    \lineheight{1}%
    \setlength\tabcolsep{0pt}%
    \put(0,0){\includegraphics[width=\unitlength,page=1]{Courbes_ab_puissance6.pdf}}%
    \put(0.46874978,0.14439381){\color[rgb]{0,0,0}\makebox(0,0)[lt]{\lineheight{1.25}\smash{\begin{tabular}[t]{l}$\scriptstyle b$\end{tabular}}}}%
    \put(0.45937478,0.03189376){\color[rgb]{0,0,0}\makebox(0,0)[lt]{\lineheight{1.25}\smash{\begin{tabular}[t]{l}$\scriptstyle a$\end{tabular}}}}%
    \put(0,0){\includegraphics[width=\unitlength,page=2]{Courbes_ab_puissance6.pdf}}%
  \end{picture}%
\endgroup%

%% file: Courbes_piqure_1.pdf_tex
\begingroup%
  \makeatletter%
  \providecommand\color[2][]{%
    \errmessage{(Inkscape) Color is used for the text in Inkscape, but the package 'color.sty' is not loaded}%
    \renewcommand\color[2][]{}%
  }%
  \providecommand\transparent[1]{%
    \errmessage{(Inkscape) Transparency is used (non-zero) for the text in Inkscape, but the package 'transparent.sty' is not loaded}%
    \renewcommand\transparent[1]{}%
  }%
  \providecommand\rotatebox[2]{#2}%
  \newcommand*\fsize{\dimexpr\f@size pt\relax}%
  \newcommand*\lineheight[1]{\fontsize{\fsize}{#1\fsize}\selectfont}%
  \ifx\svgwidth\undefined%
    \setlength{\unitlength}{453.54313407bp}%
    \ifx\svgscale\undefined%
      \relax%
    \else%
      \setlength{\unitlength}{\unitlength * \real{\svgscale}}%
    \fi%
  \else%
    \setlength{\unitlength}{\svgwidth}%
  \fi%
  \global\let\svgwidth\undefined%
  \global\let\svgscale\undefined%
  \makeatother%
  \begin{picture}(1,0.17628758)%
    \lineheight{1}%
    \setlength\tabcolsep{0pt}%
    \put(0.72564371,0.11626882){\color[rgb]{0,0,0}\makebox(0,0)[lt]{\lineheight{1.25}\smash{\begin{tabular}[t]{l}$\scriptstyle \gamma_1^k$\end{tabular}}}}%
    \put(0.46939356,0.05376879){\color[rgb]{0,0,0}\makebox(0,0)[lt]{\lineheight{1.25}\smash{\begin{tabular}[t]{l}$\scriptstyle \gamma_3^k$\end{tabular}}}}%
    \put(0.39439353,0.11626882){\color[rgb]{0,0,0}\makebox(0,0)[lt]{\lineheight{1.25}\smash{\begin{tabular}[t]{l}$\scriptstyle \gamma_2^k$\end{tabular}}}}%
    \put(0.62251614,0.03501878){\color[rgb]{0,0,0}\makebox(0,0)[lt]{\lineheight{1.25}\smash{\begin{tabular}[t]{l}$\scriptscriptstyle k-1$\end{tabular}}}}%
    \put(0.7156412,0.03501878){\color[rgb]{0,0,0}\makebox(0,0)[lt]{\lineheight{1.25}\smash{\begin{tabular}[t]{l}$\scriptscriptstyle k$\end{tabular}}}}%
    \put(0.78501622,0.03501878){\color[rgb]{0,0,0}\makebox(0,0)[lt]{\lineheight{1.25}\smash{\begin{tabular}[t]{l}$\scriptscriptstyle k+1$\end{tabular}}}}%
    \put(0,0){\includegraphics[width=\unitlength,page=1]{Courbes_piqure_1.pdf}}%
  \end{picture}%
\endgroup%

%% file: Courbes_piqure_2.pdf_tex
\begingroup%
  \makeatletter%
  \providecommand\color[2][]{%
    \errmessage{(Inkscape) Color is used for the text in Inkscape, but the package 'color.sty' is not loaded}%
    \renewcommand\color[2][]{}%
  }%
  \providecommand\transparent[1]{%
    \errmessage{(Inkscape) Transparency is used (non-zero) for the text in Inkscape, but the package 'transparent.sty' is not loaded}%
    \renewcommand\transparent[1]{}%
  }%
  \providecommand\rotatebox[2]{#2}%
  \newcommand*\fsize{\dimexpr\f@size pt\relax}%
  \newcommand*\lineheight[1]{\fontsize{\fsize}{#1\fsize}\selectfont}%
  \ifx\svgwidth\undefined%
    \setlength{\unitlength}{453.54313407bp}%
    \ifx\svgscale\undefined%
      \relax%
    \else%
      \setlength{\unitlength}{\unitlength * \real{\svgscale}}%
    \fi%
  \else%
    \setlength{\unitlength}{\svgwidth}%
  \fi%
  \global\let\svgwidth\undefined%
  \global\let\svgscale\undefined%
  \makeatother%
  \begin{picture}(1,0.17628758)%
    \lineheight{1}%
    \setlength\tabcolsep{0pt}%
    \put(0.60064365,0.1193938){\color[rgb]{0,0,0}\makebox(0,0)[lt]{\lineheight{1.25}\smash{\begin{tabular}[t]{l}$\scriptstyle y^k$\end{tabular}}}}%
    \put(0.69439367,0.1068938){\color[rgb]{0,0,0}\makebox(0,0)[lt]{\lineheight{1.25}\smash{\begin{tabular}[t]{l}$\scriptstyle z^k$\end{tabular}}}}%
    \put(0.75689366,0.0662688){\color[rgb]{0,0,0}\makebox(0,0)[lt]{\lineheight{1.25}\smash{\begin{tabular}[t]{l}$\scriptstyle x^k$\end{tabular}}}}%
    \put(0.62251614,0.03501878){\color[rgb]{0,0,0}\makebox(0,0)[lt]{\lineheight{1.25}\smash{\begin{tabular}[t]{l}$\scriptscriptstyle k-1$\end{tabular}}}}%
    \put(0.7156412,0.03501878){\color[rgb]{0,0,0}\makebox(0,0)[lt]{\lineheight{1.25}\smash{\begin{tabular}[t]{l}$\scriptscriptstyle k$\end{tabular}}}}%
    \put(0.78501622,0.03501878){\color[rgb]{0,0,0}\makebox(0,0)[lt]{\lineheight{1.25}\smash{\begin{tabular}[t]{l}$\scriptscriptstyle k+1$\end{tabular}}}}%
    \put(0,0){\includegraphics[width=\unitlength,page=1]{Courbes_piqure_2.pdf}}%
  \end{picture}%
\endgroup%

%% file: Exemple_diagramme.pdf_tex
\begingroup%
  \makeatletter%
  \providecommand\color[2][]{%
    \errmessage{(Inkscape) Color is used for the text in Inkscape, but the package 'color.sty' is not loaded}%
    \renewcommand\color[2][]{}%
  }%
  \providecommand\transparent[1]{%
    \errmessage{(Inkscape) Transparency is used (non-zero) for the text in Inkscape, but the package 'transparent.sty' is not loaded}%
    \renewcommand\transparent[1]{}%
  }%
  \providecommand\rotatebox[2]{#2}%
  \newcommand*\fsize{\dimexpr\f@size pt\relax}%
  \newcommand*\lineheight[1]{\fontsize{\fsize}{#1\fsize}\selectfont}%
  \ifx\svgwidth\undefined%
    \setlength{\unitlength}{453.54330709bp}%
    \ifx\svgscale\undefined%
      \relax%
    \else%
      \setlength{\unitlength}{\unitlength * \real{\svgscale}}%
    \fi%
  \else%
    \setlength{\unitlength}{\svgwidth}%
  \fi%
  \global\let\svgwidth\undefined%
  \global\let\svgscale\undefined%
  \makeatother%
  \begin{picture}(1,0.16649204)%
    \lineheight{1}%
    \setlength\tabcolsep{0pt}%
    \put(0,0){\includegraphics[width=\unitlength,page=1]{Exemple_diagramme.pdf}}%
    \put(0.65000152,0.00399203){\color[rgb]{0,0,0}\makebox(0,0)[lt]{\lineheight{0}\smash{\begin{tabular}[t]{l}$\scriptstyle x_{i_k}$ \end{tabular}}}}%
    \put(0.53750234,0.00399204){\color[rgb]{0,0,0}\makebox(0,0)[lt]{\lineheight{0}\smash{\begin{tabular}[t]{l}$\scriptstyle x_{i_2}$ \end{tabular}}}}%
    \put(0.42500232,0.00399204){\color[rgb]{0,0,0}\makebox(0,0)[lt]{\lineheight{0}\smash{\begin{tabular}[t]{l}$\scriptstyle x_{i_1}$ \end{tabular}}}}%
    \put(0.49470415,0.10711704){\color[rgb]{0,0,0}\makebox(0,0)[lt]{\lineheight{0}\smash{\begin{tabular}[t]{l}$\scriptstyle T$\end{tabular}}}}%
    \put(0.61874774,0.14774204){\makebox(0,0)[lt]{\lineheight{0}\smash{\begin{tabular}[t]{l}$\scriptstyle a_{i_{2t}}$\end{tabular}}}}%
    \put(0.50624953,0.14774204){\makebox(0,0)[lt]{\lineheight{0}\smash{\begin{tabular}[t]{l}$\scriptstyle a_{i_2}$\end{tabular}}}}%
    \put(0.39374954,0.14774204){\makebox(0,0)[lt]{\lineheight{0}\smash{\begin{tabular}[t]{l}$\scriptstyle a_{i_1}$\end{tabular}}}}%
    \put(0,0){\includegraphics[width=\unitlength,page=2]{Exemple_diagramme.pdf}}%
  \end{picture}%
\endgroup%

%% file: Diagramme_Dxi_Dxi_bar_et_concatenation.pdf_tex
\begingroup%
  \makeatletter%
  \providecommand\color[2][]{%
    \errmessage{(Inkscape) Color is used for the text in Inkscape, but the package 'color.sty' is not loaded}%
    \renewcommand\color[2][]{}%
  }%
  \providecommand\transparent[1]{%
    \errmessage{(Inkscape) Transparency is used (non-zero) for the text in Inkscape, but the package 'transparent.sty' is not loaded}%
    \renewcommand\transparent[1]{}%
  }%
  \providecommand\rotatebox[2]{#2}%
  \newcommand*\fsize{\dimexpr\f@size pt\relax}%
  \newcommand*\lineheight[1]{\fontsize{\fsize}{#1\fsize}\selectfont}%
  \ifx\svgwidth\undefined%
    \setlength{\unitlength}{481.88976378bp}%
    \ifx\svgscale\undefined%
      \relax%
    \else%
      \setlength{\unitlength}{\unitlength * \real{\svgscale}}%
    \fi%
  \else%
    \setlength{\unitlength}{\svgwidth}%
  \fi%
  \global\let\svgwidth\undefined%
  \global\let\svgscale\undefined%
  \makeatother%
  \begin{picture}(1,0.08567653)%
    \lineheight{1}%
    \setlength\tabcolsep{0pt}%
    \put(0,0){\includegraphics[width=\unitlength,page=1]{Diagramme_Dxi_Dxi_bar_et_concatenation.pdf}}%
    \put(0.44705884,0.04258823){\color[rgb]{0,0,0}\makebox(0,0)[lt]{\lineheight{0}\smash{\begin{tabular}[t]{l}$\overset{?}{D}_{x_{i_1}} \otimes \cdots \otimes \overset{?}{D}_{x_{i_k}} :=$ \end{tabular}}}}%
    \put(0.97352907,0.00144615){\color[rgb]{0,0,0}\makebox(0,0)[lt]{\lineheight{0}\smash{\begin{tabular}[t]{l}$\scriptstyle x_{i_k}$ \end{tabular}}}}%
    \put(0.86764675,0.00144615){\color[rgb]{0,0,0}\makebox(0,0)[lt]{\lineheight{0}\smash{\begin{tabular}[t]{l}$\scriptstyle x_{i_2}$ \end{tabular}}}}%
    \put(0.7617644,0.00144615){\color[rgb]{0,0,0}\makebox(0,0)[lt]{\lineheight{0}\smash{\begin{tabular}[t]{l}$\scriptstyle x_{i_1}$ \end{tabular}}}}%
    \put(0,0){\includegraphics[width=\unitlength,page=2]{Diagramme_Dxi_Dxi_bar_et_concatenation.pdf}}%
    \put(0.3970609,0.00141177){\color[rgb]{0,0,0}\makebox(0,0)[lt]{\lineheight{0}\smash{\begin{tabular}[t]{l}$\scriptstyle x_i$ \end{tabular}}}}%
    \put(0,0){\includegraphics[width=\unitlength,page=3]{Diagramme_Dxi_Dxi_bar_et_concatenation.pdf}}%
    \put(0.22352943,0.04258823){\color[rgb]{0,0,0}\makebox(0,0)[lt]{\lineheight{0}\smash{\begin{tabular}[t]{l}$\bar{D}_{x_i} :=$ \end{tabular}}}}%
    \put(0,0){\includegraphics[width=\unitlength,page=4]{Diagramme_Dxi_Dxi_bar_et_concatenation.pdf}}%
    \put(0.17353147,0.00141177){\color[rgb]{0,0,0}\makebox(0,0)[lt]{\lineheight{0}\smash{\begin{tabular}[t]{l}$\scriptstyle x_i$ \end{tabular}}}}%
    \put(0,0){\includegraphics[width=\unitlength,page=5]{Diagramme_Dxi_Dxi_bar_et_concatenation.pdf}}%
    \put(0,0.04258823){\color[rgb]{0,0,0}\makebox(0,0)[lt]{\lineheight{0}\smash{\begin{tabular}[t]{l}$D_{x_i} :=$ \end{tabular}}}}%
  \end{picture}%
\endgroup%

%% file: Diagramme_local_elementaire.pdf_tex
\begingroup%
  \makeatletter%
  \providecommand\color[2][]{%
    \errmessage{(Inkscape) Color is used for the text in Inkscape, but the package 'color.sty' is not loaded}%
    \renewcommand\color[2][]{}%
  }%
  \providecommand\transparent[1]{%
    \errmessage{(Inkscape) Transparency is used (non-zero) for the text in Inkscape, but the package 'transparent.sty' is not loaded}%
    \renewcommand\transparent[1]{}%
  }%
  \providecommand\rotatebox[2]{#2}%
  \newcommand*\fsize{\dimexpr\f@size pt\relax}%
  \newcommand*\lineheight[1]{\fontsize{\fsize}{#1\fsize}\selectfont}%
  \ifx\svgwidth\undefined%
    \setlength{\unitlength}{481.88976378bp}%
    \ifx\svgscale\undefined%
      \relax%
    \else%
      \setlength{\unitlength}{\unitlength * \real{\svgscale}}%
    \fi%
  \else%
    \setlength{\unitlength}{\svgwidth}%
  \fi%
  \global\let\svgwidth\undefined%
  \global\let\svgscale\undefined%
  \makeatother%
  \begin{picture}(1,0.08567794)%
    \lineheight{1}%
    \setlength\tabcolsep{0pt}%
    \put(0,0){\includegraphics[width=\unitlength,page=1]{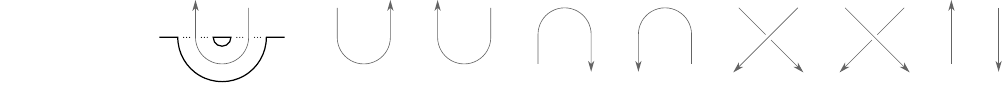}}%
    \put(0.25648953,0.00141318){\color[rgb]{0,0,0}\makebox(0,0)[lt]{\lineheight{0}\smash{\begin{tabular}[t]{l}\scriptsize $x_i$ \end{tabular}}}}%
    \put(0,0){\includegraphics[width=\unitlength,page=2]{Diagramme_local_elementaire.pdf}}%
    \put(0.097666,0.00141317){\color[rgb]{0,0,0}\makebox(0,0)[lt]{\lineheight{0}\smash{\begin{tabular}[t]{l}\scriptsize $x_i$ \end{tabular}}}}%
    \put(0,0){\includegraphics[width=\unitlength,page=3]{Diagramme_local_elementaire.pdf}}%
  \end{picture}%
\endgroup%

%% file: Regles_etiquetage_Hennings.pdf_tex
\begingroup%
  \makeatletter%
  \providecommand\color[2][]{%
    \errmessage{(Inkscape) Color is used for the text in Inkscape, but the package 'color.sty' is not loaded}%
    \renewcommand\color[2][]{}%
  }%
  \providecommand\transparent[1]{%
    \errmessage{(Inkscape) Transparency is used (non-zero) for the text in Inkscape, but the package 'transparent.sty' is not loaded}%
    \renewcommand\transparent[1]{}%
  }%
  \providecommand\rotatebox[2]{#2}%
  \newcommand*\fsize{\dimexpr\f@size pt\relax}%
  \newcommand*\lineheight[1]{\fontsize{\fsize}{#1\fsize}\selectfont}%
  \ifx\svgwidth\undefined%
    \setlength{\unitlength}{481.88976378bp}%
    \ifx\svgscale\undefined%
      \relax%
    \else%
      \setlength{\unitlength}{\unitlength * \real{\svgscale}}%
    \fi%
  \else%
    \setlength{\unitlength}{\svgwidth}%
  \fi%
  \global\let\svgwidth\undefined%
  \global\let\svgscale\undefined%
  \makeatother%
  \begin{picture}(1,0.09847492)%
    \lineheight{1}%
    \setlength\tabcolsep{0pt}%
    \put(0,0){\includegraphics[width=\unitlength,page=1]{Regles_etiquetage_Hennings.pdf}}%
    \put(0.99178531,0.00141317){\color[rgb]{0,0,0}\makebox(0,0)[lt]{\lineheight{0}\smash{\begin{tabular}[t]{l}\scriptsize $1$\end{tabular}}}}%
    \put(0.94472644,0.00141318){\color[rgb]{0,0,0}\makebox(0,0)[lt]{\lineheight{0}\smash{\begin{tabular}[t]{l}\scriptsize $1$\end{tabular}}}}%
    \put(0.88590138,0.00141317){\color[rgb]{0,0,0}\makebox(0,0)[lt]{\lineheight{0}\smash{\begin{tabular}[t]{l}\scriptsize $S(r_a)$\end{tabular}}}}%
    \put(0.89178503,0.08671021){\color[rgb]{0,0,0}\makebox(0,0)[lt]{\lineheight{0}\smash{\begin{tabular}[t]{l}\scriptsize $r^a$\end{tabular}}}}%
    \put(0.78884379,0.00141316){\color[rgb]{0,0,0}\makebox(0,0)[lt]{\lineheight{0}\smash{\begin{tabular}[t]{l}\scriptsize $r^a$\end{tabular}}}}%
    \put(0.78884316,0.08671021){\color[rgb]{0,0,0}\makebox(0,0)[lt]{\lineheight{0}\smash{\begin{tabular}[t]{l}\scriptsize $r_a$\end{tabular}}}}%
    \put(0.65942944,0.08671021){\color[rgb]{0,0,0}\makebox(0,0)[lt]{\lineheight{0}\smash{\begin{tabular}[t]{l}\scriptsize $g$\end{tabular}}}}%
    \put(0.55943231,0.08671021){\color[rgb]{0,0,0}\makebox(0,0)[lt]{\lineheight{0}\smash{\begin{tabular}[t]{l}\scriptsize $1$\end{tabular}}}}%
    \put(0.45649055,0.00435727){\color[rgb]{0,0,0}\makebox(0,0)[lt]{\lineheight{0}\smash{\begin{tabular}[t]{l}\scriptsize $g^{-1}$\end{tabular}}}}%
    \put(0.35647764,0.00435727){\color[rgb]{0,0,0}\makebox(0,0)[lt]{\lineheight{0}\smash{\begin{tabular}[t]{l}\scriptsize $1$\end{tabular}}}}%
    \put(0,0){\includegraphics[width=\unitlength,page=2]{Regles_etiquetage_Hennings.pdf}}%
    \put(0.20648824,0.08671021){\color[rgb]{0,0,0}\makebox(0,0)[lt]{\lineheight{0}\smash{\begin{tabular}[t]{l}\scriptsize ${g^{-1} S \bigl( \mathcal{X}_1(i)}_{(k)} \bigr)$ \end{tabular}}}}%
    \put(0,0){\includegraphics[width=\unitlength,page=3]{Regles_etiquetage_Hennings.pdf}}%
    \put(0.25648953,0.00435727){\color[rgb]{0,0,0}\makebox(0,0)[lt]{\lineheight{0}\smash{\begin{tabular}[t]{l}\scriptsize $x_i$ \end{tabular}}}}%
    \put(0,0){\includegraphics[width=\unitlength,page=4]{Regles_etiquetage_Hennings.pdf}}%
    \put(0.01237059,0.08671021){\color[rgb]{0,0,0}\makebox(0,0)[lt]{\lineheight{0}\smash{\begin{tabular}[t]{l}\scriptsize ${\mathcal{X}_1(i)}_{(k)}$ \end{tabular}}}}%
    \put(0.097666,0.00435728){\color[rgb]{0,0,0}\makebox(0,0)[lt]{\lineheight{0}\smash{\begin{tabular}[t]{l}\scriptsize $x_i$ \end{tabular}}}}%
    \put(0,0){\includegraphics[width=\unitlength,page=5]{Regles_etiquetage_Hennings.pdf}}%
  \end{picture}%
\endgroup%

%% file: Regles_etiquetage_consequences.pdf_tex
\begingroup%
  \makeatletter%
  \providecommand\color[2][]{%
    \errmessage{(Inkscape) Color is used for the text in Inkscape, but the package 'color.sty' is not loaded}%
    \renewcommand\color[2][]{}%
  }%
  \providecommand\transparent[1]{%
    \errmessage{(Inkscape) Transparency is used (non-zero) for the text in Inkscape, but the package 'transparent.sty' is not loaded}%
    \renewcommand\transparent[1]{}%
  }%
  \providecommand\rotatebox[2]{#2}%
  \newcommand*\fsize{\dimexpr\f@size pt\relax}%
  \newcommand*\lineheight[1]{\fontsize{\fsize}{#1\fsize}\selectfont}%
  \ifx\svgwidth\undefined%
    \setlength{\unitlength}{484.54560047bp}%
    \ifx\svgscale\undefined%
      \relax%
    \else%
      \setlength{\unitlength}{\unitlength * \real{\svgscale}}%
    \fi%
  \else%
    \setlength{\unitlength}{\svgwidth}%
  \fi%
  \global\let\svgwidth\undefined%
  \global\let\svgscale\undefined%
  \makeatother%
  \begin{picture}(1,0.91054693)%
    \lineheight{1}%
    \setlength\tabcolsep{0pt}%
    \put(0,0){\includegraphics[width=\unitlength,page=1]{Regles_etiquetage_consequences.pdf}}%
    \put(0.73413431,0.03115183){\color[rgb]{0,0,0}\makebox(0,0)[lt]{\lineheight{0}\smash{\begin{tabular}[t]{l}\scriptsize $S(r_a)gr^a = v$ \end{tabular}}}}%
    \put(0,0){\includegraphics[width=\unitlength,page=2]{Regles_etiquetage_consequences.pdf}}%
    \put(0.55863098,0.03115187){\color[rgb]{0,0,0}\makebox(0,0)[lt]{\lineheight{0}\smash{\begin{tabular}[t]{l}\scriptsize $r^ag^{-1}S(r_a) = v$ \end{tabular}}}}%
    \put(0,0){\includegraphics[width=\unitlength,page=3]{Regles_etiquetage_consequences.pdf}}%
    \put(0.38312761,0.03115186){\color[rgb]{0,0,0}\makebox(0,0)[lt]{\lineheight{0}\smash{\begin{tabular}[t]{l}\scriptsize $r^agr_a = v^{-1}$ \end{tabular}}}}%
    \put(0,0){\includegraphics[width=\unitlength,page=4]{Regles_etiquetage_consequences.pdf}}%
    \put(0.20762422,0.03115186){\color[rgb]{0,0,0}\makebox(0,0)[lt]{\lineheight{0}\smash{\begin{tabular}[t]{l}\scriptsize $r_ag^{-1}r^a = v^{-1 }$ \end{tabular}}}}%
    \put(0,0){\includegraphics[width=\unitlength,page=5]{Regles_etiquetage_consequences.pdf}}%
    \put(0.73413429,0.13060367){\color[rgb]{0,0,0}\makebox(0,0)[lt]{\lineheight{0}\smash{\begin{tabular}[t]{l}\scriptsize $r^ag^{-1}S(r_a) = v$ \end{tabular}}}}%
    \put(0,0){\includegraphics[width=\unitlength,page=6]{Regles_etiquetage_consequences.pdf}}%
    \put(0.55863092,0.13060372){\color[rgb]{0,0,0}\makebox(0,0)[lt]{\lineheight{0}\smash{\begin{tabular}[t]{l}\scriptsize $S(r_a)gr^a = v$ \end{tabular}}}}%
    \put(0,0){\includegraphics[width=\unitlength,page=7]{Regles_etiquetage_consequences.pdf}}%
    \put(0.38312761,0.13060367){\color[rgb]{0,0,0}\makebox(0,0)[lt]{\lineheight{0}\smash{\begin{tabular}[t]{l}\scriptsize $r_ag^{-1}r^a = v^{-1}$ \end{tabular}}}}%
    \put(0,0){\includegraphics[width=\unitlength,page=8]{Regles_etiquetage_consequences.pdf}}%
    \put(0.20746924,0.13122261){\color[rgb]{0,0,0}\makebox(0,0)[lt]{\lineheight{0}\smash{\begin{tabular}[t]{l}\scriptsize $r^agr_a = v^{-1}$ \end{tabular}}}}%
    \put(0,0){\includegraphics[width=\unitlength,page=9]{Regles_etiquetage_consequences.pdf}}%
    \put(0.4942797,0.21835534){\color[rgb]{0,0,0}\makebox(0,0)[lt]{\lineheight{0}\smash{\begin{tabular}[t]{l}\scriptsize $r_a $\end{tabular}}}}%
    \put(0.4942797,0.27685645){\color[rgb]{0,0,0}\makebox(0,0)[lt]{\lineheight{0}\smash{\begin{tabular}[t]{l}\scriptsize $r^a$\end{tabular}}}}%
    \put(0.38897768,0.2417558){\color[rgb]{0,0,0}\makebox(0,0)[lt]{\lineheight{0}\smash{\begin{tabular}[t]{l}$=$\end{tabular}}}}%
    \put(0.24857503,0.24175581){\color[rgb]{0,0,0}\makebox(0,0)[lt]{\lineheight{0}\smash{\begin{tabular}[t]{l}$\overset{\text{top.}}{=}$\end{tabular}}}}%
    \put(0,0){\includegraphics[width=\unitlength,page=10]{Regles_etiquetage_consequences.pdf}}%
    \put(0.49427971,0.34120768){\color[rgb]{0,0,0}\makebox(0,0)[lt]{\lineheight{0}\smash{\begin{tabular}[t]{l}\scriptsize $gr_ag^{-1} $\end{tabular}}}}%
    \put(0.49427971,0.39970879){\color[rgb]{0,0,0}\makebox(0,0)[lt]{\lineheight{0}\smash{\begin{tabular}[t]{l}\scriptsize $r^a$\end{tabular}}}}%
    \put(0.65808282,0.34120768){\color[rgb]{0,0,0}\makebox(0,0)[lt]{\lineheight{0}\smash{\begin{tabular}[t]{l}\scriptsize $S^2(r_a)$\end{tabular}}}}%
    \put(0.65808283,0.39970879){\color[rgb]{0,0,0}\makebox(0,0)[lt]{\lineheight{0}\smash{\begin{tabular}[t]{l}\scriptsize $r^a$\end{tabular}}}}%
    \put(0.55863094,0.36460813){\color[rgb]{0,0,0}\makebox(0,0)[lt]{\lineheight{0}\smash{\begin{tabular}[t]{l}$=$\end{tabular}}}}%
    \put(0.38897768,0.36460813){\color[rgb]{0,0,0}\makebox(0,0)[lt]{\lineheight{0}\smash{\begin{tabular}[t]{l}$=$\end{tabular}}}}%
    \put(0.24857503,0.36460815){\color[rgb]{0,0,0}\makebox(0,0)[lt]{\lineheight{0}\smash{\begin{tabular}[t]{l}$\overset{\text{top.}}{=}$\end{tabular}}}}%
    \put(0,0){\includegraphics[width=\unitlength,page=11]{Regles_etiquetage_consequences.pdf}}%
    \put(0.63468238,0.46406002){\color[rgb]{0,0,0}\makebox(0,0)[lt]{\lineheight{0}\smash{\begin{tabular}[t]{l}\scriptsize $gS(r_a)g^{-1} $\end{tabular}}}}%
    \put(0.63468238,0.52256113){\color[rgb]{0,0,0}\makebox(0,0)[lt]{\lineheight{0}\smash{\begin{tabular}[t]{l}\scriptsize $gr^ag^{-1}$\end{tabular}}}}%
    \put(0.81018571,0.46406002){\color[rgb]{0,0,0}\makebox(0,0)[lt]{\lineheight{0}\smash{\begin{tabular}[t]{l}\scriptsize $S(r_a)$\end{tabular}}}}%
    \put(0.81018574,0.52256113){\color[rgb]{0,0,0}\makebox(0,0)[lt]{\lineheight{0}\smash{\begin{tabular}[t]{l}\scriptsize $r^a$\end{tabular}}}}%
    \put(0.71073383,0.48746048){\color[rgb]{0,0,0}\makebox(0,0)[lt]{\lineheight{0}\smash{\begin{tabular}[t]{l}$=$\end{tabular}}}}%
    \put(0.52938036,0.48746048){\color[rgb]{0,0,0}\makebox(0,0)[lt]{\lineheight{0}\smash{\begin{tabular}[t]{l}$=$\end{tabular}}}}%
    \put(0.38897769,0.48746048){\color[rgb]{0,0,0}\makebox(0,0)[lt]{\lineheight{0}\smash{\begin{tabular}[t]{l}$\overset{\text{top.}}{=}$\end{tabular}}}}%
    \put(0.24857503,0.48746049){\color[rgb]{0,0,0}\makebox(0,0)[lt]{\lineheight{0}\smash{\begin{tabular}[t]{l}$\overset{\text{top.}}{=}$\end{tabular}}}}%
    \put(0,0){\includegraphics[width=\unitlength,page=12]{Regles_etiquetage_consequences.pdf}}%
    \put(0.4942797,0.58691236){\color[rgb]{0,0,0}\makebox(0,0)[lt]{\lineheight{0}\smash{\begin{tabular}[t]{l}\scriptsize $r^a $\end{tabular}}}}%
    \put(0.4942797,0.64541347){\color[rgb]{0,0,0}\makebox(0,0)[lt]{\lineheight{0}\smash{\begin{tabular}[t]{l}\scriptsize $S(r_a)$\end{tabular}}}}%
    \put(0.38897768,0.61031281){\color[rgb]{0,0,0}\makebox(0,0)[lt]{\lineheight{0}\smash{\begin{tabular}[t]{l}$=$\end{tabular}}}}%
    \put(0.24857503,0.61031283){\color[rgb]{0,0,0}\makebox(0,0)[lt]{\lineheight{0}\smash{\begin{tabular}[t]{l}$\overset{\text{top.}}{=}$\end{tabular}}}}%
    \put(0,0){\includegraphics[width=\unitlength,page=13]{Regles_etiquetage_consequences.pdf}}%
    \put(0.49427972,0.70976469){\color[rgb]{0,0,0}\makebox(0,0)[lt]{\lineheight{0}\smash{\begin{tabular}[t]{l}\scriptsize $gr^ag^{-1} $\end{tabular}}}}%
    \put(0.49427972,0.7682658){\color[rgb]{0,0,0}\makebox(0,0)[lt]{\lineheight{0}\smash{\begin{tabular}[t]{l}\scriptsize $S(r_a)$\end{tabular}}}}%
    \put(0.65808283,0.70976469){\color[rgb]{0,0,0}\makebox(0,0)[lt]{\lineheight{0}\smash{\begin{tabular}[t]{l}\scriptsize $S(r^a)$\end{tabular}}}}%
    \put(0.65808284,0.7682658){\color[rgb]{0,0,0}\makebox(0,0)[lt]{\lineheight{0}\smash{\begin{tabular}[t]{l}\scriptsize $r_a$\end{tabular}}}}%
    \put(0.55863095,0.73316514){\color[rgb]{0,0,0}\makebox(0,0)[lt]{\lineheight{0}\smash{\begin{tabular}[t]{l}$=$\end{tabular}}}}%
    \put(0.38897769,0.73316514){\color[rgb]{0,0,0}\makebox(0,0)[lt]{\lineheight{0}\smash{\begin{tabular}[t]{l}$=$\end{tabular}}}}%
    \put(0.24857504,0.73316516){\color[rgb]{0,0,0}\makebox(0,0)[lt]{\lineheight{0}\smash{\begin{tabular}[t]{l}$\overset{\text{top.}}{=}$\end{tabular}}}}%
    \put(0,0){\includegraphics[width=\unitlength,page=14]{Regles_etiquetage_consequences.pdf}}%
    \put(0.63468238,0.83261702){\color[rgb]{0,0,0}\makebox(0,0)[lt]{\lineheight{0}\smash{\begin{tabular}[t]{l}\scriptsize $gr^ag^{-1} $\end{tabular}}}}%
    \put(0.63468238,0.89111813){\color[rgb]{0,0,0}\makebox(0,0)[lt]{\lineheight{0}\smash{\begin{tabular}[t]{l}\scriptsize $gr_ag^{-1}$\end{tabular}}}}%
    \put(0.81018571,0.83261702){\color[rgb]{0,0,0}\makebox(0,0)[lt]{\lineheight{0}\smash{\begin{tabular}[t]{l}\scriptsize $r^a$\end{tabular}}}}%
    \put(0.81018573,0.89111813){\color[rgb]{0,0,0}\makebox(0,0)[lt]{\lineheight{0}\smash{\begin{tabular}[t]{l}\scriptsize $r_a$\end{tabular}}}}%
    \put(0.71073383,0.85601748){\color[rgb]{0,0,0}\makebox(0,0)[lt]{\lineheight{0}\smash{\begin{tabular}[t]{l}$=$\end{tabular}}}}%
    \put(0.52938036,0.85601748){\color[rgb]{0,0,0}\makebox(0,0)[lt]{\lineheight{0}\smash{\begin{tabular}[t]{l}$=$\end{tabular}}}}%
    \put(0.38897768,0.85601748){\color[rgb]{0,0,0}\makebox(0,0)[lt]{\lineheight{0}\smash{\begin{tabular}[t]{l}$\overset{\text{top.}}{=}$\end{tabular}}}}%
    \put(0.24857503,0.8560175){\color[rgb]{0,0,0}\makebox(0,0)[lt]{\lineheight{0}\smash{\begin{tabular}[t]{l}$\overset{\text{top.}}{=}$\end{tabular}}}}%
  \end{picture}%
\endgroup%

%% file: Exemple_diagramme_a_etats.pdf_tex
\begingroup%
  \makeatletter%
  \providecommand\color[2][]{%
    \errmessage{(Inkscape) Color is used for the text in Inkscape, but the package 'color.sty' is not loaded}%
    \renewcommand\color[2][]{}%
  }%
  \providecommand\transparent[1]{%
    \errmessage{(Inkscape) Transparency is used (non-zero) for the text in Inkscape, but the package 'transparent.sty' is not loaded}%
    \renewcommand\transparent[1]{}%
  }%
  \providecommand\rotatebox[2]{#2}%
  \newcommand*\fsize{\dimexpr\f@size pt\relax}%
  \newcommand*\lineheight[1]{\fontsize{\fsize}{#1\fsize}\selectfont}%
  \ifx\svgwidth\undefined%
    \setlength{\unitlength}{481.88976378bp}%
    \ifx\svgscale\undefined%
      \relax%
    \else%
      \setlength{\unitlength}{\unitlength * \real{\svgscale}}%
    \fi%
  \else%
    \setlength{\unitlength}{\svgwidth}%
  \fi%
  \global\let\svgwidth\undefined%
  \global\let\svgscale\undefined%
  \makeatother%
  \begin{picture}(1,0.16611765)%
    \lineheight{1}%
    \setlength\tabcolsep{0pt}%
    \put(0,0){\includegraphics[width=\unitlength,page=1]{Exemple_diagramme_a_etats.pdf}}%
    \put(0.64117789,0.00141176){\color[rgb]{0,0,0}\makebox(0,0)[lt]{\lineheight{0}\smash{\begin{tabular}[t]{l}$\scriptstyle x_{i_k}$ \end{tabular}}}}%
    \put(0.53529631,0.00141176){\color[rgb]{0,0,0}\makebox(0,0)[lt]{\lineheight{0}\smash{\begin{tabular}[t]{l}$\scriptstyle x_{i_2}$ \end{tabular}}}}%
    \put(0.42941395,0.00141176){\color[rgb]{0,0,0}\makebox(0,0)[lt]{\lineheight{0}\smash{\begin{tabular}[t]{l}$\scriptstyle x_{i_1}$ \end{tabular}}}}%
    \put(0.49501567,0.10141176){\color[rgb]{0,0,0}\makebox(0,0)[lt]{\lineheight{0}\smash{\begin{tabular}[t]{l}$\scriptstyle T$\end{tabular}}}}%
    \put(0.61176257,0.13670588){\makebox(0,0)[lt]{\lineheight{0}\smash{\begin{tabular}[t]{l}$\scriptstyle s_{i_{t}}$\end{tabular}}}}%
    \put(0.5058819,0.13670588){\makebox(0,0)[lt]{\lineheight{0}\smash{\begin{tabular}[t]{l}$\scriptstyle s_{i_2}$\end{tabular}}}}%
    \put(0.39999956,0.13670588){\makebox(0,0)[lt]{\lineheight{0}\smash{\begin{tabular}[t]{l}$\scriptstyle s_{i_1}$\end{tabular}}}}%
    \put(0.38823384,0.16023567){\makebox(0,0)[lt]{\lineheight{0}\smash{\begin{tabular}[t]{l}$\scriptstyle \varphi_1$\end{tabular}}}}%
    \put(0.4941162,0.16023567){\makebox(0,0)[lt]{\lineheight{0}\smash{\begin{tabular}[t]{l}$\scriptstyle \varphi_2$\end{tabular}}}}%
    \put(0.59999856,0.16023567){\makebox(0,0)[lt]{\lineheight{0}\smash{\begin{tabular}[t]{l}$\scriptstyle \varphi_t$\end{tabular}}}}%
    \put(0,0){\includegraphics[width=\unitlength,page=2]{Exemple_diagramme_a_etats.pdf}}%
  \end{picture}%
\endgroup%

%% file: Diagramme_a_etats_Dxi_phi.pdf_tex
\begingroup%
  \makeatletter%
  \providecommand\color[2][]{%
    \errmessage{(Inkscape) Color is used for the text in Inkscape, but the package 'color.sty' is not loaded}%
    \renewcommand\color[2][]{}%
  }%
  \providecommand\transparent[1]{%
    \errmessage{(Inkscape) Transparency is used (non-zero) for the text in Inkscape, but the package 'transparent.sty' is not loaded}%
    \renewcommand\transparent[1]{}%
  }%
  \providecommand\rotatebox[2]{#2}%
  \newcommand*\fsize{\dimexpr\f@size pt\relax}%
  \newcommand*\lineheight[1]{\fontsize{\fsize}{#1\fsize}\selectfont}%
  \ifx\svgwidth\undefined%
    \setlength{\unitlength}{481.88976378bp}%
    \ifx\svgscale\undefined%
      \relax%
    \else%
      \setlength{\unitlength}{\unitlength * \real{\svgscale}}%
    \fi%
  \else%
    \setlength{\unitlength}{\svgwidth}%
  \fi%
  \global\let\svgwidth\undefined%
  \global\let\svgscale\undefined%
  \makeatother%
  \begin{picture}(1,0.08964706)%
    \lineheight{1}%
    \setlength\tabcolsep{0pt}%
    \put(0,0){\includegraphics[width=\unitlength,page=1]{Diagramme_a_etats_Dxi_phi.pdf}}%
    \put(0.46764707,0.0837647){\color[rgb]{0,0,0}\makebox(0,0)[lt]{\lineheight{1.25}\smash{\begin{tabular}[t]{l}$\scriptstyle \varphi$\end{tabular}}}}%
    \put(0.53529329,0.00141177){\color[rgb]{0,0,0}\makebox(0,0)[lt]{\lineheight{0}\smash{\begin{tabular}[t]{l}$\scriptstyle x_i$ \end{tabular}}}}%
    \put(0,0){\includegraphics[width=\unitlength,page=2]{Diagramme_a_etats_Dxi_phi.pdf}}%
  \end{picture}%
\endgroup%

%% file: Diagramme_a_etats_co-produit.pdf_tex
\begingroup%
  \makeatletter%
  \providecommand\color[2][]{%
    \errmessage{(Inkscape) Color is used for the text in Inkscape, but the package 'color.sty' is not loaded}%
    \renewcommand\color[2][]{}%
  }%
  \providecommand\transparent[1]{%
    \errmessage{(Inkscape) Transparency is used (non-zero) for the text in Inkscape, but the package 'transparent.sty' is not loaded}%
    \renewcommand\transparent[1]{}%
  }%
  \providecommand\rotatebox[2]{#2}%
  \newcommand*\fsize{\dimexpr\f@size pt\relax}%
  \newcommand*\lineheight[1]{\fontsize{\fsize}{#1\fsize}\selectfont}%
  \ifx\svgwidth\undefined%
    \setlength{\unitlength}{481.88976378bp}%
    \ifx\svgscale\undefined%
      \relax%
    \else%
      \setlength{\unitlength}{\unitlength * \real{\svgscale}}%
    \fi%
  \else%
    \setlength{\unitlength}{\svgwidth}%
  \fi%
  \global\let\svgwidth\undefined%
  \global\let\svgscale\undefined%
  \makeatother%
  \begin{picture}(1,0.08823529)%
    \lineheight{1}%
    \setlength\tabcolsep{0pt}%
    \put(0,0){\includegraphics[width=\unitlength,page=1]{Diagramme_a_etats_co-produit.pdf}}%
    \put(0.38529417,0.0837647){\color[rgb]{0,0,0}\makebox(0,0)[lt]{\lineheight{1.25}\smash{\begin{tabular}[t]{l}$\scriptstyle \varphi$\end{tabular}}}}%
    \put(0.55882192,0.00141177){\color[rgb]{0,0,0}\makebox(0,0)[lt]{\lineheight{0}\smash{\begin{tabular}[t]{l}$\scriptstyle y$ \end{tabular}}}}%
    \put(0.45294031,0.00141177){\color[rgb]{0,0,0}\makebox(0,0)[lt]{\lineheight{0}\smash{\begin{tabular}[t]{l}$\scriptstyle x$ \end{tabular}}}}%
    \put(0,0){\includegraphics[width=\unitlength,page=2]{Diagramme_a_etats_co-produit.pdf}}%
  \end{picture}%
\endgroup%

%% file: Diagramme_a_etats_relations.pdf_tex
\begingroup%
  \makeatletter%
  \providecommand\color[2][]{%
    \errmessage{(Inkscape) Color is used for the text in Inkscape, but the package 'color.sty' is not loaded}%
    \renewcommand\color[2][]{}%
  }%
  \providecommand\transparent[1]{%
    \errmessage{(Inkscape) Transparency is used (non-zero) for the text in Inkscape, but the package 'transparent.sty' is not loaded}%
    \renewcommand\transparent[1]{}%
  }%
  \providecommand\rotatebox[2]{#2}%
  \newcommand*\fsize{\dimexpr\f@size pt\relax}%
  \newcommand*\lineheight[1]{\fontsize{\fsize}{#1\fsize}\selectfont}%
  \ifx\svgwidth\undefined%
    \setlength{\unitlength}{481.88976378bp}%
    \ifx\svgscale\undefined%
      \relax%
    \else%
      \setlength{\unitlength}{\unitlength * \real{\svgscale}}%
    \fi%
  \else%
    \setlength{\unitlength}{\svgwidth}%
  \fi%
  \global\let\svgwidth\undefined%
  \global\let\svgscale\undefined%
  \makeatother%
  \begin{picture}(1,0.42494256)%
    \lineheight{1}%
    \setlength\tabcolsep{0pt}%
    \put(0,0){\includegraphics[width=\unitlength,page=1]{Diagramme_a_etats_relations.pdf}}%
    \put(0.72234211,0.00141314){\color[rgb]{0,0,0}\makebox(0,0)[lt]{\lineheight{0}\smash{\begin{tabular}[t]{l}\scriptsize $y_j$\end{tabular}}}}%
    \put(0.61645975,0.00141315){\color[rgb]{0,0,0}\makebox(0,0)[lt]{\lineheight{0}\smash{\begin{tabular}[t]{l}\scriptsize $x_i$\end{tabular}}}}%
    \put(0.65469589,0.11317785){\color[rgb]{0,0,0}\makebox(0,0)[lt]{\lineheight{1.25}\smash{\begin{tabular}[t]{l}$\scriptstyle \varphi$\end{tabular}}}}%
    \put(0.54881353,0.11317785){\color[rgb]{0,0,0}\makebox(0,0)[lt]{\lineheight{1.25}\smash{\begin{tabular}[t]{l}$\scriptstyle \psi$\end{tabular}}}}%
    \put(0.48998821,0.04258961){\color[rgb]{0,0,0}\makebox(0,0)[lt]{\lineheight{0}\smash{\begin{tabular}[t]{l}$=$\end{tabular}}}}%
    \put(0,0){\includegraphics[width=\unitlength,page=2]{Diagramme_a_etats_relations.pdf}}%
    \put(0.45175377,0.00141315){\color[rgb]{0,0,0}\makebox(0,0)[lt]{\lineheight{0}\smash{\begin{tabular}[t]{l}\scriptsize $x_i$\end{tabular}}}}%
    \put(0.34587141,0.00141314){\color[rgb]{0,0,0}\makebox(0,0)[lt]{\lineheight{0}\smash{\begin{tabular}[t]{l}\scriptsize $y_j$\end{tabular}}}}%
    \put(0.38410764,0.11317784){\color[rgb]{0,0,0}\makebox(0,0)[lt]{\lineheight{1.25}\smash{\begin{tabular}[t]{l}$\scriptstyle \psi$\end{tabular}}}}%
    \put(0.27822529,0.11317785){\color[rgb]{0,0,0}\makebox(0,0)[lt]{\lineheight{1.25}\smash{\begin{tabular}[t]{l}$\scriptstyle \varphi$\end{tabular}}}}%
    \put(0,0){\includegraphics[width=\unitlength,page=3]{Diagramme_a_etats_relations.pdf}}%
    \put(0.72234248,0.15435432){\color[rgb]{0,0,0}\makebox(0,0)[lt]{\lineheight{0}\smash{\begin{tabular}[t]{l}\scriptsize $a_i$\end{tabular}}}}%
    \put(0.6164601,0.15435432){\color[rgb]{0,0,0}\makebox(0,0)[lt]{\lineheight{0}\smash{\begin{tabular}[t]{l}\scriptsize $b_i$\end{tabular}}}}%
    \put(0.65469577,0.26611903){\color[rgb]{0,0,0}\makebox(0,0)[lt]{\lineheight{1.25}\smash{\begin{tabular}[t]{l}$\scriptstyle \varphi$\end{tabular}}}}%
    \put(0.54881341,0.26611903){\color[rgb]{0,0,0}\makebox(0,0)[lt]{\lineheight{1.25}\smash{\begin{tabular}[t]{l}$\scriptstyle \psi$\end{tabular}}}}%
    \put(0.48998809,0.19553079){\color[rgb]{0,0,0}\makebox(0,0)[lt]{\lineheight{0}\smash{\begin{tabular}[t]{l}$=$\end{tabular}}}}%
    \put(0,0){\includegraphics[width=\unitlength,page=4]{Diagramme_a_etats_relations.pdf}}%
    \put(0.45175393,0.15435432){\color[rgb]{0,0,0}\makebox(0,0)[lt]{\lineheight{0}\smash{\begin{tabular}[t]{l}\scriptsize $b_i$\end{tabular}}}}%
    \put(0.34587157,0.15435432){\color[rgb]{0,0,0}\makebox(0,0)[lt]{\lineheight{0}\smash{\begin{tabular}[t]{l}\scriptsize $a_i$\end{tabular}}}}%
    \put(0.38410753,0.26611902){\color[rgb]{0,0,0}\makebox(0,0)[lt]{\lineheight{1.25}\smash{\begin{tabular}[t]{l}$\scriptstyle \psi$\end{tabular}}}}%
    \put(0.27822518,0.26611903){\color[rgb]{0,0,0}\makebox(0,0)[lt]{\lineheight{1.25}\smash{\begin{tabular}[t]{l}$\scriptstyle \varphi$\end{tabular}}}}%
    \put(0,0){\includegraphics[width=\unitlength,page=5]{Diagramme_a_etats_relations.pdf}}%
    \put(0.66940095,0.3072955){\color[rgb]{0,0,0}\makebox(0,0)[lt]{\lineheight{0}\smash{\begin{tabular}[t]{l}\scriptsize $x_i$\end{tabular}}}}%
    \put(0.64881395,0.4190602){\color[rgb]{0,0,0}\makebox(0,0)[lt]{\lineheight{1.25}\smash{\begin{tabular}[t]{l}$\scriptstyle \psi$\end{tabular}}}}%
    \put(0.58410809,0.4190602){\color[rgb]{0,0,0}\makebox(0,0)[lt]{\lineheight{1.25}\smash{\begin{tabular}[t]{l}$\scriptstyle \varphi$\end{tabular}}}}%
    \put(0.54293027,0.34847198){\color[rgb]{0,0,0}\makebox(0,0)[lt]{\lineheight{0}\smash{\begin{tabular}[t]{l}$=$\end{tabular}}}}%
    \put(0,0){\includegraphics[width=\unitlength,page=6]{Diagramme_a_etats_relations.pdf}}%
    \put(0.50469508,0.3072955){\color[rgb]{0,0,0}\makebox(0,0)[lt]{\lineheight{0}\smash{\begin{tabular}[t]{l}\scriptsize $x_i$\end{tabular}}}}%
    \put(0.39881271,0.30729549){\color[rgb]{0,0,0}\makebox(0,0)[lt]{\lineheight{0}\smash{\begin{tabular}[t]{l}\scriptsize $x_i$\end{tabular}}}}%
    \put(0.43704881,0.4190602){\color[rgb]{0,0,0}\makebox(0,0)[lt]{\lineheight{1.25}\smash{\begin{tabular}[t]{l}$\scriptstyle \psi$\end{tabular}}}}%
    \put(0.33116645,0.4190602){\color[rgb]{0,0,0}\makebox(0,0)[lt]{\lineheight{1.25}\smash{\begin{tabular}[t]{l}$\scriptstyle \varphi$\end{tabular}}}}%
  \end{picture}%
\endgroup%

%% file: Diagramme_a_etats_Dxi_inverse_phi.pdf_tex
\begingroup%
  \makeatletter%
  \providecommand\color[2][]{%
    \errmessage{(Inkscape) Color is used for the text in Inkscape, but the package 'color.sty' is not loaded}%
    \renewcommand\color[2][]{}%
  }%
  \providecommand\transparent[1]{%
    \errmessage{(Inkscape) Transparency is used (non-zero) for the text in Inkscape, but the package 'transparent.sty' is not loaded}%
    \renewcommand\transparent[1]{}%
  }%
  \providecommand\rotatebox[2]{#2}%
  \newcommand*\fsize{\dimexpr\f@size pt\relax}%
  \newcommand*\lineheight[1]{\fontsize{\fsize}{#1\fsize}\selectfont}%
  \ifx\svgwidth\undefined%
    \setlength{\unitlength}{481.88976378bp}%
    \ifx\svgscale\undefined%
      \relax%
    \else%
      \setlength{\unitlength}{\unitlength * \real{\svgscale}}%
    \fi%
  \else%
    \setlength{\unitlength}{\svgwidth}%
  \fi%
  \global\let\svgwidth\undefined%
  \global\let\svgscale\undefined%
  \makeatother%
  \begin{picture}(1,0.11905882)%
    \lineheight{1}%
    \setlength\tabcolsep{0pt}%
    \put(0,0){\includegraphics[width=\unitlength,page=1]{Diagramme_a_etats_Dxi_inverse_phi.pdf}}%
    \put(0.46636177,0.11317647){\color[rgb]{0,0,0}\makebox(0,0)[lt]{\lineheight{1.25}\smash{\begin{tabular}[t]{l}$\scriptstyle \varphi$\end{tabular}}}}%
    \put(0.534008,0.00141176){\color[rgb]{0,0,0}\makebox(0,0)[lt]{\lineheight{0}\smash{\begin{tabular}[t]{l}$\scriptstyle x_i$ \end{tabular}}}}%
    \put(0,0){\includegraphics[width=\unitlength,page=2]{Diagramme_a_etats_Dxi_inverse_phi.pdf}}%
  \end{picture}%
\endgroup%

%% file: Diagramme_a_etats_representation_de_A.pdf_tex
\begingroup%
  \makeatletter%
  \providecommand\color[2][]{%
    \errmessage{(Inkscape) Color is used for the text in Inkscape, but the package 'color.sty' is not loaded}%
    \renewcommand\color[2][]{}%
  }%
  \providecommand\transparent[1]{%
    \errmessage{(Inkscape) Transparency is used (non-zero) for the text in Inkscape, but the package 'transparent.sty' is not loaded}%
    \renewcommand\transparent[1]{}%
  }%
  \providecommand\rotatebox[2]{#2}%
  \newcommand*\fsize{\dimexpr\f@size pt\relax}%
  \newcommand*\lineheight[1]{\fontsize{\fsize}{#1\fsize}\selectfont}%
  \ifx\svgwidth\undefined%
    \setlength{\unitlength}{481.88976378bp}%
    \ifx\svgscale\undefined%
      \relax%
    \else%
      \setlength{\unitlength}{\unitlength * \real{\svgscale}}%
    \fi%
  \else%
    \setlength{\unitlength}{\svgwidth}%
  \fi%
  \global\let\svgwidth\undefined%
  \global\let\svgscale\undefined%
  \makeatother%
  \begin{picture}(1,0.08823529)%
    \lineheight{1}%
    \setlength\tabcolsep{0pt}%
    \put(0,0){\includegraphics[width=\unitlength,page=1]{Diagramme_a_etats_representation_de_A.pdf}}%
    \put(0.54982687,0.04411772){\color[rgb]{0,0,0}\makebox(0,0)[lt]{\lineheight{1.25}\smash{\begin{tabular}[t]{l}$\scriptstyle =$\end{tabular}}}}%
    \put(0.39472022,0.08235294){\color[rgb]{0,0,0}\makebox(0,0)[lt]{\lineheight{1.25}\smash{\begin{tabular}[t]{l}$\scriptstyle \alpha_j$\end{tabular}}}}%
    \put(0.57335641,0.08235288){\color[rgb]{0,0,0}\makebox(0,0)[lt]{\lineheight{1.25}\smash{\begin{tabular}[t]{l}$\scriptstyle \alpha_j$\end{tabular}}}}%
    \put(0,0){\includegraphics[width=\unitlength,page=2]{Diagramme_a_etats_representation_de_A.pdf}}%
    \put(0.46236747,0){\color[rgb]{0,0,0}\makebox(0,0)[lt]{\lineheight{0}\smash{\begin{tabular}[t]{l}$\scriptstyle a_j$ \end{tabular}}}}%
    \put(0.51003201,0.04411764){\color[rgb]{0,0,0}\makebox(0,0)[lt]{\lineheight{0}\smash{\begin{tabular}[t]{l}$1_\mathbb{K}$ \end{tabular}}}}%
    \put(0.64394452,0.04411763){\color[rgb]{0,0,0}\makebox(0,0)[lt]{\lineheight{0}\smash{\begin{tabular}[t]{l}$1_\mathbb{K}$ \end{tabular}}}}%
  \end{picture}%
\endgroup%

%% file: Diagramme_a_etats_representation_induite.pdf_tex
\begingroup%
  \makeatletter%
  \providecommand\color[2][]{%
    \errmessage{(Inkscape) Color is used for the text in Inkscape, but the package 'color.sty' is not loaded}%
    \renewcommand\color[2][]{}%
  }%
  \providecommand\transparent[1]{%
    \errmessage{(Inkscape) Transparency is used (non-zero) for the text in Inkscape, but the package 'transparent.sty' is not loaded}%
    \renewcommand\transparent[1]{}%
  }%
  \providecommand\rotatebox[2]{#2}%
  \newcommand*\fsize{\dimexpr\f@size pt\relax}%
  \newcommand*\lineheight[1]{\fontsize{\fsize}{#1\fsize}\selectfont}%
  \ifx\svgwidth\undefined%
    \setlength{\unitlength}{481.88980703bp}%
    \ifx\svgscale\undefined%
      \relax%
    \else%
      \setlength{\unitlength}{\unitlength * \real{\svgscale}}%
    \fi%
  \else%
    \setlength{\unitlength}{\svgwidth}%
  \fi%
  \global\let\svgwidth\undefined%
  \global\let\svgscale\undefined%
  \makeatother%
  \begin{picture}(1,0.47058819)%
    \lineheight{1}%
    \setlength\tabcolsep{0pt}%
    \put(0,0){\includegraphics[width=\unitlength,page=1]{Diagramme_a_etats_representation_induite.pdf}}%
    \put(0.40882353,0.05000001){\color[rgb]{0,0,0}\makebox(0,0)[lt]{\lineheight{1.25}\smash{\begin{tabular}[t]{l}$\scriptstyle =$\end{tabular}}}}%
    \put(0.86764876,0.12058821){\color[rgb]{0,0,0}\makebox(0,0)[lt]{\lineheight{1.25}\smash{\begin{tabular}[t]{l}$\scriptstyle \psi_{g+n}$\end{tabular}}}}%
    \put(0.54411938,0.12058818){\color[rgb]{0,0,0}\makebox(0,0)[lt]{\lineheight{1.25}\smash{\begin{tabular}[t]{l}$\scriptstyle \psi_{g+k-1}$\end{tabular}}}}%
    \put(0.65588408,0.1205855){\color[rgb]{0,0,0}\makebox(0,0)[lt]{\lineheight{1.25}\smash{\begin{tabular}[t]{l}$\scriptstyle \psi_{g+k}$\end{tabular}}}}%
    \put(0.75588406,0.12058821){\color[rgb]{0,0,0}\makebox(0,0)[lt]{\lineheight{1.25}\smash{\begin{tabular}[t]{l}$\scriptstyle \psi_{g+k+1}$\end{tabular}}}}%
    \put(0.46176465,0.12058816){\color[rgb]{0,0,0}\makebox(0,0)[lt]{\lineheight{1.25}\smash{\begin{tabular}[t]{l}$\scriptstyle \psi_{1}$\end{tabular}}}}%
    \put(0.42058822,0.12058816){\color[rgb]{0,0,0}\makebox(0,0)[lt]{\lineheight{1.25}\smash{\begin{tabular}[t]{l}$\scriptstyle \varphi_{k}$\end{tabular}}}}%
    \put(0.2852959,0.12058815){\color[rgb]{0,0,0}\makebox(0,0)[lt]{\lineheight{1.25}\smash{\begin{tabular}[t]{l}$\scriptstyle \psi_{g+n}$\end{tabular}}}}%
    \put(0.19705878,0.12058815){\color[rgb]{0,0,0}\makebox(0,0)[lt]{\lineheight{1.25}\smash{\begin{tabular}[t]{l}$\scriptstyle \psi_1$\end{tabular}}}}%
    \put(0.04999994,0.12058821){\color[rgb]{0,0,0}\makebox(0,0)[lt]{\lineheight{1.25}\smash{\begin{tabular}[t]{l}$\scriptstyle \varphi_{k}$\end{tabular}}}}%
    \put(0,0){\includegraphics[width=\unitlength,page=2]{Diagramme_a_etats_representation_induite.pdf}}%
    \put(0.9500003,0.00294255){\color[rgb]{0,0,0}\makebox(0,0)[lt]{\lineheight{0}\smash{\begin{tabular}[t]{l}\scriptsize $m_{g+n}$\end{tabular}}}}%
    \put(0.83235221,0.00294255){\color[rgb]{0,0,0}\makebox(0,0)[lt]{\lineheight{0}\smash{\begin{tabular}[t]{l}\scriptsize $m_{g+k+1}$\end{tabular}}}}%
    \put(0.73823561,0.00294255){\color[rgb]{0,0,0}\makebox(0,0)[lt]{\lineheight{0}\smash{\begin{tabular}[t]{l}\scriptsize $m_{g+k}$\end{tabular}}}}%
    \put(0.62058666,0.00294256){\color[rgb]{0,0,0}\makebox(0,0)[lt]{\lineheight{0}\smash{\begin{tabular}[t]{l}\scriptsize $m_{g+k-1}$\end{tabular}}}}%
    \put(0.52647089,0.00294255){\color[rgb]{0,0,0}\makebox(0,0)[lt]{\lineheight{0}\smash{\begin{tabular}[t]{l}\scriptsize $m_{1}$\end{tabular}}}}%
    \put(0.36764735,0.00294256){\color[rgb]{0,0,0}\makebox(0,0)[lt]{\lineheight{0}\smash{\begin{tabular}[t]{l}\scriptsize $m_{g+n}$\end{tabular}}}}%
    \put(0.26176499,0.00294257){\color[rgb]{0,0,0}\makebox(0,0)[lt]{\lineheight{0}\smash{\begin{tabular}[t]{l}\scriptsize $m_{1}$\end{tabular}}}}%
    \put(0.11470619,0.00294256){\color[rgb]{0,0,0}\makebox(0,0)[lt]{\lineheight{0}\smash{\begin{tabular}[t]{l}$\scriptstyle m_{g+k}$ \end{tabular}}}}%
    \put(0,0){\includegraphics[width=\unitlength,page=3]{Diagramme_a_etats_representation_induite.pdf}}%
    \put(0.40882353,0.22058823){\color[rgb]{0,0,0}\makebox(0,0)[lt]{\lineheight{1.25}\smash{\begin{tabular}[t]{l}$\scriptstyle =$\end{tabular}}}}%
    \put(0.86765061,0.29117636){\color[rgb]{0,0,0}\makebox(0,0)[lt]{\lineheight{1.25}\smash{\begin{tabular}[t]{l}$\scriptstyle \psi_{g+n}$\end{tabular}}}}%
    \put(0.55000362,0.29117633){\color[rgb]{0,0,0}\makebox(0,0)[lt]{\lineheight{1.25}\smash{\begin{tabular}[t]{l}$\scriptstyle \psi_{j-1}$\end{tabular}}}}%
    \put(0.67353125,0.29117365){\color[rgb]{0,0,0}\makebox(0,0)[lt]{\lineheight{1.25}\smash{\begin{tabular}[t]{l}$\scriptstyle \psi_{j}$\end{tabular}}}}%
    \put(0.7617683,0.29117644){\color[rgb]{0,0,0}\makebox(0,0)[lt]{\lineheight{1.25}\smash{\begin{tabular}[t]{l}$\scriptstyle \psi_{j+1}$\end{tabular}}}}%
    \put(0.4617665,0.29117631){\color[rgb]{0,0,0}\makebox(0,0)[lt]{\lineheight{1.25}\smash{\begin{tabular}[t]{l}$\scriptstyle \psi_{1}$\end{tabular}}}}%
    \put(0.42059007,0.29117631){\color[rgb]{0,0,0}\makebox(0,0)[lt]{\lineheight{1.25}\smash{\begin{tabular}[t]{l}$\scriptstyle \beta_j$\end{tabular}}}}%
    \put(0.28529775,0.2911763){\color[rgb]{0,0,0}\makebox(0,0)[lt]{\lineheight{1.25}\smash{\begin{tabular}[t]{l}$\scriptstyle \psi_{g+n}$\end{tabular}}}}%
    \put(0.19706063,0.2911763){\color[rgb]{0,0,0}\makebox(0,0)[lt]{\lineheight{1.25}\smash{\begin{tabular}[t]{l}$\scriptstyle \psi_1$\end{tabular}}}}%
    \put(0.05000179,0.29117636){\color[rgb]{0,0,0}\makebox(0,0)[lt]{\lineheight{1.25}\smash{\begin{tabular}[t]{l}$\scriptstyle \beta_j$\end{tabular}}}}%
    \put(0,0){\includegraphics[width=\unitlength,page=4]{Diagramme_a_etats_representation_induite.pdf}}%
    \put(0.94999998,0.17352939){\color[rgb]{0,0,0}\makebox(0,0)[lt]{\lineheight{0}\smash{\begin{tabular}[t]{l}\scriptsize $m_{g+n}$\end{tabular}}}}%
    \put(0.84411759,0.1735294){\color[rgb]{0,0,0}\makebox(0,0)[lt]{\lineheight{0}\smash{\begin{tabular}[t]{l}\scriptsize $m_{j+1}$\end{tabular}}}}%
    \put(0.7382353,0.1735294){\color[rgb]{0,0,0}\makebox(0,0)[lt]{\lineheight{0}\smash{\begin{tabular}[t]{l}\scriptsize $m_{j}$\end{tabular}}}}%
    \put(0.63235291,0.1735294){\color[rgb]{0,0,0}\makebox(0,0)[lt]{\lineheight{0}\smash{\begin{tabular}[t]{l}\scriptsize $m_{j-1}$\end{tabular}}}}%
    \put(0.52647057,0.1735294){\color[rgb]{0,0,0}\makebox(0,0)[lt]{\lineheight{0}\smash{\begin{tabular}[t]{l}\scriptsize $m_{1}$\end{tabular}}}}%
    \put(0.36764703,0.17352941){\color[rgb]{0,0,0}\makebox(0,0)[lt]{\lineheight{0}\smash{\begin{tabular}[t]{l}\scriptsize $m_{g+n}$\end{tabular}}}}%
    \put(0.26176467,0.17352942){\color[rgb]{0,0,0}\makebox(0,0)[lt]{\lineheight{0}\smash{\begin{tabular}[t]{l}\scriptsize $m_{1}$\end{tabular}}}}%
    \put(0.11470587,0.17352941){\color[rgb]{0,0,0}\makebox(0,0)[lt]{\lineheight{0}\smash{\begin{tabular}[t]{l}$\scriptstyle b_{j}$ \end{tabular}}}}%
    \put(0,0){\includegraphics[width=\unitlength,page=5]{Diagramme_a_etats_representation_induite.pdf}}%
    \put(0.40882353,0.39117652){\color[rgb]{0,0,0}\makebox(0,0)[lt]{\lineheight{1.25}\smash{\begin{tabular}[t]{l}$\scriptstyle =$\end{tabular}}}}%
    \put(0.04999996,0.46176467){\color[rgb]{0,0,0}\makebox(0,0)[lt]{\lineheight{1.25}\smash{\begin{tabular}[t]{l}$\scriptstyle \alpha_j$\end{tabular}}}}%
    \put(0.19705878,0.4617646){\color[rgb]{0,0,0}\makebox(0,0)[lt]{\lineheight{1.25}\smash{\begin{tabular}[t]{l}$\scriptstyle \psi_1$\end{tabular}}}}%
    \put(0.2852959,0.4617646){\color[rgb]{0,0,0}\makebox(0,0)[lt]{\lineheight{1.25}\smash{\begin{tabular}[t]{l}$\scriptstyle \psi_{g+n}$\end{tabular}}}}%
    \put(0.46176465,0.46176461){\color[rgb]{0,0,0}\makebox(0,0)[lt]{\lineheight{1.25}\smash{\begin{tabular}[t]{l}$\scriptstyle \psi_1$\end{tabular}}}}%
    \put(0.42058822,0.46176461){\color[rgb]{0,0,0}\makebox(0,0)[lt]{\lineheight{1.25}\smash{\begin{tabular}[t]{l}$\scriptstyle \alpha_j$\end{tabular}}}}%
    \put(0.55000177,0.46176461){\color[rgb]{0,0,0}\makebox(0,0)[lt]{\lineheight{1.25}\smash{\begin{tabular}[t]{l}$\scriptstyle \psi_{j-1}$\end{tabular}}}}%
    \put(0.6735294,0.46176467){\color[rgb]{0,0,0}\makebox(0,0)[lt]{\lineheight{1.25}\smash{\begin{tabular}[t]{l}$\scriptstyle \psi_{j}$\end{tabular}}}}%
    \put(0.76176645,0.46176467){\color[rgb]{0,0,0}\makebox(0,0)[lt]{\lineheight{1.25}\smash{\begin{tabular}[t]{l}$\scriptstyle \psi_{j+1}$\end{tabular}}}}%
    \put(0.86764876,0.46176467){\color[rgb]{0,0,0}\makebox(0,0)[lt]{\lineheight{1.25}\smash{\begin{tabular}[t]{l}$\scriptstyle \psi_{g+n}$\end{tabular}}}}%
    \put(0,0){\includegraphics[width=\unitlength,page=6]{Diagramme_a_etats_representation_induite.pdf}}%
    \put(0.94999997,0.34411762){\color[rgb]{0,0,0}\makebox(0,0)[lt]{\lineheight{0}\smash{\begin{tabular}[t]{l}\scriptsize $m_{g+n}$\end{tabular}}}}%
    \put(0.84411761,0.34411761){\color[rgb]{0,0,0}\makebox(0,0)[lt]{\lineheight{0}\smash{\begin{tabular}[t]{l}\scriptsize $m_{i+1}$\end{tabular}}}}%
    \put(0.73823526,0.34411761){\color[rgb]{0,0,0}\makebox(0,0)[lt]{\lineheight{0}\smash{\begin{tabular}[t]{l}\scriptsize $m_{j}$\end{tabular}}}}%
    \put(0.63235291,0.34411761){\color[rgb]{0,0,0}\makebox(0,0)[lt]{\lineheight{0}\smash{\begin{tabular}[t]{l}\scriptsize $m_{j-1}$\end{tabular}}}}%
    \put(0.52647057,0.34411761){\color[rgb]{0,0,0}\makebox(0,0)[lt]{\lineheight{0}\smash{\begin{tabular}[t]{l}\scriptsize $m_{1}$\end{tabular}}}}%
    \put(0.1176472,0.34411761){\color[rgb]{0,0,0}\makebox(0,0)[lt]{\lineheight{0}\smash{\begin{tabular}[t]{l}$\scriptstyle a_j$ \end{tabular}}}}%
    \put(0.36764703,0.34411761){\color[rgb]{0,0,0}\makebox(0,0)[lt]{\lineheight{0}\smash{\begin{tabular}[t]{l}\scriptsize $m_{g+n}$\end{tabular}}}}%
    \put(0.26176467,0.34411761){\color[rgb]{0,0,0}\makebox(0,0)[lt]{\lineheight{0}\smash{\begin{tabular}[t]{l}\scriptsize $m_{1}$\end{tabular}}}}%
  \end{picture}%
\endgroup%

%% file: Action_ai_1.pdf_tex
\begingroup%
  \makeatletter%
  \providecommand\color[2][]{%
    \errmessage{(Inkscape) Color is used for the text in Inkscape, but the package 'color.sty' is not loaded}%
    \renewcommand\color[2][]{}%
  }%
  \providecommand\transparent[1]{%
    \errmessage{(Inkscape) Transparency is used (non-zero) for the text in Inkscape, but the package 'transparent.sty' is not loaded}%
    \renewcommand\transparent[1]{}%
  }%
  \providecommand\rotatebox[2]{#2}%
  \newcommand*\fsize{\dimexpr\f@size pt\relax}%
  \newcommand*\lineheight[1]{\fontsize{\fsize}{#1\fsize}\selectfont}%
  \ifx\svgwidth\undefined%
    \setlength{\unitlength}{481.88976378bp}%
    \ifx\svgscale\undefined%
      \relax%
    \else%
      \setlength{\unitlength}{\unitlength * \real{\svgscale}}%
    \fi%
  \else%
    \setlength{\unitlength}{\svgwidth}%
  \fi%
  \global\let\svgwidth\undefined%
  \global\let\svgscale\undefined%
  \makeatother%
  \begin{picture}(1,0.51955881)%
    \lineheight{1}%
    \setlength\tabcolsep{0pt}%
    \put(0,0){\includegraphics[width=\unitlength,page=1]{Action_ai_1.pdf}}%
    \put(0.56035124,0.00588237){\color[rgb]{0,0,0}\makebox(0,0)[lt]{\lineheight{0}\smash{\begin{tabular}[t]{l}\scriptsize $m_{g+n}$\end{tabular}}}}%
    \put(0.45446887,0.00588235){\color[rgb]{0,0,0}\makebox(0,0)[lt]{\lineheight{0}\smash{\begin{tabular}[t]{l}\scriptsize $b_i$\end{tabular}}}}%
    \put(0.34858651,0.00588235){\color[rgb]{0,0,0}\makebox(0,0)[lt]{\lineheight{0}\smash{\begin{tabular}[t]{l}\scriptsize $b_2$\end{tabular}}}}%
    \put(0.24270414,0.00588237){\color[rgb]{0,0,0}\makebox(0,0)[lt]{\lineheight{0}\smash{\begin{tabular}[t]{l}\scriptsize $a_i$\end{tabular}}}}%
    \put(0.1368218,0.00588235){\color[rgb]{0,0,0}\makebox(0,0)[lt]{\lineheight{0}\smash{\begin{tabular}[t]{l}\scriptsize $b_1$\end{tabular}}}}%
    \put(0.59717303,0.04740116){\color[rgb]{0,0,0}\makebox(0,0)[lt]{\lineheight{0}\smash{\begin{tabular}[t]{l}$1_\mathbb{K}$\end{tabular}}}}%
    \put(0,0){\includegraphics[width=\unitlength,page=2]{Action_ai_1.pdf}}%
    \put(0.00446885,0.04705881){\color[rgb]{0,0,0}\makebox(0,0)[lt]{\lineheight{0}\smash{\begin{tabular}[t]{l}$\overset{\text{top.}}{=}$\end{tabular}}}}%
    \put(0,0){\includegraphics[width=\unitlength,page=3]{Action_ai_1.pdf}}%
    \put(0.56035124,0.15882354){\color[rgb]{0,0,0}\makebox(0,0)[lt]{\lineheight{0}\smash{\begin{tabular}[t]{l}\scriptsize $m_{g+n}$\end{tabular}}}}%
    \put(0.45446887,0.15882353){\color[rgb]{0,0,0}\makebox(0,0)[lt]{\lineheight{0}\smash{\begin{tabular}[t]{l}\scriptsize $b_i$\end{tabular}}}}%
    \put(0.34858652,0.15882353){\color[rgb]{0,0,0}\makebox(0,0)[lt]{\lineheight{0}\smash{\begin{tabular}[t]{l}\scriptsize $b_2$\end{tabular}}}}%
    \put(0.24270415,0.15882354){\color[rgb]{0,0,0}\makebox(0,0)[lt]{\lineheight{0}\smash{\begin{tabular}[t]{l}\scriptsize $a_i$\end{tabular}}}}%
    \put(0.1368218,0.15882353){\color[rgb]{0,0,0}\makebox(0,0)[lt]{\lineheight{0}\smash{\begin{tabular}[t]{l}\scriptsize $b_1$\end{tabular}}}}%
    \put(0.59717308,0.20034235){\color[rgb]{0,0,0}\makebox(0,0)[lt]{\lineheight{0}\smash{\begin{tabular}[t]{l}$1_\mathbb{K}$\end{tabular}}}}%
    \put(0,0){\includegraphics[width=\unitlength,page=4]{Action_ai_1.pdf}}%
    \put(-0.0014135,0.19999999){\color[rgb]{0,0,0}\makebox(0,0)[lt]{\lineheight{0}\smash{\begin{tabular}[t]{l}$\overset{\scriptstyle\labelcref{fig-diagrammes_a_etats_relations}}{=}$\end{tabular}}}}%
    \put(0,0){\includegraphics[width=\unitlength,page=5]{Action_ai_1.pdf}}%
    \put(0.6662361,0.31176469){\color[rgb]{0,0,0}\makebox(0,0)[lt]{\lineheight{0}\smash{\begin{tabular}[t]{l}\scriptsize $m_{g+n}$\end{tabular}}}}%
    \put(0.56035371,0.31176469){\color[rgb]{0,0,0}\makebox(0,0)[lt]{\lineheight{0}\smash{\begin{tabular}[t]{l}\scriptsize $m_{g+1}$\end{tabular}}}}%
    \put(0.45446887,0.31176469){\color[rgb]{0,0,0}\makebox(0,0)[lt]{\lineheight{0}\smash{\begin{tabular}[t]{l}\scriptsize $b_{g}$\end{tabular}}}}%
    \put(0.34858652,0.31176469){\color[rgb]{0,0,0}\makebox(0,0)[lt]{\lineheight{0}\smash{\begin{tabular}[t]{l}\scriptsize $b_i$\end{tabular}}}}%
    \put(0.24270414,0.31176469){\color[rgb]{0,0,0}\makebox(0,0)[lt]{\lineheight{0}\smash{\begin{tabular}[t]{l}\scriptsize $b_1$\end{tabular}}}}%
    \put(0.1368218,0.31176469){\color[rgb]{0,0,0}\makebox(0,0)[lt]{\lineheight{0}\smash{\begin{tabular}[t]{l}\scriptsize $a_i$\end{tabular}}}}%
    \put(0.70305539,0.35328349){\color[rgb]{0,0,0}\makebox(0,0)[lt]{\lineheight{0}\smash{\begin{tabular}[t]{l}$1_\mathbb{K}$\end{tabular}}}}%
    \put(0,0){\includegraphics[width=\unitlength,page=6]{Action_ai_1.pdf}}%
    \put(0.0103512,0.35294116){\color[rgb]{0,0,0}\makebox(0,0)[lt]{\lineheight{0}\smash{\begin{tabular}[t]{l}$=$\end{tabular}}}}%
    \put(0,0){\includegraphics[width=\unitlength,page=7]{Action_ai_1.pdf}}%
    \put(0.49564535,0.4352941){\color[rgb]{0,0,0}\makebox(0,0)[lt]{\lineheight{0}\smash{\begin{tabular}[t]{l}\scriptsize $m_{g+n}$\end{tabular}}}}%
    \put(0.38976295,0.4352941){\color[rgb]{0,0,0}\makebox(0,0)[lt]{\lineheight{0}\smash{\begin{tabular}[t]{l}\scriptsize $m_{i}$\end{tabular}}}}%
    \put(0.28388062,0.4352941){\color[rgb]{0,0,0}\makebox(0,0)[lt]{\lineheight{0}\smash{\begin{tabular}[t]{l}\scriptsize $m_{1}$\end{tabular}}}}%
    \put(0.1368218,0.4352941){\color[rgb]{0,0,0}\makebox(0,0)[lt]{\lineheight{0}\smash{\begin{tabular}[t]{l}\scriptsize $a_i$\end{tabular}}}}%
    \put(0,0){\includegraphics[width=\unitlength,page=8]{Action_ai_1.pdf}}%
  \end{picture}%
\endgroup%

%% file: Action_ai_2.pdf_tex
\begingroup%
  \makeatletter%
  \providecommand\color[2][]{%
    \errmessage{(Inkscape) Color is used for the text in Inkscape, but the package 'color.sty' is not loaded}%
    \renewcommand\color[2][]{}%
  }%
  \providecommand\transparent[1]{%
    \errmessage{(Inkscape) Transparency is used (non-zero) for the text in Inkscape, but the package 'transparent.sty' is not loaded}%
    \renewcommand\transparent[1]{}%
  }%
  \providecommand\rotatebox[2]{#2}%
  \newcommand*\fsize{\dimexpr\f@size pt\relax}%
  \newcommand*\lineheight[1]{\fontsize{\fsize}{#1\fsize}\selectfont}%
  \ifx\svgwidth\undefined%
    \setlength{\unitlength}{481.88976378bp}%
    \ifx\svgscale\undefined%
      \relax%
    \else%
      \setlength{\unitlength}{\unitlength * \real{\svgscale}}%
    \fi%
  \else%
    \setlength{\unitlength}{\svgwidth}%
  \fi%
  \global\let\svgwidth\undefined%
  \global\let\svgscale\undefined%
  \makeatother%
  \begin{picture}(1,1.04117638)%
    \lineheight{1}%
    \setlength\tabcolsep{0pt}%
    \put(0,0){\includegraphics[width=\unitlength,page=1]{Action_ai_2.pdf}}%
    \put(0.56035124,0.00743427){\color[rgb]{0,0,0}\makebox(0,0)[lt]{\lineheight{0}\smash{\begin{tabular}[t]{l}\scriptsize $m_{g+n}$\end{tabular}}}}%
    \put(0.45446888,0.00743429){\color[rgb]{0,0,0}\makebox(0,0)[lt]{\lineheight{0}\smash{\begin{tabular}[t]{l}\scriptsize $m_{i+1}$\end{tabular}}}}%
    \put(0.34862002,0.00743429){\color[rgb]{0,0,0}\makebox(0,0)[lt]{\lineheight{0}\smash{\begin{tabular}[t]{l}\scriptsize $m_{i}$\end{tabular}}}}%
    \put(0.24270414,0.00743429){\color[rgb]{0,0,0}\makebox(0,0)[lt]{\lineheight{0}\smash{\begin{tabular}[t]{l}\scriptsize $m_{i-1}$\end{tabular}}}}%
    \put(0.13682179,0.00743429){\color[rgb]{0,0,0}\makebox(0,0)[lt]{\lineheight{0}\smash{\begin{tabular}[t]{l}\scriptsize $m_{1}$\end{tabular}}}}%
    \put(0.0103512,0.04861073){\color[rgb]{0,0,0}\makebox(0,0)[lt]{\lineheight{0}\smash{\begin{tabular}[t]{l}$=$\end{tabular}}}}%
    \put(0,0){\includegraphics[width=\unitlength,page=2]{Action_ai_2.pdf}}%
    \put(0.56035125,0.15882356){\color[rgb]{0,0,0}\makebox(0,0)[lt]{\lineheight{0}\smash{\begin{tabular}[t]{l}\scriptsize $m_{g+n}$\end{tabular}}}}%
    \put(0.45446888,0.15882355){\color[rgb]{0,0,0}\makebox(0,0)[lt]{\lineheight{0}\smash{\begin{tabular}[t]{l}\scriptsize $b_{i+1}$\end{tabular}}}}%
    \put(0.34858652,0.15882355){\color[rgb]{0,0,0}\makebox(0,0)[lt]{\lineheight{0}\smash{\begin{tabular}[t]{l}\scriptsize $b_i$\end{tabular}}}}%
    \put(0.24270415,0.15882356){\color[rgb]{0,0,0}\makebox(0,0)[lt]{\lineheight{0}\smash{\begin{tabular}[t]{l}\scriptsize $b_{i-1}$\end{tabular}}}}%
    \put(0.13682181,0.15882355){\color[rgb]{0,0,0}\makebox(0,0)[lt]{\lineheight{0}\smash{\begin{tabular}[t]{l}\scriptsize $b_1$\end{tabular}}}}%
    \put(0.59717296,0.20034237){\color[rgb]{0,0,0}\makebox(0,0)[lt]{\lineheight{0}\smash{\begin{tabular}[t]{l}$1_\mathbb{K}$\end{tabular}}}}%
    \put(0,0){\includegraphics[width=\unitlength,page=3]{Action_ai_2.pdf}}%
    \put(0.00446886,0.19999999){\color[rgb]{0,0,0}\makebox(0,0)[lt]{\lineheight{0}\smash{\begin{tabular}[t]{l}$\overset{\text{top.}}{=}$\end{tabular}}}}%
    \put(0,0){\includegraphics[width=\unitlength,page=4]{Action_ai_2.pdf}}%
    \put(0.56035124,0.31176471){\color[rgb]{0,0,0}\makebox(0,0)[lt]{\lineheight{0}\smash{\begin{tabular}[t]{l}\scriptsize $m_{g+n}$\end{tabular}}}}%
    \put(0.45446887,0.31176469){\color[rgb]{0,0,0}\makebox(0,0)[lt]{\lineheight{0}\smash{\begin{tabular}[t]{l}\scriptsize $b_{i+1}$\end{tabular}}}}%
    \put(0.34858652,0.31176469){\color[rgb]{0,0,0}\makebox(0,0)[lt]{\lineheight{0}\smash{\begin{tabular}[t]{l}\scriptsize $b_i$\end{tabular}}}}%
    \put(0.24270415,0.31176471){\color[rgb]{0,0,0}\makebox(0,0)[lt]{\lineheight{0}\smash{\begin{tabular}[t]{l}\scriptsize $b_{i-1}$\end{tabular}}}}%
    \put(0.1368218,0.31176469){\color[rgb]{0,0,0}\makebox(0,0)[lt]{\lineheight{0}\smash{\begin{tabular}[t]{l}\scriptsize $b_1$\end{tabular}}}}%
    \put(0.66776118,0.35328351){\color[rgb]{0,0,0}\makebox(0,0)[lt]{\lineheight{0}\smash{\begin{tabular}[t]{l}$1_\mathbb{K}$\end{tabular}}}}%
    \put(0,0){\includegraphics[width=\unitlength,page=5]{Action_ai_2.pdf}}%
    \put(-0.0014135,0.35294119){\color[rgb]{0,0,0}\makebox(0,0)[lt]{\lineheight{0}\smash{\begin{tabular}[t]{l}$\overset{\scriptstyle\labelcref{fig-diagramme_a_etats_representation_de_A}}{=}$\end{tabular}}}}%
    \put(0,0){\includegraphics[width=\unitlength,page=6]{Action_ai_2.pdf}}%
    \put(0.66623363,0.4647058){\color[rgb]{0,0,0}\makebox(0,0)[lt]{\lineheight{0}\smash{\begin{tabular}[t]{l}\scriptsize $a_i$\end{tabular}}}}%
    \put(0.56035124,0.4647058){\color[rgb]{0,0,0}\makebox(0,0)[lt]{\lineheight{0}\smash{\begin{tabular}[t]{l}\scriptsize $m_{g+n}$\end{tabular}}}}%
    \put(0.45446886,0.46470579){\color[rgb]{0,0,0}\makebox(0,0)[lt]{\lineheight{0}\smash{\begin{tabular}[t]{l}\scriptsize $b_{i+1}$\end{tabular}}}}%
    \put(0.34858651,0.46470579){\color[rgb]{0,0,0}\makebox(0,0)[lt]{\lineheight{0}\smash{\begin{tabular}[t]{l}\scriptsize $b_i$\end{tabular}}}}%
    \put(0.24270414,0.4647058){\color[rgb]{0,0,0}\makebox(0,0)[lt]{\lineheight{0}\smash{\begin{tabular}[t]{l}\scriptsize $b_{i-1}$\end{tabular}}}}%
    \put(0.1368218,0.46470579){\color[rgb]{0,0,0}\makebox(0,0)[lt]{\lineheight{0}\smash{\begin{tabular}[t]{l}\scriptsize $b_1$\end{tabular}}}}%
    \put(0.7030553,0.5062246){\color[rgb]{0,0,0}\makebox(0,0)[lt]{\lineheight{0}\smash{\begin{tabular}[t]{l}$1_\mathbb{K}$\end{tabular}}}}%
    \put(0,0){\includegraphics[width=\unitlength,page=7]{Action_ai_2.pdf}}%
    \put(0.0103512,0.50588228){\color[rgb]{0,0,0}\makebox(0,0)[lt]{\lineheight{0}\smash{\begin{tabular}[t]{l}$=$\end{tabular}}}}%
    \put(0,0){\includegraphics[width=\unitlength,page=8]{Action_ai_2.pdf}}%
    \put(0.77211599,0.61764699){\color[rgb]{0,0,0}\makebox(0,0)[lt]{\lineheight{0}\smash{\begin{tabular}[t]{l}\scriptsize $m_{g+n}$\end{tabular}}}}%
    \put(0.66623363,0.61764699){\color[rgb]{0,0,0}\makebox(0,0)[lt]{\lineheight{0}\smash{\begin{tabular}[t]{l}\scriptsize $b_{i+2}$\end{tabular}}}}%
    \put(0.45446887,0.61764698){\color[rgb]{0,0,0}\makebox(0,0)[lt]{\lineheight{0}\smash{\begin{tabular}[t]{l}\scriptsize $b_{i+1}$\end{tabular}}}}%
    \put(0.56035124,0.61764699){\color[rgb]{0,0,0}\makebox(0,0)[lt]{\lineheight{0}\smash{\begin{tabular}[t]{l}\scriptsize $a_i$\end{tabular}}}}%
    \put(0.34858652,0.61764698){\color[rgb]{0,0,0}\makebox(0,0)[lt]{\lineheight{0}\smash{\begin{tabular}[t]{l}\scriptsize $b_i$\end{tabular}}}}%
    \put(0.24270414,0.61764699){\color[rgb]{0,0,0}\makebox(0,0)[lt]{\lineheight{0}\smash{\begin{tabular}[t]{l}\scriptsize $b_{i-1}$\end{tabular}}}}%
    \put(0.1368218,0.61764698){\color[rgb]{0,0,0}\makebox(0,0)[lt]{\lineheight{0}\smash{\begin{tabular}[t]{l}\scriptsize $b_1$\end{tabular}}}}%
    \put(0.8089377,0.65916579){\color[rgb]{0,0,0}\makebox(0,0)[lt]{\lineheight{0}\smash{\begin{tabular}[t]{l}$1_\mathbb{K}$\end{tabular}}}}%
    \put(0,0){\includegraphics[width=\unitlength,page=9]{Action_ai_2.pdf}}%
    \put(0.0103512,0.65882346){\color[rgb]{0,0,0}\makebox(0,0)[lt]{\lineheight{0}\smash{\begin{tabular}[t]{l}$=$\end{tabular}}}}%
    \put(0,0){\includegraphics[width=\unitlength,page=10]{Action_ai_2.pdf}}%
    \put(0.66623364,0.77058816){\color[rgb]{0,0,0}\makebox(0,0)[lt]{\lineheight{0}\smash{\begin{tabular}[t]{l}\scriptsize $m_{g+n}$\end{tabular}}}}%
    \put(0.56035124,0.77058816){\color[rgb]{0,0,0}\makebox(0,0)[lt]{\lineheight{0}\smash{\begin{tabular}[t]{l}\scriptsize $b_{i+1}$\end{tabular}}}}%
    \put(0.45446887,0.77058815){\color[rgb]{0,0,0}\makebox(0,0)[lt]{\lineheight{0}\smash{\begin{tabular}[t]{l}\scriptsize $a_i$\end{tabular}}}}%
    \put(0.34858652,0.77058815){\color[rgb]{0,0,0}\makebox(0,0)[lt]{\lineheight{0}\smash{\begin{tabular}[t]{l}\scriptsize $b_i$\end{tabular}}}}%
    \put(0.24270415,0.77058816){\color[rgb]{0,0,0}\makebox(0,0)[lt]{\lineheight{0}\smash{\begin{tabular}[t]{l}\scriptsize $b_{i-1}$\end{tabular}}}}%
    \put(0.1368218,0.77058815){\color[rgb]{0,0,0}\makebox(0,0)[lt]{\lineheight{0}\smash{\begin{tabular}[t]{l}\scriptsize $b_1$\end{tabular}}}}%
    \put(0.7030553,0.81210696){\color[rgb]{0,0,0}\makebox(0,0)[lt]{\lineheight{0}\smash{\begin{tabular}[t]{l}$1_\mathbb{K}$\end{tabular}}}}%
    \put(0,0){\includegraphics[width=\unitlength,page=11]{Action_ai_2.pdf}}%
    \put(-0.0014135,0.81176463){\color[rgb]{0,0,0}\makebox(0,0)[lt]{\lineheight{0}\smash{\begin{tabular}[t]{l}$\overset{\scriptstyle\labelcref{fig-diagrammes_a_etats_relations}}{=}$\end{tabular}}}}%
    \put(0,0){\includegraphics[width=\unitlength,page=12]{Action_ai_2.pdf}}%
    \put(0.56035124,0.92352934){\color[rgb]{0,0,0}\makebox(0,0)[lt]{\lineheight{0}\smash{\begin{tabular}[t]{l}\scriptsize $m_{g+n}$\end{tabular}}}}%
    \put(0.45446887,0.92352933){\color[rgb]{0,0,0}\makebox(0,0)[lt]{\lineheight{0}\smash{\begin{tabular}[t]{l}\scriptsize $b_i$\end{tabular}}}}%
    \put(0.34858652,0.92352933){\color[rgb]{0,0,0}\makebox(0,0)[lt]{\lineheight{0}\smash{\begin{tabular}[t]{l}\scriptsize $a_i$\end{tabular}}}}%
    \put(0.24270415,0.92352934){\color[rgb]{0,0,0}\makebox(0,0)[lt]{\lineheight{0}\smash{\begin{tabular}[t]{l}\scriptsize $b_{i-1}$\end{tabular}}}}%
    \put(0.1368218,0.92352933){\color[rgb]{0,0,0}\makebox(0,0)[lt]{\lineheight{0}\smash{\begin{tabular}[t]{l}\scriptsize $b_1$\end{tabular}}}}%
    \put(0.59717304,0.96504814){\color[rgb]{0,0,0}\makebox(0,0)[lt]{\lineheight{0}\smash{\begin{tabular}[t]{l}$1_\mathbb{K}$\end{tabular}}}}%
    \put(0,0){\includegraphics[width=\unitlength,page=13]{Action_ai_2.pdf}}%
    \put(-0.0014135,0.9647058){\color[rgb]{0,0,0}\makebox(0,0)[lt]{\lineheight{0}\smash{\begin{tabular}[t]{l}$\overset{\scriptstyle\labelcref{fig-diagrammes_a_etats_relations}}{=}$\end{tabular}}}}%
  \end{picture}%
\endgroup%

%% file: Courbe_ci_surface.pdf_tex
\begingroup%
  \makeatletter%
  \providecommand\color[2][]{%
    \errmessage{(Inkscape) Color is used for the text in Inkscape, but the package 'color.sty' is not loaded}%
    \renewcommand\color[2][]{}%
  }%
  \providecommand\transparent[1]{%
    \errmessage{(Inkscape) Transparency is used (non-zero) for the text in Inkscape, but the package 'transparent.sty' is not loaded}%
    \renewcommand\transparent[1]{}%
  }%
  \providecommand\rotatebox[2]{#2}%
  \newcommand*\fsize{\dimexpr\f@size pt\relax}%
  \newcommand*\lineheight[1]{\fontsize{\fsize}{#1\fsize}\selectfont}%
  \ifx\svgwidth\undefined%
    \setlength{\unitlength}{453.54313407bp}%
    \ifx\svgscale\undefined%
      \relax%
    \else%
      \setlength{\unitlength}{\unitlength * \real{\svgscale}}%
    \fi%
  \else%
    \setlength{\unitlength}{\svgwidth}%
  \fi%
  \global\let\svgwidth\undefined%
  \global\let\svgscale\undefined%
  \makeatother%
  \begin{picture}(1,0.17628758)%
    \lineheight{1}%
    \setlength\tabcolsep{0pt}%
    \put(0,0){\includegraphics[width=\unitlength,page=1]{Courbe_ci_surface.pdf}}%
    \put(0.6068936,0.14439381){\color[rgb]{0,0,0}\makebox(0,0)[lt]{\lineheight{1.25}\smash{\begin{tabular}[t]{l}$\scriptstyle c_i$\end{tabular}}}}%
    \put(0.29439349,0.05064377){\color[rgb]{0,0,0}\makebox(0,0)[lt]{\lineheight{1.25}\smash{\begin{tabular}[t]{l}$\scriptstyle i$\end{tabular}}}}%
    \put(0,0){\includegraphics[width=\unitlength,page=2]{Courbe_ci_surface.pdf}}%
  \end{picture}%
\endgroup%

%% file: Courbe_ci_ruban.pdf_tex
\begingroup%
  \makeatletter%
  \providecommand\color[2][]{%
    \errmessage{(Inkscape) Color is used for the text in Inkscape, but the package 'color.sty' is not loaded}%
    \renewcommand\color[2][]{}%
  }%
  \providecommand\transparent[1]{%
    \errmessage{(Inkscape) Transparency is used (non-zero) for the text in Inkscape, but the package 'transparent.sty' is not loaded}%
    \renewcommand\transparent[1]{}%
  }%
  \providecommand\rotatebox[2]{#2}%
  \newcommand*\fsize{\dimexpr\f@size pt\relax}%
  \newcommand*\lineheight[1]{\fontsize{\fsize}{#1\fsize}\selectfont}%
  \ifx\svgwidth\undefined%
    \setlength{\unitlength}{453.54330709bp}%
    \ifx\svgscale\undefined%
      \relax%
    \else%
      \setlength{\unitlength}{\unitlength * \real{\svgscale}}%
    \fi%
  \else%
    \setlength{\unitlength}{\svgwidth}%
  \fi%
  \global\let\svgwidth\undefined%
  \global\let\svgscale\undefined%
  \makeatother%
  \begin{picture}(1,0.21291177)%
    \lineheight{1}%
    \setlength\tabcolsep{0pt}%
    \put(0,0){\includegraphics[width=\unitlength,page=1]{Courbe_ci_ruban.pdf}}%
    \put(0.07251875,0.00409859){\color[rgb]{0,0,0}\makebox(0,0)[lt]{\lineheight{0}\smash{\begin{tabular}[t]{l}\scriptsize $b_1$\end{tabular}}}}%
    \put(0.12251875,0.00409859){\color[rgb]{0,0,0}\makebox(0,0)[lt]{\lineheight{0}\smash{\begin{tabular}[t]{l}\scriptsize $a_1$\end{tabular}}}}%
    \put(0.27251876,0.00409859){\color[rgb]{0,0,0}\makebox(0,0)[lt]{\lineheight{0}\smash{\begin{tabular}[t]{l}\scriptsize $b_i$\end{tabular}}}}%
    \put(0.32251876,0.00409859){\color[rgb]{0,0,0}\makebox(0,0)[lt]{\lineheight{0}\smash{\begin{tabular}[t]{l}\scriptsize $a_i$\end{tabular}}}}%
    \put(0.47251877,0.00409859){\color[rgb]{0,0,0}\makebox(0,0)[lt]{\lineheight{0}\smash{\begin{tabular}[t]{l}\scriptsize $b_{i+1}$\end{tabular}}}}%
    \put(0.52251875,0.00409859){\color[rgb]{0,0,0}\makebox(0,0)[lt]{\lineheight{0}\smash{\begin{tabular}[t]{l}\scriptsize $a_{i+1}$\end{tabular}}}}%
    \put(0.67251877,0.00409859){\color[rgb]{0,0,0}\makebox(0,0)[lt]{\lineheight{0}\smash{\begin{tabular}[t]{l}\scriptsize $b_g$\end{tabular}}}}%
    \put(0.72251866,0.00409859){\color[rgb]{0,0,0}\makebox(0,0)[lt]{\lineheight{0}\smash{\begin{tabular}[t]{l}\scriptsize $a_g$\end{tabular}}}}%
    \put(0.83814336,0.03222361){\color[rgb]{0,0,0}\makebox(0,0)[lt]{\lineheight{0}\smash{\begin{tabular}[t]{l}\scriptsize $m_{g+1}$\end{tabular}}}}%
    \put(0.93189332,0.03222361){\color[rgb]{0,0,0}\makebox(0,0)[lt]{\lineheight{0}\smash{\begin{tabular}[t]{l}\scriptsize $m_{g+n}$\end{tabular}}}}%
    \put(0.50064378,0.1790986){\color[rgb]{0,0,0}\makebox(0,0)[lt]{\lineheight{1.25}\smash{\begin{tabular}[t]{l}\scriptsize $c_i$\end{tabular}}}}%
  \end{picture}%
\endgroup%

%% file: Diagramme_ci.pdf_tex
\begingroup%
  \makeatletter%
  \providecommand\color[2][]{%
    \errmessage{(Inkscape) Color is used for the text in Inkscape, but the package 'color.sty' is not loaded}%
    \renewcommand\color[2][]{}%
  }%
  \providecommand\transparent[1]{%
    \errmessage{(Inkscape) Transparency is used (non-zero) for the text in Inkscape, but the package 'transparent.sty' is not loaded}%
    \renewcommand\transparent[1]{}%
  }%
  \providecommand\rotatebox[2]{#2}%
  \newcommand*\fsize{\dimexpr\f@size pt\relax}%
  \newcommand*\lineheight[1]{\fontsize{\fsize}{#1\fsize}\selectfont}%
  \ifx\svgwidth\undefined%
    \setlength{\unitlength}{300.13833114bp}%
    \ifx\svgscale\undefined%
      \relax%
    \else%
      \setlength{\unitlength}{\unitlength * \real{\svgscale}}%
    \fi%
  \else%
    \setlength{\unitlength}{\svgwidth}%
  \fi%
  \global\let\svgwidth\undefined%
  \global\let\svgscale\undefined%
  \makeatother%
  \begin{picture}(1,0.46240614)%
    \lineheight{1}%
    \setlength\tabcolsep{0pt}%
    \put(0,0){\includegraphics[width=\unitlength,page=1]{Diagramme_ci.pdf}}%
    \put(0.92653162,0.00603238){\color[rgb]{0,0,0}\makebox(0,0)[lt]{\lineheight{0}\smash{\begin{tabular}[t]{l}\scriptsize $b_{i+1}$\end{tabular}}}}%
    \put(0.7565314,0.00603238){\color[rgb]{0,0,0}\makebox(0,0)[lt]{\lineheight{0}\smash{\begin{tabular}[t]{l}\scriptsize $a_{i+1}$\end{tabular}}}}%
    \put(0.58653104,0.00603238){\color[rgb]{0,0,0}\makebox(0,0)[lt]{\lineheight{0}\smash{\begin{tabular}[t]{l}\scriptsize $b_{i+1}$\end{tabular}}}}%
    \put(0.4165306,0.00603238){\color[rgb]{0,0,0}\makebox(0,0)[lt]{\lineheight{0}\smash{\begin{tabular}[t]{l}\scriptsize $a_i$\end{tabular}}}}%
    \put(0.20403017,0.07211364){\color[rgb]{0,0,0}\makebox(0,0)[lt]{\lineheight{0}\smash{\begin{tabular}[t]{l}$\overset{\text{top.}}{=}$\end{tabular}}}}%
    \put(0,0){\includegraphics[width=\unitlength,page=2]{Diagramme_ci.pdf}}%
    \put(0.9265316,0.1665583){\color[rgb]{0,0,0}\makebox(0,0)[lt]{\lineheight{0}\smash{\begin{tabular}[t]{l}\scriptsize $b_{i+1}$\end{tabular}}}}%
    \put(0.75653137,0.1665583){\color[rgb]{0,0,0}\makebox(0,0)[lt]{\lineheight{0}\smash{\begin{tabular}[t]{l}\scriptsize $a_{i+1}$\end{tabular}}}}%
    \put(0.58653102,0.1665583){\color[rgb]{0,0,0}\makebox(0,0)[lt]{\lineheight{0}\smash{\begin{tabular}[t]{l}\scriptsize $b_{i+1}$\end{tabular}}}}%
    \put(0.41653064,0.1665583){\color[rgb]{0,0,0}\makebox(0,0)[lt]{\lineheight{0}\smash{\begin{tabular}[t]{l}\scriptsize $a_i$\end{tabular}}}}%
    \put(0.21347463,0.23266955){\color[rgb]{0,0,0}\makebox(0,0)[lt]{\lineheight{0}\smash{\begin{tabular}[t]{l}$=$\end{tabular}}}}%
    \put(0.34569715,0.25628078){\color[rgb]{0,0,0}\makebox(0,0)[lt]{\lineheight{0}\smash{\begin{tabular}[t]{l}\scriptsize $2$\end{tabular}}}}%
    \put(0,0){\includegraphics[width=\unitlength,page=3]{Diagramme_ci.pdf}}%
    \put(0.92653146,0.32711419){\color[rgb]{0,0,0}\makebox(0,0)[lt]{\lineheight{0}\smash{\begin{tabular}[t]{l}\scriptsize $b_{i+1}^{-1}$\end{tabular}}}}%
    \put(0.75653117,0.32711419){\color[rgb]{0,0,0}\makebox(0,0)[lt]{\lineheight{0}\smash{\begin{tabular}[t]{l}\scriptsize $a_{i+1}^{-1}$\end{tabular}}}}%
    \put(0.58653082,0.32711419){\color[rgb]{0,0,0}\makebox(0,0)[lt]{\lineheight{0}\smash{\begin{tabular}[t]{l}\scriptsize $b_{i+1}$\end{tabular}}}}%
    \put(0.41653046,0.32711419){\color[rgb]{0,0,0}\makebox(0,0)[lt]{\lineheight{0}\smash{\begin{tabular}[t]{l}\scriptsize $a_i$\end{tabular}}}}%
    \put(0.15680645,0.32711419){\color[rgb]{0,0,0}\makebox(0,0)[lt]{\lineheight{0}\smash{\begin{tabular}[t]{l}\scriptsize $c_i$\end{tabular}}}}%
    \put(0.34569698,0.4168366){\color[rgb]{0,0,0}\makebox(0,0)[lt]{\lineheight{0}\smash{\begin{tabular}[t]{l}\scriptsize $2$\end{tabular}}}}%
    \put(0.21111353,0.39322544){\color[rgb]{0,0,0}\makebox(0,0)[lt]{\lineheight{0}\smash{\begin{tabular}[t]{l}$:=$\end{tabular}}}}%
  \end{picture}%
\endgroup%

%% file: Action_ci.pdf_tex
\begingroup%
  \makeatletter%
  \providecommand\color[2][]{%
    \errmessage{(Inkscape) Color is used for the text in Inkscape, but the package 'color.sty' is not loaded}%
    \renewcommand\color[2][]{}%
  }%
  \providecommand\transparent[1]{%
    \errmessage{(Inkscape) Transparency is used (non-zero) for the text in Inkscape, but the package 'transparent.sty' is not loaded}%
    \renewcommand\transparent[1]{}%
  }%
  \providecommand\rotatebox[2]{#2}%
  \newcommand*\fsize{\dimexpr\f@size pt\relax}%
  \newcommand*\lineheight[1]{\fontsize{\fsize}{#1\fsize}\selectfont}%
  \ifx\svgwidth\undefined%
    \setlength{\unitlength}{482.01290581bp}%
    \ifx\svgscale\undefined%
      \relax%
    \else%
      \setlength{\unitlength}{\unitlength * \real{\svgscale}}%
    \fi%
  \else%
    \setlength{\unitlength}{\svgwidth}%
  \fi%
  \global\let\svgwidth\undefined%
  \global\let\svgscale\undefined%
  \makeatother%
  \begin{picture}(1,1.34132675)%
    \lineheight{1}%
    \setlength\tabcolsep{0pt}%
    \put(0,0){\includegraphics[width=\unitlength,page=1]{Action_ci.pdf}}%
    \put(0.96330039,0.00448122){\color[rgb]{0,0,0}\makebox(0,0)[lt]{\lineheight{0}\smash{\begin{tabular}[t]{l}\scriptsize $m_{i+1}$\end{tabular}}}}%
    \put(0.86038836,0.00448122){\color[rgb]{0,0,0}\makebox(0,0)[lt]{\lineheight{0}\smash{\begin{tabular}[t]{l}\scriptsize $m_{i}$\end{tabular}}}}%
    \put(0.75159109,0.00448122){\color[rgb]{0,0,0}\makebox(0,0)[lt]{\lineheight{0}\smash{\begin{tabular}[t]{l}\scriptsize $m_{i-1}$\end{tabular}}}}%
    \put(0.64867778,0.00448122){\color[rgb]{0,0,0}\makebox(0,0)[lt]{\lineheight{0}\smash{\begin{tabular}[t]{l}\scriptsize $m_{1}$\end{tabular}}}}%
    \put(0.52223806,0.04562848){\color[rgb]{0,0,0}\makebox(0,0)[lt]{\lineheight{0}\smash{\begin{tabular}[t]{l}$=$\end{tabular}}}}%
    \put(0,0){\includegraphics[width=\unitlength,page=2]{Action_ci.pdf}}%
    \put(0.44872752,0.00448122){\color[rgb]{0,0,0}\makebox(0,0)[lt]{\lineheight{0}\smash{\begin{tabular}[t]{l}\scriptsize $b_{i+1}$\end{tabular}}}}%
    \put(0.3428723,0.00448122){\color[rgb]{0,0,0}\makebox(0,0)[lt]{\lineheight{0}\smash{\begin{tabular}[t]{l}\scriptsize $b_i$\end{tabular}}}}%
    \put(0.23701692,0.00448122){\color[rgb]{0,0,0}\makebox(0,0)[lt]{\lineheight{0}\smash{\begin{tabular}[t]{l}\scriptsize $b_{i-1}$\end{tabular}}}}%
    \put(0.13116154,0.00448122){\color[rgb]{0,0,0}\makebox(0,0)[lt]{\lineheight{0}\smash{\begin{tabular}[t]{l}\scriptsize $b_1$\end{tabular}}}}%
    \put(0,0){\includegraphics[width=\unitlength,page=3]{Action_ci.pdf}}%
    \put(0.48695311,0.04562849){\color[rgb]{0,0,0}\makebox(0,0)[lt]{\lineheight{0}\smash{\begin{tabular}[t]{l}$1_\mathbb{K}$\end{tabular}}}}%
    \put(-0.00115764,0.04562848){\color[rgb]{0,0,0}\makebox(0,0)[lt]{\lineheight{0}\smash{\begin{tabular}[t]{l}$\overset{\text{top.}}{=}$\end{tabular}}}}%
    \put(0,0){\includegraphics[width=\unitlength,page=4]{Action_ci.pdf}}%
    \put(0.44872752,0.13385987){\color[rgb]{0,0,0}\makebox(0,0)[lt]{\lineheight{0}\smash{\begin{tabular}[t]{l}\scriptsize $b_{i+1}$\end{tabular}}}}%
    \put(0.34287229,0.13385987){\color[rgb]{0,0,0}\makebox(0,0)[lt]{\lineheight{0}\smash{\begin{tabular}[t]{l}\scriptsize $b_i$\end{tabular}}}}%
    \put(0.23701691,0.13385987){\color[rgb]{0,0,0}\makebox(0,0)[lt]{\lineheight{0}\smash{\begin{tabular}[t]{l}\scriptsize $b_{i-1}$\end{tabular}}}}%
    \put(0.13116153,0.13385987){\color[rgb]{0,0,0}\makebox(0,0)[lt]{\lineheight{0}\smash{\begin{tabular}[t]{l}\scriptsize $b_1$\end{tabular}}}}%
    \put(0,0){\includegraphics[width=\unitlength,page=5]{Action_ci.pdf}}%
    \put(0.48695313,0.17500715){\color[rgb]{0,0,0}\makebox(0,0)[lt]{\lineheight{0}\smash{\begin{tabular}[t]{l}$1_\mathbb{K}$\end{tabular}}}}%
    \put(0.00472322,0.17500714){\color[rgb]{0,0,0}\makebox(0,0)[lt]{\lineheight{0}\smash{\begin{tabular}[t]{l}$=$\end{tabular}}}}%
    \put(0,0){\includegraphics[width=\unitlength,page=6]{Action_ci.pdf}}%
    \put(0.55458275,0.29264283){\color[rgb]{0,0,0}\makebox(0,0)[lt]{\lineheight{0}\smash{\begin{tabular}[t]{l}\scriptsize $b_{i+1}$\end{tabular}}}}%
    \put(0.44872752,0.29264283){\color[rgb]{0,0,0}\makebox(0,0)[lt]{\lineheight{0}\smash{\begin{tabular}[t]{l}\scriptsize $b_i$\end{tabular}}}}%
    \put(0.34287218,0.29264283){\color[rgb]{0,0,0}\makebox(0,0)[lt]{\lineheight{0}\smash{\begin{tabular}[t]{l}\scriptsize $b_{i-1}$\end{tabular}}}}%
    \put(0.2370168,0.29264283){\color[rgb]{0,0,0}\makebox(0,0)[lt]{\lineheight{0}\smash{\begin{tabular}[t]{l}\scriptsize $b_1$\end{tabular}}}}%
    \put(0.13116152,0.29264281){\color[rgb]{0,0,0}\makebox(0,0)[lt]{\lineheight{0}\smash{\begin{tabular}[t]{l}\scriptsize $a_i$\end{tabular}}}}%
    \put(0,0){\includegraphics[width=\unitlength,page=7]{Action_ci.pdf}}%
    \put(0.59280843,0.3337901){\color[rgb]{0,0,0}\makebox(0,0)[lt]{\lineheight{0}\smash{\begin{tabular}[t]{l}$1_\mathbb{K}$\end{tabular}}}}%
    \put(-0.00115758,0.33379009){\color[rgb]{0,0,0}\makebox(0,0)[lt]{\lineheight{0}\smash{\begin{tabular}[t]{l}$\overset{\text{top.}}{=}$\end{tabular}}}}%
    \put(0,0){\includegraphics[width=\unitlength,page=8]{Action_ci.pdf}}%
    \put(0.5545828,0.45142578){\color[rgb]{0,0,0}\makebox(0,0)[lt]{\lineheight{0}\smash{\begin{tabular}[t]{l}\scriptsize $b_{i+1}$\end{tabular}}}}%
    \put(0.44872757,0.45142578){\color[rgb]{0,0,0}\makebox(0,0)[lt]{\lineheight{0}\smash{\begin{tabular}[t]{l}\scriptsize $b_i$\end{tabular}}}}%
    \put(0.34287223,0.45142578){\color[rgb]{0,0,0}\makebox(0,0)[lt]{\lineheight{0}\smash{\begin{tabular}[t]{l}\scriptsize $b_{i-1}$\end{tabular}}}}%
    \put(0.23701685,0.45142578){\color[rgb]{0,0,0}\makebox(0,0)[lt]{\lineheight{0}\smash{\begin{tabular}[t]{l}\scriptsize $b_1$\end{tabular}}}}%
    \put(0.13116157,0.45142577){\color[rgb]{0,0,0}\makebox(0,0)[lt]{\lineheight{0}\smash{\begin{tabular}[t]{l}\scriptsize $a_i$\end{tabular}}}}%
    \put(0,0){\includegraphics[width=\unitlength,page=9]{Action_ci.pdf}}%
    \put(0.59280844,0.49257306){\color[rgb]{0,0,0}\makebox(0,0)[lt]{\lineheight{0}\smash{\begin{tabular}[t]{l}$1_\mathbb{K}$\end{tabular}}}}%
    \put(0.00472327,0.49257305){\color[rgb]{0,0,0}\makebox(0,0)[lt]{\lineheight{0}\smash{\begin{tabular}[t]{l}$=$\end{tabular}}}}%
    \put(0,0){\includegraphics[width=\unitlength,page=10]{Action_ci.pdf}}%
    \put(0.66043809,0.61020874){\color[rgb]{0,0,0}\makebox(0,0)[lt]{\lineheight{0}\smash{\begin{tabular}[t]{l}\scriptsize $b_{i+1}$\end{tabular}}}}%
    \put(0.55458282,0.61020874){\color[rgb]{0,0,0}\makebox(0,0)[lt]{\lineheight{0}\smash{\begin{tabular}[t]{l}\scriptsize $b_i$\end{tabular}}}}%
    \put(0.44872748,0.61020874){\color[rgb]{0,0,0}\makebox(0,0)[lt]{\lineheight{0}\smash{\begin{tabular}[t]{l}\scriptsize $b_{i-1}$\end{tabular}}}}%
    \put(0.3428721,0.61020874){\color[rgb]{0,0,0}\makebox(0,0)[lt]{\lineheight{0}\smash{\begin{tabular}[t]{l}\scriptsize $b_1$\end{tabular}}}}%
    \put(0.23701688,0.61020872){\color[rgb]{0,0,0}\makebox(0,0)[lt]{\lineheight{0}\smash{\begin{tabular}[t]{l}\scriptsize $b_{i+1}$\end{tabular}}}}%
    \put(0.13116152,0.61020872){\color[rgb]{0,0,0}\makebox(0,0)[lt]{\lineheight{0}\smash{\begin{tabular}[t]{l}\scriptsize $a_i$\end{tabular}}}}%
    \put(0,0){\includegraphics[width=\unitlength,page=11]{Action_ci.pdf}}%
    \put(0.6986636,0.65135601){\color[rgb]{0,0,0}\makebox(0,0)[lt]{\lineheight{0}\smash{\begin{tabular}[t]{l}$1_\mathbb{K}$\end{tabular}}}}%
    \put(-0.00115758,0.651356){\color[rgb]{0,0,0}\makebox(0,0)[lt]{\lineheight{0}\smash{\begin{tabular}[t]{l}$\overset{\text{top.}}{=}$\end{tabular}}}}%
    \put(0,0){\includegraphics[width=\unitlength,page=12]{Action_ci.pdf}}%
    \put(0.66043813,0.76899167){\color[rgb]{0,0,0}\makebox(0,0)[lt]{\lineheight{0}\smash{\begin{tabular}[t]{l}\scriptsize $b_{i+1}$\end{tabular}}}}%
    \put(0.55458287,0.76899167){\color[rgb]{0,0,0}\makebox(0,0)[lt]{\lineheight{0}\smash{\begin{tabular}[t]{l}\scriptsize $b_i$\end{tabular}}}}%
    \put(0.44872753,0.76899167){\color[rgb]{0,0,0}\makebox(0,0)[lt]{\lineheight{0}\smash{\begin{tabular}[t]{l}\scriptsize $b_{i-1}$\end{tabular}}}}%
    \put(0.34287215,0.76899167){\color[rgb]{0,0,0}\makebox(0,0)[lt]{\lineheight{0}\smash{\begin{tabular}[t]{l}\scriptsize $b_1$\end{tabular}}}}%
    \put(0.23701692,0.76899166){\color[rgb]{0,0,0}\makebox(0,0)[lt]{\lineheight{0}\smash{\begin{tabular}[t]{l}\scriptsize $b_{i+1}$\end{tabular}}}}%
    \put(0.13116156,0.76899166){\color[rgb]{0,0,0}\makebox(0,0)[lt]{\lineheight{0}\smash{\begin{tabular}[t]{l}\scriptsize $a_i$\end{tabular}}}}%
    \put(0,0){\includegraphics[width=\unitlength,page=13]{Action_ci.pdf}}%
    \put(0.69866369,0.81013895){\color[rgb]{0,0,0}\makebox(0,0)[lt]{\lineheight{0}\smash{\begin{tabular}[t]{l}$1_\mathbb{K}$\end{tabular}}}}%
    \put(0.00472327,0.81013894){\color[rgb]{0,0,0}\makebox(0,0)[lt]{\lineheight{0}\smash{\begin{tabular}[t]{l}$=$\end{tabular}}}}%
    \put(0,0){\includegraphics[width=\unitlength,page=14]{Action_ci.pdf}}%
    \put(0.76629348,0.92777463){\color[rgb]{0,0,0}\makebox(0,0)[lt]{\lineheight{0}\smash{\begin{tabular}[t]{l}\scriptsize $b_{i+1}$\end{tabular}}}}%
    \put(0.66043821,0.92777463){\color[rgb]{0,0,0}\makebox(0,0)[lt]{\lineheight{0}\smash{\begin{tabular}[t]{l}\scriptsize $b_i$\end{tabular}}}}%
    \put(0.55458287,0.92777463){\color[rgb]{0,0,0}\makebox(0,0)[lt]{\lineheight{0}\smash{\begin{tabular}[t]{l}\scriptsize $b_{i-1}$\end{tabular}}}}%
    \put(0.44872749,0.92777463){\color[rgb]{0,0,0}\makebox(0,0)[lt]{\lineheight{0}\smash{\begin{tabular}[t]{l}\scriptsize $b_1$\end{tabular}}}}%
    \put(0.34287223,0.92777462){\color[rgb]{0,0,0}\makebox(0,0)[lt]{\lineheight{0}\smash{\begin{tabular}[t]{l}\scriptsize $a_{i+1}$\end{tabular}}}}%
    \put(0.23701692,0.92777462){\color[rgb]{0,0,0}\makebox(0,0)[lt]{\lineheight{0}\smash{\begin{tabular}[t]{l}\scriptsize $b_{i+1}$\end{tabular}}}}%
    \put(0.13116156,0.92777462){\color[rgb]{0,0,0}\makebox(0,0)[lt]{\lineheight{0}\smash{\begin{tabular}[t]{l}\scriptsize $a_i$\end{tabular}}}}%
    \put(0,0){\includegraphics[width=\unitlength,page=15]{Action_ci.pdf}}%
    \put(0.80451899,0.96892191){\color[rgb]{0,0,0}\makebox(0,0)[lt]{\lineheight{0}\smash{\begin{tabular}[t]{l}$1_\mathbb{K}$\end{tabular}}}}%
    \put(0.00472327,0.9689219){\color[rgb]{0,0,0}\makebox(0,0)[lt]{\lineheight{0}\smash{\begin{tabular}[t]{l}$=$\end{tabular}}}}%
    \put(0,0){\includegraphics[width=\unitlength,page=16]{Action_ci.pdf}}%
    \put(0.87214875,1.05715334){\color[rgb]{0,0,0}\makebox(0,0)[lt]{\lineheight{0}\smash{\begin{tabular}[t]{l}\scriptsize $b_{i+1}$\end{tabular}}}}%
    \put(0.76629348,1.05715334){\color[rgb]{0,0,0}\makebox(0,0)[lt]{\lineheight{0}\smash{\begin{tabular}[t]{l}\scriptsize $b_i$\end{tabular}}}}%
    \put(0.66043814,1.05715334){\color[rgb]{0,0,0}\makebox(0,0)[lt]{\lineheight{0}\smash{\begin{tabular}[t]{l}\scriptsize $b_{i-1}$\end{tabular}}}}%
    \put(0.5545828,1.05715334){\color[rgb]{0,0,0}\makebox(0,0)[lt]{\lineheight{0}\smash{\begin{tabular}[t]{l}\scriptsize $b_1$\end{tabular}}}}%
    \put(0.44872746,1.05715332){\color[rgb]{0,0,0}\makebox(0,0)[lt]{\lineheight{0}\smash{\begin{tabular}[t]{l}\scriptsize $b_{i+1}$\end{tabular}}}}%
    \put(0.34287223,1.05715332){\color[rgb]{0,0,0}\makebox(0,0)[lt]{\lineheight{0}\smash{\begin{tabular}[t]{l}\scriptsize $a_{i+1}$\end{tabular}}}}%
    \put(0.23701693,1.05715332){\color[rgb]{0,0,0}\makebox(0,0)[lt]{\lineheight{0}\smash{\begin{tabular}[t]{l}\scriptsize $b_{i+1}$\end{tabular}}}}%
    \put(0.13116157,1.05715332){\color[rgb]{0,0,0}\makebox(0,0)[lt]{\lineheight{0}\smash{\begin{tabular}[t]{l}\scriptsize $a_i$\end{tabular}}}}%
    \put(0,0){\includegraphics[width=\unitlength,page=17]{Action_ci.pdf}}%
    \put(0.9103743,1.09830061){\color[rgb]{0,0,0}\makebox(0,0)[lt]{\lineheight{0}\smash{\begin{tabular}[t]{l}$1_\mathbb{K}$\end{tabular}}}}%
    \put(0.00472327,1.0983006){\color[rgb]{0,0,0}\makebox(0,0)[lt]{\lineheight{0}\smash{\begin{tabular}[t]{l}$=$\end{tabular}}}}%
    \put(0,0){\includegraphics[width=\unitlength,page=18]{Action_ci.pdf}}%
    \put(0.91331468,1.1571278){\color[rgb]{0,0,0}\makebox(0,0)[lt]{\lineheight{0}\smash{\begin{tabular}[t]{l}\scriptsize $m_{i+1}$\end{tabular}}}}%
    \put(0.80745942,1.1571278){\color[rgb]{0,0,0}\makebox(0,0)[lt]{\lineheight{0}\smash{\begin{tabular}[t]{l}\scriptsize $m_i$\end{tabular}}}}%
    \put(0.70160408,1.1571278){\color[rgb]{0,0,0}\makebox(0,0)[lt]{\lineheight{0}\smash{\begin{tabular}[t]{l}\scriptsize $m_{i-1}$\end{tabular}}}}%
    \put(0.59574873,1.1571278){\color[rgb]{0,0,0}\makebox(0,0)[lt]{\lineheight{0}\smash{\begin{tabular}[t]{l}\scriptsize $m_1$\end{tabular}}}}%
    \put(0.44872746,1.15712779){\color[rgb]{0,0,0}\makebox(0,0)[lt]{\lineheight{0}\smash{\begin{tabular}[t]{l}\scriptsize $b_{i+1}$\end{tabular}}}}%
    \put(0.34287223,1.15712779){\color[rgb]{0,0,0}\makebox(0,0)[lt]{\lineheight{0}\smash{\begin{tabular}[t]{l}\scriptsize $a_{i+1}$\end{tabular}}}}%
    \put(0.23701693,1.15712779){\color[rgb]{0,0,0}\makebox(0,0)[lt]{\lineheight{0}\smash{\begin{tabular}[t]{l}\scriptsize $b_{i+1}$\end{tabular}}}}%
    \put(0.13116157,1.15712779){\color[rgb]{0,0,0}\makebox(0,0)[lt]{\lineheight{0}\smash{\begin{tabular}[t]{l}\scriptsize $a_i$\end{tabular}}}}%
    \put(0,0){\includegraphics[width=\unitlength,page=19]{Action_ci.pdf}}%
    \put(0.00472327,1.19827506){\color[rgb]{0,0,0}\makebox(0,0)[lt]{\lineheight{0}\smash{\begin{tabular}[t]{l}$=$\end{tabular}}}}%
    \put(0,0){\includegraphics[width=\unitlength,page=20]{Action_ci.pdf}}%
    \put(0.59574866,1.25710224){\color[rgb]{0,0,0}\makebox(0,0)[lt]{\lineheight{0}\smash{\begin{tabular}[t]{l}\scriptsize $m_{i+1}$\end{tabular}}}}%
    \put(0.48989479,1.25710224){\color[rgb]{0,0,0}\makebox(0,0)[lt]{\lineheight{0}\smash{\begin{tabular}[t]{l}\scriptsize $m_{i}$\end{tabular}}}}%
    \put(0.38403948,1.25710224){\color[rgb]{0,0,0}\makebox(0,0)[lt]{\lineheight{0}\smash{\begin{tabular}[t]{l}\scriptsize $m_{i-1}$\end{tabular}}}}%
    \put(0.2781842,1.25710224){\color[rgb]{0,0,0}\makebox(0,0)[lt]{\lineheight{0}\smash{\begin{tabular}[t]{l}\scriptsize $m_{1}$\end{tabular}}}}%
    \put(0.13116156,1.25708357){\color[rgb]{0,0,0}\makebox(0,0)[lt]{\lineheight{0}\smash{\begin{tabular}[t]{l}\scriptsize $c_i$\end{tabular}}}}%
    \put(0,0){\includegraphics[width=\unitlength,page=21]{Action_ci.pdf}}%
  \end{picture}%
\endgroup%

%% file: Etiquetage_ci.pdf_tex
\begingroup%
  \makeatletter%
  \providecommand\color[2][]{%
    \errmessage{(Inkscape) Color is used for the text in Inkscape, but the package 'color.sty' is not loaded}%
    \renewcommand\color[2][]{}%
  }%
  \providecommand\transparent[1]{%
    \errmessage{(Inkscape) Transparency is used (non-zero) for the text in Inkscape, but the package 'transparent.sty' is not loaded}%
    \renewcommand\transparent[1]{}%
  }%
  \providecommand\rotatebox[2]{#2}%
  \newcommand*\fsize{\dimexpr\f@size pt\relax}%
  \newcommand*\lineheight[1]{\fontsize{\fsize}{#1\fsize}\selectfont}%
  \ifx\svgwidth\undefined%
    \setlength{\unitlength}{453.54330709bp}%
    \ifx\svgscale\undefined%
      \relax%
    \else%
      \setlength{\unitlength}{\unitlength * \real{\svgscale}}%
    \fi%
  \else%
    \setlength{\unitlength}{\svgwidth}%
  \fi%
  \global\let\svgwidth\undefined%
  \global\let\svgscale\undefined%
  \makeatother%
  \begin{picture}(1,0.5286675)%
    \lineheight{1}%
    \setlength\tabcolsep{0pt}%
    \put(0,0){\includegraphics[width=\unitlength,page=1]{Etiquetage_ci.pdf}}%
    \put(0.69359382,0.12663622){\color[rgb]{0,0,0}\makebox(0,0)[lt]{\lineheight{0}\smash{\begin{tabular}[t]{l}\tiny $z_{i+1}$\end{tabular}}}}%
    \put(0.69359382,0.10163622){\color[rgb]{0,0,0}\makebox(0,0)[lt]{\lineheight{0}\smash{\begin{tabular}[t]{l}\tiny $d$\end{tabular}}}}%
    \put(0.61859378,0.10163622){\color[rgb]{0,0,0}\makebox(0,0)[lt]{\lineheight{0}\smash{\begin{tabular}[t]{l}\tiny $c$\end{tabular}}}}%
    \put(0.56859377,0.126636){\color[rgb]{0,0,0}\makebox(0,0)[lt]{\lineheight{0}\smash{\begin{tabular}[t]{l}\tiny $z_{i}$\end{tabular}}}}%
    \put(0.61859378,0.12663604){\color[rgb]{0,0,0}\makebox(0,0)[lt]{\lineheight{0}\smash{\begin{tabular}[t]{l}\tiny $y_{i}$\end{tabular}}}}%
    \put(0.55609382,0.10163622){\color[rgb]{0,0,0}\makebox(0,0)[lt]{\lineheight{0}\smash{\begin{tabular}[t]{l}\tiny $a_{i}$\end{tabular}}}}%
    \put(0.41859382,0.126636){\color[rgb]{0,0,0}\makebox(0,0)[lt]{\lineheight{0}\smash{\begin{tabular}[t]{l}\tiny $z_{i-1}$\end{tabular}}}}%
    \put(0.49359379,0.12663604){\color[rgb]{0,0,0}\makebox(0,0)[lt]{\lineheight{0}\smash{\begin{tabular}[t]{l}\tiny $y_{i-1}$\end{tabular}}}}%
    \put(0.48109384,0.10163622){\color[rgb]{0,0,0}\makebox(0,0)[lt]{\lineheight{0}\smash{\begin{tabular}[t]{l}\tiny $b_{i-1}$\end{tabular}}}}%
    \put(0.40609387,0.10163622){\color[rgb]{0,0,0}\makebox(0,0)[lt]{\lineheight{0}\smash{\begin{tabular}[t]{l}\tiny $a_{i-1}$\end{tabular}}}}%
    \put(0.2685938,0.126636){\color[rgb]{0,0,0}\makebox(0,0)[lt]{\lineheight{0}\smash{\begin{tabular}[t]{l}\tiny $z_1$\end{tabular}}}}%
    \put(0.34359381,0.12663604){\color[rgb]{0,0,0}\makebox(0,0)[lt]{\lineheight{0}\smash{\begin{tabular}[t]{l}\tiny $y_1$\end{tabular}}}}%
    \put(0.33109386,0.10163622){\color[rgb]{0,0,0}\makebox(0,0)[lt]{\lineheight{0}\smash{\begin{tabular}[t]{l}\tiny $b_1$\end{tabular}}}}%
    \put(0.25609382,0.10163622){\color[rgb]{0,0,0}\makebox(0,0)[lt]{\lineheight{0}\smash{\begin{tabular}[t]{l}\tiny $a_1$\end{tabular}}}}%
    \put(0,0){\includegraphics[width=\unitlength,page=2]{Etiquetage_ci.pdf}}%
    \put(0.76859466,0.0047625){\color[rgb]{0,0,0}\makebox(0,0)[lt]{\lineheight{0}\smash{\begin{tabular}[t]{l}\scriptsize $m_{i+1}$\end{tabular}}}}%
    \put(0.61859457,0.0047625){\color[rgb]{0,0,0}\makebox(0,0)[lt]{\lineheight{0}\smash{\begin{tabular}[t]{l}\scriptsize $m_{i}$\end{tabular}}}}%
    \put(0.46859372,0.00476251){\color[rgb]{0,0,0}\makebox(0,0)[lt]{\lineheight{0}\smash{\begin{tabular}[t]{l}\scriptsize $m_{i-1}$\end{tabular}}}}%
    \put(0.31859455,0.00476251){\color[rgb]{0,0,0}\makebox(0,0)[lt]{\lineheight{0}\smash{\begin{tabular}[t]{l}\scriptsize $m_1$\end{tabular}}}}%
    \put(0,0){\includegraphics[width=\unitlength,page=3]{Etiquetage_ci.pdf}}%
    \put(0.15609458,0.06413623){\color[rgb]{0,0,0}\makebox(0,0)[lt]{\lineheight{0}\smash{\begin{tabular}[t]{l}$=$\end{tabular}}}}%
    \put(0,0){\includegraphics[width=\unitlength,page=4]{Etiquetage_ci.pdf}}%
    \put(0.76859465,0.18601233){\color[rgb]{0,0,0}\makebox(0,0)[lt]{\lineheight{0}\smash{\begin{tabular}[t]{l}\scriptsize $m_{i+1}$\end{tabular}}}}%
    \put(0.61859456,0.18601233){\color[rgb]{0,0,0}\makebox(0,0)[lt]{\lineheight{0}\smash{\begin{tabular}[t]{l}\scriptsize $m_{i}$\end{tabular}}}}%
    \put(0.46859371,0.18601234){\color[rgb]{0,0,0}\makebox(0,0)[lt]{\lineheight{0}\smash{\begin{tabular}[t]{l}\scriptsize $m_{i-1}$\end{tabular}}}}%
    \put(0.31859455,0.18601234){\color[rgb]{0,0,0}\makebox(0,0)[lt]{\lineheight{0}\smash{\begin{tabular}[t]{l}\scriptsize $m_1$\end{tabular}}}}%
    \put(0,0){\includegraphics[width=\unitlength,page=5]{Etiquetage_ci.pdf}}%
    \put(0.15609457,0.24538606){\color[rgb]{0,0,0}\makebox(0,0)[lt]{\lineheight{0}\smash{\begin{tabular}[t]{l}$\overset{\text{top.}}{=}$\end{tabular}}}}%
    \put(0.76859466,0.36726233){\color[rgb]{0,0,0}\makebox(0,0)[lt]{\lineheight{0}\smash{\begin{tabular}[t]{l}\scriptsize $m_{i+1}$\end{tabular}}}}%
    \put(0.61859456,0.36726233){\color[rgb]{0,0,0}\makebox(0,0)[lt]{\lineheight{0}\smash{\begin{tabular}[t]{l}\scriptsize $m_{i}$\end{tabular}}}}%
    \put(0.46859372,0.36726234){\color[rgb]{0,0,0}\makebox(0,0)[lt]{\lineheight{0}\smash{\begin{tabular}[t]{l}\scriptsize $m_{i-1}$\end{tabular}}}}%
    \put(0.31859455,0.36726234){\color[rgb]{0,0,0}\makebox(0,0)[lt]{\lineheight{0}\smash{\begin{tabular}[t]{l}\scriptsize $m_1$\end{tabular}}}}%
    \put(0,0){\includegraphics[width=\unitlength,page=6]{Etiquetage_ci.pdf}}%
  \end{picture}%
\endgroup%

%% file: Courbe_di_surface.pdf_tex
\begingroup%
  \makeatletter%
  \providecommand\color[2][]{%
    \errmessage{(Inkscape) Color is used for the text in Inkscape, but the package 'color.sty' is not loaded}%
    \renewcommand\color[2][]{}%
  }%
  \providecommand\transparent[1]{%
    \errmessage{(Inkscape) Transparency is used (non-zero) for the text in Inkscape, but the package 'transparent.sty' is not loaded}%
    \renewcommand\transparent[1]{}%
  }%
  \providecommand\rotatebox[2]{#2}%
  \newcommand*\fsize{\dimexpr\f@size pt\relax}%
  \newcommand*\lineheight[1]{\fontsize{\fsize}{#1\fsize}\selectfont}%
  \ifx\svgwidth\undefined%
    \setlength{\unitlength}{453.54313407bp}%
    \ifx\svgscale\undefined%
      \relax%
    \else%
      \setlength{\unitlength}{\unitlength * \real{\svgscale}}%
    \fi%
  \else%
    \setlength{\unitlength}{\svgwidth}%
  \fi%
  \global\let\svgwidth\undefined%
  \global\let\svgscale\undefined%
  \makeatother%
  \begin{picture}(1,0.17628758)%
    \lineheight{1}%
    \setlength\tabcolsep{0pt}%
    \put(0,0){\includegraphics[width=\unitlength,page=1]{Courbe_di_surface.pdf}}%
    \put(0.6068936,0.14439381){\color[rgb]{0,0,0}\makebox(0,0)[lt]{\lineheight{1.25}\smash{\begin{tabular}[t]{l}$\scriptstyle d_i$\end{tabular}}}}%
    \put(0.29439349,0.05064377){\color[rgb]{0,0,0}\makebox(0,0)[lt]{\lineheight{1.25}\smash{\begin{tabular}[t]{l}$\scriptstyle i$\end{tabular}}}}%
    \put(0,0){\includegraphics[width=\unitlength,page=2]{Courbe_di_surface.pdf}}%
  \end{picture}%
\endgroup%

%% file: Courbe_y_surface.pdf_tex
\begingroup%
  \makeatletter%
  \providecommand\color[2][]{%
    \errmessage{(Inkscape) Color is used for the text in Inkscape, but the package 'color.sty' is not loaded}%
    \renewcommand\color[2][]{}%
  }%
  \providecommand\transparent[1]{%
    \errmessage{(Inkscape) Transparency is used (non-zero) for the text in Inkscape, but the package 'transparent.sty' is not loaded}%
    \renewcommand\transparent[1]{}%
  }%
  \providecommand\rotatebox[2]{#2}%
  \newcommand*\fsize{\dimexpr\f@size pt\relax}%
  \newcommand*\lineheight[1]{\fontsize{\fsize}{#1\fsize}\selectfont}%
  \ifx\svgwidth\undefined%
    \setlength{\unitlength}{453.54313407bp}%
    \ifx\svgscale\undefined%
      \relax%
    \else%
      \setlength{\unitlength}{\unitlength * \real{\svgscale}}%
    \fi%
  \else%
    \setlength{\unitlength}{\svgwidth}%
  \fi%
  \global\let\svgwidth\undefined%
  \global\let\svgscale\undefined%
  \makeatother%
  \begin{picture}(1,0.17628758)%
    \lineheight{1}%
    \setlength\tabcolsep{0pt}%
    \put(0,0){\includegraphics[width=\unitlength,page=1]{Courbe_y_surface.pdf}}%
    \put(0.60689363,0.1287688){\color[rgb]{0,0,0}\makebox(0,0)[lt]{\lineheight{1.25}\smash{\begin{tabular}[t]{l}$\scriptstyle y$\end{tabular}}}}%
    \put(0,0){\includegraphics[width=\unitlength,page=2]{Courbe_y_surface.pdf}}%
  \end{picture}%
\endgroup%

%% file: Courbe_z_surface.pdf_tex
\begingroup%
  \makeatletter%
  \providecommand\color[2][]{%
    \errmessage{(Inkscape) Color is used for the text in Inkscape, but the package 'color.sty' is not loaded}%
    \renewcommand\color[2][]{}%
  }%
  \providecommand\transparent[1]{%
    \errmessage{(Inkscape) Transparency is used (non-zero) for the text in Inkscape, but the package 'transparent.sty' is not loaded}%
    \renewcommand\transparent[1]{}%
  }%
  \providecommand\rotatebox[2]{#2}%
  \newcommand*\fsize{\dimexpr\f@size pt\relax}%
  \newcommand*\lineheight[1]{\fontsize{\fsize}{#1\fsize}\selectfont}%
  \ifx\svgwidth\undefined%
    \setlength{\unitlength}{453.54313407bp}%
    \ifx\svgscale\undefined%
      \relax%
    \else%
      \setlength{\unitlength}{\unitlength * \real{\svgscale}}%
    \fi%
  \else%
    \setlength{\unitlength}{\svgwidth}%
  \fi%
  \global\let\svgwidth\undefined%
  \global\let\svgscale\undefined%
  \makeatother%
  \begin{picture}(1,0.17628758)%
    \lineheight{1}%
    \setlength\tabcolsep{0pt}%
    \put(0,0){\includegraphics[width=\unitlength,page=1]{Courbe_z_surface.pdf}}%
    \put(0.60689362,0.1256438){\color[rgb]{0,0,0}\makebox(0,0)[lt]{\lineheight{1.25}\smash{\begin{tabular}[t]{l}$\scriptstyle z$\end{tabular}}}}%
    \put(0,0){\includegraphics[width=\unitlength,page=2]{Courbe_z_surface.pdf}}%
  \end{picture}%
\endgroup%

%% file: Action_z_courte.pdf_tex
\begingroup%
  \makeatletter%
  \providecommand\color[2][]{%
    \errmessage{(Inkscape) Color is used for the text in Inkscape, but the package 'color.sty' is not loaded}%
    \renewcommand\color[2][]{}%
  }%
  \providecommand\transparent[1]{%
    \errmessage{(Inkscape) Transparency is used (non-zero) for the text in Inkscape, but the package 'transparent.sty' is not loaded}%
    \renewcommand\transparent[1]{}%
  }%
  \providecommand\rotatebox[2]{#2}%
  \newcommand*\fsize{\dimexpr\f@size pt\relax}%
  \newcommand*\lineheight[1]{\fontsize{\fsize}{#1\fsize}\selectfont}%
  \ifx\svgwidth\undefined%
    \setlength{\unitlength}{481.88976378bp}%
    \ifx\svgscale\undefined%
      \relax%
    \else%
      \setlength{\unitlength}{\unitlength * \real{\svgscale}}%
    \fi%
  \else%
    \setlength{\unitlength}{\svgwidth}%
  \fi%
  \global\let\svgwidth\undefined%
  \global\let\svgscale\undefined%
  \makeatother%
  \begin{picture}(1,0.11815962)%
    \lineheight{1}%
    \setlength\tabcolsep{0pt}%
    \put(0,0){\includegraphics[width=\unitlength,page=1]{Action_z_courte.pdf}}%
    \put(0.9012826,0.00448234){\color[rgb]{0,0,0}\makebox(0,0)[lt]{\lineheight{0}\smash{\begin{tabular}[t]{l}\scriptsize $m_{3}$\end{tabular}}}}%
    \put(0.79540156,0.00448237){\color[rgb]{0,0,0}\makebox(0,0)[lt]{\lineheight{0}\smash{\begin{tabular}[t]{l}\scriptsize $m_{2}$\end{tabular}}}}%
    \put(0.68951922,0.00448237){\color[rgb]{0,0,0}\makebox(0,0)[lt]{\lineheight{0}\smash{\begin{tabular}[t]{l}\scriptsize $m_{1}$\end{tabular}}}}%
    \put(0.56304721,0.04565883){\color[rgb]{0,0,0}\makebox(0,0)[lt]{\lineheight{0}\smash{\begin{tabular}[t]{l}$=$\end{tabular}}}}%
    \put(0,0){\includegraphics[width=\unitlength,page=2]{Action_z_courte.pdf}}%
    \put(0.52481187,0.00448236){\color[rgb]{0,0,0}\makebox(0,0)[lt]{\lineheight{0}\smash{\begin{tabular}[t]{l}\scriptsize $m_{3}$\end{tabular}}}}%
    \put(0.41893087,0.00448238){\color[rgb]{0,0,0}\makebox(0,0)[lt]{\lineheight{0}\smash{\begin{tabular}[t]{l}\scriptsize $m_{2}$\end{tabular}}}}%
    \put(0.3130485,0.00448238){\color[rgb]{0,0,0}\makebox(0,0)[lt]{\lineheight{0}\smash{\begin{tabular}[t]{l}\scriptsize $m_{1}$\end{tabular}}}}%
    \put(0.16598833,0.00448239){\color[rgb]{0,0,0}\makebox(0,0)[lt]{\lineheight{0}\smash{\begin{tabular}[t]{l}\scriptsize $z$\end{tabular}}}}%
    \put(0,0){\includegraphics[width=\unitlength,page=3]{Action_z_courte.pdf}}%
  \end{picture}%
\endgroup%

%% file: Etiquetage_z.pdf_tex
\begingroup%
  \makeatletter%
  \providecommand\color[2][]{%
    \errmessage{(Inkscape) Color is used for the text in Inkscape, but the package 'color.sty' is not loaded}%
    \renewcommand\color[2][]{}%
  }%
  \providecommand\transparent[1]{%
    \errmessage{(Inkscape) Transparency is used (non-zero) for the text in Inkscape, but the package 'transparent.sty' is not loaded}%
    \renewcommand\transparent[1]{}%
  }%
  \providecommand\rotatebox[2]{#2}%
  \newcommand*\fsize{\dimexpr\f@size pt\relax}%
  \newcommand*\lineheight[1]{\fontsize{\fsize}{#1\fsize}\selectfont}%
  \ifx\svgwidth\undefined%
    \setlength{\unitlength}{453.54330709bp}%
    \ifx\svgscale\undefined%
      \relax%
    \else%
      \setlength{\unitlength}{\unitlength * \real{\svgscale}}%
    \fi%
  \else%
    \setlength{\unitlength}{\svgwidth}%
  \fi%
  \global\let\svgwidth\undefined%
  \global\let\svgscale\undefined%
  \makeatother%
  \begin{picture}(1,0.34352326)%
    \lineheight{1}%
    \setlength\tabcolsep{0pt}%
    \put(0,0){\includegraphics[width=\unitlength,page=1]{Etiquetage_z.pdf}}%
    \put(0.94984461,0.00086813){\color[rgb]{0,0,0}\makebox(0,0)[lt]{\lineheight{0}\smash{\begin{tabular}[t]{l}\scriptsize $m_{3}$\end{tabular}}}}%
    \put(0.79984458,0.00086813){\color[rgb]{0,0,0}\makebox(0,0)[lt]{\lineheight{0}\smash{\begin{tabular}[t]{l}\scriptsize $m_{2}$\end{tabular}}}}%
    \put(0.64984459,0.00086813){\color[rgb]{0,0,0}\makebox(0,0)[lt]{\lineheight{0}\smash{\begin{tabular}[t]{l}\scriptsize $m_{1}$\end{tabular}}}}%
    \put(0.59984377,0.11961685){\color[rgb]{0,0,0}\makebox(0,0)[lt]{\lineheight{1.25}\smash{\begin{tabular}[t]{l}\tiny $h$\end{tabular}}}}%
    \put(0.67484375,0.11961685){\color[rgb]{0,0,0}\makebox(0,0)[lt]{\lineheight{1.25}\smash{\begin{tabular}[t]{l}\tiny $g$\end{tabular}}}}%
    \put(0.79984372,0.09461686){\color[rgb]{0,0,0}\makebox(0,0)[lt]{\lineheight{1.25}\smash{\begin{tabular}[t]{l}\tiny $f$\end{tabular}}}}%
    \put(0.79984372,0.11961685){\color[rgb]{0,0,0}\makebox(0,0)[lt]{\lineheight{1.25}\smash{\begin{tabular}[t]{l}\tiny $e$\end{tabular}}}}%
    \put(0.8748437,0.11961685){\color[rgb]{0,0,0}\makebox(0,0)[lt]{\lineheight{1.25}\smash{\begin{tabular}[t]{l}\tiny $d$\end{tabular}}}}%
    \put(0.8748437,0.09461686){\color[rgb]{0,0,0}\makebox(0,0)[lt]{\lineheight{1.25}\smash{\begin{tabular}[t]{l}\tiny $c$\end{tabular}}}}%
    \put(0.66234376,0.09461686){\color[rgb]{0,0,0}\makebox(0,0)[lt]{\lineheight{1.25}\smash{\begin{tabular}[t]{l}\tiny $b$\end{tabular}}}}%
    \put(0.58734371,0.09461673){\color[rgb]{0,0,0}\makebox(0,0)[lt]{\lineheight{1.25}\smash{\begin{tabular}[t]{l}\tiny $a$\end{tabular}}}}%
    \put(0,0){\includegraphics[width=\unitlength,page=2]{Etiquetage_z.pdf}}%
    \put(0.49046874,0.06024198){\color[rgb]{0,0,0}\makebox(0,0)[lt]{\lineheight{0}\smash{\begin{tabular}[t]{l}$=$\end{tabular}}}}%
    \put(0,0){\includegraphics[width=\unitlength,page=3]{Etiquetage_z.pdf}}%
    \put(0.94984457,0.18211813){\color[rgb]{0,0,0}\makebox(0,0)[lt]{\lineheight{0}\smash{\begin{tabular}[t]{l}\scriptsize $m_{3}$\end{tabular}}}}%
    \put(0.79984458,0.18211813){\color[rgb]{0,0,0}\makebox(0,0)[lt]{\lineheight{0}\smash{\begin{tabular}[t]{l}\scriptsize $m_{2}$\end{tabular}}}}%
    \put(0.6498446,0.18211813){\color[rgb]{0,0,0}\makebox(0,0)[lt]{\lineheight{0}\smash{\begin{tabular}[t]{l}\scriptsize $m_{1}$\end{tabular}}}}%
    \put(0,0){\includegraphics[width=\unitlength,page=4]{Etiquetage_z.pdf}}%
    \put(0.48734469,0.241492){\color[rgb]{0,0,0}\makebox(0,0)[lt]{\lineheight{0}\smash{\begin{tabular}[t]{l}$\overset{\text{top.}}{=}$\end{tabular}}}}%
    \put(0,0){\includegraphics[width=\unitlength,page=5]{Etiquetage_z.pdf}}%
    \put(0.43734377,0.18211813){\color[rgb]{0,0,0}\makebox(0,0)[lt]{\lineheight{0}\smash{\begin{tabular}[t]{l}\scriptsize $m_{3}$\end{tabular}}}}%
    \put(0.2873446,0.18211813){\color[rgb]{0,0,0}\makebox(0,0)[lt]{\lineheight{0}\smash{\begin{tabular}[t]{l}\scriptsize $m_{2}$\end{tabular}}}}%
    \put(0.13734458,0.18211813){\color[rgb]{0,0,0}\makebox(0,0)[lt]{\lineheight{0}\smash{\begin{tabular}[t]{l}\scriptsize $m_{1}$\end{tabular}}}}%
  \end{picture}%
\endgroup%

%% file: Courbe_e_surface.pdf_tex
\begingroup%
  \makeatletter%
  \providecommand\color[2][]{%
    \errmessage{(Inkscape) Color is used for the text in Inkscape, but the package 'color.sty' is not loaded}%
    \renewcommand\color[2][]{}%
  }%
  \providecommand\transparent[1]{%
    \errmessage{(Inkscape) Transparency is used (non-zero) for the text in Inkscape, but the package 'transparent.sty' is not loaded}%
    \renewcommand\transparent[1]{}%
  }%
  \providecommand\rotatebox[2]{#2}%
  \newcommand*\fsize{\dimexpr\f@size pt\relax}%
  \newcommand*\lineheight[1]{\fontsize{\fsize}{#1\fsize}\selectfont}%
  \ifx\svgwidth\undefined%
    \setlength{\unitlength}{453.54313407bp}%
    \ifx\svgscale\undefined%
      \relax%
    \else%
      \setlength{\unitlength}{\unitlength * \real{\svgscale}}%
    \fi%
  \else%
    \setlength{\unitlength}{\svgwidth}%
  \fi%
  \global\let\svgwidth\undefined%
  \global\let\svgscale\undefined%
  \makeatother%
  \begin{picture}(1,0.17628758)%
    \lineheight{1}%
    \setlength\tabcolsep{0pt}%
    \put(0,0){\includegraphics[width=\unitlength,page=1]{Courbe_e_surface.pdf}}%
    \put(0.60689362,0.1225188){\color[rgb]{0,0,0}\makebox(0,0)[lt]{\lineheight{1.25}\smash{\begin{tabular}[t]{l}$\scriptstyle e$\end{tabular}}}}%
    \put(0,0){\includegraphics[width=\unitlength,page=2]{Courbe_e_surface.pdf}}%
  \end{picture}%
\endgroup%